\documentclass[11pt]{amsart}
\usepackage[foot]{amsaddr}
\usepackage[a4paper,margin=2.65cm]{geometry}
\usepackage{amsmath,amssymb,amsthm,mathtools,mathrsfs}
\mathtoolsset{showonlyrefs}
\usepackage[renew-dots,renew-matrix]{nicematrix}
\usepackage{enumitem}
\usepackage[hidelinks]{hyperref}
\usepackage{orcidlink}
\usepackage{microtype}

\usepackage[
backend=biber,
style=numeric,
sorting=nyt,
giveninits=true,
maxnames=99,
doi=true,
url=false,
isbn=false
]{biblatex}

\DeclareFieldFormat[article]{title}{\mkbibemph{#1}}

\DeclareFieldFormat[article]{journaltitle}{#1}

\DeclareFieldFormat[article]{volume}{\textbf{#1}}

\DeclareFieldFormat{pages}{#1}

\renewbibmacro{in:}{%
	\ifentrytype{article}
	{}
	{\printtext{\bibstring{in}\intitlepunct}}%
}

\renewbibmacro*{journal+issuetitle}{%
	\usebibmacro{journal}%
	\setunit*{\addspace}%
	\printfield{volume}%
	\setunit{\addspace}%
	\printtext[parens]{\printfield{year}}%
	\setunit{\addspace}%
	\iffieldundef{pages}
	{\printfield{eid}}
	{\printfield{pages}}%
	\newunit
}

\renewbibmacro*{note+pages}{%
	\ifentrytype{article}
	{}
	{\printfield{note}%
		\setunit{\bibpagespunct}%
		\printfield{pages}}%
}
\usepackage[T1]{fontenc}
\usepackage{textcomp}
\usepackage{newtxtext}
\usepackage{newtxmath}
\usepackage[bb=boondox,cal=boondoxo,scr=boondoxo]{mathalfa}

\newcommand{\dz}{\,\mathrm{d}z}
\newcommand{\one}{\mathbf{1}}
\newcommand{\pFq}[5]{\;{}_{#1}F_{#2}\left(\begin{matrix}#3\\#4\end{matrix};#5\right)}
\newcommand{\WB}{W}
\newcommand{\CB}{C}
\newcommand{\MB}{\mathscr M}
\newcommand{\LC}{\operatorname{LC}}
\newcommand{\Torus}{\mathbb T}
\newcommand{\N}{\mathbb N}
\newcommand{\Poly}{\mathbb P}
\renewcommand{\top}{\mathsf T}

\theoremstyle{plain}
\newtheorem{theorem}{Theorem}[section]
\newtheorem{proposition}[theorem]{Proposition}
\newtheorem{lemma}[theorem]{Lemma}
\newtheorem{corollary}[theorem]{Corollary}

\newtheoremstyle{definitionstyle}
{6pt}
{12pt}
{\normalfont}
{}
{\bfseries}
{.}
{0.5em}
{}

\theoremstyle{definitionstyle}
\newtheorem{definition}[theorem]{Definition}
\newtheorem{example}[theorem]{Example}

\newtheoremstyle{remarkstyle}
{6pt}
{12pt}
{\normalfont}
{}
{\itshape}
{.}
{0.5em}
{}

\theoremstyle{remarkstyle}
\newtheorem{remark}[theorem]{Remark}

\title[Bessel mixed-type multiple orthogonality]
{Bessel Multiple Orthogonal Polynomials of Mixed Type}
\author{Manuel Ma\~nas\,\orcidlink{0000-0003-3764-5737}}
\address{Department of Theoretical Physics, Faculty of Physical Sciences,
	Complutense University of Madrid, 28040 Madrid, Spain}
\email{manuel.manas@ucm.es}
\date{August 20, 2026}

\hypersetup{
	pdftitle={Bessel Multiple Orthogonal Polynomials of Mixed Type},
	pdfauthor={Manuel Manas}
}

\begin{document}

\begin{abstract}
	This article constructs a family of Bessel multiple orthogonal
	polynomials of mixed type for a \(q\times p\) matrix weight on the unit circle.
	The weight is not of rank-one product form; for generic regular parameters
	it has maximal rank
	\(\min\{q,p\}\) outside a finite subset of the circle. Its reciprocal-Gamma
	moments recover the multiple Bessel system when \(q=1\) and the Bessel-like
	system of Wolfs when \(p=1\). The same matrix weight arises as a scaled
	Markov--Stieltjes limit of a rank-one Jacobi system of mixed type, although the
	interval measures themselves have no finite limit.

	For balanced index pairs with a near-diagonal row multi-index, explicit
	formulas are obtained for the mixed \(A\)- and \(B\)-polynomial vectors;
	the column multi-index is otherwise unrestricted, subject to the stated
	parameter admissibility conditions. Their orthogonality and weak
	normality are proved, and componentwise strong normality is characterized.
	Their components have terminating generalized
	hypergeometric representations; the
	\(B\)-components also have finite Kamp\'e de F\'eriet representations and a
	matrix Rodrigues-type formula. In the one-row reduction, the sole Kamp\'e de
	F\'eriet block, evaluated at \((-z,1)\), reduces to a generalized
	hypergeometric polynomial governed by
	a reflected type-II multiple Hahn polynomial.

	Finite Gamma--Pochhammer formulas give the near-diagonal and step-line
	recurrence coefficients. The corresponding banded recurrence matrix has a
	bidiagonal Christoffel factorization. The lower factors are evaluated from
	transformed polynomial vectors, and the upper factors are given by finite
	tau-determinants. When \(q=1\), every Christoffel step stays
	within the multiple Bessel family and the complete factorization follows
	from Gamma--Vandermonde determinants. Thus the matrix weight, polynomial
	vectors, normality conditions, recurrence coefficients, and bidiagonal
	factors are all given by explicit formulas.
\end{abstract}



\maketitle

\tableofcontents

\section{Introduction}
\label{sec:introduction}

Multiple orthogonal polynomials arose from simultaneous rational
approximation and Hermite--Pad\'e approximation. They extend classical
orthogonality from one measure or moment functional to a system of measures;
standard references include the monograph of Nikishin and Sorokin, Ismail's
book, and the account of Mart\'inez-Finkelshtein and Van Assche
\cite{NikishinSorokin1991,Ismail2005,MartinezVanAssche2016}. Their scope now
extends well beyond this original setting. They occur in Diophantine
approximation and in
irrationality constructions in the line of Ap\'ery's proof for
\(\zeta(2)\) and \(\zeta(3)\), in models of non-intersecting Brownian motions,
and in integrable systems governed by multicomponent Toda hierarchies; see,
among others,
\cite{Apery1979,NikishinSorokin1991,DaemsKuijlaars2007,AlvarezFidalgoManas2011}.
At the algebraic level, the corresponding systems of measures give rise to
structured moment matrices and higher-order recurrence relations. Mixed-type
multiple orthogonal polynomials also play a fundamental role in the spectral
theory of banded matrices with positive bidiagonal factorizations, where they
enter naturally in the corresponding Favard-type spectral representations
\cite{BranquinhoFoulquieManas2023Spectral,BranquinhoFoulquieManas2026Unbounded}.
The relation between banded total positivity and step-line normality for
mixed-type multiple orthogonality is developed in
\cite{BranquinhoFoulquieManas2026BTP}.
This gives a broader spectral motivation for the recurrence and factorization
questions considered here; the Bessel factorization constructed below is
algebraic, and no positivity is asserted.
Mixed-type multiple orthogonality is closely related to matrix Hermite--Pad\'e
approximation and matrix orthogonality; see the works of Sorokin and Van
Iseghem
\cite{SorokinVanIseghem1997,SorokinVanIseghem1999,SorokinVanIseghem2000}, the
mixed-type Nikishin construction of Fidalgo, L\'opez-Garc\'ia,
L\'opez-Lagomasino and Sorokin \cite{FidalgoLopezLopezSorokin2010}, and the
framework of Daems and Kuijlaars \cite{DaemsKuijlaars2007}. A formulation
through Gauss--Borel factorization and the multicomponent Toda hierarchy was
developed in \cite{AlvarezFidalgoManas2011}. On the step-line, the two
polynomial vectors form dual biorthogonal sequences governed by a banded
recurrence matrix.

Explicit formulas are well established in the one-sided type-I and type-II
settings. Van Assche and Coussement described several classical continuous
multiple families, while Arves\'u, Coussement, and Van Assche developed the
corresponding classical discrete families and obtained explicit type-II
systems \cite{VanAsscheCoussement2001,ArvesuCoussementVanAssche2003}.
Beckermann, Coussement, and Van Assche subsequently constructed the multiple
Wilson and Jacobi--Pi\~neiro families
\cite{BeckermannCoussementVanAssche2005}. More recently, explicit
hypergeometric representations for type-I multiple Hahn polynomials were
derived in \cite{BranquinhoDiazFoulquieManas2023Hahn}, and integral and
hypergeometric representations for both types and an arbitrary number of
weights were obtained in
\cite{BranquinhoDiazFoulquieManasWolfs2025}. Explicit mixed-type systems, in
which two multi-indices act on opposite sides of a matrix of measures, are
much less common. Besides the Gaussian system of Daems and Kuijlaars already
mentioned, representative explicit examples include Zhang's restricted
\(2\times2\) construction associated with modified Bessel
functions~\cite{Zhang2016BesselMixed}, the exponential-integral system of Van
Assche and Wolfs~\cite{VanAsscheWolfs2023}, the mixed-type Pi\~neiro system
of Branquinho, D\'iaz, Foulqui\'e-Moreno, and
Ma\~nas~\cite{PineiroMixed2026}, and the Jacobi and Laguerre systems of mixed
type introduced in \cite{JacobiLaguerreMixed}. All these constructions have
separable, hence rank-one, matrices of
weights. Such a matrix does not capture a full interaction between several
row and column components. Full-rank examples for which the polynomial
vectors, normality conditions, recurrence coefficients, and factorizations
of the associated recurrence matrices can all be determined explicitly are
therefore especially valuable.

Bessel polynomials provide a natural setting for such a construction. They
combine an explicit hypergeometric structure with a nonstandard contour
orthogonality. Their extensions to multiple
orthogonality include the multiple Bessel system of Aptekarev, Branquinho and
Van Assche \cite{ABV2003} and the Bessel-like system studied by Wolfs
\cite{Wolfs2024}. In the orientation used here, these appear as the one-row
case \(q=1\) and the one-column case \(p=1\), respectively. The full
\(q\times p\) family is called the Bessel system of mixed type; the
suffix ``-like'' is reserved for its one-column reduction \(p=1\). The problem
addressed here is to pass simultaneously beyond both one-sided boundaries:
to construct a Bessel matrix weight of maximal rank, rather than a
separable rank-one product, while retaining the explicit structure of the
cases \(q=1\) and \(p=1\).

Matrix Bessel biorthogonal polynomials associated with matrix Pearson
equations have also been studied through a Riemann--Hilbert problem,
leading to differential relations and nonlinear equations for the
recurrence coefficients \cite{MatrixBessel2026}.

The construction in this article uses reciprocal-Gamma moments with
noninteger displacements in the column parameters. For generic regular parameters, the resulting
\(q\times p\) matrix has maximal rank \(\min\{q,p\}\) away from finitely many
points of the unit circle. At the same time, it recovers the two Bessel
systems just mentioned when \(q=1\) or \(p=1\). It is the reciprocal-Gamma
counterpart of the Jacobi and Laguerre systems of mixed type in
\cite{JacobiLaguerreMixed}. It is also obtained by confluence from a rank-one
system: after one passes from a Jacobi interval system of mixed type to its
Markov--Stieltjes contour
representative and rescales the spectral variable, the matrix converges
locally uniformly to the present Bessel weight. The interval measures
themselves have no finite measure limit, whereas the complex contour
representative has a well-defined limit.

The choice of contour representation is necessary in the Bessel case.
Classical generalized Bessel polynomials admit several
orthogonality realizations. Krall and Frink used a single-valued Laurent
weight on the unit circle \cite{KrallFrink1949}; Burchnall obtained an
elementary closed-contour formula in the integral-exponent case
\cite{Burchnall1951}; and Exton treated nonintegral exponents through a
cut-contour formulation \cite{Exton1986}. For nonintegral powers, an
expression integrated around the circle must be accompanied by a choice of
branch and by the corresponding
boundary values across the cut. This point is essential when comparing the
present moments with the multiple Bessel formulas in \cite{ABV2003}. After an
Exton-type interpretation, an explicit normalization, and a rescaling of the
spectral variable, those functionals give exactly the reciprocal-Gamma
moments of the one-row reduction.

The contour notation should not be confused with the standard sesquilinear
theory of orthogonal polynomials on the unit circle
\cite{MinguezVanAssche2008}. The pairing in this article is bilinear and has
no complex conjugation. Thus
\(\langle zf,g\rangle=\langle f,zg\rangle,\)
whereas
\(\langle zf,g\rangle_{\mathrm{sesq}}
=\langle f,z^{-1}g\rangle_{\mathrm{sesq}}.\)
The circle is used as a Laurent-coefficient contour realization of the
reciprocal-Gamma bimoments, not as the support of a positive sesquilinear
measure. The bilinear structure leads on the step-line to dual recurrences
governed by transposed band matrices rather than to a Szeg\H{o} or CMV
structure.

Beyond constructing the weight, the article gives an explicit analysis of
the associated mixed system. For balanced index pairs whose row multi-index
is near the diagonal, the \(A\)- and \(B\)-polynomial vectors are evaluated
by terminating generalized hypergeometric formulas. The column multi-index
need not be near the diagonal, provided the stated regularity and
admissibility conditions on the parameters hold. Their mixed orthogonality
and weak normality are
proved, and strong normality is characterized component by component,
including the exceptional loci on which a \(B\)-component loses degree. The
\(B\)-components also have finite Kamp\'e de F\'eriet representations and a
matrix Rodrigues-type formula. When \(q=1\), the sole Kamp\'e de F\'eriet
block, evaluated at \((-z,1)\), reduces to a single generalized
hypergeometric polynomial, and its
unit-difference parameters are determined by a reflected type-II multiple
Hahn polynomial.

The recurrence coefficients can also be computed explicitly. A finite
Gamma--Pochhammer pairing formula gives the entries of the canonical
\((p,q)\)-banded step-line recurrence matrix and the local \(A\)- and
\(B\)-recurrences of length \(p+q+1\). Applying the mixed Christoffel
factorization of \cite{BranquinhoFoulquieManas2026}, the successive column
Christoffel transformations stay within the Bessel family and give the
lower factors explicitly. For
\(q>1\), the upper factors are finite \(\tau^B\)-determinants. For \(q=1\),
both the column and row Christoffel transformations stay within the multiple
Bessel family, and
Gamma--Vandermonde determinants give the complete bidiagonal factorization.
Thus the weight, polynomial vectors, normality conditions, recurrence
coefficients, and bidiagonal factors are all given by explicit formulas.

The paper is organized as follows. Section~\ref{sec:mixed-framework}
introduces the general mixed-type framework, the moment matrix, the step-line
recurrence, and the notation used throughout. Section~\ref{sec:bessel-system}
defines the full-rank Bessel matrix, computes its reciprocal-Gamma
moments, proves maximal rank, and establishes the Jacobi confluence.
Section~\ref{sec:final-hypergeometric-bessel-forms} derives the explicit
hypergeometric polynomial vectors, their mixed orthogonality, and the weak
and strong normality results. The Rodrigues-type representation is obtained
in Section~\ref{sec:matrix-differential-B}, and the one-column and one-row
reductions are identified in Section~\ref{sec:final-reductions}.
Sections~\ref{sec:final-recurrences} and
\ref{sec:final-bidiagonal-factorizations} give, respectively, the local and
step-line recurrences and the bidiagonal factorizations, including the first
mixed-type row case \(q=2\) and the numerical example with \(p=3\), \(q=2\).
The final section summarizes the results and records the remaining problems.

\section{Mixed-type framework and notation}
\label{sec:mixed-framework}

This section fixes the general framework and notation used throughout the
paper. It first introduces the two polynomial vectors associated with a
rectangular matrix of contour measures, together with the corresponding
notions of normality. The moment matrix, its Gauss--Borel factorization, and
the induced step-line recurrence are then recalled before the notation needed
for the Bessel construction is collected. The specific full-rank matrix
of weights and its basic analytic properties are developed in the following
section.

	\subsection{Mixed-type multiple orthogonality for a matrix of measures}
	\label{sec:general-mixed}

	Let
	\[
	\mathrm d\boldsymbol\mu(z)
	=
	W(z)\frac{\dz}{2\pi\mathrm i},
	\qquad
	W(z)
	=
	\begin{bmatrix}
		w_{1,1}(z)&w_{1,2}(z)&\cdots&w_{1,p}(z)\\
		w_{2,1}(z)&w_{2,2}(z)&\cdots&w_{2,p}(z)\\
		\vdots&\vdots&\ddots&\vdots\\
		w_{q,1}(z)&w_{q,2}(z)&\cdots&w_{q,p}(z)
	\end{bmatrix},
	\qquad z\in\Torus.
\]
	be a \(q\times p\) matrix of contour measures.  Write
\(\Poly_N\coloneq\{P\in\mathbb C[z]:\deg P\le N\},\) \(N\in\N_0,\)
	and set \(\Poly_{-1}\coloneq\{0\}\).  For a row polynomial vector
	\(\mathbf B\in\mathbb C^{1\times q}[z]\) and a column polynomial vector
	\(\mathbf A\in\mathbb C^{p\times1}[z]\), set
	\begin{equation}
		\label{eq:general-pairing}
		\left\langle \mathbf B,\mathbf A\right\rangle
		\coloneq
		\oint_{\Torus}
		\mathbf B(z)W(z)\mathbf A(z)\frac{\dz}{2\pi\mathrm i}.
	\end{equation}
	
	Let
	\[
	\boldsymbol n=\begin{bNiceMatrix}n_1&\Cdots&n_p\end{bNiceMatrix}\in\N_0^p,
	\qquad
	\boldsymbol m=\begin{bNiceMatrix}m_1&\Cdots&m_q\end{bNiceMatrix}\in\N_0^q.
	\]
	If \(|\boldsymbol n|=|\boldsymbol m|+1\), the mixed \(A\)-polynomial vector, expanded on the column side, is the
	column vector
	\[
	\mathbf A_{\boldsymbol n,\boldsymbol m}(z)
	=
	\begin{bNiceMatrix}
		A_{\boldsymbol n,\boldsymbol m}^{(1)}(z)\\
		\Vdots\\
		A_{\boldsymbol n,\boldsymbol m}^{(p)}(z)
	\end{bNiceMatrix},\qquad 	A_{\boldsymbol n,\boldsymbol m}^{(i)}\in\Poly_{n_i-1},
	\qquad i\in\{1,\ldots,p\}.
		\]
		It satisfies
	\begin{equation}
		\label{eq:general-A-orth}
		\sum_{i=1}^{p}
		\oint_{\Torus}
		z^\ell w_{j,i}(z)A_{\boldsymbol n,\boldsymbol m}^{(i)}(z)
		\frac{\dz}{2\pi\mathrm i}
		=0,
		\qquad
		\ell\in\{0,\ldots,m_j-1\},
		\quad j\in\{1,\ldots,q\}.
	\end{equation}
	If \(|\boldsymbol m|=|\boldsymbol n|+1\), the mixed \(B\)-polynomial vector, expanded on the row side, is the
	row vector
	\[
	\mathbf B_{\boldsymbol n,\boldsymbol m}(z)
	=
	\begin{bNiceMatrix}
		B_{\boldsymbol n,\boldsymbol m}^{(1)}(z)&\Cdots&B_{\boldsymbol n,\boldsymbol m}^{(q)}(z)
	\end{bNiceMatrix},\qquad 	B_{\boldsymbol n,\boldsymbol m}^{(j)}\in\Poly_{m_j-1},
	\qquad j\in\{1,\ldots,q\}.
		\]
		It satisfies
	\begin{equation}
		\label{eq:general-B-orth}
		\sum_{j=1}^{q}
		\oint_{\Torus}
		z^\ell B_{\boldsymbol n,\boldsymbol m}^{(j)}(z)w_{j,i}(z)
		\frac{\dz}{2\pi\mathrm i}
		=0,
		\qquad
		\ell\in\{0,\ldots,n_i-1\},
		\quad i\in\{1,\ldots,p\}.
	\end{equation}
	When \(p=1\), the \(A\)-polynomial vector has one component and contains the usual scalar type-II multiple orthogonal polynomial for the \(q\) measures in the single column.  When \(q=1\), the \(B\)-polynomial vector is the scalar type-II object of the dual multiple Bessel case.
\begin{definition}[Balanced index pairs and normality]
	\label{def:balanced-normality}
	Let \(\boldsymbol n\in\N_0^p\) and \(\boldsymbol m\in\N_0^q\).  The index pair
	\((\boldsymbol n,\boldsymbol m)\) is called \(A\)-balanced when
\(|\boldsymbol n|=|\boldsymbol m|+1,\)
	and it is called \(B\)-balanced when
\(|\boldsymbol m|=|\boldsymbol n|+1.\)
	When the side is clear from the context, the pair is simply said to be
	balanced.
	
	An \(A\)-balanced index pair is called weakly normal for the
	\(A\)-polynomial problem if the corresponding \(A\)-polynomial vector exists
	and is uniquely determined up to multiplication by a nonzero constant.
	It is called strongly normal if, in addition, every component with a nontrivial
	prescribed polynomial space is nonzero and attains its prescribed maximal
	degree, that is,
\(\deg A_{\boldsymbol n,\boldsymbol m}^{(i)}=n_i-1,\) \(i\in\{1,\ldots,p\},\) \(n_i\ge1.\)
	Similarly, a \(B\)-balanced index pair is called weakly normal for the
	\(B\)-polynomial problem if the corresponding \(B\)-polynomial vector exists
	and is uniquely determined up to multiplication by a nonzero constant.
	It is called strongly normal if, in addition, every component with a nontrivial
	prescribed polynomial space is nonzero and attains its prescribed maximal
	degree, that is,
\(\deg B_{\boldsymbol n,\boldsymbol m}^{(j)}=m_j-1,\) \(j\in\{1,\ldots,q\},\) \(m_j\ge1.\)
\end{definition}

\begin{remark}
	In much of the literature on multiple orthogonality, normality refers to
	existence and uniqueness, up to the natural normalization, of the solution
	associated with a prescribed multi-index.  This property is called weak
	normality here, in order to distinguish it from the stronger componentwise
	degree condition.  Thus strong normality implies weak normality, but not
	conversely in general.
\end{remark}
	
	\subsection{Moment matrix and the step-line}
	
Introduce the block monomial vector
\[
X_{[d]}(z)
\coloneq
\begin{bNiceMatrix}
	I_d\\ zI_d\\ z^2I_d\\ \Vdots
\end{bNiceMatrix},
\]
and let \(\Lambda_{[d]}\) be the block shift characterized by
\(\Lambda_{[d]}X_{[d]}(z)=zX_{[d]}(z).\)
Equivalently,
\[
\Lambda_{[d]}
=
\begin{bNiceMatrix}[margin=5pt]
	0&I_d&0&0&\Cdots\\
	0&0&I_d&0&\Cdots\\
	0&0&0&I_d&\Ddots\\
	\Vdots[shorten-end=-10pt]&\Vdots[shorten-end=-10pt]&\Vdots[shorten-end=-10pt]&\Ddots[shorten-end=-15pt]&\Ddots\\
	\phantom{AA} &\phantom{AA} &\phantom{AA} &\phantom{AA} &\phantom{AA} 
\end{bNiceMatrix}.
\]
The moment matrix is
\begin{equation}
	\label{eq:general-moment-matrix}
	\mathscr M
	\coloneq
	\oint_{\Torus}
	X_{[q]}(z)W(z)X_{[p]}^{\top}(z)
	\frac{\dz}{2\pi\mathrm i}.
\end{equation}
Since the scalar powers on both sides enter through their sum, it satisfies
\begin{equation}
	\label{eq:general-hankel}
	\Lambda_{[q]}\mathscr M
	=
	\mathscr M\Lambda_{[p]}^{\top}.
\end{equation}
For \(d\in\mathbb N\), a multi-index \(\boldsymbol\nu\in\mathbb N_0^d\)
is called near-diagonal if
\(|\nu_i-\nu_j|\le 1,\) \(i,j\in\{1,\ldots,d\}.\)
For each \(N\in\mathbb N_0\), write \(N=da+r\), with
\(a\in\mathbb N_0\) and \(r\in\{0,\ldots,d-1\}\). Denote by
\(\boldsymbol\sigma_d(N)\) the standard near-diagonal multi-index of norm \(N\),
defined by
\[
\boldsymbol\sigma_d(N):=
a\one_d+\sum_{\alpha=1}^r\boldsymbol e_\alpha
=
\begin{bNiceMatrix}
	a+1&
	\Cdots&
	a+1&
	a&
	\Cdots&
	a
\end{bNiceMatrix}^\top,
\]
where the first block has length \(r\) and the second block has length
\(d-r\); the empty sum is understood as zero. Thus
\(\boldsymbol\sigma_d(N)\) has \(r\) components equal to \(a+1\) and \(d-r\)
components equal to \(a\), the larger components being placed first in the
standard cyclic order.
At level \(N\), the step-line selects a unique column
\(c_N\in\{1,\ldots,p\}\) and a unique row
\(r_N\in\{1,\ldots,q\}\), determined by
\begin{equation}
	\label{eq:general-active-indices}
	\boldsymbol\sigma_p(N+1)-\boldsymbol\sigma_p(N)=\boldsymbol e_{c_N},
	\qquad
	\boldsymbol\sigma_q(N+1)-\boldsymbol\sigma_q(N)=\boldsymbol e_{r_N}.
\end{equation}
Thus \(r_N\) is the row component whose degree increases on the
\(B\)-side at level \(N\), while \(c_N\) is the column component whose
degree increases on the dual \(A\)-side.

Assuming that the leading principal truncations of \(\mathscr M\) are
nonsingular, write its Gauss--Borel factorization in the form
\begin{equation}
	\label{eq:general-GB-factorization}
	\mathscr M=S^{-1}H\bar S^{-\top},
\end{equation}
where \(S\) and \(\bar S\) are lower unitriangular and \(H\) is diagonal.
This factorization fixes the step-line normalization: the
\(B\)-vectors are taken from the lower unitriangular factor and are monic in
the row \(r_N\) selected by the step-line, while the diagonal factor is
absorbed into the dual \(A\)-vectors.  In the explicit near-diagonal formulas below, the
remaining scalar freedom is used to choose the constants compatibly with this
convention.  Consequently, when a near-diagonal formula is specialized to a
step-line pair, it gives the canonically normalized step-line vector
directly, with no additional Gauss--Borel rescaling.
Define
\begin{equation}
	\label{eq:general-step-vectors}
	\mathbf B_N
	\coloneq
	\mathbf B_{\boldsymbol\sigma_p(N),\boldsymbol\sigma_q(N+1)},
	\qquad
	\mathbf A_N
	\coloneq
	\mathbf A_{\boldsymbol\sigma_p(N+1),\boldsymbol\sigma_q(N)}.
\end{equation}
By the compatible normalization just described, these are already the
Gauss--Borel step-line vectors.  The \(B\)-vector at level \(N\) is monic in
the row \(r_N\) defined in \eqref{eq:general-active-indices}, whereas the
\(A\)-vectors are the dual normalized ones.  They satisfy
\begin{equation}
	\label{eq:general-biorthogonality}
	\left\langle\mathbf B_N,\mathbf A_M\right\rangle
	=\delta_{N,M},
	\qquad N,M\in\N_0.
\end{equation}
Thus the step-line is a specialization of the near-diagonal formulas, not a
separate family obtained by a later renormalization.
	\begin{proposition}[Step-line recurrence matrix]
		\label{prop:general-step-recurrence}
		Under the preceding Gauss--Borel assumption, there exists a \((p,q)\)-banded
		matrix \(T\), in the sense of \(p\) subdiagonals and \(q\) superdiagonals,
		such that, for every \(N\in\N_0\),
		\begin{equation}
			\label{eq:general-step-recurrence}
			z\mathbf B_N(z)
			=
			\sum_{\substack{k=-p\\N+k\ge0}}^{q}
			t_{N,k}\mathbf B_{N+k}(z),
			\qquad
			t_{N,k}
			=
			\left\langle z\mathbf B_N,\mathbf A_{N+k}\right\rangle.
		\end{equation}
		Dually,
		\[
		z\mathbf A_N(z)
		=
		\sum_{\substack{k=-q\\N+k\ge0}}^{p}
			t_{N+k,-k}\mathbf A_{N+k}(z).
		\]
		The recurrence matrix is represented by
\(T=S\Lambda_{[q]}S^{-1} =H\bar S^{-\top}\Lambda_{[p]}^{\top}\bar S^{\top}H^{-1}.\)
	\end{proposition}
	
	\begin{proof}
		Equation~\eqref{eq:general-hankel}, dressed by the Gauss--Borel factors,
		gives the two matrix representations of \(T\).  The bandwidth follows from
		the block shifts.  Pairing the \(B\)-side recurrence with
		\(\mathbf A_{N+k}\) and using \eqref{eq:general-biorthogonality} gives
\(t_{N,k} = \left\langle z\mathbf B_N,\mathbf A_{N+k}\right\rangle.\)
		Since multiplication by \(z\) on the \(A\)-side is represented by
		\(T^{\top}\), its \(N\)-th component is
		\[
		z\mathbf A_N(z)
		=
		\sum_{\substack{k=-q\\N+k\ge0}}^{p}
		T_{N+k,N}\mathbf A_{N+k}(z)
		=
		\sum_{\substack{k=-q\\N+k\ge0}}^{p}
		t_{N+k,-k}\mathbf A_{N+k}(z).
		\]
	\end{proof}
	
	The Bessel system below supplies explicit hypergeometric polynomial
	vectors for a balanced class of general index pairs.  Its general
	recurrence coefficients and the specialization of
	\eqref{eq:general-step-recurrence} are evaluated explicitly in
	Section~\ref{sec:final-general-recurrence}.
	
\subsection{Notation}
For \(d\in\mathbb N\), denote by
\[
\one_d\coloneq\begin{bNiceMatrix}1&\Cdots&1\end{bNiceMatrix}\in\mathbb C^d
\]
the vector whose \(d\) entries are equal to one. Use the notation
\[
\Gamma(s\one_r+\boldsymbol a)\coloneq\prod_{\rho=1}^{r}\Gamma(s+a_\rho),
\qquad
\Gamma(s\one_q+\boldsymbol b)\coloneq\prod_{i=1}^{q}\Gamma(s+b_i),
\]
with the empty product equal to one. In particular,
\[
\Gamma(s\one_r+\boldsymbol a+\one_r)
=
\prod_{\rho=1}^{r}\Gamma(s+a_\rho+1),
\qquad
\Gamma(s\one_q+\boldsymbol b+\one_q)
=
\prod_{i=1}^{q}\Gamma(s+b_i+1).
\]
For vectors \(\boldsymbol c\) and \(\boldsymbol m\) of the same length, with
integer components in \(\boldsymbol m\), write
\((\boldsymbol c)_{\boldsymbol m}\coloneq\prod_h(c_h)_{m_h}.\)
If the subscript is a scalar, it is understood componentwise; thus
\((\boldsymbol c)_k=(\boldsymbol c)_{k\one}\).

Use the Euler differential operator
\[
\vartheta\coloneq z\frac{\mathrm d}{\mathrm dz}.
\]
For a polynomial \(F\in\mathbb C[z]\), let
\([z^k]F(z)\) denote the coefficient of \(z^k\) in \(F(z)\). Also let
\[
\Delta F(t)\coloneq F(t+1)-F(t)
\]
denote the forward-difference operator, and let \(\Delta^m\) denote its
\(m\)-fold iteration.

For vectors of the same length, inequalities are understood componentwise;
that is,
\(\boldsymbol a\le \boldsymbol b\) \(\Longleftrightarrow\) \(a_i\le b_i\) \(\text{for every }i.\)
The generalized hypergeometric function is denoted by
\begin{equation}
	\label{eq:generalized-hypergeometric-definition}
	\pFq{p}{q}
	{\boldsymbol a}
	{\boldsymbol b}
	{z}
	\coloneq
	\sum_{k=0}^{\infty}
	\frac{(\boldsymbol a)_k}{(\boldsymbol b)_k}
	\frac{z^k}{k!},
\end{equation}
where the series is understood in its domain of convergence, or by analytic
continuation when appropriate.  The hypergeometric series occurring as
polynomial components below are terminating series.

The standard bivariate Kamp\'e de F\'eriet series~\cite{SrivastavaKarlsson1985}
will also be used. For finite parameter strings
\(\boldsymbol A,\boldsymbol B,\boldsymbol C,\boldsymbol D,\boldsymbol E,\boldsymbol F\), write
\begin{equation}
	\label{eq:standard-KdF-definition}
	F_{\lvert\boldsymbol D\rvert:\lvert\boldsymbol E\rvert;\lvert\boldsymbol F\rvert}
	^{\lvert\boldsymbol A\rvert:\lvert\boldsymbol B\rvert;\lvert\boldsymbol C\rvert}
	\left[
	\begin{array}{c}
		\boldsymbol A:\boldsymbol B;\boldsymbol C\\
		\boldsymbol D:\boldsymbol E;\boldsymbol F
	\end{array}
	\middle|x,y
	\right]
	\coloneq
	\sum_{u=0}^{\infty}\sum_{\lambda=0}^{\infty}
	\frac{
		(\boldsymbol A)_{u+\lambda}
		(\boldsymbol B)_u
		(\boldsymbol C)_\lambda
	}{
		(\boldsymbol D)_{u+\lambda}
		(\boldsymbol E)_u
		(\boldsymbol F)_\lambda
	}
	\frac{x^u}{u!}\frac{y^\lambda}{\lambda!}.
\end{equation}
Empty parameter strings are interpreted as contributing the factor one.
The series is understood formally whenever convergence is not under
consideration; all instances used below are terminating. The generalized
hypergeometric series is recovered by setting \(y=0\) and suppressing the
corresponding parameter blocks.
	
	\section{The full-rank matrix of weights}
	\label{sec:bessel-system}
	
Let
\[
0\le r<q,
\qquad
\boldsymbol a=\begin{bNiceMatrix}a_1&\Cdots&a_r\end{bNiceMatrix},
\qquad
\boldsymbol b=\begin{bNiceMatrix}b_1&\Cdots&b_q\end{bNiceMatrix},
\]
and let
\[
\boldsymbol\kappa=\begin{bNiceMatrix}\kappa_1&\Cdots&\kappa_p\end{bNiceMatrix}\in\mathbb C^p.
\]

Following Wolfs~\cite{Wolfs2024}, call
\[
f_0(z;\boldsymbol a,\boldsymbol b)
\coloneq
\sum_{k=0}^{\infty}
\frac{\Gamma(k\one_r+\boldsymbol a+\one_r)}
{\Gamma(k\one_q+\boldsymbol b+\one_q)}z^k
=
\frac{\Gamma(\boldsymbol a+\one_r)}
{\Gamma(\boldsymbol b+\one_q)}
\pFq{r+1}{q}
{1,\boldsymbol a+\one_r}
{\boldsymbol b+\one_q}
{z}
\]
the base hypergeometric moment-generating function. In the Bessel regime
\(r<q\), it is an entire function of \(z\).
\begin{definition}[Regular Bessel parameters]
	\label{def:final-parameter-regularity}
	The Bessel parameters \((\boldsymbol a,\boldsymbol b,\boldsymbol\kappa)\) are called
	regular if
	\[
	0\le r<q,
	\qquad
	b_j-b_h\notin\mathbb Z
	\quad (j\ne h),
	\qquad
	\kappa_i-\kappa_h\notin\mathbb Z
	\quad (i\ne h).
	\]
\end{definition}

It is useful to distinguish regularity from the absence of poles in a particular formula.
Given a displayed expression \(\mathcal F\), a regular parameter tuple is
called \(\mathcal F\)-admissible if every Gamma factor occurring explicitly
in \(\mathcal F\), and both Gamma factors defining each
parameter-dependent Pochhammer quotient in \(\mathcal F\), are finite. For a
result involving several expressions, admissibility means simultaneous
admissibility for all of them. Unless an extension by analytic continuation
is stated explicitly, every formula-based assertion below is made on this
corresponding admissible subset of the regular parameter domain.

Thus parameter-dependent Pochhammer symbols are interpreted as Gamma
quotients and are nonzero on the stated admissible domain. In contrast,
Pochhammer symbols with integer arguments independent of the parameters,
such as \((-n_i+1)_k\) or \((-\ell)_{n_i}\), retain their usual finite-product
meaning; their zeros are structural and produce the terminating sums.
Admissibility concerns the parameters only, whereas balancedness concerns
only the sizes of the two multi-indices.

The following definition is modeled on the Bessel-like weights of
Wolfs~\cite{Wolfs2024}.  The multiple Bessel case will be described
later at the level of its moment functionals.  The standard mixed-type
orientation is retained: rows are indexed by \(j\in\{1,\ldots,q\}\), columns by
\(i\in\{1,\ldots,p\}\), so the matrix of weights is \(q\times p\).

\begin{definition}[Bessel matrix of weights of mixed type]
	For each row index  $j\in\{1,\ldots,q\}$ and column index
	$i\in\{1,\ldots,p\}$, define
	\begin{equation}
		\label{eq:final-weight}
		w_{j,i}(z)
		\coloneq
		f_0\left(
		z;
		\boldsymbol a+\kappa_i\one_r,
		\boldsymbol b+\kappa_i\one_q+\boldsymbol e_j
		\right),
		\qquad z\in\mathbb C,
	\end{equation}
	where $\boldsymbol e_j\in\mathbb C^q$ is the $j$-th standard basis vector.  The
	Bessel matrix of weights of mixed type is the $q\times p$ matrix
	\begin{equation}
		\label{eq:final-matrix-weight-direct}
		\WB(z)
		\coloneq
		\begin{bmatrix}
			w_{1,1}(z)&w_{1,2}(z)&\cdots&w_{1,p}(z)\\
			w_{2,1}(z)&w_{2,2}(z)&\cdots&w_{2,p}(z)\\
			\vdots&\vdots&\ddots&\vdots\\
			w_{q,1}(z)&w_{q,2}(z)&\cdots&w_{q,p}(z)
		\end{bmatrix}.
	\end{equation}
	The matrix \(\WB\) is the direct coefficient representative on the unit
	circle.  Its Laurent coefficient moments are extracted with the contour form
	\begin{equation}
		\label{eq:final-direct-contour-form}
		\WB(z)\frac{\dz}{2\pi\mathrm i},
		\qquad z\in\Torus .
	\end{equation}
\end{definition}

The reciprocal representative introduced below contains the additional
factor \(z^{-1}\); it is that representative which will be used in the
mixed Gauss--Borel pairing.

\subsection{Analyticity and maximal rank}

The matrix representatives \(\WB\) and \(\CB\) are studied next. First,
the analyticity of \(\WB\) is established; the reciprocal relation between
\(\WB\) and \(\CB\) then gives the analyticity of \(\CB\). Finally, both
representatives are shown to have maximal rank on \(\Torus\) outside a
finite set.

\begin{proposition}[Analyticity and moments of the direct weights]
\label{prop:final-moments}
	Assume that the Bessel parameters are regular in the
	sense of Definition~\ref{def:final-parameter-regularity}.  Then, for
	every $j\in\{1,\ldots,q\}$ and $i\in\{1,\ldots,p\}$,
	\begin{equation}
		\label{eq:final-direct-weight-series}
		w_{j,i}(z)
		=
		\sum_{k=0}^{\infty}
		\frac{
			\Gamma((k+\kappa_i)\one_r+\boldsymbol a+\one_r)
		}{
			\Gamma((k+\kappa_i)\one_q+\boldsymbol b+\boldsymbol e_j+\one_q)
		}
		z^k.
	\end{equation}
	In particular, each entry of \(\WB\) is an entire function of \(z\).
	Moreover, for $s\in\N$,
	\begin{equation}
		\label{eq:final-moments}
		\oint_{\Torus}
		z^{-s}w_{j,i}(z)
		\frac{\dz}{2\pi\mathrm i}
		=
		\frac{
			\Gamma((s+\kappa_i)\one_r+\boldsymbol a)
		}{
			\Gamma((s+\kappa_i)\one_q+\boldsymbol b+\boldsymbol e_j)
		}.
	\end{equation}
\end{proposition}

\begin{proof}
	Formula~\eqref{eq:final-direct-weight-series} is just the defining series
	of \(f_0\) with the shifted parameters appearing in
	\eqref{eq:final-weight}.  Let
\(c_{k;j,i} \coloneq \frac{ \Gamma((k+\kappa_i)\one_r+\boldsymbol a+\one_r) }{ \Gamma((k+\kappa_i)\one_q+\boldsymbol b+\boldsymbol e_j+\one_q) }.\)
	Using \(\Gamma(z+1)=z\Gamma(z)\), the quotient of two consecutive
	coefficients is
	\[
		\frac{c_{k+1;j,i}}{c_{k;j,i}}
		=
		\frac{
			\prod_{\rho=1}^{r}(k+\kappa_i+a_\rho+1)
		}{
			\prod_{\ell=1}^{q}(k+\kappa_i+b_\ell+\delta_{j\ell}+1)
		}.
	\]
	Hence
\(\left| \frac{c_{k+1;j,i}z^{k+1}}{c_{k;j,i}z^k} \right| = |z|\,\mathrm{O}\left(k^{r-q}\right),\) \(k\to\infty .\)
	Since the Bessel regime \(r<q\) is assumed, this quantity tends to
	zero for every fixed \(z\in\mathbb C\).  Thus the series has infinite
	radius of convergence, and \(w_{j,i}\) is entire.  Since this
	argument is uniform for each fixed pair \((j,i)\), the matrix \(\WB\) is a
	\(q\times p\) entire matrix-valued function.
	
	It remains to identify the moments.  Since \(w_{j,i}\) is
	holomorphic in a neighborhood of \(\Torus\), the contour integral extracts
	the coefficient of \(z^{-1}\).  Fix \(s\in\N\). Multiplying
	\eqref{eq:final-direct-weight-series} by \(z^{-s}\) gives
\(z^{-s}w_{j,i}(z) = \sum_{k=0}^{\infty}c_{k;j,i}z^{k-s}.\)
	The coefficient of \(z^{-1}\) occurs for \(k=s-1\).  Therefore
	\[
		\oint_{\Torus}
		z^{-s}w_{j,i}(z)
		\frac{\dz}{2\pi\mathrm i}
		=
		c_{s-1;j,i}
		=
		\frac{
			\Gamma((s+\kappa_i)\one_r+\boldsymbol a)
		}{
			\Gamma((s+\kappa_i)\one_q+\boldsymbol b+\boldsymbol e_j)
		},
	\]
	which is \eqref{eq:final-moments}.
\end{proof}

The explicit formulas below are most conveniently written with ordinary
polynomials in the variable \(z\).  For this reason, the
reciprocal Laurent representation of the same coefficient sequence is also used.
This is  only a change of
presentation of the moments in \eqref{eq:final-moments}.

\begin{definition}[Reciprocal Laurent representation]
	For $j\in\{1,\ldots,q\}$ and $i\in\{1,\ldots,p\}$, define
	\begin{equation}
		\label{eq:final-cauchy-representative}
		C_{j,i}(z)
		\coloneq
		\frac{1}{z}
		w_{j,i}\left(\frac{1}{z}\right)
		=
		\frac{1}{z}
		f_0\left(
		\frac{1}{z};
		\boldsymbol a+\kappa_i\one_r,
		\boldsymbol b+\kappa_i\one_q+\boldsymbol e_j
		\right).
	\end{equation}
	Write
\begin{equation}
	\label{eq:final-matrix-weight}
	\CB(z)
	\coloneq
	\begin{bmatrix}
		C_{1,1}(z) & C_{1,2}(z) & \cdots & C_{1,p}(z)\\
		C_{2,1}(z) & C_{2,2}(z) & \cdots & C_{2,p}(z)\\
		\vdots & \vdots & \ddots & \vdots\\
		C_{q,1}(z) & C_{q,2}(z) & \cdots & C_{q,p}(z)
	\end{bmatrix}.
\end{equation}
\end{definition}

\begin{proposition}[Reciprocal form of the same moment sequence]
	For $z$ in a neighborhood of infinity,
	\begin{equation}
		\label{eq:final-weight-series}
		C_{j,i}(z)
		=
		\sum_{k=0}^{\infty}
		\frac{
			\Gamma((k+\kappa_i)\one_r+\boldsymbol a+\one_r)
		}{
			\Gamma((k+\kappa_i)\one_q+\boldsymbol b+\boldsymbol e_j+\one_q)
		}
		z^{-k-1}.
	\end{equation}
	Consequently, for $s\in\N$,
	\begin{equation}
		\label{eq:final-reciprocal-moments}
		\oint_{\Torus}
		z^{s-1}C_{j,i}(z)
		\frac{\dz}{2\pi\mathrm i}
		=
		\oint_{\Torus}
		z^{-s}w_{j,i}(z)
		\frac{\dz}{2\pi\mathrm i}.
	\end{equation}
\end{proposition}

\begin{proof}
	Substituting \(z^{-1}\) for \(z\) in
	\eqref{eq:final-direct-weight-series} and multiplying by \(z^{-1}\) gives
	\eqref{eq:final-weight-series}.  The coefficient of \(z^{-s}\) in
	\eqref{eq:final-weight-series} is exactly the coefficient of \(z^{s-1}\)
	in \eqref{eq:final-direct-weight-series}.  Thus the two contour integrals
	in \eqref{eq:final-reciprocal-moments} extract the same coefficient.
\end{proof}

\begin{proposition}[Contiguity of the reciprocal representatives]
	\label{prop:Bessel-reciprocal-contiguity}
	For a common shift parameter \(\kappa\), define
	\[
		w_j^{[\kappa]}(z)
		\coloneq
		f_0\left(
		z;
		\boldsymbol a+\kappa\one_r,
		\boldsymbol b+\kappa\one_q+\boldsymbol e_j
		\right),
		\qquad
		j\in\{1,\ldots,q\},
	\]
	and
	\[
		C_j^{[\kappa]}(z)
		\coloneq
		\frac{1}{z}
		w_j^{[\kappa]}\left(\frac{1}{z}\right).
	\]
	Then
	\[
		w_j^{[\kappa+1]}(z)
		=
		\frac{
			w_j^{[\kappa]}(z)-w_j^{[\kappa]}(0)
		}{z},
	\]
	and
	\[
		C_j^{[\kappa+1]}(z)
		=
		zC_j^{[\kappa]}(z)-w_j^{[\kappa]}(0).
	\]
	Consequently,
	\[
		\oint_{\Torus}
		z^nC_j^{[\kappa+1]}(z)
		\frac{\dz}{2\pi\mathrm i}
		=
		\oint_{\Torus}
		z^{n+1}C_j^{[\kappa]}(z)
		\frac{\dz}{2\pi\mathrm i},
		\qquad
		n\in\N_0.
	\]
\end{proposition}

\begin{proof}
	Write
	\[
		w_j^{[\kappa]}(z)
		=
		\sum_{k=0}^{\infty}
		c_{k;j}^{[\kappa]}z^k,
		\qquad
		c_{k;j}^{[\kappa]}
		\coloneq
		\frac{
			\Gamma((k+\kappa)\one_r+\boldsymbol a+\one_r)
		}{
			\Gamma((k+\kappa)\one_q+\boldsymbol b+\boldsymbol e_j+\one_q)
		}.
	\]
	For every \(k\in\N_0\),
	\[
		c_{k;j}^{[\kappa+1]}
		=
		c_{k+1;j}^{[\kappa]}.
	\]
	Hence
	\[
		w_j^{[\kappa+1]}(z)
		=
		\sum_{k=0}^{\infty}
		c_{k+1;j}^{[\kappa]}z^k
		=
		\frac{
			w_j^{[\kappa]}(z)-w_j^{[\kappa]}(0)
		}{z}.
	\]
	Replacing \(z\) by \(z^{-1}\) and multiplying by \(z^{-1}\) gives
	\[
		C_j^{[\kappa+1]}(z)
		=
		zC_j^{[\kappa]}(z)-w_j^{[\kappa]}(0).
	\]
	Finally, multiplication by \(z^n\) and contour integration yield
	\[
		\oint_{\Torus}
		z^nC_j^{[\kappa+1]}(z)
		\frac{\dz}{2\pi\mathrm i}
		=
		\oint_{\Torus}
		z^{n+1}C_j^{[\kappa]}(z)
		\frac{\dz}{2\pi\mathrm i},
	\]
	because the constant term \(w_j^{[\kappa]}(0)\) has zero contour
	pairing against \(z^n\) for every \(n\in\N_0\).
\end{proof}

\begin{remark}[Relation with the two Bessel reductions]
	The weights of the mixed Bessel system are the entries of \(\WB\)
	restricted to \(\Torus\).  This follows the Bessel-like construction of
	Wolfs~\cite{Wolfs2024}: when \(p=1\), after the common parameter shift
	\((\boldsymbol a,\boldsymbol b)\mapsto(\boldsymbol a+\kappa_1\one_r,
	\boldsymbol b+\kappa_1\one_q)\), the weights become
\(w_{j,1}(z) = f_0\left(z;\boldsymbol a+\kappa_1\one_r, \boldsymbol b+\kappa_1\one_q+\boldsymbol e_j\right),\) \(j\in\{1,\ldots,q\},\)
	which are Wolfs' Bessel-like weights on the unit circle.  The Laurent
		presentation in \eqref{eq:final-cauchy-representative} is used only to
		rewrite the same coefficient sequence in the moment computations below.
	
	The other case, \(q=1\), is the multiple Bessel case.  In that
	case the row direction is scalar and the remaining column parameters
	\(\kappa_i\) produce the family of shifted Bessel moment functionals
	described in Proposition~\ref{prop:multiple-bessel-case}.
	
\end{remark}

\begin{proposition}[Rank of the Bessel weight matrix on the unit circle]
	\label{prop:final-rank-circle}
	Let \(m=\min\{p,q\}\).  Assume that there are ordered indices
	\(1\le j_1<\cdots<j_m\le q\) and
	\(1\le i_1<\cdots<i_m\le p\) such that the numbers
	\(b_{j_1},\ldots,b_{j_m}\) are pairwise distinct, the numbers
	\(\kappa_{i_1},\ldots,\kappa_{i_m}\) are pairwise distinct, and
\(1+\kappa_{i_\beta}+b_{j_\alpha}\ne0,\) \(\alpha,\beta\in\{1,\ldots,m\},\)
	with all involved Gamma factors finite.  Then a maximal minor of
	\(\WB\) is not identically zero.  Consequently,
\(\operatorname{rank}\WB(z)=m\)
	for every \(z\in\Torus\) except, possibly, for finitely many points.
	The same assertion holds for the reciprocal Laurent representative \(\CB\).
\end{proposition}

\begin{proof}
	It suffices to prove the assertion for \(\CB\); the statement for \(\WB\) follows from
	\eqref{eq:final-cauchy-representative}, because the transformation
	\(z\mapsto z^{-1}\) preserves nonidentical vanishing and maps \(\Torus\)
	onto itself.

	From \eqref{eq:final-weight-series}, the coefficient of \(z^{-1}\) in
	\(C_{j,i}\) is
\(c_{j,i} = \frac{\Gamma((1+\kappa_i)\one_r+\boldsymbol a)} {\Gamma((1+\kappa_i)\one_q+\boldsymbol b+\boldsymbol e_j)}.\)
	Since
	\[
		\Gamma((1+\kappa_i)\one_q+\boldsymbol b+\boldsymbol e_j)
		=
		\Gamma(2+\kappa_i+b_j)
		\prod_{\substack{h=1\\ h\ne j}}^{q}
		\Gamma(1+\kappa_i+b_h),
	\]
	this coefficient can be written as
	\[
		c_{j,i}
		=
		R_i\frac{1}{1+\kappa_i+b_j},
		\qquad
		R_i\coloneq
		\frac{\Gamma((1+\kappa_i)\one_r+\boldsymbol a)}
		{\Gamma((1+\kappa_i)\one_q+\boldsymbol b)}.
	\]
	Consider the ordered submatrix determined by the row list
	\(J=(j_1,\ldots,j_m)\) and the column list
	\(I=(i_1,\ldots,i_m)\).  The coefficient of
	\(z^{-m}\) in \(\det \CB_{J,I}(z)\) is
	\[
		\left(\prod_{\beta=1}^{m}R_{i_\beta}\right)
		\det\left[
		\frac{1}{1+\kappa_{i_\beta}+b_{j_\alpha}}
		\right]_{\alpha,\beta=1}^{m}.
	\]
	The last determinant is a Cauchy determinant.  Under the stated assumptions it
	is nonzero, because
	\[
		\det\left[
		\frac{1}{1+\kappa_{i_\beta}+b_{j_\alpha}}
		\right]_{\alpha,\beta=1}^{m}
		=
		\frac{
		\prod_{1\le\alpha<\alpha'\le m}(b_{j_{\alpha'}}-b_{j_\alpha})
		\prod_{1\le\beta<\beta'\le m}(\kappa_{i_{\beta'}}-\kappa_{i_\beta})
		}{
		\prod_{\alpha=1}^{m}\prod_{\beta=1}^{m}
		(1+\kappa_{i_\beta}+b_{j_\alpha})
		}.
	\]
	Hence \(\det \CB_{J,I}\) is not identically
	zero.

	Each entry of \(\CB\) is holomorphic in \(\mathbb C\setminus\{0\}\), and so
	each minor is holomorphic in an open annulus containing \(\Torus\).  If the
	minor \(\det \CB_{J,I}\) had infinitely many zeros on \(\Torus\), compactness
	of \(\Torus\) would give an accumulation point on \(\Torus\); the identity
	theorem would then force this minor to vanish identically in the annulus,
	contradicting the previous paragraph.  Thus the exceptional set on
	\(\Torus\) is finite.  Away from this finite set, one maximal minor is
	nonzero, and therefore \(\operatorname{rank}\CB(z)=m\).  The same conclusion
	holds for \(\WB\).
\end{proof}

\begin{remark}[Why this phenomenon is specific to the Bessel realization on the unit circle]
	For a real-line Mellin weight, shifting the argument of the Mellin transform is multiplication by a power:
\(\mathcal M[x^{\kappa}w](s)=\mathcal M[w](s+\kappa),\)
	whenever both Mellin transforms are defined.
	Consequently, the analogous Gamma-parameter shifts in the Jacobi and Laguerre systems of mixed type introduced in \cite{JacobiLaguerreMixed} give the rank-one matrices already present in the real-line construction.  The Bessel presentation on the unit circle is different: in the reciprocal Cauchy representation, its moments are extracted from Laurent coefficients.  Shifting the Gamma parameters by a nonintegral \(\kappa_i\) again gives a legitimate reciprocal Laurent representative, but it cannot be represented in the same Laurent framework by multiplication by \(z^{\kappa_i}\).  This is why the present weight matrix goes beyond the rank-one form \(w_{j,i}(z)=v_j(z)u_i(z)\).
\end{remark}

\subsection{Jacobi confluence through the Markov--Stieltjes matrix}
\label{subsec:Jacobi-Markov-Bessel-confluence}

The maximal-rank matrix introduced above has a precise confluent
origin in the rank-one Jacobi system of mixed type introduced in
\cite{JacobiLaguerreMixed}.  The confluence does not take place as a
pointwise limit of the Jacobi matrix densities: such a limit would necessarily
have rank at most one.  Instead, one first passes from the interval
orthogonality to its exactly equivalent complex-contour realization through
the Markov--Stieltjes matrix.  After the appropriate rescaling, these
Markov--Stieltjes representatives converge to the reciprocal matrix \(\CB\),
or equivalently to the entire direct representative \(\WB\).

Put \(d\coloneq q-r\) and introduce the confluent parameters
\(\boldsymbol\lambda=\begin{bNiceMatrix}\lambda_1&\Cdots&\lambda_d\end{bNiceMatrix}\) and
\(\boldsymbol a^{\,\mathrm J,\boldsymbol\lambda}
\coloneq\begin{bNiceMatrix}\boldsymbol a&\boldsymbol\lambda\end{bNiceMatrix}\in\mathbb C^q\).

For \(u,v\in\mathbb C\) satisfying
\(\operatorname{Re}u>-1\) and
\(\operatorname{Re}(v-u)>0\), define the normalized beta weight
\[
\mathcal B_{u,v}(x)
\coloneq
\frac{x^u(1-x)^{v-u-1}}{\Gamma(v-u)},
\qquad 0<x<1,
\]
where the powers are defined using the real logarithm on \((0,1)\). Its
Mellin transform is
\[
\int_0^1x^{s-1}\mathcal B_{u,v}(x)\,\mathrm dx
=
\frac{\Gamma(s+u)}{\Gamma(s+v)},
\qquad
\operatorname{Re}(s+u)>0.
\]

For integrable functions \(f\) and \(g\) supported on \((0,1)\), their Mellin
convolution is
\[
(f*g)(x)
\coloneq
\int_x^1
f(t)g\left(\frac{x}{t}\right)\frac{\mathrm dt}{t},
\qquad 0<x<1.
\]
Multiple convolutions are defined recursively by
\(f_1*\cdots*f_q\coloneq(f_1*\cdots*f_{q-1})*f_q\), with the convention that
the convolution reduces to \(f_1\) when \(q=1\). Whenever the corresponding
Mellin transforms exist, the Mellin transform of \(f*g\) is the product of
the Mellin transforms of \(f\) and \(g\).

Assume that
\(\operatorname{Re}a_h>-1\) and
\(\operatorname{Re}(b_h-a_h)>0\) for \(h\in\{1,\ldots,r\}\), and that
\(\operatorname{Re}\lambda_\ell>-1\) and
\(\operatorname{Re}(b_{r+\ell}-\lambda_\ell)>0\) for
\(\ell\in\{1,\ldots,d\}\). Writing
\(a_h^{\,\mathrm J,\boldsymbol\lambda}\) for the \(h\)-th component of
\(\boldsymbol a^{\,\mathrm J,\boldsymbol\lambda}\), define the Jacobi-like row weights by
\[
w_j^{\mathrm J,\boldsymbol\lambda}
\coloneq
\mathcal B_{a_1^{\,\mathrm J,\boldsymbol\lambda},\,b_1+\delta_{j,1}}
*\cdots*
\mathcal B_{a_q^{\,\mathrm J,\boldsymbol\lambda},\,b_q+\delta_{j,q}},
\qquad
j\in\{1,\ldots,q\}.
\]
Each \(w_j^{\mathrm J,\boldsymbol\lambda}\) is an integrable complex-valued weight
on \((0,1)\), and the Mellin convolution rule gives
\begin{equation}
	\label{eq:Jacobi-confluent-Mellin-transform}
	\begin{aligned}
	\int_0^1x^{s-1}w_j^{\mathrm J,\boldsymbol\lambda}(x)\,\mathrm dx
	&=
	\frac{
		\Gamma(s\one_r+\boldsymbol a)
		\prod_{\ell=1}^{d}\Gamma(s+\lambda_\ell)
	}{
		\Gamma(s\one_q+\boldsymbol b+\boldsymbol e_j)
	},\\
	\operatorname{Re}s
	&>
	\max\left\{
	-\min_{1\le h\le r}\operatorname{Re}a_h,
	-\min_{1\le\ell\le d}\operatorname{Re}\lambda_\ell
	\right\}.
	\end{aligned}
\end{equation}

The Jacobi power parameters are identified with the Bessel column shifts,
and the corresponding rank-one matrix of interval measures is defined by
\begin{equation}
	\label{eq:Jacobi-confluent-matrix-measure}
	\mathrm d\mathcal J^{\boldsymbol\lambda}(x)
	\coloneq
\begin{bNiceMatrix}
	w_1^{\mathrm J,\boldsymbol\lambda}(x)\\
	\Vdots\\
	w_q^{\mathrm J,\boldsymbol\lambda}(x)
\end{bNiceMatrix}
\begin{bNiceMatrix}
	x^{\kappa_1}&\Cdots&x^{\kappa_p}
\end{bNiceMatrix}
	\,\mathrm dx,
\end{equation}
where \(x^{\kappa_i}\coloneq\exp(\kappa_i\log x)\), with the real logarithm
on \((0,1)\). Hence its pointwise rank is at most one.

To let the parameters \(\lambda_\ell\) tend to infinity while keeping the
Jacobi-like weights integrable on \((0,1)\), assume that
\(\operatorname{Re}(\kappa_i+a_h+1)>0\) for
\(i\in\{1,\ldots,p\}\) and \(h\in\{1,\ldots,r\}\), and choose real numbers
\(c_\ell\) satisfying
\[
\max\left\{
-1,
\max_{1\le i\le p}
\bigl(-1-\operatorname{Re}\kappa_i\bigr)
\right\}
<
c_\ell
<
\operatorname{Re}b_{r+\ell},
\qquad
\ell\in\{1,\ldots,d\}.
\]
Set
\[
\lambda_\ell(T)\coloneq c_\ell+\mathrm iT,
\qquad
\ell\in\{1,\ldots,d\},
\quad
T>0.
\]
Since the real part of \(\lambda_\ell(T)\) is the fixed number \(c_\ell\),
the beta factors defining \(w_j^{\mathrm J,\boldsymbol\lambda(T)}\) remain
integrable on \((0,1)\), and so do the functions
\(x^{\kappa_i}w_j^{\mathrm J,\boldsymbol\lambda(T)}(x)\). Therefore
\(\mathrm d\mathcal J^{\boldsymbol\lambda(T)}\) is a finite matrix of complex
measures for every \(T>0\). At the same time,
\[
\min_{1\le\ell\le d}|\lambda_\ell(T)|
\longrightarrow\infty,
\qquad
T\longrightarrow+\infty.
\]

The moments of \eqref{eq:Jacobi-confluent-matrix-measure} are therefore
\begin{equation}
	\label{eq:Jacobi-confluent-unscaled-moments}
	\mu_{j,i}^{\mathrm J,\boldsymbol\lambda}(n)
	\coloneq
	\int_0^1
	x^n\,\mathrm d\mathcal J_{j,i}^{\boldsymbol\lambda}(x)
	=
	\frac{
		\Gamma((n+\kappa_i)\one_r+\boldsymbol a+\one_r)
		\prod_{\ell=1}^{d}
		\Gamma(n+\kappa_i+\lambda_\ell+1)
	}{
		\Gamma((n+\kappa_i)\one_q+\boldsymbol b+\boldsymbol e_j+\one_q)
	}.
\end{equation}
This identity holds for every \(n\in\N_0\).

The Markov--Stieltjes matrix of \(\mathrm d\mathcal J^{\boldsymbol\lambda}\) is
\begin{equation}
\label{eq:Jacobi-Markov-Stieltjes-matrix}
\mathsf S^{\mathrm J,\boldsymbol\lambda}(\zeta)
\coloneq
\int_0^1\frac{\mathrm d\mathcal J^{\boldsymbol\lambda}(x)}{\zeta-x},
\qquad \zeta\in\mathbb C\setminus[0,1].
\end{equation}
Set
\[
\Lambda_{\boldsymbol\lambda}
\coloneq
\prod_{h=1}^{d}\lambda_h,
\qquad
\mathsf D_{\boldsymbol\lambda}
\coloneq
\operatorname{diag}\left[
\prod_{h=1}^{d}\Gamma(\lambda_h+\kappa_i+1)
\right]_{i=1}^{p}.
\]
Introduce the rescaled matrix measure by push-forward under
\(x\mapsto x/\Lambda_{\boldsymbol\lambda}\) and by the column normalization
\(\mathsf D_{\boldsymbol\lambda}^{-1}\):
\begin{equation}
\label{eq:scaled-Jacobi-measure-functional-definition}
\int P(y)\,\mathrm d\widehat{\mathcal J}^{\boldsymbol\lambda}(y)
\coloneq
\int_0^1P\left(\frac{x}{\Lambda_{\boldsymbol\lambda}}\right)
\mathrm d\mathcal J^{\boldsymbol\lambda}(x)\mathsf D_{\boldsymbol\lambda}^{-1},
\qquad P\in\mathbb C[y].
\end{equation}
Its support is the segment
\(\Delta_{\boldsymbol\lambda}\coloneq\Lambda_{\boldsymbol\lambda}^{-1}[0,1]\), and its
Markov--Stieltjes matrix is
\begin{equation}
\label{eq:scaled-Jacobi-Markov-matrix}
\widehat{\CB}^{\,\boldsymbol\lambda}(z)
\coloneq
\int_{\Delta_{\boldsymbol\lambda}}
\frac{\mathrm d\widehat{\mathcal J}^{\boldsymbol\lambda}(y)}{z-y}
=
\Lambda_{\boldsymbol\lambda}\mathsf S^{\mathrm J,\boldsymbol\lambda}
\bigl(\Lambda_{\boldsymbol\lambda}z\bigr)\mathsf D_{\boldsymbol\lambda}^{-1},
\qquad
z\in\mathbb C\setminus\Delta_{\boldsymbol\lambda}.
\end{equation}

\begin{proposition}[Interval-to-circle equivalence]
	\label{prop:interval-contour-Markov-equivalence}
	For \(T\) sufficiently large, let
	\(\mathbf B\in\mathbb C^{1\times q}[z]\) and
	\(\mathbf A\in\mathbb C^{p\times1}[z]\). Then
	\begin{equation}
		\label{eq:interval-contour-bilinear-equivalence}
		\int_{\Delta_{\boldsymbol\lambda(T)}}
		\mathbf B(y)\,
		\mathrm d\widehat{\mathcal J}^{\boldsymbol\lambda(T)}(y)\,
		\mathbf A(y)
		=
		\oint_{\Torus}
		\mathbf B(z)
		\widehat{\CB}^{\,\boldsymbol\lambda(T)}(z)
		\mathbf A(z)
		\frac{\dz}{2\pi\mathrm i}.
	\end{equation}
	Consequently, the interval and unit-circle realizations induce the same
	bilinear form on polynomial vectors.
\end{proposition}

\begin{proof}
	Since \(\Delta_{\boldsymbol\lambda}\) is compact and contained in the open unit
	disk, one has
\(	\delta
	\coloneq
	\operatorname{dist}(\Torus,\Delta_{\boldsymbol\lambda})
	>0\).
	The polynomial vectors \(\mathbf B\) and \(\mathbf A\) are bounded on
	\(\Torus\), and
	\(\mathrm d\widehat{\mathcal J}^{\boldsymbol\lambda}\) is a finite matrix of
	complex measures. Hence each scalar entry of
	\[
	\frac{
		\mathbf B(z)\,
		\mathrm d\widehat{\mathcal J}^{\boldsymbol\lambda}(y)\,
		\mathbf A(z)
	}{
		z-y
	}
	\]
	is absolutely integrable with respect to arc length on \(\Torus\) and the
	total variation of the corresponding measure, so Fubini's theorem applies
	entrywise. Therefore,
	\[
	\oint_{\Torus}
	\mathbf B(z)\widehat{\CB}^{\,\boldsymbol\lambda}(z)\mathbf A(z)
	\frac{\dz}{2\pi\mathrm i}
	=
	\int_{\Delta_{\boldsymbol\lambda}}
	\left(
	\oint_{\Torus}
	\frac{
		\mathbf B(z)\,
		\mathrm d\widehat{\mathcal J}^{\boldsymbol\lambda}(y)\,
		\mathbf A(z)
	}{
		z-y
	}
	\frac{\dz}{2\pi\mathrm i}
	\right).
	\]
	For each fixed \(y\in\Delta_{\boldsymbol\lambda}\), the numerator is entire in
	\(z\), and \(y\) lies inside \(\Torus\). The Cauchy integral formula then
	gives
	\[
	\oint_{\Torus}
	\frac{
		\mathbf B(z)\,
		\mathrm d\widehat{\mathcal J}^{\boldsymbol\lambda}(y)\,
		\mathbf A(z)
	}{
		z-y
	}
	\frac{\dz}{2\pi\mathrm i}
	=
	\mathbf B(y)\,
	\mathrm d\widehat{\mathcal J}^{\boldsymbol\lambda}(y)\,
	\mathbf A(y),
	\]
	which proves the identity.
\end{proof}

	Put
\(L_{\boldsymbol\lambda}
\coloneq
\min_{1\le h\le d}|\lambda_h|\).
The moments of the rescaled measure are
\begin{equation}
	\label{eq:exact-rescaled-Jacobi-moments}
	\begin{aligned}
	\widehat\mu_{j,i}^{\,\mathrm J,\boldsymbol\lambda}(n)
	&\coloneq
	\int_{\Delta_{\boldsymbol\lambda}}
	y^n\,\mathrm d\widehat{\mathcal J}_{j,i}^{\boldsymbol\lambda}(y)\\
	&=
	\frac{
		\mu_{j,i}^{\mathrm J,\boldsymbol\lambda}(n)
	}{
		\Lambda_{\boldsymbol\lambda}^{n}
		\prod_{h=1}^{d}
		\Gamma(\lambda_h+\kappa_i+1)
	}\\
	&=\frac{
		\Gamma((n+\kappa_i)\one_r+\boldsymbol a+\one_r)
	}{
		\Gamma((n+\kappa_i)\one_q+\boldsymbol b+\boldsymbol e_j+\one_q)
	}
	\prod_{h=1}^{d}
	\frac{(\lambda_h+\kappa_i+1)_n}{\lambda_h^n}.
	\end{aligned}
\end{equation}

In particular, for every fixed \(n\in\N_0\),
\(j\in\{1,\ldots,q\}\), and \(i\in\{1,\ldots,p\}\),
\begin{equation}
	\label{eq:rescaled-Jacobi-moment-convergence}
	\widehat\mu_{j,i}^{\,\mathrm J,\boldsymbol\lambda}(n)
	\xlongrightarrow[
	L_{\boldsymbol\lambda}\to\infty
	]{}
	\frac{\Gamma((n+\kappa_i)\one_r+\boldsymbol a+\one_r)}
	{\Gamma((n+\kappa_i)\one_q+\boldsymbol b+\boldsymbol e_j+\one_q)}.
\end{equation}

Its direct reciprocal representative is
\begin{equation}
	\label{eq:scaled-Jacobi-direct-representative}
	\widehat{\WB}^{\,\boldsymbol\lambda}(z)
	\coloneq
	\int_{\Delta_{\boldsymbol\lambda}}
	\frac{\mathrm d\widehat{\mathcal J}^{\boldsymbol\lambda}(y)}{1-yz}
	=
	\frac{1}{z}
	\widehat{\CB}^{\,\boldsymbol\lambda}\left(\frac{1}{z}\right),
	\qquad z\ne0,
\end{equation}
where the integral on the left gives its analytic continuation to \(z=0\).
Since
\(\Delta_{\boldsymbol\lambda}
=\Lambda_{\boldsymbol\lambda}^{-1}[0,1]\),
expansion of the Cauchy kernel gives
\begin{equation}
	\label{eq:scaled-Jacobi-direct-series}
	\widehat W_{j,i}^{\,\boldsymbol\lambda}(z)
	=
	\sum_{n=0}^{\infty}
	\widehat\mu_{j,i}^{\mathrm J,\boldsymbol\lambda}(n)z^n,
	\qquad
	|z|<|\Lambda_{\boldsymbol\lambda}|.
\end{equation}
Using \eqref{eq:exact-rescaled-Jacobi-moments}, this series can equivalently
be written as
\begin{equation}
	\label{eq:scaled-Jacobi-hypergeometric-representative}
	\widehat W_{j,i}^{\,\boldsymbol\lambda}(z)
	=
	\frac{\Gamma(\kappa_i\one_r+\boldsymbol a+\one_r)}
	{\Gamma(\kappa_i\one_q+\boldsymbol b+\boldsymbol e_j+\one_q)}
	\pFq{q+1}{q}
	{1,\,\kappa_i\one_r+\boldsymbol a+\one_r,\,
		\kappa_i\one_d+\boldsymbol\lambda+\one_d}
	{\kappa_i\one_q+\boldsymbol b+\boldsymbol e_j+\one_q}
	{\dfrac{z}{\Lambda_{\boldsymbol\lambda}}}.
\end{equation}
\begin{theorem}[Jacobi-to-Bessel Markov--Stieltjes confluence]
\label{thm:Jacobi-Bessel-Markov-confluence}
Along the admissible path \(\boldsymbol\lambda=\boldsymbol\lambda(T)\) introduced above,
assume that all parameter-dependent Gamma factors remain finite. Then, as
\(T\to+\infty\), or equivalently \(L_{\boldsymbol\lambda}\to\infty\),
\begin{equation}
\label{eq:direct-compact-open-convergence}
\widehat{\WB}^{\,\boldsymbol\lambda}\longrightarrow\WB
\end{equation}
locally uniformly on \(\mathbb C\), i.e.,  for every
compact set \(K\subset\mathbb C\), one has
\(K\subset\{z:|z|<|\Lambda_{\boldsymbol\lambda}|\}\) for all sufficiently large
\(\boldsymbol\lambda\), and
\(\sup_{z\in K}\|\widehat{\WB}^{\,\boldsymbol\lambda}(z)-\WB(z)\|\to0\).
Equivalently,
\begin{equation}
\label{eq:reciprocal-compact-open-convergence}
\widehat{\CB}^{\,\boldsymbol\lambda}\longrightarrow\CB
\end{equation}
locally uniformly on \(\mathbb C^*\).  In particular,
\begin{equation}
\label{eq:reciprocal-uniform-circle-convergence}
\sup_{z\in\Torus}
\left\|\widehat{\CB}^{\,\boldsymbol\lambda}(z)-\CB(z)\right\|
\longrightarrow0.
\end{equation}
\end{theorem}

\begin{proof}
	Fix \(j\in\{1,\ldots,q\}\) and \(i\in\{1,\ldots,p\}\), and set
	\[
	c_{n;j,i}
	\coloneq
	\frac{\Gamma((n+\kappa_i)\one_r+\boldsymbol a+\one_r)}
	{\Gamma((n+\kappa_i)\one_q+\boldsymbol b+\boldsymbol e_j+\one_q)},
	\qquad
	R_{n;i}^{\boldsymbol\lambda}
	\coloneq
	\prod_{h=1}^{d}
	\frac{(\lambda_h+\kappa_i+1)_n}{\lambda_h^n}.
	\]
	Then
	\[
	\widehat W_{j,i}^{\,\boldsymbol\lambda}(z)
	=
	\sum_{n=0}^{\infty}
	c_{n;j,i}R_{n;i}^{\boldsymbol\lambda}z^n,
	\qquad
	W_{j,i}(z)
	=
	\sum_{n=0}^{\infty}c_{n;j,i}z^n.
	\]
	For every fixed \(n\in\N_0\),
	\[
	R_{n;i}^{\boldsymbol\lambda}
	=
	\prod_{h=1}^{d}
	\prod_{\ell=0}^{n-1}
	\left(
	1+\frac{\kappa_i+1+\ell}{\lambda_h}
	\right)
	\xlongrightarrow[L_{\boldsymbol\lambda}\to\infty]{}1
	\]
	Thus every coefficient of
	\(\widehat W_{j,i}^{\,\boldsymbol\lambda}\) converges to the corresponding
	coefficient of \(W_{j,i}\). To obtain uniform convergence on compact
	sets, however, the infinitely many terms with large \(n\) must also be
	controlled simultaneously.

The quotient of two consecutive coefficients of
\(\widehat W_{j,i}^{\,\boldsymbol\lambda}\) is
\[
\frac{c_{n+1;j,i}R_{n+1;i}^{\boldsymbol\lambda}}
{c_{n;j,i}R_{n;i}^{\boldsymbol\lambda}}
=
\frac{
	\prod_{\rho=1}^{r}
	(n+\kappa_i+a_\rho+1)
}{
	\prod_{\ell=1}^{q}
	(n+\kappa_i+b_\ell+\delta_{j\ell}+1)
}
\prod_{h=1}^{d}
\left(
1+\frac{n+\kappa_i+1}{\lambda_h}
\right).
\]
Since \(d=q-r\), define
\[
\mathfrak C_{j,i}
\coloneq
\sup_{n\in\N_0}
(n+1)^d
\left|
\frac{
	\prod_{\rho=1}^{r}
	(n+\kappa_i+a_\rho+1)
}{
	\prod_{\ell=1}^{q}
	(n+\kappa_i+b_\ell+\delta_{j\ell}+1)
}
\right|.
\]
The standing assumptions ensure that the factors in the denominator do
not vanish for \(n\in\N_0\). Moreover,
\[
(n+1)^d
\left|
\frac{
	\prod_{\rho=1}^{r}
	(n+\kappa_i+a_\rho+1)
}{
	\prod_{\ell=1}^{q}
	(n+\kappa_i+b_\ell+\delta_{j\ell}+1)
}
\right|
=
\frac{
	\prod_{\rho=1}^{r}
	\left|
	1+\frac{\kappa_i+a_\rho}{n+1}
	\right|
}{
	\prod_{\ell=1}^{q}
	\left|
	1+\frac{\kappa_i+b_\ell+\delta_{j\ell}}{n+1}
	\right|
}
\longrightarrow1
\]
as \(n\to\infty\). Hence \(\mathfrak C_{j,i}<\infty\), and
\[
\left|
\frac{
	\prod_{\rho=1}^{r}
	(n+\kappa_i+a_\rho+1)
}{
	\prod_{\ell=1}^{q}
	(n+\kappa_i+b_\ell+\delta_{j\ell}+1)
}
\right|
\le
\frac{\mathfrak C_{j,i}}{(n+1)^d},
\qquad n\in\N_0.
\]

Now set
\(M_i\coloneq\max\{1,|\kappa_i+1|\}\).
If \(L_{\boldsymbol\lambda}\ge1\), then, for every
\(h\in\{1,\ldots,d\}\),
\[
\left|
1+\frac{n+\kappa_i+1}{\lambda_h}
\right|
\le
1+\frac{n+|\kappa_i+1|}{L_{\boldsymbol\lambda}}
\le
M_i\left(
1+\frac{n+1}{L_{\boldsymbol\lambda}}
\right).
\]
Therefore,
\[
\prod_{h=1}^{d}
\left|
1+\frac{n+\kappa_i+1}{\lambda_h}
\right|
\le
M_i^d
\left(
1+\frac{n+1}{L_{\boldsymbol\lambda}}
\right)^d.
\]
Combining the last two estimates yields
\begin{equation}
	\label{eq:uniform-confluent-ratio-bound}
	\left|
	\frac{c_{n+1;j,i}R_{n+1;i}^{\boldsymbol\lambda}}
	{c_{n;j,i}R_{n;i}^{\boldsymbol\lambda}}
	\right|
	\le
	\mathfrak C_{j,i}M_i^d
	\left(
	\frac{1}{n+1}
	+
	\frac{1}{L_{\boldsymbol\lambda}}
	\right)^d,
	\qquad
	n\in\N_0,\quad L_{\boldsymbol\lambda}\ge1.
\end{equation}
	Set	\(C\coloneq \mathfrak C_{j,i}M_i^d\).
	
	Let \(R>0\) and \(\varepsilon>0\). It remains to prove uniform convergence on the
	closed disk \(\{z\in\mathbb C:|z|\le R\}\). Choose \(N\) sufficiently
	large that
	\[
	RC\left(\frac{1}{n+1}\right)^d
	\le\frac14,
	\qquad n\ge N,
\qquad
	2|c_{N;j,i}|R^N<\frac{\varepsilon}{6}.
	\]
The second condition is possible because, for the fixed value of \(R\),
the convergence of the series
\[
W_{j,i}(R)=\sum_{n=0}^{\infty}c_{n;j,i}R^n
\]
implies
\(|c_{n;j,i}|R^n\longrightarrow0\)
as \(n\to\infty\).
Hence \(N\) can be chosen sufficiently large so that
\(2|c_{N;j,i}|R^N<\frac{\varepsilon}{6}\).
	
	Next choose \(L_0\) sufficiently large that, whenever
	\(L_{\boldsymbol\lambda}\ge L_0\),
	\[
	RC\left(
	\frac{1}{n+1}
	+
	\frac{1}{L_{\boldsymbol\lambda}}
	\right)^d
	\le\frac12,
	\qquad n\ge N,
\qquad
	2|c_{N;j,i}R_{N;i}^{\boldsymbol\lambda}|R^N
	<\frac{\varepsilon}{3}.
	\]
The latter follows from
\(R_{N;i}^{\boldsymbol\lambda}
\longrightarrow1\)
  as 
\(L_{\boldsymbol\lambda}\longrightarrow\infty\),
since \(N\) is now fixed.
	
	Call
\(	\sum_{n=N}^{\infty}
	c_{n;j,i}R_{n;i}^{\boldsymbol\lambda}z^n\)
	the tail, or remainder, of the series for
	\(\widehat W_{j,i}^{\,\boldsymbol\lambda}\) after the first \(N\) terms.
	For \(|z|\le R\), estimate
	\eqref{eq:uniform-confluent-ratio-bound} gives
	\[
	\left|
	\frac{
		c_{n+1;j,i}R_{n+1;i}^{\boldsymbol\lambda}z^{n+1}
	}{
		c_{n;j,i}R_{n;i}^{\boldsymbol\lambda}z^n
	}
	\right|
	\le\frac12,
	\qquad n\ge N.
	\]
	Thus every term of the tail has modulus at most one half of the preceding
	one. The entire tail is therefore bounded by the geometric series
	generated by its first term:
	\[
	\sup_{|z|\le R}
	\left|
	\sum_{n=N}^{\infty}
	c_{n;j,i}R_{n;i}^{\boldsymbol\lambda}z^n
	\right|
	\le
	|c_{N;j,i}R_{N;i}^{\boldsymbol\lambda}|R^N
	\sum_{m=0}^{\infty}2^{-m}
	=
	2|c_{N;j,i}R_{N;i}^{\boldsymbol\lambda}|R^N
	<
	\frac{\varepsilon}{3}.
	\]
	The same argument, now using
	\[
	\left|\frac{c_{n+1;j,i}}{c_{n;j,i}}\right|
	\le
	\frac{C}{(n+1)^d},
	\]
	shows that the tail of the limiting series satisfies
	\[
	\sup_{|z|\le R}
	\left|
	\sum_{n=N}^{\infty}c_{n;j,i}z^n
	\right|
	\le
	2|c_{N;j,i}|R^N
	<
	\frac{\varepsilon}{6}.
	\]
	
	It remains only to compare the first \(N\) terms. Since
	\(R_{n;i}^{\boldsymbol\lambda}\to1\) for each
	\(n\in\{0,\ldots,N-1\}\), and this is a finite set, there exists
	\(L_1\ge L_0\) such that
	\[
	\sup_{|z|\le R}
	\left|
	\sum_{n=0}^{N-1}
	c_{n;j,i}
	\bigl(R_{n;i}^{\boldsymbol\lambda}-1\bigr)z^n
	\right|
	<
	\frac{\varepsilon}{3}
	\]
	whenever \(L_{\boldsymbol\lambda}\ge L_1\). Splitting both series into their
	first \(N\) terms and their respective tails therefore gives
	\[
	\sup_{|z|\le R}
	\left|
	\widehat W_{j,i}^{\,\boldsymbol\lambda}(z)-W_{j,i}(z)
	\right|
	<\varepsilon
	\]
	for all sufficiently large \(L_{\boldsymbol\lambda}\). Since \(R>0\) is
	arbitrary, this proves locally uniform convergence on \(\mathbb C\).
	As the matrix has only finitely many entries, it also proves
	\eqref{eq:direct-compact-open-convergence} for the full matrix.
	
	Finally, the reciprocal relations
	\[
	\widehat{\CB}^{\,\boldsymbol\lambda}(z)
	=
	\frac1z
	\widehat{\WB}^{\,\boldsymbol\lambda}\left(\frac1z\right),
	\qquad
	\CB(z)
	=
	\frac1z\WB\left(\frac1z\right)
	\]
	transfer this convergence to the reciprocal representatives. Indeed, let
	\(K\subset\mathbb C^*\) be compact and set
	\(K^{-1}\coloneq\{z^{-1}:z\in K\}\).
	Since \(K\) is separated from the origin, \(K^{-1}\) is also compact.
	Moreover, for sufficiently large \(\boldsymbol\lambda\),
\(	\sup_{z\in K}\frac1{|z|}
	<
	|\Lambda_{\boldsymbol\lambda}|\),
	so that \(\widehat{\WB}^{\,\boldsymbol\lambda}(z^{-1})\) is represented by its
	convergent power series for every \(z\in K\). Therefore,
	\[
	\sup_{z\in K}
	\left\|
	\widehat{\CB}^{\,\boldsymbol\lambda}(z)-\CB(z)
	\right\|
	\le
	\left(\sup_{z\in K}\frac1{|z|}\right)
	\sup_{\zeta\in K^{-1}}
	\left\|
	\widehat{\WB}^{\,\boldsymbol\lambda}(\zeta)-\WB(\zeta)
	\right\|
	\xlongrightarrow[L_{\boldsymbol\lambda}\to\infty]{}
	0.
	\]
	This proves \eqref{eq:reciprocal-compact-open-convergence}. Taking
	\(K=\Torus\) gives
	\eqref{eq:reciprocal-uniform-circle-convergence}.
\end{proof}

The geometry of the two representatives explains why an entire matrix
appears in the limit.  The possible singular set of
\(\widehat{\CB}^{\,\boldsymbol\lambda}\) is contained in the scaled segment
\(\Delta_{\boldsymbol\lambda}=\Lambda_{\boldsymbol\lambda}^{-1}[0,1]\), which collapses
to the origin.  Under the reciprocal change \(z\mapsto z^{-1}\), the
corresponding cut becomes the ray \(\Lambda_{\boldsymbol\lambda}[1,\infty)\); for
positive \(\Lambda_{\boldsymbol\lambda}\), this is
\([\Lambda_{\boldsymbol\lambda},\infty)\).  Thus the possible finite singularities
of \(\widehat{\WB}^{\,\boldsymbol\lambda}\) escape to infinity, and the locally
uniform limit is the entire matrix \(\WB\).

For every sufficiently large \(\boldsymbol\lambda\), the segment
\(\Delta_{\boldsymbol\lambda}\) is enclosed by \(\Torus\).  Proposition
\ref{prop:interval-contour-Markov-equivalence} gives
\begin{equation}
\label{eq:scaled-Jacobi-circle-pairing}
\int_{\Delta_{\boldsymbol\lambda}}
\mathbf B(y)\,\mathrm d\widehat{\mathcal J}^{\boldsymbol\lambda}(y)\,\mathbf A(y)
=
\oint_{\Torus}
\mathbf B(z)\widehat{\CB}^{\,\boldsymbol\lambda}(z)\mathbf A(z)
\frac{\mathrm dz}{2\pi\mathrm i}.
\end{equation}
By \eqref{eq:reciprocal-uniform-circle-convergence}, for every pair of fixed
polynomial vectors,
\begin{equation}
\label{eq:pairing-convergence-to-Bessel}
\oint_{\Torus}
\mathbf B(z)\widehat{\CB}^{\,\boldsymbol\lambda}(z)\mathbf A(z)
\frac{\mathrm dz}{2\pi\mathrm i}
\longrightarrow
\oint_{\Torus}
\mathbf B(z)\CB(z)\mathbf A(z)
\frac{\mathrm dz}{2\pi\mathrm i}.
\end{equation}
The right-hand side is precisely the Bessel bilinear form used in
\eqref{eq:final-moment-entry}.  Thus it is the complex contour orthogonality,
rather than a limiting interval measure, that survives the confluence.

For completeness, the polynomial vectors associated with the scaled Jacobi
problem are related to the original Jacobi polynomial vectors by
\begin{equation}
	\label{eq:scaled-Jacobi-polynomial-vectors}
	\widehat A_{\boldsymbol n,\boldsymbol m}^{\mathrm J,\boldsymbol\lambda,(i)}(z)
	=
	\left(
	\prod_{h=1}^{d}
	\Gamma(\lambda_h+\kappa_i+1)
	\right)
	A_{\boldsymbol n,\boldsymbol m}^{\,\mathrm J,\boldsymbol\lambda,(i)}
	\bigl(\Lambda_{\boldsymbol\lambda}z\bigr),
	\qquad
	\widehat B_{\boldsymbol n,\boldsymbol m}^{\,\mathrm J,\boldsymbol\lambda,(j)}(z)
	=
	B_{\boldsymbol n,\boldsymbol m}^{\,\mathrm J,\boldsymbol\lambda,(j)}
	\bigl(\Lambda_{\boldsymbol\lambda}z\bigr).
\end{equation}
Indeed, substitution of \eqref{eq:scaled-Jacobi-polynomial-vectors} into the
scaled \(A\)- and \(B\)-orthogonality relations gives, respectively, the
original Jacobi equations multiplied by nonzero scalar factors.

\begin{corollary}[Convergence of the polynomial vectors]
	\label{cor:Jacobi-Bessel-polynomial-confluence}
	Fix an \(A\)- or \(B\)-balanced index pair, and assume that the
	corresponding Bessel problem is weakly normal. Choose a coefficient
	of the Bessel polynomial vector which is nonzero and normalize this
	coefficient to \(1\). Then, for all sufficiently large
	\(L_{\boldsymbol\lambda}\), the same normalization uniquely determines the
	corresponding scaled Jacobi polynomial vector in
	\eqref{eq:scaled-Jacobi-polynomial-vectors}. Under this normalization,
	each of its coefficients converges to the corresponding coefficient of
	the Bessel polynomial vector as
	\(L_{\boldsymbol\lambda}\to\infty\).
\end{corollary}

\begin{proof}
	Write the components of the scaled Jacobi and Bessel polynomial
	vectors in the monomial basis. Since the index pair is fixed, only
	finitely many polynomial coefficients are involved. Substitution of these
	expansions into the defining orthogonality conditions gives finitely many
	linear equations for those coefficients. The coefficients of these
	equations are precisely moments of the corresponding matrix measures.
	
	By \eqref{eq:exact-rescaled-Jacobi-moments}, every scaled Jacobi moment
	appearing in these equations converges to the corresponding Bessel
	moment as \(L_{\boldsymbol\lambda}\to\infty\). Hence all coefficients of the
	finite Jacobi linear equations converge to those of the limiting
	Bessel equations.
	
	Weak normality of the Bessel problem means that its polynomial vector
	is unique up to multiplication by a nonzero scalar. After fixing a
	coefficient which is nonzero in the limiting vector and normalizing it to
	\(1\), the remaining polynomial coefficients are therefore the unique
	solution of a nonsingular square linear system.
	
	The determinant of this limiting system is nonzero. Since the entries of
	the corresponding scaled Jacobi system converge to the limiting entries,
	its determinant also converges to this nonzero value and is therefore
	nonzero for all sufficiently large \(L_{\boldsymbol\lambda}\). Thus the same
	normalization uniquely determines the scaled Jacobi polynomial vector.
	Finally, Cramer's rule shows that each of its coefficients converges to
	the corresponding coefficient of the normalized Bessel polynomial
	vector.
\end{proof}

\subsubsection{Disappearance of the interval measure and maximal rank}

The preceding result must not be interpreted as weak convergence of the
rank-one interval measures.  Suppose that, along a path on which the measures
in \eqref{eq:scaled-Jacobi-measure-functional-definition} are finite, they
converged weakly to a finite matrix measure \(\mathrm d\nu\).  Since
\(\Delta_{\boldsymbol\lambda}=\Lambda_{\boldsymbol\lambda}^{-1}[0,1]\to\{0\}\), the
limiting measure would be supported at the origin and hence
\(\int y^n\,\mathrm d\nu(y)=0\) for every \(n\ge1\).  On the other hand,
\eqref{eq:rescaled-Jacobi-moment-convergence} gives
\begin{equation*}
\int_{\Delta_{\boldsymbol\lambda}}
 y^n\,\mathrm d\widehat{\mathcal J}_{j,i}^{\boldsymbol\lambda}(y)
\longrightarrow
\frac{\Gamma((n+\kappa_i)\one_r+\boldsymbol a+\one_r)}
{\Gamma((n+\kappa_i)\one_q+\boldsymbol b+\boldsymbol e_j+\one_q)},
\end{equation*}
which is nonzero on the admissible regular parameter domain.  Hence no such finite
measure limit exists.

Accordingly, the confluence has the following structure:
\begingroup
\small
\[
\begin{array}{ccc}
\text{rank-one scaled Jacobi matrix measure on }\Delta_{\boldsymbol\lambda}
&
\overset{\textnormal{Cauchy}}{\longleftrightarrow}
&
\text{contour weight }
\widehat{\CB}^{\,\boldsymbol\lambda}(z)\dfrac{\mathrm dz}{2\pi\mathrm i}
\\[2mm]
\text{no limit as a finite matrix measure}
&&
\Big\downarrow\textnormal{ locally uniform convergence}
\\[2mm]
&&
\text{Bessel contour weight }
\CB(z)\dfrac{\mathrm dz}{2\pi\mathrm i}.
\end{array}
\]
\endgroup
The horizontal arrow is an exact equality of polynomial bilinear forms, not
an asymptotic correspondence.  The vertical arrow is the analytic confluence
established in Theorem~\ref{thm:Jacobi-Bessel-Markov-confluence}.

The maximal-rank assertion of Proposition~\ref{prop:final-rank-circle} is
therefore a structural part of the limiting orthogonality.  Before
confluence, the interval density has pointwise rank one.  That interval
realization has no limiting matrix measure.  What survives is the unit-circle
contour representative \(\CB\), or equivalently its entire reciprocal
representative \(\WB\), and these matrices have generic maximal rank
\(\min\{p,q\}\).  Thus the limiting orthogonality is genuinely matrix-valued
on the circle and cannot be represented there by a single-valued pointwise
factorization \(C_{j,i}(z)=v_j(z)u_i(z)\) or
\(W_{j,i}(z)=\widetilde v_j(z)\widetilde u_i(z)\).  The increase of rank is
not produced by the reciprocal change between \(\WB\) and \(\CB\), which
contain exactly the same coefficient sequence.  The rank-one interval measure
has no finite measure limit, whereas its equivalent complex-contour
orthogonality converges to a nondegenerate matrix of maximal rank.

When \(q=1\), necessarily \(r=0\) and \(d=1\).  The construction above
reduces to the Markov--Stieltjes form of the
Jacobi--Pi\~neiro-to-multiple-Bessel confluence described by Aptekarev,
Branquinho and Van Assche~\cite{ABV2003}.  The present \(q\times p\)
construction shows that, beyond this one-row boundary, the limiting complex
orthogonality is represented by a matrix weight of maximal rank.

The Jacobi confluence allows some properties of each fixed Bessel
mixed orthogonality problem to be obtained by passage to the limit. It does
not, however, provide all the results developed below. The confluence will
therefore be used as an independent interpretation of the Bessel
system, while deriving the explicit polynomial vectors, normality,
recurrences, and factorizations directly from the hypergeometric Bessel
matrix of weights and its reciprocal-Gamma moments.

\section{Hypergeometric representation}
\label{sec:final-hypergeometric-bessel-forms}

The preceding section defines the Bessel matrix of weights and the
corresponding moment functionals.  The associated mixed orthogonality
problems are solved for the balanced index pairs needed below.  The
\(A\)-components are terminating generalized hypergeometric polynomials,
whereas the \(B\)-components are finite combinations of terminating
hypergeometric blocks.

\subsection{Explicit polynomial vectors and mixed orthogonality}
\label{subsec:explicit-polynomial-vectors}
	
	For \(\boldsymbol n\in\N_0^p\) and \(\boldsymbol m\in\N_0^q\), write
	\[
	\mathbf A_{\boldsymbol n,\boldsymbol m}(z)
	=
	\begin{bNiceMatrix}
		A_{\boldsymbol n,\boldsymbol m}^{(1)}(z)\\ \Vdots\\ A_{\boldsymbol n,\boldsymbol m}^{(p)}(z)
	\end{bNiceMatrix},
	\qquad
	\mathbf B_{\boldsymbol n,\boldsymbol m}(z)
	=
	\begin{bNiceMatrix}
		B_{\boldsymbol n,\boldsymbol m}^{(1)}(z)&\Cdots&B_{\boldsymbol n,\boldsymbol m}^{(q)}(z)
	\end{bNiceMatrix}.
	\]
	The pairing used in this section is the general pairing
	\eqref{eq:general-pairing} with the weight matrix \(W\) replaced by the
	reciprocal representative \(\CB\).  Denote the bimoment array of this
	pairing by \(\MB\).  Its entries are
	\begin{equation}
		\label{eq:final-moment-entry}
		\MB_{(u,j),(v,i)}
		=
		\oint_{\Torus}
		z^{u+v}C_{j,i}(z)\frac{\dz}{2\pi\mathrm i}
		=
		\frac{
			\Gamma((u+v+1+\kappa_i)\one_r+\boldsymbol a)
		}{
			\Gamma((u+v+1+\kappa_i)\one_q+\boldsymbol b+\boldsymbol e_j)
		},
		\qquad u,v\in\N_0.
	\end{equation}
	The right vector polynomial has component constraints
\(A_{\boldsymbol n,\boldsymbol m}^{(i)}\in\Poly_{n_i-1},\) \(i\in\{1,\ldots,p\},\)
	and satisfies
	\begin{equation}
		\label{eq:final-A-orth}
		\sum_{i=1}^{p}
		\oint_{\Torus}
		z^{\ell}C_{j,i}(z)
		A_{\boldsymbol n,\boldsymbol m}^{(i)}(z)
		\frac{\dz}{2\pi\mathrm i}=0,
		\qquad
		\ell\in\{0,\ldots,m_j-1\},
		\quad j\in\{1,\ldots,q\}.
	\end{equation}
	The left vector polynomial has component constraints
\(B_{\boldsymbol n,\boldsymbol m}^{(j)}\in\Poly_{m_j-1},\) \(j\in\{1,\ldots,q\},\)
	and satisfies
	\begin{equation}
		\label{eq:final-B-orth}
		\sum_{j=1}^{q}
		\oint_{\Torus}
		z^{\ell}B_{\boldsymbol n,\boldsymbol m}^{(j)}(z)
		C_{j,i}(z)
		\frac{\dz}{2\pi\mathrm i}=0,
		\qquad
		\ell\in\{0,\ldots,n_i-1\},
		\quad i\in\{1,\ldots,p\}.
	\end{equation}

For a vector \(\boldsymbol c=(c_1,\ldots,c_d)\), write \(\boldsymbol c^{\,*i}\) for
the vector obtained by removing its \(i\)-th component.

\medskip
\noindent\textbf{Definition (Near-diagonal row multi-indices).}
A multi-index \(\boldsymbol m\in\N_0^q\) is said to be \emph{near the diagonal} if
\begin{equation}
	\label{eq:final-near-diagonal-row-index}
	\left|m_j-m_h\right|\le1,
	\qquad j,h\in\{1,\ldots,q\}.
\end{equation}
For such a multi-index, the following elementary extrema will be used:
\begin{equation}
	\label{eq:final-row-active-indices}
	j_{\min}(\boldsymbol m)
	\coloneq
	\min\{j\in\{1,\ldots,q\}:m_j=\min_h m_h\},
	\qquad
	m_{\min}(\boldsymbol m)
	\coloneq
	m_{j_{\min}(\boldsymbol m)},
\end{equation}
\begin{equation}
	\label{eq:final-row-extreme-indices}
	j_{\max}(\boldsymbol m)
	\coloneq
	\max\{j\in\{1,\ldots,q\}:m_j=\max_h m_h\},
	\qquad
	m_{\max}(\boldsymbol m)
	\coloneq
	m_{j_{\max}(\boldsymbol m)}.
\end{equation}
Thus \(m_{\min}(\boldsymbol m)=\min_h m_h\) and
\(m_{\max}(\boldsymbol m)=\max_h m_h\).  On the step-line these extrema coincide
with the rows selected by the next increase and by the monic component,
respectively.  The near-diagonal formulas below use only
\(j_{\min},m_{\min},j_{\max},m_{\max}\).
The formula for the \(B\)-polynomial vector below uses the finite Bessel reconstruction of simple
fractions in the Gamma-quotient moment basis. The coefficients
\(\pi_{H,K}\), defined later in \eqref{eq:final-B-pi}, are partial-fraction data
combined with terminating hypergeometric polynomial blocks. The contribution
corresponding to \(K=0\) collapses to an explicit constant term. This is the
Bessel-type analogue of the polynomial reconstruction used by Wolfs in
the one-column case: the same partial-fraction coefficients are combined
with hypergeometric polynomial blocks instead of monomials.

The explicit near-diagonal representatives are normalized as follows.
On the \(A\)-side, normalization is by the first moment beyond the imposed
orthogonality in the row \(j_{\min}\); on the \(B\)-side, it is fixed by
requiring the component \(B^{(j_{\max})}\) to be monic.  After specialization to the
balanced step-line, this gives the monic \(B\)-vector supplied by the lower
unitriangular Gauss--Borel factor and the dual \(A\)-vector with the diagonal
normalization absorbed on the \(A\)-side.  For general near-diagonal pairs this
is the corresponding extension of the same scalar choice.

\begin{theorem}[Explicit \(A\)- and \(B\)-polynomial vectors]
	\label{thm:final-explicit}
	Let \(\boldsymbol n\in\N_0^p\) and \(\boldsymbol m\in\N_0^q\). Assume that the
	Bessel parameters are regular in the sense of
	Definition~\ref{def:final-parameter-regularity} and simultaneously
	admissible for the formulas displayed below.
	
	\begin{enumerate}[label=\textnormal{(\roman*)}]
		
		\item
		Assume that \((\boldsymbol n,\boldsymbol m)\) is \(A\)-balanced in the sense of
		Definition~\ref{def:balanced-normality}, and that \(\boldsymbol m\)
		is near the diagonal in the sense of
		\eqref{eq:final-near-diagonal-row-index}. For every
		\(i\in\{1,\ldots,p\}\) with \(n_i\ge1\), set
		\begin{equation}
			\label{eq:final-A-prefactor}
			C_{\boldsymbol n,\boldsymbol m;i}
			\coloneq
			(-1)^{|\boldsymbol n|}
			\frac{\Gamma(\kappa_i\one_q+\boldsymbol b+\one_q)}
			{\Gamma(\kappa_i\one_r+\boldsymbol a+\one_r)}
			\frac{
				\prod_{j=1}^{q}(\kappa_i+b_j+1)_{m_j}
			}{
				(-1)^{n_i-1}(n_i-1)!
				\prod_{\substack{\ell=1\\\ell\ne i}}^{p}
				(\kappa_i-\kappa_\ell-n_\ell+1)_{n_\ell}
			}.
		\end{equation}
		For every nonzero scalar \(\nu\), define
		\begingroup
		\small
		\begin{equation}
			\label{eq:final-A-components}
			\begin{aligned}
			A_{\boldsymbol n,\boldsymbol m}^{(i)}(z)
			&=
			\nu\,
			C_{\boldsymbol n,\boldsymbol m;i}\\
			&\quad\times
			\pFq{q+p}{r+p-1}
			{
				-n_i+1,\ 
				\kappa_i\one_q+\boldsymbol b+\boldsymbol m+\one_q,\
				\substack{
				\kappa_i\one_{p-1}-\boldsymbol\kappa^{\,*i}\\
				{}-\boldsymbol n^{\,*i}+\one_{p-1}}
			}
			{
				\kappa_i\one_r+\boldsymbol a+\one_r,\
				\substack{
				\kappa_i\one_{p-1}-\boldsymbol\kappa^{\,*i}\\
				{}+\one_{p-1}}
			}
			{z},
			\end{aligned}
		\end{equation}
		\endgroup
		for \(i\in\{1,\ldots,p\}\) with \(n_i\ge1\). If \(n_i=0\), set
		\(A_{\boldsymbol n,\boldsymbol m}^{(i)}\equiv0\). For \(n_i\ge1\), the
		hypergeometric series terminates because of the numerator parameter
		\(-n_i+1\), and hence
		\begin{equation}
			\label{eq:final-A-degree}
			\deg A_{\boldsymbol n,\boldsymbol m}^{(i)}\le n_i-1.
		\end{equation}
		Then the polynomials \(A_{\boldsymbol n,\boldsymbol m}^{(i)}\) satisfy the
		\(A\)-orthogonality conditions \eqref{eq:final-A-orth}, namely
		\begin{equation}
			\sum_{i=1}^{p}
			\oint_{\Torus}
			z^\ell
			C_{j,i}(z)
			A_{\boldsymbol n,\boldsymbol m}^{(i)}(z)
			\frac{\dz}{2\pi\mathrm i}
			=0,
			\qquad
			j\in\{1,\ldots,q\},
			\qquad
			\ell\in\{0,\ldots,m_j-1\}.
		\end{equation}
		
		To fix the scalar normalization, write
		\begin{equation}
			\label{eq:final-A-min-data}
			j_{\min}\coloneq j_{\min}(\boldsymbol m),
			\qquad
			m_{\min}\coloneq m_{\min}(\boldsymbol m).
		\end{equation}
		If \(\nu\) is chosen as
		\begin{equation}
			\label{eq:final-A-nu}
			\nu_{\boldsymbol n,\boldsymbol m}^{A}
			\coloneq
			-
			\frac{
				\prod_{i=1}^{p}
				(\kappa_i+b_{j_{\min}}+m_{\min}+1)_{n_i}
			}{
				(\boldsymbol a-b_{j_{\min}}\one_r-m_{\min}\one_r)_{m_{\min}}
				\prod_{\substack{h=1\\m_h=m_{\min}+1}}^{q}
				(b_h-b_{j_{\min}})
			},
		\end{equation}
		then the normalized vector satisfies
		\begin{equation}
			\label{eq:final-A-normalizing-moment}
			\sum_{i=1}^{p}
			\oint_{\Torus}
			z^{m_{\min}}
			C_{j_{\min},i}(z)
			A_{\boldsymbol n,\boldsymbol m}^{(i)}(z)
			\frac{\dz}{2\pi\mathrm i}
			=1.
		\end{equation}
		
		\item
		Assume that \((\boldsymbol n,\boldsymbol m)\) is \(B\)-balanced in the sense of
		Definition~\ref{def:balanced-normality}, and that \(\boldsymbol m\)
		is near the diagonal in the sense of
		\eqref{eq:final-near-diagonal-row-index}. For
		\(H\in\{1,\ldots,q\}\) and \(K\in\{0,\ldots,m_H-1\}\), set
		\begin{equation}
			\label{eq:final-B-pi}
			\pi_{H,K}
			\coloneq
			-
			\frac{
				\prod_{i=1}^{p}(\kappa_i+b_H+K+1)_{n_i}
			}{
				(-1)^K K!(m_H-K-1)!
				\prod_{\substack{I=1\\I\ne H}}^{q}
				(b_I-b_H-K)_{m_I}
			}.
		\end{equation}
		For each \(j\in\{1,\ldots,q\}\) with \(m_j\ge1\), the specialization
		\(H=j\), \(K=0\) gives
		\begin{equation}
			\label{eq:final-B-pi-zero}
			\pi_{j,0}
			=
			-
			\frac{
				\prod_{i=1}^{p}(\kappa_i+b_j+1)_{n_i}
			}{
				(m_j-1)!
				\prod_{\substack{I=1\\I\ne j}}^{q}
				(b_I-b_j)_{m_I}
			}.
		\end{equation}
		When \(m_j=0\), set \(\pi_{j,0}\coloneq0\).
		
		For every nonzero scalar \(\nu\) and every
		\(j\in\{1,\ldots,q\}\) with \(m_j\ge1\), define
		\begin{multline}
			\label{eq:final-B-components}
			B_{\boldsymbol n,\boldsymbol m}^{(j)}(z)
			=
			\nu
			\Biggl[
			\pi_{j,0}
			+
			\frac{(\boldsymbol a-b_j\one_r)_1}
			{(\boldsymbol b^{\,*j}-b_j\one_{q-1})_1}
			\sum_{H=1}^{q}
			\sum_{K=1}^{m_H-1}
			\pi_{H,K}
			\frac{
				(\boldsymbol b^{\,*j}-(b_H+K)\one_{q-1})_1
			}{
				(\boldsymbol a-(b_H+K)\one_r)_1
			}
			\\
			\times
			\pFq{q+1}{r}
			{
				1,\ \boldsymbol b-(b_H+K)\one_q+\one_q-\boldsymbol e_j
			}
			{
				\boldsymbol a-(b_H+K)\one_r+\one_r
			}
			{z}
			\Biggr].
		\end{multline}
		If \(m_j=0\), set \(B_{\boldsymbol n,\boldsymbol m}^{(j)}\equiv0\). For every
		\(j\in\{1,\ldots,q\}\) with \(m_j\ge1\), one has
		\begin{equation}
			\label{eq:final-B-degree}
			\deg B_{\boldsymbol n,\boldsymbol m}^{(j)}\le m_j-1.
		\end{equation}
		The components \(B_{\boldsymbol n,\boldsymbol m}^{(j)}\) satisfy the
		\(B\)-orthogonality conditions \eqref{eq:final-B-orth}, namely
		\begin{equation}
			\sum_{j=1}^{q}
			\oint_{\Torus}
			z^\ell
			B_{\boldsymbol n,\boldsymbol m}^{(j)}(z)
			C_{j,i}(z)
			\frac{\dz}{2\pi\mathrm i}
			=0,
			\qquad
			i\in\{1,\ldots,p\},
			\qquad
			\ell\in\{0,\ldots,n_i-1\}.
		\end{equation}
		
		To fix the monic normalization, write
		\begin{equation}
			\label{eq:final-B-max-data}
			j_{\max}\coloneq j_{\max}(\boldsymbol m),
			\qquad
			m_{\max}\coloneq m_{\max}(\boldsymbol m).
		\end{equation}
		If \(\nu\) is chosen as
		\begin{equation}
			\label{eq:final-B-nu}
			\nu_{\boldsymbol n,\boldsymbol m}^{B}
			\coloneq
			-
			\frac{
				(\boldsymbol a-b_{j_{\max}}\one_r-(m_{\max}-1)\one_r)_{m_{\max}}
				\prod_{\substack{h=1\\h\ne j_{\max},\ m_h=m_{\max}}}^{q}
				(b_h-b_{j_{\max}})
			}{
				(\boldsymbol a-b_{j_{\max}}\one_r)_1
				\prod_{i=1}^{p}
				(\kappa_i+b_{j_{\max}}+m_{\max})_{n_i}
			},
		\end{equation}
		then the component \(B_{\boldsymbol n,\boldsymbol m}^{(j_{\max})}\) is monic.
		After step-line specialization, this is the monic \(B\)-vector
		supplied by the lower unitriangular Gauss--Borel factor.
		
	\end{enumerate}
\end{theorem}

The proof of the \(B\)-part of
Theorem~\ref{thm:final-explicit} requires two auxiliary rational
identities. The first one identifies the coefficients
\(\pi_{H,K}\) as the residues in the partial-fraction decomposition of
the rational function that encodes the \(B\)-orthogonality conditions.

\begin{lemma}[Partial-fraction decomposition]
	\label{lem:final-B-partial-fractions}
	Assume that \((\boldsymbol n,\boldsymbol m)\) is \(B\)-balanced and that the
	Bessel parameters are regular. Consider
	\[
	R(t) \coloneq - \frac{ \prod_{i=1}^{p}(\kappa_i+1-t)_{n_i} }{ \prod_{j=1}^{q}(t+b_j)_{m_j} }.
	\]
	Then
	\begin{equation}
		\label{eq:final-B-partial-fractions}
		R(t)
		=
		\sum_{H=1}^{q}\sum_{K=0}^{m_H-1}
		\frac{\pi_{H,K}}{t+b_H+K}.
	\end{equation}
\end{lemma}

\begin{proof}
	By the regularity assumptions in
	Definition~\ref{def:final-parameter-regularity}, one has
	\(b_H-b_I\notin\mathbb Z\) whenever \(H\ne I\). Hence the points
	\(-b_H-K\), with \(H\in\{1,\ldots,q\}\) and
	\(K\in\{0,\ldots,m_H-1\}\), are pairwise distinct. Indeed,
	\(-b_H-K=-b_I-L\) would imply \(b_H-b_I=L-K\in\mathbb Z\).
	If \(H\ne I\), this contradicts regularity, while if \(H=I\), it
	implies \(K=L\).
	
	Since \(|\boldsymbol m|=|\boldsymbol n|+1\), the degree of the denominator of
	\(R(t)\) is one greater than the degree of its numerator. Hence
	\(R(t)\) is a proper rational function and its partial-fraction
	decomposition has no polynomial part. Since all its poles are simple,
	there exist coefficients \(c_{H,K}\) such that
	\[
	R(t) = \sum_{H=1}^{q}\sum_{K=0}^{m_H-1} \frac{c_{H,K}}{t+b_H+K}.
	\]
	The coefficient corresponding to the pole \(t=-b_H-K\) is its residue:
	\begin{align*}
		c_{H,K}
		&=
		\operatorname*{Res}_{t=-b_H-K}R(t)
		\\
		&=
		\lim_{t\to-b_H-K}(t+b_H+K)R(t)
		\\
		&=
		-
		\frac{
			\prod_{i=1}^{p}
			(\kappa_i+b_H+K+1)_{n_i}
		}{
			\prod_{\substack{L=0\\L\ne K}}^{m_H-1}(L-K)
			\prod_{\substack{I=1\\I\ne H}}^{q}
			(b_I-b_H-K)_{m_I}
		}.
	\end{align*}
	Indeed, after the factor \(t+b_H+K\) is removed from the \(H\)-th
	denominator block \((t+b_H)_{m_H}\), evaluation at
	\(t=-b_H-K\) gives
	\[
	\prod_{\substack{L=0\\L\ne K}}^{m_H-1}(L-K)
	=
	(-1)^K K!(m_H-K-1)!.
	\]
	Therefore
	\[
	c_{H,K}
	=
	-
	\frac{
		\prod_{i=1}^{p}
		(\kappa_i+b_H+K+1)_{n_i}
	}{
		(-1)^K K!(m_H-K-1)!
		\prod_{\substack{I=1\\I\ne H}}^{q}
		(b_I-b_H-K)_{m_I}
	}
	=
	\pi_{H,K}.
	\]
	Consequently,
	\eqref{eq:final-B-partial-fractions} follows.
\end{proof}

The second auxiliary identity reconstructs each simple fraction occurring
in \eqref{eq:final-B-partial-fractions} as a finite linear combination of
the rational terms that arise from the moments of the hypergeometric
components in \eqref{eq:final-B-components}. This is the step that makes
it possible to pass from the partial-fraction decomposition to the
polynomial identity used in the \(B\)-orthogonality proof.

\begin{lemma}[Finite reconstruction of a simple fraction]
	\label{lem:final-B-simple-fraction-reconstruction}
	Fix \(H\in\{1,\ldots,q\}\) and
	\(K\in\{0,\ldots,m_H-1\}\), and set \(c\coloneq b_H+K\). Then
	\begin{equation}
		\label{eq:final-B-simple-fraction-reconstruction}
		\begin{aligned}
		\frac{1}{t+c}
		&=
		\sum_{j=1}^{q}
		\frac{(\boldsymbol a-b_j\one_r)_1}
		{(\boldsymbol b^{\,*j}-b_j\one_{q-1})_1}
		\frac{(\boldsymbol b^{\,*j}-c\one_{q-1})_1}
		{(\boldsymbol a-c\one_r)_1}\\
		&\quad\times
		\sum_{d=0}^{K-1+\delta_{j,H}}
		\frac{
			(\boldsymbol b-c\one_q+\one_q-\boldsymbol e_j)_d
		}{
			(\boldsymbol a-c\one_r+\one_r)_d
		}
		\frac{(t\one_r+\boldsymbol a)_d}
		{(t\one_q+\boldsymbol b)_d(t+b_j+d)}.
		\end{aligned}
	\end{equation}
\end{lemma}

\begin{proof}
	To prove this identity, introduce
	\[
	\omega_j
	\coloneq
	\frac{(\boldsymbol a-b_j\one_r)_1}
	{(\boldsymbol b^{\,*j}-b_j\one_{q-1})_1},
	\qquad
	\varrho(x)
	\coloneq
	\frac{(x\one_r+\boldsymbol a)_1}{(x\one_q+\boldsymbol b)_1}.
	\]
	First, the following barycentric identity is established:
	\begin{equation}
		\label{eq:final-B-barycentric-identity}
		\sum_{j=1}^{q}
		\omega_j
		\frac{
			(\boldsymbol b^{\,*j}-y\one_{q-1})_1
		}{
			(\boldsymbol a-y\one_r)_1
		}
		\frac{1}{x+b_j}
		=
		\frac{1-\varrho(x)/\varrho(-y)}{x+y}.
	\end{equation}
	Assume first that the displayed denominators are nonzero. The apparent
	singularity of the right-hand side at \(x=-y\) is removable. Apart from
	this point, both sides are proper rational functions of \(x\), with
	possible poles only at \(-b_1,\ldots,-b_q\). Since
	\(\omega_j=\operatorname*{Res}_{x=-b_j}\varrho(x)\), the residue of the
	right-hand side at \(x=-b_j\) is
	\(-\frac{\omega_j}{\varrho(-y)(y-b_j)}\).
	Now
	\[
	\varrho(-y)
	=
	\frac{(\boldsymbol a-y\one_r)_1}{(\boldsymbol b-y\one_q)_1},
	\qquad
	(\boldsymbol b-y\one_q)_1
	=
	-(y-b_j)(\boldsymbol b^{\,*j}-y\one_{q-1})_1,
	\]
	and therefore
	\[
	-\frac{\omega_j}{\varrho(-y)(y-b_j)}
	=
	\omega_j
	\frac{
		(\boldsymbol b^{\,*j}-y\one_{q-1})_1
	}{
		(\boldsymbol a-y\one_r)_1
	}.
	\]
	This is precisely the residue of the left-hand side at \(x=-b_j\).
	Consequently, the difference between the two sides has no finite poles.
	Since \(r<q\), one has \(\varrho(x)\to0\) as \(x\to\infty\), and both
	sides tend to zero. Their difference therefore vanishes identically,
	which proves \eqref{eq:final-B-barycentric-identity} whenever the
	displayed denominators are nonzero.
	
	The identity \eqref{eq:final-B-barycentric-identity} will also be applied when
	\(y=b_H\). Its two sides are therefore evaluated as
	\(y\to b_H\). Consider first the left-hand side. If \(j\ne H\), then
	\(b_H\) remains among the components of \(\boldsymbol b^{\,*j}\), so
	\((\boldsymbol b^{\,*j}-y\one_{q-1})_1\) contains the factor \(b_H-y\).
	All the other factors in that summand remain finite by regularity and the
	standing admissibility convention, and
	hence the summand tends to zero. For \(j=H\), the coefficient tends to
	\[
	\omega_H
	\frac{
		(\boldsymbol b^{\,*H}-b_H\one_{q-1})_1
	}{
		(\boldsymbol a-b_H\one_r)_1
	}
	=1.
	\]
	Thus the left-hand side tends to \(1/(x+b_H)\).
	
	For the right-hand side, one has
	\[
	\frac{1}{\varrho(-y)}
	=
	\frac{(\boldsymbol b-y\one_q)_1}{(\boldsymbol a-y\one_r)_1}.
	\]
	Its numerator contains the factor \(b_H-y\), whereas its denominator
	is finite and nonzero at \(y=b_H\) by the standing admissibility convention. Hence
	\(1/\varrho(-y)\to0\), and the right-hand side also tends to
	\(1/(x+b_H)\). Therefore, when \(y=b_H\) occurs below,
	\eqref{eq:final-B-barycentric-identity} remains valid after taking this
	limit.
	
	Return to
	\eqref{eq:final-B-simple-fraction-reconstruction}. Its right-hand side
	is a finite nested sum: one first sums over \(j\), and, for each fixed
	\(j\), one then sums over the index \(d\) in
	\(\sum_{d=0}^{K-1+\delta_{j,H}}\).
	The upper limit of the sum over \(d\) depends on \(j\). More precisely,
	\[
	K-1+\delta_{j,H}
	=
	\begin{cases}
		K,   & j=H,\\
		K-1, & j\ne H.
	\end{cases}
	\]
	Thus the term with \(d=K\) is present from the outset when \(j=H\), but
	it is not present when \(j\ne H\).
	
	The purpose is to reverse the order in which the two finite sums are
	performed. The next step is to fix \(d\), collect the terms corresponding to all
	\(j\in\{1,\ldots,q\}\), and apply the barycentric identity to the
	resulting sum over \(j\). To do this, the sum over \(d\) must have the
	same range for every \(j\). It is therefore enough to show that all these sums may
	be written over \(d\in\{0,\ldots,K\}\) without changing their values.
	
	Suppose first that \(K\ge1\) and \(j\ne H\). Replacing the original
	range \(d\in\{0,\ldots,K-1\}\) by \(d\in\{0,\ldots,K\}\) introduces
	only the term with \(d=K\). This additional term is zero. Since
	\(c=b_H+K\) and \(j\ne H\), the \(H\)-th component of
	\(\boldsymbol b-c\one_q+\one_q-\boldsymbol e_j\) is
	\(b_H-c+1=b_H-(b_H+K)+1=1-K\).
	Consequently, the vectorial Pochhammer symbol
	\((\boldsymbol b-c\one_q+\one_q-\boldsymbol e_j)_K\) contains the scalar factor
	\((1-K)_K=(1-K)(2-K)\cdots(-1)\,0=0\).
	Hence the term with \(d=K\) vanishes.
	
	Consider now \(K=0\) and \(j\ne H\). The original upper limit is then
	\(-1\), so the original sum over \(d\) contains no terms. If it is written
	over the common range \(d\in\{0\}\), the only possible contribution
	is the term with \(d=0\). All Pochhammer symbols of order zero in this
	term are equal to one, but the prefactor contains
\(	(\boldsymbol b^{\,*j}-c\one_{q-1})_1
	=
	(\boldsymbol b^{\,*j}-b_H\one_{q-1})_1\).
	Because \(j\ne H\), deleting the \(j\)-th component from \(\boldsymbol b\)
	does not delete the component \(b_H\). The displayed product therefore
	contains the factor \(b_H-b_H=0\), and the term with \(d=0\) also
	vanishes.
	
	It follows that, for every \(j\), the sum over \(d\) in
	\eqref{eq:final-B-simple-fraction-reconstruction} may be written over
	the common range \(d\in\{0,\ldots,K\}\). Since all sums are finite, the terms may now be regrouped by first fixing \(d\) and then summing over
	\(j\in\{1,\ldots,q\}\).
	
	For \(d\in\{0,\ldots,K+1\}\), define
	\[
	M_d(c)
	\coloneq
	\frac{(\boldsymbol b-c\one_q)_d}{(\boldsymbol a-c\one_r)_d},
	\qquad
	F_d(t)
	\coloneq
	\frac{(t\one_r+\boldsymbol a)_d}{(t\one_q+\boldsymbol b)_d}.
	\]
	Fix \(d\in\{0,\ldots,K\}\). For each
	\(j\in\{1,\ldots,q\}\), the term corresponding to the pair \((j,d)\)
	on the right-hand side of
	\eqref{eq:final-B-simple-fraction-reconstruction} is
	\[
	\omega_j
	\frac{(\boldsymbol b^{\,*j}-c\one_{q-1})_1}
	{(\boldsymbol a-c\one_r)_1}
	\frac{(\boldsymbol b-c\one_q+\one_q-\boldsymbol e_j)_d}
	{(\boldsymbol a-c\one_r+\one_r)_d}
	\frac{(t\one_r+\boldsymbol a)_d}
	{(t\one_q+\boldsymbol b)_d(t+b_j+d)}.
	\]
	The quotient containing \(t\) is
	\(\frac{(t\one_r+\boldsymbol a)_d}{(t\one_q+\boldsymbol b)_d}=F_d(t)\),
	so the preceding term can be written as
	\[
	\omega_j
	\frac{(\boldsymbol b^{\,*j}-c\one_{q-1})_1}
	{(\boldsymbol a-c\one_r)_1}
	\frac{(\boldsymbol b-c\one_q+\one_q-\boldsymbol e_j)_d}
	{(\boldsymbol a-c\one_r+\one_r)_d}
	\frac{F_d(t)}{t+b_j+d}.
	\]
	Next, rewrite the remaining coefficient in the precise form required
	by the barycentric identity.
	
	The \(j\)-th component of
	\(\boldsymbol b-c\one_q+\one_q-\boldsymbol e_j\) is \(b_j-c\), because the \(j\)-th
	component of \(\one_q-\boldsymbol e_j\) is zero. For \(h\ne j\), the
	\(h\)-th component is \(b_h-c+1\). Hence
	\[
	(\boldsymbol b-c\one_q+\one_q-\boldsymbol e_j)_d
	=
	(b_j-c)_d
	(\boldsymbol b^{\,*j}-c\one_{q-1}+\one_{q-1})_d.
	\]
	Substituting this identity into the coefficient gives
	\begin{equation*}
		\frac{(\boldsymbol b^{\,*j}-c\one_{q-1})_1}
		{(\boldsymbol a-c\one_r)_1}
		\frac{(\boldsymbol b-c\one_q+\one_q-\boldsymbol e_j)_d}
		{(\boldsymbol a-c\one_r+\one_r)_d}
		=
		\frac{
			(b_j-c)_d
			(\boldsymbol b^{\,*j}-c\one_{q-1})_1
			(\boldsymbol b^{\,*j}-c\one_{q-1}+\one_{q-1})_d
		}{
			(\boldsymbol a-c\one_r)_1
			(\boldsymbol a-c\one_r+\one_r)_d
		}.
	\end{equation*}
	For a scalar parameter \(x\), one has
	\((x)_1(x+1)_d=(x)_{d+1}\).
	Applying this identity componentwise to the vectorial Pochhammer
	symbols in the numerator and denominator of the right-hand side yields
	\[
	\frac{
		(b_j-c)_d
		(\boldsymbol b^{\,*j}-c\one_{q-1})_{d+1}
	}{
		(\boldsymbol a-c\one_r)_{d+1}
	}.
	\]
	Split each Pochhammer symbol of order \(d+1\) into a Pochhammer
	symbol of order \(d\) and its last factor. Thus
	\begin{align*}
	(\boldsymbol b^{\,*j}-c\one_{q-1})_{d+1}
	&=
	(\boldsymbol b^{\,*j}-c\one_{q-1})_d
	(\boldsymbol b^{\,*j}-c\one_{q-1}+d\one_{q-1})_1,
	\\
	(\boldsymbol a-c\one_r)_{d+1}
	&=
	(\boldsymbol a-c\one_r)_d
	(\boldsymbol a-c\one_r+d\one_r)_1.
	\end{align*}
	Since
	\[
	\boldsymbol b^{\,*j}-c\one_{q-1}+d\one_{q-1}
	=
	\boldsymbol b^{\,*j}-(c-d)\one_{q-1},
	\qquad
	\boldsymbol a-c\one_r+d\one_r
	=
	\boldsymbol a-(c-d)\one_r,
	\]
	the preceding quotient becomes
	\[
	\frac{
		(b_j-c)_d
		(\boldsymbol b^{\,*j}-c\one_{q-1})_d
	}{
		(\boldsymbol a-c\one_r)_d
	}
	\frac{
		(\boldsymbol b^{\,*j}-(c-d)\one_{q-1})_1
	}{
		(\boldsymbol a-(c-d)\one_r)_1
	}.
	\]
	
	The first quotient can now be identified with \(M_d(c)\). Indeed,
	\((b_j-c)_d\) is the Pochhammer symbol associated with the \(j\)-th
	component of \(\boldsymbol b-c\one_q\), whereas
	\((\boldsymbol b^{\,*j}-c\one_{q-1})_d\) is the product of the Pochhammer
	symbols associated with all the remaining components. Therefore
	\[
	(b_j-c)_d
	(\boldsymbol b^{\,*j}-c\one_{q-1})_d
	=
	(\boldsymbol b-c\one_q)_d,
	\]
	and hence
	\[
	\frac{
		(b_j-c)_d
		(\boldsymbol b^{\,*j}-c\one_{q-1})_d
	}{
		(\boldsymbol a-c\one_r)_d
	}
	=
	M_d(c).
	\]
	This proves that
	\[
	\frac{(\boldsymbol b^{\,*j}-c\one_{q-1})_1}
	{(\boldsymbol a-c\one_r)_1}
	\frac{(\boldsymbol b-c\one_q+\one_q-\boldsymbol e_j)_d}
	{(\boldsymbol a-c\one_r+\one_r)_d}
	=
	M_d(c)
	\frac{
		(\boldsymbol b^{\,*j}-(c-d)\one_{q-1})_1
	}{
		(\boldsymbol a-(c-d)\one_r)_1
	}.
	\]
	
	Consequently, the complete term corresponding to the fixed pair
	\((j,d)\) is
	\[
	M_d(c)F_d(t)\,
	\omega_j
	\frac{
		(\boldsymbol b^{\,*j}-(c-d)\one_{q-1})_1
	}{
		(\boldsymbol a-(c-d)\one_r)_1
	}
	\frac{1}{t+d+b_j}.
	\]
	For the fixed value of \(d\), the factors \(M_d(c)\) and \(F_d(t)\) do
	not depend on \(j\). Therefore the sum of all the terms corresponding
	to this value of \(d\) is
	\[
	M_d(c)F_d(t)
	\sum_{j=1}^{q}
	\omega_j
	\frac{
		(\boldsymbol b^{\,*j}-(c-d)\one_{q-1})_1
	}{
		(\boldsymbol a-(c-d)\one_r)_1
	}
	\frac{1}{t+d+b_j}.
	\]
	
	Assume first that \(d<K\). Then apply
	\eqref{eq:final-B-barycentric-identity} with
	\(x=t+d\),
	\(y=c-d\).
	Indeed, with this choice,
	\[
	x+b_j=t+d+b_j,
	\qquad
	x+y=t+d+c-d=t+c,
	\qquad
	-y=d-c.
	\]
	The sum over \(j\) is therefore
	\(\frac{1-\varrho(t+d)/\varrho(d-c)}{t+c}\),
	and the total contribution corresponding to the fixed value of \(d\)
	is
	\[
	M_d(c)F_d(t)
	\frac{1-\varrho(t+d)/\varrho(d-c)}{t+c}.
	\]
	This expression gives a telescoping difference. Indeed, from
	the definition of \(F_d(t)\),
	\begin{equation*}
		F_{d+1}(t)
		=
		\frac{(t\one_r+\boldsymbol a)_{d+1}}{(t\one_q+\boldsymbol b)_{d+1}}
		=
		\frac{(t\one_r+\boldsymbol a)_d}{(t\one_q+\boldsymbol b)_d}
		\frac{((t+d)\one_r+\boldsymbol a)_1}{((t+d)\one_q+\boldsymbol b)_1}
		=
		F_d(t)\varrho(t+d).
	\end{equation*}
	Similarly, from the definition of \(M_d(c)\),
	\begin{align*}
		M_{d+1}(c)
		&=
		\frac{(\boldsymbol b-c\one_q)_{d+1}}
		{(\boldsymbol a-c\one_r)_{d+1}}
		\\
		&=
		\frac{(\boldsymbol b-c\one_q)_d}
		{(\boldsymbol a-c\one_r)_d}
		\frac{(\boldsymbol b-c\one_q+d\one_q)_1}
		{(\boldsymbol a-c\one_r+d\one_r)_1}
		\\
		&=
		M_d(c)
		\frac{(\boldsymbol b+(d-c)\one_q)_1}
		{(\boldsymbol a+(d-c)\one_r)_1}
		\\
		&=
		\frac{M_d(c)}{\varrho(d-c)}.
	\end{align*}
	It follows that
	\[
	M_d(c)F_d(t)
	\frac{\varrho(t+d)}{\varrho(d-c)}
	=
	M_{d+1}(c)F_{d+1}(t).
	\]
	Hence, for \(d<K\), the contribution corresponding to the fixed value
	of \(d\) is
	\[
	\frac{
		M_d(c)F_d(t)-M_{d+1}(c)F_{d+1}(t)
	}{
		t+c
	}.
	\]
	It remains to consider \(d=K\). In this case,
	\(c-d=b_H+K-K=b_H\).
	The limiting form of
	\eqref{eq:final-B-barycentric-identity} established above must therefore
	be used. With
	\(x=t+K\) and \(y=b_H\), that limiting identity gives
	\[
	\sum_{j=1}^{q}
	\omega_j
	\frac{
		(\boldsymbol b^{\,*j}-b_H\one_{q-1})_1
	}{
		(\boldsymbol a-b_H\one_r)_1
	}
	\frac{1}{t+K+b_j}
	=
	\frac{1}{t+K+b_H}
	=
	\frac{1}{t+c}.
	\]
	Therefore the total contribution corresponding to \(d=K\) is
	\(\frac{M_K(c)F_K(t)}{t+c}\).
	On the other hand,
	\[
	M_{K+1}(c)
	=
	\frac{(\boldsymbol b-c\one_q)_{K+1}}
	{(\boldsymbol a-c\one_r)_{K+1}}
	=
	0,
	\]
	because the \(H\)-th Pochhammer symbol in its numerator is
	\((b_H-c)_{K+1}=(-K)_{K+1}=0\).
	Thus the contribution corresponding to \(d=K\) may also be written as
	\[
	\frac{
		M_K(c)F_K(t)-M_{K+1}(c)F_{K+1}(t)
	}{
		t+c
	}.
	\]
	This proves that, for every
	\(d\in\{0,\ldots,K\}\), the contribution obtained by fixing \(d\) and
	summing over \(j\in\{1,\ldots,q\}\) is
	\begin{equation}
		\label{eq:final-B-level-difference}
		\frac{
			M_d(c)F_d(t)-M_{d+1}(c)F_{d+1}(t)
		}{
			t+c
		}.
	\end{equation}
	It remains only to sum these contributions over \(d\in\{0,\ldots,K\}\). I
	obtain
	\begin{multline*}
		\sum_{d=0}^{K}
		\frac{
			M_d(c)F_d(t)-M_{d+1}(c)F_{d+1}(t)
		}{
			t+c
		}
		=
		\frac{1}{t+c}
		\Bigl(
		M_0(c)F_0(t)-M_1(c)F_1(t)
		+M_1(c)F_1(t)
		\\
		-M_2(c)F_2(t)
		+\cdots
		+M_K(c)F_K(t)-M_{K+1}(c)F_{K+1}(t)
		\Bigr).
	\end{multline*}
	Every intermediate term appears once with a positive sign and once with
	a negative sign, so all such terms cancel. Therefore
	\[
	\sum_{d=0}^{K}
	\frac{
		M_d(c)F_d(t)-M_{d+1}(c)F_{d+1}(t)
	}{
		t+c
	}
	=
	\frac{
		M_0(c)F_0(t)-M_{K+1}(c)F_{K+1}(t)
	}{
		t+c
	}.
	\]
	Finally,
	\(M_0(c)=1\),
	\(F_0(t)=1\),
	\(M_{K+1}(c)=0\).
	The preceding expression is therefore \(1/(t+c)\), which proves
	\eqref{eq:final-B-simple-fraction-reconstruction}.
\end{proof}

\begin{proof}[Proof of Theorem~\ref{thm:final-explicit}]
	The two assertions are proved separately.
	
	\noindent\emph{Proof of \textnormal{(i)}.}
	Begin by constructing an auxiliary polynomial vector. Set
	\[
	D_{\boldsymbol\kappa,\boldsymbol n}(t)
	\coloneq
	\prod_{i=1}^{p}
	\prod_{k=0}^{n_i-1}
	(t-\kappa_i-k),
	\qquad
	P_{\boldsymbol b,\boldsymbol m}(t)
	\coloneq
	\prod_{j=1}^{q}
	(t+b_j+1)_{m_j}.
	\]
	The noninteger-separation part of regularity ensures that the nodes
	\(\kappa_i+k\),
	\(i\in\{1,\ldots,p\}\),
	\(k\in\{0,\ldots,n_i-1\}\),
	are pairwise distinct. Define
	\[
	\mathcal R(t)
	\coloneq
	(-1)^{|\boldsymbol n|}
	\frac{P_{\boldsymbol b,\boldsymbol m}(t)}
	{D_{\boldsymbol\kappa,\boldsymbol n}(t)}
	\frac{\Gamma(t\one_q+\boldsymbol b+\one_q)}
	{\Gamma(t\one_r+\boldsymbol a+\one_r)}.
	\]
	For \(i\in\{1,\ldots,p\}\) and
	\(k\in\{0,\ldots,n_i-1\}\), put
	\[
	\mathsf c_{i,k}
	\coloneq
	\operatorname*{Res}_{t=\kappa_i+k}\mathcal R(t),
	\qquad
	Q_i(z)
	\coloneq
	\sum_{k=0}^{n_i-1}\mathsf c_{i,k}z^k.
	\]
	If \(n_i=0\), the sum defining \(Q_i\) is empty; set
	\(Q_i\equiv0\).
	
	First, compute the coefficients \(\mathsf c_{i,k}\). At
	\(t=\kappa_i+k\), the denominator factors as
	\[
	D_{\boldsymbol\kappa,\boldsymbol n}(t)
	=
	(t-\kappa_i-k)
	\prod_{\substack{k'=0\\k'\ne k}}^{n_i-1}
	(t-\kappa_i-k')
	\prod_{\substack{\ell=1\\\ell\ne i}}^{p}
	\prod_{s=0}^{n_\ell-1}
	(t-\kappa_\ell-s).
	\]
	Hence
	\begin{align*}
		\mathsf c_{i,k}
		&=
		\lim_{t\to\kappa_i+k}
		(t-\kappa_i-k)\mathcal R(t)
		\\
		&=
		(-1)^{|\boldsymbol n|}
		\frac{
			P_{\boldsymbol b,\boldsymbol m}(\kappa_i+k)
		}{
			\prod_{\substack{k'=0\\k'\ne k}}^{n_i-1}
			(k-k')
			\prod_{\substack{\ell=1\\\ell\ne i}}^{p}
			\prod_{s=0}^{n_\ell-1}
			(\kappa_i+k-\kappa_\ell-s)
		}
		\frac{
			\Gamma((\kappa_i+k)\one_q+\boldsymbol b+\one_q)
		}{
			\Gamma((\kappa_i+k)\one_r+\boldsymbol a+\one_r)
		}.
	\end{align*}
	The factors belonging to the \(i\)-th block satisfy
	\[
	\prod_{\substack{k'=0\\k'\ne k}}^{n_i-1}(k-k')
	=
	\left(
	\prod_{k'=0}^{k-1}(k-k')
	\right)
	\left(
	\prod_{k'=k+1}^{n_i-1}(k-k')
	\right)
	=
	(-1)^{n_i-1-k}k!(n_i-k-1)!.
	\]
	
	Rewrite the remaining factors to separate their
	\(k\)-independent parts from their \(k\)-dependence. For every
	\(\ell\ne i\), the contribution of the \(\ell\)-th block of
	\(D_{\boldsymbol\kappa,\boldsymbol n}\) is
	\begin{equation*}
		\prod_{s=0}^{n_\ell-1}
		(\kappa_i+k-\kappa_\ell-s)
		=
		(\kappa_i+k-\kappa_\ell-n_\ell+1)_{n_\ell}
		=
		(\kappa_i-\kappa_\ell-n_\ell+1)_{n_\ell}
		\frac{
			(\kappa_i-\kappa_\ell+1)_k
		}{
			(\kappa_i-\kappa_\ell-n_\ell+1)_k
		}.
	\end{equation*}
	Thus, after inversion, this block contributes the
	\(k\)-independent factor
	\(\frac{1}{
		(\kappa_i-\kappa_\ell-n_\ell+1)_{n_\ell}
	}\)
	and the \(k\)-dependent quotient
	\(\frac{
		(\kappa_i-\kappa_\ell-n_\ell+1)_k
	}{
		(\kappa_i-\kappa_\ell+1)_k
	}\).
	
	For every \(j\in\{1,\ldots,q\}\), the corresponding factor in
	\(P_{\boldsymbol b,\boldsymbol m}(\kappa_i+k)\) can be written as
	\[
	(\kappa_i+k+b_j+1)_{m_j}
	=
	(\kappa_i+b_j+1)_{m_j}
	\frac{
		(\kappa_i+b_j+m_j+1)_k
	}{
		(\kappa_i+b_j+1)_k
	}.
	\]
	On the other hand,
	\[
	\Gamma(\kappa_i+k+b_j+1)
	=
	\Gamma(\kappa_i+b_j+1)
	(\kappa_i+b_j+1)_k.
	\]
	The factors \((\kappa_i+b_j+1)_k\) therefore cancel. Consequently, the
	combined contribution of \(P_{\boldsymbol b,\boldsymbol m}(\kappa_i+k)\) and the
	Gamma factors in the numerator is
	\[
	\Gamma(\kappa_i\one_q+\boldsymbol b+\one_q)
	\prod_{j=1}^{q}(\kappa_i+b_j+1)_{m_j}
	\prod_{j=1}^{q}(\kappa_i+b_j+m_j+1)_k.
	\]
	Finally, for every \(\rho\in\{1,\ldots,r\}\),
	\[
	\Gamma(\kappa_i+k+a_\rho+1)
	=
	\Gamma(\kappa_i+a_\rho+1)
	(\kappa_i+a_\rho+1)_k.
	\]
	Collecting all these factors gives
	\begin{multline*}
		\mathsf c_{i,k}
		=
		(-1)^{|\boldsymbol n|}
		\frac{\Gamma(\kappa_i\one_q+\boldsymbol b+\one_q)}
		{\Gamma(\kappa_i\one_r+\boldsymbol a+\one_r)}
		\frac{
			\prod_{j=1}^{q}(\kappa_i+b_j+1)_{m_j}
		}{
			(-1)^{n_i-1-k}k!(n_i-k-1)!
			\prod_{\substack{\ell=1\\\ell\ne i}}^{p}
			(\kappa_i-\kappa_\ell-n_\ell+1)_{n_\ell}
		}
		\\
		\times
		\frac{
			\prod_{j=1}^{q}(\kappa_i+b_j+m_j+1)_k
			\prod_{\substack{\ell=1\\\ell\ne i}}^{p}
			(\kappa_i-\kappa_\ell-n_\ell+1)_k
		}{
			\prod_{\rho=1}^{r}(\kappa_i+a_\rho+1)_k
			\prod_{\substack{\ell=1\\\ell\ne i}}^{p}
			(\kappa_i-\kappa_\ell+1)_k
		}.
	\end{multline*}
	Using
\(	\frac{(-n_i+1)_k}{k!}
	=
	\frac{(-1)^k(n_i-1)!}
	{k!(n_i-k-1)!}\),
	this coefficient may be rewritten as
	\[
	\mathsf c_{i,k}
	=
	C_{\boldsymbol n,\boldsymbol m;i}
	\frac{
		(-n_i+1)_k
		\prod_{j=1}^{q}
		(\kappa_i+b_j+m_j+1)_k
		\prod_{\substack{\ell=1\\\ell\ne i}}^{p}
		(\kappa_i-\kappa_\ell-n_\ell+1)_k
	}{
		\prod_{\rho=1}^{r}
		(\kappa_i+a_\rho+1)_k
		\prod_{\substack{\ell=1\\\ell\ne i}}^{p}
		(\kappa_i-\kappa_\ell+1)_k
	}
	\frac{1}{k!}.
	\]
	Therefore, for \(n_i\ge1\), the residue-defined polynomial \(Q_i\)
	admits the hypergeometric representation
	\[
	Q_i(z)
	=
	C_{\boldsymbol n,\boldsymbol m;i}
	\pFq{q+p}{r+p-1}
	{
		-n_i+1,\ 
		\kappa_i\one_q+\boldsymbol b+\boldsymbol m+\one_q,\
		\kappa_i\one_{p-1}-\boldsymbol\kappa^{\,*i}
		-\boldsymbol n^{\,*i}+\one_{p-1}
	}
	{
		\kappa_i\one_r+\boldsymbol a+\one_r,\
		\kappa_i\one_{p-1}-\boldsymbol\kappa^{\,*i}
		+\one_{p-1}
	}
	{z}.
	\]
	In particular,
	\(\deg Q_i\le n_i-1\).
	
	Set
\(\mathbf Q_{\boldsymbol n,\boldsymbol m}
	\coloneq
	\begin{bNiceMatrix}
		Q_1&\Cdots&Q_p
	\end{bNiceMatrix}^{\top}\).
	Next, it is shown that \(\mathbf Q_{\boldsymbol n,\boldsymbol m}\) satisfies the required
	orthogonality relations.
	Fix
	\(j\in\{1,\ldots,q\}\),
	\(\ell\in\{0,\ldots,m_j-1\}\).
	Using
	\(Q_i(z)
	=
	\sum_{k=0}^{n_i-1}\mathsf c_{i,k}z^k\),
	the corresponding moment of
	\(\mathbf Q_{\boldsymbol n,\boldsymbol m}\) can be written as
	\begin{align*}
		\sum_{i=1}^{p}
		\oint_{\Torus}
		z^\ell C_{j,i}(z)Q_i(z)
		\frac{\dz}{2\pi\mathrm i}
		&=
		\sum_{i=1}^{p}
		\sum_{k=0}^{n_i-1}
		\mathsf c_{i,k}
		\oint_{\Torus}
		z^{\ell+k}C_{j,i}(z)
		\frac{\dz}{2\pi\mathrm i}
		\\
		&=
		\sum_{i=1}^{p}
		\sum_{k=0}^{n_i-1}
		\mathsf c_{i,k}
		\frac{
			\Gamma((\kappa_i+k+\ell+1)\one_r+\boldsymbol a)
		}{
			\Gamma((\kappa_i+k+\ell+1)\one_q+\boldsymbol b+\boldsymbol e_j)
		},
	\end{align*}
	where the second equality uses the moment formula
	\eqref{eq:final-moment-entry}.
	
	Since
	\(\mathsf c_{i,k}
	=
	\operatorname*{Res}_{t=\kappa_i+k}\mathcal R(t)\),
	each term in the last double sum is the residue at
	\(t=\kappa_i+k\) of
	\begin{equation}
		\label{eq:final-A-orthogonality-residue-function}
		\mathcal R_{j,\ell}(t)
		\coloneq
		\mathcal R(t)
		\frac{\Gamma((t+\ell+1)\one_r+\boldsymbol a)}
		{\Gamma((t+\ell+1)\one_q+\boldsymbol b+\boldsymbol e_j)}.
	\end{equation}
	Hence the required moment is the sum of the residues of
	\(\mathcal R_{j,\ell}\) at the nodes \(t=\kappa_i+k\).
	
	Substitution of the definition of \(\mathcal R\) gives
	\[
	\mathcal R_{j,\ell}(t)
	=
	(-1)^{|\boldsymbol n|}
	\frac{P_{\boldsymbol b,\boldsymbol m}(t)}
	{D_{\boldsymbol\kappa,\boldsymbol n}(t)}
	\frac{\Gamma(t\one_q+\boldsymbol b+\one_q)}
	{\Gamma(t\one_r+\boldsymbol a+\one_r)}
	\frac{\Gamma((t+\ell+1)\one_r+\boldsymbol a)}
	{\Gamma((t+\ell+1)\one_q+\boldsymbol b+\boldsymbol e_j)}.
	\]
	Now
	\[
	\begin{aligned}
		\frac{\Gamma((t+\ell+1)\one_r+\boldsymbol a)}
		{\Gamma(t\one_r+\boldsymbol a+\one_r)}
		&=
		(t\one_r+\boldsymbol a+\one_r)_\ell,
		\\
		\frac{\Gamma(t\one_q+\boldsymbol b+\one_q)}
		{\Gamma((t+\ell+1)\one_q+\boldsymbol b+\boldsymbol e_j)}
		&=
		\frac{1}
		{(t\one_q+\boldsymbol b+\one_q)_\ell(t+b_j+\ell+1)}.
	\end{aligned}
	\]
	Therefore
	\begin{equation}
		\label{eq:final-A-residue-rational-form}
		\mathcal R_{j,\ell}(t)
		=
		(-1)^{|\boldsymbol n|}
		\frac{
			P_{\boldsymbol b,\boldsymbol m}(t)
			(t\one_r+\boldsymbol a+\one_r)_\ell
		}{
			D_{\boldsymbol\kappa,\boldsymbol n}(t)
			(t\one_q+\boldsymbol b+\one_q)_\ell
			(t+b_j+\ell+1)
		}.
	\end{equation}
	Since \(\boldsymbol m\) is near the diagonal and
	\(\ell\le m_j-1\), one has \(\ell\le m_h\) for every
	\(h\in\{1,\ldots,q\}\). Therefore, for each \(h\),
	\[
	(t+b_h+1)_{m_h}
	=
	(t+b_h+1)_\ell
	(t+b_h+\ell+1)_{m_h-\ell}.
	\]
	Hence
	\[
	P_{\boldsymbol b,\boldsymbol m}(t)
	=
	(t\one_q+\boldsymbol b+\one_q)_\ell
	\prod_{h=1}^{q}
	(t+b_h+\ell+1)_{m_h-\ell},
	\]
	so \((t\one_q+\boldsymbol b+\one_q)_\ell\) divides
	\(P_{\boldsymbol b,\boldsymbol m}(t)\). After this cancellation, the factor
	corresponding to the \(j\)-th component is
	\((t+b_j+\ell+1)_{m_j-\ell}\),
	which is divisible by \(t+b_j+\ell+1\). Thus all possible auxiliary
	poles in \eqref{eq:final-A-residue-rational-form} are removable, and
	its only finite poles are the simple poles at the zeros of
	\(D_{\boldsymbol\kappa,\boldsymbol n}\).
	
	The degree of the numerator in
	\eqref{eq:final-A-residue-rational-form} minus the degree of the
	denominator is
	\[
	|\boldsymbol m|+r\ell-\bigl(|\boldsymbol n|+q\ell+1\bigr)
	=
	-2-(q-r)\ell.
	\]
	Here the identity \(|\boldsymbol n|=|\boldsymbol m|+1\) has been used. Since \(0\le r<q\),
	\(\mathcal R_{j,\ell}(t)
	=
	\mathrm{O}(t^{-2})\),
	\(t\to\infty\).
	Hence the sum of its residues at all finite poles is zero. This sum is
	\[
	\sum_{i=1}^{p}
	\sum_{k=0}^{n_i-1}
	\mathsf c_{i,k}
	\frac{\Gamma((\ell+k+1+\kappa_i)\one_r+\boldsymbol a)}
	{\Gamma((\ell+k+1+\kappa_i)\one_q+\boldsymbol b+\boldsymbol e_j)},
	\]
	which, by \eqref{eq:final-moment-entry}, is precisely
	\[
	\sum_{i=1}^{p}
	\oint_{\Torus}
	z^\ell C_{j,i}(z)Q_i(z)
	\frac{\dz}{2\pi\mathrm i}.
	\]
	Therefore
	\[
	\sum_{i=1}^{p}
	\oint_{\Torus}
	z^\ell C_{j,i}(z)Q_i(z)
	\frac{\dz}{2\pi\mathrm i}
	=
	0,
	\qquad
	\ell\in\{0,\ldots,m_j-1\},
	\]
	and \(\mathbf Q_{\boldsymbol n,\boldsymbol m}\) satisfies the required orthogonality
	relations.
	
	It remains to compute the moment at the first level not covered by the
	orthogonality conditions. Take \(j=j_{\min}\) and
	\(\ell=m_{\min}\). Since \(\boldsymbol m\) is near the diagonal,
	\(m_h\in\{m_{\min},m_{\min}+1\}\) for every
	\(h\in\{1,\ldots,q\}\). Therefore,
	\[
	\frac{P_{\boldsymbol b,\boldsymbol m}(t)}
	{(t\one_q+\boldsymbol b+\one_q)_{m_{\min}}}
	=
	\prod_{h=1}^{q}
	(t+b_h+m_{\min}+1)_{m_h-m_{\min}}
	=
	\prod_{\substack{h=1\\m_h=m_{\min}+1}}^{q}
	(t+b_h+m_{\min}+1).
	\]
	Hence
	\[
	\mathcal R_{j_{\min},m_{\min}}(t)
	=
	(-1)^{|\boldsymbol n|}
	\frac{
		(t\one_r+\boldsymbol a+\one_r)_{m_{\min}}
	}{
		D_{\boldsymbol\kappa,\boldsymbol n}(t)
		(t+b_{j_{\min}}+m_{\min}+1)
	}
	\prod_{\substack{h=1\\m_h=m_{\min}+1}}^{q}
	(t+b_h+m_{\min}+1).
	\]
	Since \(m_{j_{\min}}=m_{\min}\), the product in the numerator does not
	contain the factor \(t+b_{j_{\min}}+m_{\min}+1\). Thus, apart from the
	simple poles at the zeros of \(D_{\boldsymbol\kappa,\boldsymbol n}\), the only
	additional pole is \(t_0=-b_{j_{\min}}-m_{\min}-1\).
	
	Let
	\[
	N_A
	\coloneq
	\sum_{i=1}^{p}
	\oint_{\Torus}
	z^{m_{\min}}C_{j_{\min},i}(z)Q_i(z)
	\frac{\dz}{2\pi\mathrm i}.
	\]
	The same degree computation gives
	\(\mathcal R_{j_{\min},m_{\min}}(t)
	=
	\mathrm{O}(t^{-2})\),
	\(t\to\infty\).
	Therefore the sum of all its finite residues is zero, and hence
	\(N_A
	=
	-\operatorname*{Res}_{t=t_0}
	\mathcal R_{j_{\min},m_{\min}}(t)\).
	At \(t=t_0\),
	\[
	D_{\boldsymbol\kappa,\boldsymbol n}(t_0)
	=
	(-1)^{|\boldsymbol n|}
	\prod_{i=1}^{p}
	(\kappa_i+b_{j_{\min}}+m_{\min}+1)_{n_i}.
	\]
	Moreover,
	\[
	(t_0\one_r+\boldsymbol a+\one_r)_{m_{\min}}
	=
	(\boldsymbol a-b_{j_{\min}}\one_r-m_{\min}\one_r)_{m_{\min}},
	\]
	while, for every \(h\) such that \(m_h=m_{\min}+1\),
	\(t_0+b_h+m_{\min}+1=b_h-b_{j_{\min}}\).
	Therefore
	\[
	\operatorname*{Res}_{t=t_0}
	\mathcal R_{j_{\min},m_{\min}}(t)
	=
	\frac{
		(\boldsymbol a-b_{j_{\min}}\one_r-m_{\min}\one_r)_{m_{\min}}
	}{
		\prod_{i=1}^{p}
		(\kappa_i+b_{j_{\min}}+m_{\min}+1)_{n_i}
	}
	\prod_{\substack{h=1\\m_h=m_{\min}+1}}^{q}
	(b_h-b_{j_{\min}}).
	\]
	Consequently,
	\[
	N_A
	=
	-
	\frac{
		(\boldsymbol a-b_{j_{\min}}\one_r-m_{\min}\one_r)_{m_{\min}}
	}{	\prod_{i=1}^{p}
		(\kappa_i+b_{j_{\min}}+m_{\min}+1)_{n_i}
	}
	\prod_{\substack{h=1\\m_h=m_{\min}+1}}^{q}
	(b_h-b_{j_{\min}}).
	\]
	By regularity and the standing admissibility convention, none of the
	Pochhammer symbols or linear factors in this
	expression vanishes, and its denominator is nonzero. Hence
	\(N_A\ne0\). Therefore every nonzero scalar multiple
	\[
	\mathbf A_{\boldsymbol n,\boldsymbol m}
	=
	\nu\mathbf Q_{\boldsymbol n,\boldsymbol m},
	\qquad
	\nu\ne0,
	\]
	satisfies the degree bounds and the orthogonality relations. Since the
	components of \(\mathbf Q_{\boldsymbol n,\boldsymbol m}\) have the hypergeometric
	representations obtained above, this gives
	\eqref{eq:final-A-components} and \eqref{eq:final-A-degree}.
	Finally, choosing
	\[
	\nu=\nu_{\boldsymbol n,\boldsymbol m}^{A}=N_A^{-1}
	\]
	gives \eqref{eq:final-A-nu} and the normalizing condition
	\eqref{eq:final-A-normalizing-moment}.

	\noindent\emph{Proof of \textnormal{(ii)}.}
	For each \(j\in\{1,\ldots,q\}\) with \(m_j\ge1\), let \(Q_j\) denote
	the expression inside brackets in \eqref{eq:final-B-components} with
	\(\nu=1\).
	If \(m_j=0\), set \(Q_j\equiv0\).

	For \(j,H\in\{1,\ldots,q\}\) and
	\(K\in\{1,\ldots,m_H-1\}\), define
	\begin{equation}
		\label{eq:final-B-Psi}
		\Psi_{j;H,K}(z)
		\coloneq
		\frac{
			(\boldsymbol b^{\,*j}-(b_H+K)\one_{q-1})_1
		}{
			(\boldsymbol a-(b_H+K)\one_r)_1
		}
		\pFq{q+1}{r}
		{
			1,\ \boldsymbol b-(b_H+K)\one_q+\one_q-\boldsymbol e_j
		}
		{
			\boldsymbol a-(b_H+K)\one_r+\one_r
		}
		{z}.
	\end{equation}
	Its \(H\)-th numerator parameter is
	\(1-K-\delta_{j,H}\), so \(\Psi_{j;H,K}\) is a terminating polynomial
	of degree at most \(K-1+\delta_{j,H}\). Its finite expansion is
	\begin{equation}
		\label{eq:final-B-Psi-finite}
		\Psi_{j;H,K}(z)
		=
		\frac{
			(\boldsymbol b^{\,*j}-(b_H+K)\one_{q-1})_1
		}{
			(\boldsymbol a-(b_H+K)\one_r)_1
		}
		\sum_{d=0}^{K-1+\delta_{j,H}}
		\frac{
			(\boldsymbol b-(b_H+K)\one_q+\one_q-\boldsymbol e_j)_d
		}{
			(\boldsymbol a-(b_H+K)\one_r+\one_r)_d
		}
		z^d.
	\end{equation}
	Since \(\boldsymbol m\) is near the diagonal and \(K\le m_H-1\), one has
	\(K-1+\delta_{j,H}\le m_j-1\). Hence
	\(\deg Q_j\le m_j-1\), and write
	\[
		Q_j(z)=\sum_{d=0}^{m_j-1}Q_j[d]z^d.
	\]
	
	By Lemma~\ref{lem:final-B-partial-fractions}, the rational function
	\[
	R(t)
	\coloneq
	-
	\frac{
		\prod_{i=1}^{p}(\kappa_i+1-t)_{n_i}
	}{
		\prod_{j=1}^{q}(t+b_j)_{m_j}
	}
	\]
	has the partial-fraction decomposition
	\eqref{eq:final-B-partial-fractions}. Moreover,
	Lemma~\ref{lem:final-B-simple-fraction-reconstruction} gives
	\eqref{eq:final-B-simple-fraction-reconstruction} for every
	\(H\in\{1,\ldots,q\}\) and
	\(K\in\{0,\ldots,m_H-1\}\).
	
	For every term occurring in
	\eqref{eq:final-B-simple-fraction-reconstruction}, the multi-index
	\(\boldsymbol m-d\one_q-\boldsymbol e_j\) belongs to \(\N_0^q\). Indeed, suppose
	first that \(j=H\). Then
	\[
	d\le K\le m_H-1=m_j-1,
	\]
	so the \(j\)-th component \(m_j-d-1\) is nonnegative. Moreover,
	near-diagonality gives \(m_h\ge m_H-1\) for every \(h\ne H\), and hence
	\(d\le m_h\). Thus all the remaining components \(m_h-d\) are also
	nonnegative.
	
	Suppose now that \(j\ne H\). Whenever such a term occurs,
	\(d\le K-1\le m_H-2\).
	Near-diagonality implies \(m_j\ge m_H-1\), and therefore
	\(d\le m_j-1\). It also implies \(m_h\ge m_H-1\) for every
	\(h\ne j\), so \(d\le m_h\). Hence
	\(\boldsymbol m-d\one_q-\boldsymbol e_j\in\N_0^q\) in this case as well.
	Consequently,
	\[
	\frac{(t\one_q+\boldsymbol b)_{\boldsymbol m}}
	{(t\one_q+\boldsymbol b)_d(t+b_j+d)}
	=
	(t\one_q+\boldsymbol b+d\one_q+\boldsymbol e_j)_{
		\boldsymbol m-d\one_q-\boldsymbol e_j}.
	\]
	Multiplying \eqref{eq:final-B-simple-fraction-reconstruction} by
	\((t\one_q+\boldsymbol b)_{\boldsymbol m}\) gives
	\begin{multline}
		\label{eq:final-B-polynomial-reconstruction}
		\frac{(t\one_q+\boldsymbol b)_{\boldsymbol m}}{t+b_H+K}
		=
		\sum_{j=1}^{q}
		\frac{(\boldsymbol a-b_j\one_r)_1}
		{(\boldsymbol b^{\,*j}-b_j\one_{q-1})_1}
		\frac{
			(\boldsymbol b^{\,*j}-(b_H+K)\one_{q-1})_1
		}{
			(\boldsymbol a-(b_H+K)\one_r)_1
		}
		\\
		\times
		\sum_{d=0}^{K-1+\delta_{j,H}}
		\frac{
			(\boldsymbol b-(b_H+K)\one_q+\one_q-\boldsymbol e_j)_d
		}{
			(\boldsymbol a-(b_H+K)\one_r+\one_r)_d
		}
		(t\one_r+\boldsymbol a)_d
		(t\one_q+\boldsymbol b+d\one_q+\boldsymbol e_j)_{
			\boldsymbol m-d\one_q-\boldsymbol e_j}.
	\end{multline}
	
	Multiplying \eqref{eq:final-B-partial-fractions} by the common
	denominator
\(	(t\one_q+\boldsymbol b)_{\boldsymbol m}
	=
	\prod_{h=1}^{q}(t+b_h)_{m_h}\),
	clears all the simple fractions. Next substitute, for each pair
	\((H,K)\), the reconstruction formula
	\eqref{eq:final-B-polynomial-reconstruction}. Every term produced by
	this substitution is a scalar multiple of a polynomial of the form
	\[
	(t\one_r+\boldsymbol a)_d
	(t\one_q+\boldsymbol b+d\one_q+\boldsymbol e_j)_{
		\boldsymbol m-d\one_q-\boldsymbol e_j},
	\qquad
	j\in\{1,\ldots,q\}.
	\]
	Indeed,
	\[
	(t\one_q+\boldsymbol b+d\one_q+\boldsymbol e_j)_{
		\boldsymbol m-d\one_q-\boldsymbol e_j}
	=
	\frac{
		(t\one_q+\boldsymbol b)_{\boldsymbol m}
	}{
		(t\one_q+\boldsymbol b)_d(t+b_j+d)
	},
	\]
	so this factor is exactly the polynomial obtained after multiplying
	the corresponding rational basis element by the common denominator
	\((t\one_q+\boldsymbol b)_{\boldsymbol m}\).
	
	All the sums involved are finite, so their order may be interchanged
	and the terms having the same indices \(j\) and \(d\) may be collected
	together. By the definition of
	\(Q_j(z)=\sum_{d=0}^{m_j-1}Q_j[d]z^d\),
	the sum of all scalar coefficients multiplying
	\((t\one_r+\boldsymbol a)_d
	(t\one_q+\boldsymbol b+d\one_q+\boldsymbol e_j)_{
		\boldsymbol m-d\one_q-\boldsymbol e_j}\)
	is precisely \(Q_j[d]\). Thus the right-hand side obtained after
	clearing denominators and performing the reconstruction is
	\[
	\sum_{j=1}^{q}\sum_{d=0}^{m_j-1}
	Q_j[d]
	(t\one_r+\boldsymbol a)_d
	(t\one_q+\boldsymbol b+d\one_q+\boldsymbol e_j)_{
		\boldsymbol m-d\one_q-\boldsymbol e_j}.
	\]
	On the other hand, by the definition of the rational function in
	\eqref{eq:final-B-partial-fractions},
	\[
	(t\one_q+\boldsymbol b)_{\boldsymbol m}R(t)
	=
	-
	\prod_{i=1}^{p}(\kappa_i+1-t)_{n_i}.
	\]
	Equating these two expressions for
	\((t\one_q+\boldsymbol b)_{\boldsymbol m}R(t)\) gives
	\begin{equation}
		\label{eq:final-B-polynomial-identity}
		\sum_{j=1}^{q}\sum_{d=0}^{m_j-1}
		Q_j[d]
		(t\one_r+\boldsymbol a)_d
		(t\one_q+\boldsymbol b+d\one_q+\boldsymbol e_j)_{
			\boldsymbol m-d\one_q-\boldsymbol e_j}
		=
		-
		\prod_{i=1}^{p}(\kappa_i+1-t)_{n_i},
	\end{equation}
	For \(K=0\) in
	\eqref{eq:final-B-polynomial-reconstruction}, only the term \(j=H\)
	survives, and its two prefactors cancel:
	\[
	\frac{(\boldsymbol a-b_H\one_r)_1}
	{(\boldsymbol b^{\,*H}-b_H\one_{q-1})_1}
	\frac{(\boldsymbol b^{\,*H}-b_H\one_{q-1})_1}
	{(\boldsymbol a-b_H\one_r)_1}
	=1.
	\]
	Thus the \(K=0\) blocks contribute precisely the constant coefficients
	\(\pi_{j,0}\), see \eqref{eq:final-B-pi-zero}.
	
	Fix \(i\in\{1,\ldots,p\}\), let \(s\ge1\), and put
	\(t=s+\kappa_i\). By \eqref{eq:final-moment-entry},
	\[
	\sum_{j=1}^{q}
	\oint_{\Torus}
	z^{s-1}Q_j(z)C_{j,i}(z)
	\frac{\dz}{2\pi\mathrm i}
	=
	\sum_{j=1}^{q}\sum_{d=0}^{m_j-1}
	Q_j[d]
	\frac{\Gamma((t+d)\one_r+\boldsymbol a)}
	{\Gamma((t+d)\one_q+\boldsymbol b+\boldsymbol e_j)}.
	\]
	Moreover,
	\[
	\frac{\Gamma((t+d)\one_r+\boldsymbol a)}
	{\Gamma((t+d)\one_q+\boldsymbol b+\boldsymbol e_j)}
	=
	\frac{\Gamma(t\one_r+\boldsymbol a)}
	{\Gamma(t\one_q+\boldsymbol b+\boldsymbol m)}
	(t\one_r+\boldsymbol a)_d
	(t\one_q+\boldsymbol b+d\one_q+\boldsymbol e_j)_{
		\boldsymbol m-d\one_q-\boldsymbol e_j}.
	\]
	Equation~\eqref{eq:final-B-polynomial-identity} gives
	\begin{equation}
		\label{eq:final-B-moment-identity}
		\sum_{j=1}^{q}
		\oint_{\Torus}
		z^{s-1}Q_j(z)C_{j,i}(z)
		\frac{\dz}{2\pi\mathrm i}
		=
		-
		\frac{
			\Gamma((s+\kappa_i)\one_r+\boldsymbol a)
		}{
			\Gamma((s+\kappa_i)\one_q+\boldsymbol b+\boldsymbol m)
		}
		\prod_{h=1}^{p}
		(\kappa_h-\kappa_i+1-s)_{n_h}.
	\end{equation}
	Taking \(s=\ell+1\), with
	\(\ell\in\{0,\ldots,n_i-1\}\), the factor corresponding to \(h=i\)
	is \((-\ell)_{n_i}=0\). Thus the \(B\)-orthogonality relations hold.
	
	The finite summation ranges give \(\deg Q_j\le m_j-1\). If \(m_j=0\),
	the theorem defines \(B_{\boldsymbol n,\boldsymbol m}^{(j)}\equiv0\). Therefore
	\(\mathbf B_{\boldsymbol n,\boldsymbol m}=\nu\begin{bNiceMatrix}
		Q_1&\Cdots&Q_q	\end{bNiceMatrix}\) satisfies all the
	asserted degree bounds and orthogonality relations.
	
	It remains to compute the monic normalization. Let
	\(j=j_{\max}\) and \(m_{\max}=m_{j_{\max}}\). If \(m_{\max}=1\), then
	\(Q_{j_{\max}}\) is constant and its leading coefficient is
	\(\pi_{j_{\max},0}\). Assume \(m_{\max}\ge2\). A block \((H,K)\) can
	contribute to the coefficient of \(z^{m_{\max}-1}\) only if
	\(K-1+\delta_{j_{\max},H}\ge m_{\max}-1\). Since
	\(K\le m_H-1\le m_{\max}-1\), this is possible only for
	\(H=j_{\max}\) and \(K=m_{\max}-1\). Hence only the term with
	\(H=j_{\max}\) and \(K=m_{\max}-1\) contributes to the coefficient of
	\(z^{m_{\max}-1}\) in \(Q_{j_{\max}}(z)\). Therefore, that coefficient is
	\begin{equation*}
		\frac{(\boldsymbol a-b_{j_{\max}}\one_r)_1}
		{(\boldsymbol b^{\,*j_{\max}}-b_{j_{\max}}\one_{q-1})_1}
		\pi_{j_{\max},m_{\max}-1}
		[z^{m_{\max}-1}]\Psi_{j_{\max};j_{\max},m_{\max}-1}(z).
	\end{equation*}
	By \eqref{eq:final-B-Psi-finite},
	\begin{multline*}
		[z^{m_{\max}-1}]
		\Psi_{j_{\max};j_{\max},m_{\max}-1}(z)
		=
		\frac{
			(\boldsymbol b^{\,*j_{\max}}
			-(b_{j_{\max}}+m_{\max}-1)\one_{q-1})_1
		}{
			(\boldsymbol a
			-(b_{j_{\max}}+m_{\max}-1)\one_r)_1
		}
		\\
		\times
		\frac{
			(\boldsymbol b
			-(b_{j_{\max}}+m_{\max}-1)\one_q
			+\one_q-\boldsymbol e_{j_{\max}})_{m_{\max}-1}
		}{
			(\boldsymbol a
			-(b_{j_{\max}}+m_{\max}-1)\one_r
			+\one_r)_{m_{\max}-1}
		}.
	\end{multline*}
	The cancellations are made explicit. The \(j_{\max}\)-th entry in the
	vectorial Pochhammer symbol in the numerator is
		\((-m_{\max}+1)_{m_{\max}-1}
		=
		(-1)^{m_{\max}-1}(m_{\max}-1)!\),
	which cancels the corresponding factorial factor in
	\(\pi_{j_{\max},m_{\max}-1}\). Moreover,
	\begin{align*}
	&\bigl(\boldsymbol a-b_{j_{\max}}\one_r
	-(m_{\max}-1)\one_r\bigr)_1
	\bigl(\boldsymbol a-b_{j_{\max}}\one_r
	-(m_{\max}-2)\one_r\bigr)_{m_{\max}-1}\\
	&\qquad=
	\bigl(\boldsymbol a-b_{j_{\max}}\one_r
	-(m_{\max}-1)\one_r\bigr)_{m_{\max}}.
	\end{align*}
	Finally, for every \(h\ne j_{\max}\), near-diagonality gives
	\(m_h\in\{m_{\max},m_{\max}-1\}\), and
	\begingroup
	\small
	\[
	\begin{aligned}
	&
	\frac{
		(b_h-b_{j_{\max}}-m_{\max}+1)
		(b_h-b_{j_{\max}}-m_{\max}+2)_{m_{\max}-1}
	}{
		(b_h-b_{j_{\max}}-m_{\max}+1)_{m_h}
		(b_h-b_{j_{\max}})
	}
	\\
	&\qquad=
	\begin{cases}
		\dfrac{1}{b_h-b_{j_{\max}}},&m_h=m_{\max},\\[2mm]
		1,&m_h=m_{\max}-1.
	\end{cases}
	\end{aligned}
	\]
	Substituting \eqref{eq:final-B-pi} and applying these identities gives
	\begin{align*}
	[z^{m_{\max}-1}]Q_{j_{\max}}(z)
	&=
	-
	\frac{
		(\boldsymbol a-b_{j_{\max}}\one_r)_1
	}{
		(\boldsymbol a-b_{j_{\max}}\one_r
		-(m_{\max}-1)\one_r)_{m_{\max}}
	}
	\\
	&\quad\times
	\frac{
		\prod_{i=1}^{p}
		(\kappa_i+b_{j_{\max}}+m_{\max})_{n_i}
	}{
		\prod_{\substack{h=1\\
				h\ne j_{\max},\ m_h=m_{\max}}}^{q}
		(b_h-b_{j_{\max}})
	}.
	\end{align*}
	\endgroup
	The same formula follows directly from \(\pi_{j_{\max},0}\) when
	\(m_{\max}=1\). This coefficient is finite and nonzero by regularity and the
	standing admissibility convention.
	Therefore the choice \(\nu=\nu_{\boldsymbol n,\boldsymbol m}^{B}\) in
	\eqref{eq:final-B-nu} makes
	\(B_{\boldsymbol n,\boldsymbol m}^{(j_{\max})}\) monic.
\end{proof}

The explicit formula for the \(B\)-polynomial vector obtained above is already
finite and directly computable. However, its coefficients contain nested
finite sums whose structure is naturally recognized as a terminating
bivariate Kamp\'e de F\'eriet series.

\subsection{Kamp\'e de F\'eriet representation of the
\texorpdfstring{\(B\)}{B}-components}
\label{subsec:B-Kampe-de-Feriet}

Each pole-family contribution is rewritten in this form. This reformulation is not needed for the
orthogonality proof, but it displays the hypergeometric structure of the
\(B\)-components explicitly and gives a compact closed form for the blocks
appearing in each component.

\begin{corollary}[Kamp\'e de F\'eriet form of the \(B\)-components]
	\label{cor:final-B-KdF}
	Under the assumptions of
	Theorem~\ref{thm:final-explicit}\textnormal{(ii)}, let
	\(j\in\{1,\ldots,q\}\). If \(m_j=0\), then
	\(B_{\boldsymbol n,\boldsymbol m}^{(j)}\equiv0\). If \(m_j\ge1\), then
	\begin{equation}
		\label{eq:final-B-KdF-representation}
		B_{\boldsymbol n,\boldsymbol m}^{(j)}(z)
		=
		\nu
		\Big(
		\pi_{j,0}\mathcal K_{j}(z)
		+
		\sum_{\substack{H=1\\H\ne j,\ m_H\ge2}}^{q}
		\mathcal C_{j,H}
		\mathcal K_{j,H}(z)
		\Big),
	\end{equation}
	where
	\begingroup
	\scriptsize
	\begin{equation}
		\label{eq:final-B-KdF-diagonal}
		\begin{aligned}
		\mathcal K_{j}(z)
		&\coloneq\\[-1mm]
		&\quad
		F_{p+r:0;q-1}^{p+q:1;r}
		\left[
		\begin{array}{c}
			\substack{
			1-m_j,\ \boldsymbol\kappa+(b_j+1)\one_p+\boldsymbol n,\\
			b_j\one_{q-1}-\boldsymbol b^{\,*j}
			+\one_{q-1}-\boldsymbol m^{\,*j}
			:1;\ b_j\one_r-\boldsymbol a}
			\\
			\substack{
			\boldsymbol\kappa+(b_j+1)\one_p,\ 
			b_j\one_r-\boldsymbol a+\one_r,\\
			:\text{---};\ b_j\one_{q-1}-\boldsymbol b^{\,*j}}
		\end{array}
		\middle|
		(-1)^{q-r}z,1
		\right],
		\end{aligned}
	\end{equation}
	\endgroup
	and, for \(H\ne j\) with \(m_H\ge2\),
	\begin{equation}
		\label{eq:final-B-KdF-offdiagonal-prefactor}
		\mathcal C_{j,H}
		\coloneq
		\pi_{H,1}
		\frac{(\boldsymbol a-b_j\one_r)_1}
		{(\boldsymbol b^{\,*j}-b_j\one_{q-1})_1}
		\frac{
			(\boldsymbol b^{\,*j}-(b_H+1)\one_{q-1})_1
		}{
			(\boldsymbol a-(b_H+1)\one_r)_1
		},
	\end{equation}
	while
	\begin{equation}
		\label{eq:final-B-KdF-offdiagonal}
		\resizebox{\textwidth}{!}{$
		\begin{multlined}[t]
			\mathcal K_{j,H}(z)
			\coloneq
			F_{p+r:0;q-1}^{p+q:1;r}
			\left[
			\begin{array}{c}
				2-m_H,\,
				\boldsymbol\kappa+(b_H+2)\one_p+\boldsymbol n,\,
				b_H\one_{q-1}-\boldsymbol b^{\,*H}
				+2\one_{q-1}-\boldsymbol m^{\,*H}
				:1;\,
				b_H\one_r-\boldsymbol a+\one_r
				\\
				\boldsymbol\kappa+(b_H+2)\one_p,\,
				b_H\one_r-\boldsymbol a+2\one_r
				:\text{---};\,
				b_H\one_{q-1}-\boldsymbol b^{\,*H}
				+\one_{q-1}+\boldsymbol e_j^{\,*H}
			\end{array}
			\middle|
			(-1)^{q-r}z,1
			\right].
		\end{multlined}
		$}
	\end{equation}
	Here \(\boldsymbol e_j^{\,*H}\) denotes the vector obtained from
	\(\boldsymbol e_j\) by deleting its \(H\)-th component. The parameter
	\(1-m_j\) makes \(\mathcal K_j\) terminate on
	\(u+\lambda\le m_j-1\), while \(2-m_H\) makes
	\(\mathcal K_{j,H}\) terminate on
	\(u+\lambda\le m_H-2\).
\end{corollary}

\begin{proof}
	Fix \(j\in\{1,\ldots,q\}\) and expand the terminating polynomials
	\(\Psi_{j;H,K}\) according to
	\eqref{eq:final-B-Psi-finite}. The pole families are treated separately
	with \(H=j\) and \(H\ne j\).
	
	Assume first that \(H=j\). Set \(K=u+\lambda\) and \(d=u\). The
	conditions \(0\le d\le K\) and \(K\le m_j-1\) are equivalent to
	\(u,\lambda\in\N_0\) and \(u+\lambda\le m_j-1\). This change of
	variables also incorporates the collapsed term \(K=0\): the point
	\(u=\lambda=0\) contributes precisely \(\pi_{j,0}\).
	Put \(s\coloneq u+\lambda\). Directly from
	\eqref{eq:final-B-pi},
	\begin{equation*}
		\frac{\pi_{j,s}}{\pi_{j,0}}
		=
		\frac{(1-m_j)_s}{s!}
		\frac{
			\bigl(\boldsymbol\kappa+(b_j+1)\one_p+\boldsymbol n\bigr)_s
		}{
			\bigl(\boldsymbol\kappa+(b_j+1)\one_p\bigr)_s
		}
		\frac{
			\bigl(b_j\one_{q-1}-\boldsymbol b^{\,*j}
			+\one_{q-1}-\boldsymbol m^{\,*j}\bigr)_s
		}{
			\bigl(b_j\one_{q-1}-\boldsymbol b^{\,*j}
			+\one_{q-1}\bigr)_s
		}.
	\end{equation*}
	The terminating parameter supplied by the \(H=j\) entry is \(-s\).
	Since \(s-u=\lambda\),
		\((-s)_u=(-1)^u\dfrac{s!}{\lambda!}\).
	Thus the factor \(s!\) cancels the factorial in the preceding quotient,
	while \(\lambda!\) becomes the second factorial of the bivariate series.
	
	Using
	\((a+s)_n=(a)_n(a+n)_s/(a)_s\) and
	\((a-s)_n=(a)_n(1-a)_s/(1-a-n)_s\), with
	\(s=u+\lambda\), in the factors coming from
	\eqref{eq:final-B-pi}, and then applying the same identities to the
	factors in \eqref{eq:final-B-Psi-finite}, the summand corresponding to
	\((u,\lambda)\) becomes
	\begin{multline*}
		\pi_{j,0}\,
		\frac{
			(1-m_j)_{u+\lambda}
			\bigl(\boldsymbol\kappa+(b_j+1)\one_p+\boldsymbol n\bigr)_{u+\lambda}
		}{
			\bigl(\boldsymbol\kappa+(b_j+1)\one_p\bigr)_{u+\lambda}
		}
		\\
		\times
		\frac{
			\bigl(
			b_j\one_{q-1}-\boldsymbol b^{\,*j}
			+\one_{q-1}-\boldsymbol m^{\,*j}
			\bigr)_{u+\lambda}
		}{
			\bigl(b_j\one_r-\boldsymbol a+\one_r\bigr)_{u+\lambda}
		}
		\frac{
			(1)_u\bigl(b_j\one_r-\boldsymbol a\bigr)_\lambda
		}{
			\bigl(b_j\one_{q-1}-\boldsymbol b^{\,*j}\bigr)_\lambda
		}
		\frac{\bigl((-1)^{q-r}z\bigr)^u}{u!}
		\frac{1}{\lambda!}.
	\end{multline*}
	The cancellations arising when the two
	shift identities are applied to
	\eqref{eq:final-B-pi} and
	\eqref{eq:final-B-Psi-finite} are as follows. In the coupled \(u+\lambda\) block,
	the \(j\)-th entry of
	\(b_j\one_q-\boldsymbol b+\boldsymbol e_j\) is \(1\), while its remaining entries
	form the string \(b_j\one_{q-1}-\boldsymbol b^{\,*j}\). Thus this \(q\)-string,
	together with
	\(b_j\one_{q-1}-\boldsymbol b^{\,*j}+\one_{q-1}\), cancels the scalar
	parameter \(1\) and the corresponding two \((q-1)\)-strings in the
	denominator. The common string \(b_j\one_r-\boldsymbol a\) also cancels in the
	coupled block. In the \(\lambda\)-dependent block, the common parameter
	\(1\) cancels, leaving \(b_j\one_r-\boldsymbol a\) in the numerator and
	\(b_j\one_{q-1}-\boldsymbol b^{\,*j}\) in the denominator.
	
	The reflected Pochhammer factors contribute
	\((-1)^{(q-r)u}\), which combines with \(z^u\) to give
	\(\bigl((-1)^{q-r}z\bigr)^u\). Hence the preceding expression is
	exactly the \((u,\lambda)\)-term of
	\(\pi_{j,0}\mathcal K_j(z)\). The coupled parameter \(1-m_j\)
	restricts the sum to \(u+\lambda\le m_j-1\).
	
	Assume now that \(H\ne j\). Set \(K=1+u+\lambda\) and \(d=u\). The
	conditions \(0\le d\le K-1\) and \(K\le m_H-1\) are equivalent to
	\(u,\lambda\in\N_0\) and \(u+\lambda\le m_H-2\). In particular, this
	family is empty unless \(m_H\ge2\). At \(u=\lambda=0\), one has
	\(K=1\) and \(d=0\), and the corresponding coefficient is
	\(\mathcal C_{j,H}\) by its definition.
	Put again \(s\coloneq u+\lambda\). The shift from \(K=1\) to
	\(K=1+s\) gives
	\begin{equation*}
		\frac{\pi_{H,1+s}}{\pi_{H,1}}
		=
		\frac{(2-m_H)_s}{(2)_s}
		\frac{
			\bigl(\boldsymbol\kappa+(b_H+2)\one_p+\boldsymbol n\bigr)_s
		}{
			\bigl(\boldsymbol\kappa+(b_H+2)\one_p\bigr)_s
		}
		\frac{
			\bigl(b_H\one_{q-1}-\boldsymbol b^{\,*H}
			+2\one_{q-1}-\boldsymbol m^{\,*H}\bigr)_s
		}{
			\bigl(b_H\one_{q-1}-\boldsymbol b^{\,*H}
			+2\one_{q-1}\bigr)_s
		}.
	\end{equation*}
	The \(H\)-th numerator parameter is now \(1-K=-s\), and therefore
		\((-s)_u=(-1)^u\frac{s!}{\lambda!}\).
	These two identities account explicitly for the coupled terminating
	parameter and for the factorial split in the off-diagonal block.
	
	Applying the same two shift identities, now with
	\(s=u+\lambda\) and base point \(K=1\), the summand corresponding to
	\((u,\lambda)\) becomes
	\begin{multline*}
		\mathcal C_{j,H}\,
		\frac{
			(2-m_H)_{u+\lambda}
			\bigl(\boldsymbol\kappa+(b_H+2)\one_p+\boldsymbol n\bigr)_{u+\lambda}
		}{
			\bigl(\boldsymbol\kappa+(b_H+2)\one_p\bigr)_{u+\lambda}
		}
		\\
		\times
		\frac{
			\bigl(
			b_H\one_{q-1}-\boldsymbol b^{\,*H}
			+2\one_{q-1}-\boldsymbol m^{\,*H}
			\bigr)_{u+\lambda}
		}{
			\bigl(b_H\one_r-\boldsymbol a+2\one_r\bigr)_{u+\lambda}
		}
		\frac{
			(1)_u
			\bigl(b_H\one_r-\boldsymbol a+\one_r\bigr)_\lambda
		}{
			\bigl(
			b_H\one_{q-1}-\boldsymbol b^{\,*H}
			+\one_{q-1}+\boldsymbol e_j^{\,*H}
			\bigr)_\lambda
		}
		\\
		\times
		\frac{\bigl((-1)^{q-r}z\bigr)^u}{u!}
		\frac{1}{\lambda!}.
	\end{multline*}
	The cancellations arising from the application
	of the shift identities to \eqref{eq:final-B-pi} and
	\eqref{eq:final-B-Psi-finite} are described next. Since \(H\ne j\), the \(H\)-th entry of
	\(b_H\one_q-\boldsymbol b+\one_q+\boldsymbol e_j\) is \(1\), its \(j\)-th entry is
	\(b_H-b_j+2\), and its remaining entries are \(b_H-b_h+1\).
	Together with the string
	\(b_H\one_{q-1}-\boldsymbol b^{\,*j}+2\one_{q-1}\), these entries cancel the
	scalar parameter \(2\), the string
	\(b_H\one_{q-1}-\boldsymbol b^{\,*H}+2\one_{q-1}\), and the string
	\(b_H\one_{q-1}-\boldsymbol b^{\,*j}+\one_{q-1}\) in the coupled denominator.
	The common string \(b_H\one_r-\boldsymbol a+\one_r\) also cancels in that
	block.
	
	In the \(\lambda\)-dependent block, the numerator contains the scalar
	parameter \(1\). This cancels the \(H\)-th entry of
	\(b_H\one_q-\boldsymbol b+\one_q+\boldsymbol e_j\). After deleting that entry, the
	remaining denominator string is
	\(b_H\one_{q-1}-\boldsymbol b^{\,*H}
	+\one_{q-1}+\boldsymbol e_j^{\,*H}\). As in the diagonal case, the reflected
	Pochhammer factors contribute \((-1)^{(q-r)u}\). Therefore the
	preceding expression is exactly the \((u,\lambda)\)-term of
	\(\mathcal C_{j,H}\mathcal K_{j,H}(z)\). The coupled parameter
	\(2-m_H\) restricts the sum to \(u+\lambda\le m_H-2\).
	
	Summing the contribution \(\pi_{j,0}\mathcal K_j(z)\) of the family
	\(H=j\) and the contributions
	\(\mathcal C_{j,H}\mathcal K_{j,H}(z)\) of the families \(H\ne j\)
	gives \eqref{eq:final-B-KdF-representation}.
\end{proof}

\subsection{Weak and strong normality}
\label{subsec:weak-strong-normality}

Theorem~\ref{thm:final-explicit} constructs explicit solutions of the
near-diagonal mixed problems. It is shown below that, under the same regularity
assumptions, these solutions exhaust the corresponding solution spaces. The
argument is entirely finite. On the \(A\)-side, the orthogonality conditions
force the numerator of a rational-Gamma interpolant to be divisible by a
prescribed polynomial. On the \(B\)-side, they force a reconstruction
polynomial to vanish at a prescribed set of nodes, while a partial-fraction
argument shows that this polynomial determines the original vector uniquely.

\begin{proposition}[Finite rational characterization and weak normality]
	\label{prop:finite-rational-characterization}
	Assume that the Bessel parameters are regular in the sense of
	Definition~\ref{def:final-parameter-regularity}. Then, for the
	near-diagonal balanced index pairs covered by
	Theorem~\ref{thm:final-explicit}, the corresponding mixed problems are
	weakly normal. Equivalently, their solution spaces are one-dimensional,
	and any two nonzero solutions differ by multiplication by a nonzero
	constant.
	
	More precisely, consider first an \(A\)-balanced pair. Every column
	vector
	\(\mathbf P
	=
	\begin{bNiceMatrix}
		P^{(1)}&
		\Cdots&
		P^{(p)}
	\end{bNiceMatrix}^\top\),
	with \(P^{(i)}\in\Poly_{n_i-1}\) when \(n_i\ge1\) and
	\(P^{(i)}\equiv0\) when \(n_i=0\), satisfying
	\eqref{eq:final-A-orth}, is a scalar multiple of
	\(\mathbf A_{\boldsymbol n,\boldsymbol m}\).
	
	For a \(B\)-balanced pair, every row vector
\(	\mathbf Q
	=
	\begin{bNiceMatrix}
		Q^{(1)}&\Cdots&Q^{(q)}
	\end{bNiceMatrix}\),
	with \(Q^{(j)}\in\Poly_{m_j-1}\) when \(m_j\ge1\) and
	\(Q^{(j)}\equiv0\) when \(m_j=0\), satisfying
	\eqref{eq:final-B-orth}, is a scalar multiple of
	\(\mathbf B_{\boldsymbol n,\boldsymbol m}\).
\end{proposition}

\begin{proof}
	The two assertions are proved separately. Both arguments are finite and
	algebraic. They are mixed analogues of the Mellin-transform
	characterization used in the one-column Bessel-like setting, but no
	external uniqueness theorem is needed.
	
	\medskip
	\noindent\emph{The \(A\)-side.}
	Let
	\(
	\mathbf P(z)
	=
	\begin{bNiceMatrix}
		P^{(1)}(z)&
		\Cdots&
		P^{(p)}(z)
	\end{bNiceMatrix}^\top
	\)
	be a solution of \eqref{eq:final-A-orth}. For every
	\(i\in\{1,\ldots,p\}\) with \(n_i\ge1\), write
\(P^{(i)}(z) = \sum_{k=0}^{n_i-1}p_{i,k}z^k.\)
	If \(n_i=0\), then \(P^{(i)}\equiv0\), and there are no coefficients
	\(p_{i,k}\) associated with that component.
	
	The coefficients \(p_{i,k}\) will be encoded in a single
	interpolation polynomial. Recall that
\(D_{\boldsymbol\kappa,\boldsymbol n}(t) = \prod_{i=1}^{p} \prod_{k=0}^{n_i-1} (t-\kappa_i-k).\)
	The interpolation nodes are
\(\kappa_i+k,\) \(i\in\{1,\ldots,p\},\) \(k\in\{0,\ldots,n_i-1\}.\)
	There are \(|\boldsymbol n|\) such nodes, and they are pairwise distinct by
	the noninteger-separation assumptions in
	Definition~\ref{def:final-parameter-regularity}.
	
	Consequently, there exists a unique polynomial \(Q_{\mathbf P}\) of
	degree at most \(|\boldsymbol n|-1\) satisfying
	\begin{equation}
		\label{eq:general-A-interpolation}
		Q_{\mathbf P}(\kappa_i+k)
		=
		p_{i,k}
		D_{\boldsymbol\kappa,\boldsymbol n}'(\kappa_i+k)
		\frac{
			\Gamma((\kappa_i+k)\one_r+\boldsymbol a+\one_r)
		}{
			\Gamma((\kappa_i+k)\one_q+\boldsymbol b+\one_q)
		},
	\end{equation}
	for all \(i\in\{1,\ldots,p\}\) and
	\(k\in\{0,\ldots,n_i-1\}\).
	
	Using this interpolation polynomial, define
	\begin{equation}
		\label{eq:general-A-rational-gamma}
		\mathcal R_{\mathbf P}(t)
		\coloneq
		\frac{Q_{\mathbf P}(t)}
		{D_{\boldsymbol\kappa,\boldsymbol n}(t)}
		\frac{
			\Gamma(t\one_q+\boldsymbol b+\one_q)
		}{
			\Gamma(t\one_r+\boldsymbol a+\one_r)
		}.
	\end{equation}
	The purpose of the interpolation conditions
	\eqref{eq:general-A-interpolation} is that the residues of
	\(\mathcal R_{\mathbf P}\) reproduce the coefficients of the
	polynomial vector. Indeed, since the zero of
	\(D_{\boldsymbol\kappa,\boldsymbol n}\) at \(t=\kappa_i+k\) is simple,
	\begin{align*}
		\operatorname*{Res}_{t=\kappa_i+k}
		\mathcal R_{\mathbf P}(t)
		&=
		\frac{
			Q_{\mathbf P}(\kappa_i+k)
		}{
			D_{\boldsymbol\kappa,\boldsymbol n}'(\kappa_i+k)
		}
		\frac{
			\Gamma((\kappa_i+k)\one_q+\boldsymbol b+\one_q)
		}{
			\Gamma((\kappa_i+k)\one_r+\boldsymbol a+\one_r)
		}=
		p_{i,k}.
	\end{align*}
	
	The \(A\)-orthogonality conditions are translated into information
	about \(Q_{\mathbf P}\). Fix
\(j\in\{1,\ldots,q\}\),  \(\ell\in\{0,\ldots,m_j-1\}\).
	Using the expansion of the components of \(\mathbf P\) and the moment
	formula \eqref{eq:final-moment-entry}, the corresponding
	orthogonality condition is
	\begin{equation*}
		0=
		\sum_{i=1}^{p}
		\oint_{\Torus}
		z^\ell C_{j,i}(z)P^{(i)}(z)
		\frac{\dz}{2\pi\mathrm i}=
		\sum_{i=1}^{p}
		\sum_{k=0}^{n_i-1}
		p_{i,k}
		\frac{
			\Gamma((\kappa_i+k+\ell+1)\one_r+\boldsymbol a)
		}{
			\Gamma((\kappa_i+k+\ell+1)\one_q+\boldsymbol b+\boldsymbol e_j)
		}.
	\end{equation*}
	Since \(p_{i,k}\) is the residue of
	\(\mathcal R_{\mathbf P}\) at \(t=\kappa_i+k\), the last expression is
	the sum of the residues at these points of the function
	\begin{equation}
		\label{eq:general-A-tested-rational}
		\mathcal R_{\mathbf P}(t)
		\frac{
			\Gamma((t+\ell+1)\one_r+\boldsymbol a)
		}{
			\Gamma((t+\ell+1)\one_q+\boldsymbol b+\boldsymbol e_j)
		}.
	\end{equation}
	
	Next, rewrite this function as an ordinary rational function. By
	\eqref{eq:general-A-rational-gamma},
	\[
	\mathcal R_{\mathbf P}(t)
	\frac{
		\Gamma((t+\ell+1)\one_r+\boldsymbol a)
	}{
		\Gamma((t+\ell+1)\one_q+\boldsymbol b+\boldsymbol e_j)
	}
	=
	\frac{Q_{\mathbf P}(t)}
	{D_{\boldsymbol\kappa,\boldsymbol n}(t)}
	\frac{
		\Gamma(t\one_q+\boldsymbol b+\one_q)
	}{
		\Gamma(t\one_r+\boldsymbol a+\one_r)
	}
	\frac{
		\Gamma((t+\ell+1)\one_r+\boldsymbol a)
	}{
		\Gamma((t+\ell+1)\one_q+\boldsymbol b+\boldsymbol e_j)
	}.
	\]
	The Gamma quotients satisfy
\(\dfrac{ \Gamma((t+\ell+1)\one_r+\boldsymbol a) }{ \Gamma(t\one_r+\boldsymbol a+\one_r) } = (t\one_r+\boldsymbol a+\one_r)_\ell\)
	and
	\[
	\frac{
		\Gamma(t\one_q+\boldsymbol b+\one_q)
	}{
		\Gamma((t+\ell+1)\one_q+\boldsymbol b+\boldsymbol e_j)
	}
	=
	\frac{1}{
		(t\one_q+\boldsymbol b+\one_q)_\ell
		(t+b_j+\ell+1)
	}.
	\]
	Therefore the function in
	\eqref{eq:general-A-tested-rational} is
	\begin{equation}
		\label{eq:general-A-tested-rational-expanded}
		\frac{
			Q_{\mathbf P}(t)
			(t\one_r+\boldsymbol a+\one_r)_\ell
		}{
			D_{\boldsymbol\kappa,\boldsymbol n}(t)
			(t\one_q+\boldsymbol b+\one_q)_\ell
			(t+b_j+\ell+1)
		}.
	\end{equation}
	
	The residue theorem will be applied to this rational function. First
	observe that its decay at infinity is sufficiently fast. Its numerator
	has degree at most
\(|\boldsymbol n|-1+r\ell,\)
	while its denominator has degree
\(|\boldsymbol n|+q\ell+1.\)
	Hence the degree of the numerator minus that of the denominator is at
	most
\(|\boldsymbol n|-1+r\ell-\bigl(|\boldsymbol n|+q\ell+1\bigr) = -2-(q-r)\ell.\)
	Since \(r<q\), the function in
	\eqref{eq:general-A-tested-rational-expanded} is
	\(\mathrm{O}(t^{-2})\) as \(t\to\infty\). Its residue at infinity
	therefore vanishes, and the sum of all its finite residues is zero.
	
	The orthogonality condition says that the sum of the residues at the
	zeros of \(D_{\boldsymbol\kappa,\boldsymbol n}\) is zero. It follows that the
	sum of the residues at the remaining poles, namely the poles produced
	by
\((t\one_q+\boldsymbol b+\one_q)_\ell(t+b_j+\ell+1),\)
	is also zero.
	
	These residue identities determine the zeros of
	\(Q_{\mathbf P}\). It will be proved, level by level, that
	\begin{equation}
		\label{eq:A-Q-zeros}
		Q_{\mathbf P}(-b_j-s-1)=0,
		\qquad
		j\in\{1,\ldots,q\},
		\qquad
		s\in\{0,\ldots,m_j-1\}.
	\end{equation}
	The induction is simultaneous in the row index \(j\): at level
	\(\ell\), the required zero is proved for every \(j\) such that
	\(\ell\le m_j-1\), assuming that all zeros at the preceding levels
	have already been established.
	
	Fix \(j\in\{1,\ldots,q\}\) and
	\(\ell\in\{0,\ldots,m_j-1\}\). Assume that
\(Q_{\mathbf P}(-b_h-s-1)=0\)
	has already been proved whenever
\(h\in\{1,\ldots,q\},\) \(0\le s<\ell,\) \(s\le m_h-1.\)
	When \(\ell=0\), there are no preceding levels, so the induction
	hypothesis is empty.
	
	The possible poles produced by
	\((t\one_q+\boldsymbol b+\one_q)_\ell\) are
\(t=-b_h-s-1,\) \(h\in\{1,\ldots,q\},\) \(s\in\{0,\ldots,\ell-1\}.\)
	It remains to verify that the induction hypothesis applies to every one of
	these points. Since \(\boldsymbol m\) is near the diagonal, any two
	components of \(\boldsymbol m\) differ by at most one. In particular,
\(m_h\ge m_j-1.\)
	Because \(\ell\le m_j-1\), this gives
\(m_h\ge\ell.\)
	Thus, for every \(s\in\{0,\ldots,\ell-1\}\),
\(s\le\ell-1\le m_h-1.\)
	Hence all the points
	\(-b_h-s-1\) occurring at the preceding levels are zeros of
	\(Q_{\mathbf P}\) by the induction hypothesis. Since the auxiliary
	poles are simple and pairwise distinct by regularity, the corresponding
	residues vanish.
	
	The only auxiliary pole not covered by the induction hypothesis comes
	from the additional factor \(t+b_j+\ell+1\). It is the point
\(t_\ell\coloneq-b_j-\ell-1.\)
	The regularity assumptions and the standing admissibility convention ensure
	that \(t_\ell\) does not coincide
	with any zero of \(D_{\boldsymbol\kappa,\boldsymbol n}\), nor with any
	auxiliary pole from a preceding level. Therefore it is a simple pole
	of \eqref{eq:general-A-tested-rational-expanded}.
	
	Its residue is
	\begin{align*}
		\operatorname*{Res}_{t=t_\ell}
		\frac{
			Q_{\mathbf P}(t)
			(t\one_r+\boldsymbol a+\one_r)_\ell
		}{
			D_{\boldsymbol\kappa,\boldsymbol n}(t)
			(t\one_q+\boldsymbol b+\one_q)_\ell
			(t+b_j+\ell+1)
		}
		& =
		Q_{\mathbf P}(t_\ell)
		\frac{
			(t_\ell\one_r+\boldsymbol a+\one_r)_\ell
		}{
			D_{\boldsymbol\kappa,\boldsymbol n}(t_\ell)
			(t_\ell\one_q+\boldsymbol b+\one_q)_\ell
		}
		\\
		& =
		Q_{\mathbf P}(t_\ell)
		\frac{
			(\boldsymbol a-b_j\one_r-\ell\one_r)_\ell
		}{
			D_{\boldsymbol\kappa,\boldsymbol n}(t_\ell)
			(-\ell)_\ell
			\prod_{\substack{h=1\\h\ne j}}^{q}
			(b_h-b_j-\ell)_\ell
		}.
	\end{align*}
	All the factors multiplying \(Q_{\mathbf P}(t_\ell)\) are finite and
	nonzero. More precisely, the standing admissibility convention following
	Definition~\ref{def:final-parameter-regularity} implies
	\((\boldsymbol a-b_j\one_r-\ell\one_r)_\ell\ne0\),
	\(D_{\boldsymbol\kappa,\boldsymbol n}(t_\ell)\ne0\),
	while the noninteger-separation condition for the parameters \(b_h\) gives
	\(\prod_{\substack{h=1\\h\ne j}}^{q}
	(b_h-b_j-\ell)_\ell\ne0\),
	Moreover,
\((-\ell)_\ell=(-1)^\ell\ell!\ne0\).
	The sum of all auxiliary residues is zero. Every residue belonging to
	a preceding level vanishes by the induction hypothesis. Therefore the
	residue at the only new auxiliary pole \(t=t_\ell\) must also vanish.
	Since its coefficient is nonzero, it follows that
\(Q_{\mathbf P}(t_\ell)=0,\)
	that is,
\(Q_{\mathbf P}(-b_j-\ell-1)=0.\)
	This completes the induction and proves
	\eqref{eq:A-Q-zeros}.
	
	Thus \(Q_{\mathbf P}\) vanishes at all the points
\(-b_j-s-1,\) \(j\in\{1,\ldots,q\},\) \(s\in\{0,\ldots,m_j-1\}.\)
	These points are pairwise distinct. Indeed, if
\(-b_j-s-1=-b_h-u-1,\)
	then
\(b_j-b_h=u-s\in\mathbb Z.\)
	By the noninteger-separation condition for the parameters \(b_j\),
	this forces \(j=h\), and then \(s=u\).
	
	The polynomial whose zeros are precisely these points is
\(P_{\boldsymbol b,\boldsymbol m}(t) = \prod_{j=1}^{q}(t+b_j+1)_{m_j}.\)
	Therefore \(P_{\boldsymbol b,\boldsymbol m}\) divides \(Q_{\mathbf P}\). Moreover,
\(\deg P_{\boldsymbol b,\boldsymbol m} = |\boldsymbol m|\)
	and, since the pair is \(A\)-balanced,
\(\deg Q_{\mathbf P} \le |\boldsymbol n|-1 = |\boldsymbol m|.\)
	It follows that there exists \(c\in\mathbb C\) such that
\(Q_{\mathbf P}(t) = cP_{\boldsymbol b,\boldsymbol m}(t).\)
	Substituting this identity into
	\eqref{eq:general-A-rational-gamma} gives
	\[
	\mathcal R_{\mathbf P}(t)
	=
	c\,
	\frac{P_{\boldsymbol b,\boldsymbol m}(t)}
	{D_{\boldsymbol\kappa,\boldsymbol n}(t)}
	\frac{
		\Gamma(t\one_q+\boldsymbol b+\one_q)
	}{
		\Gamma(t\one_r+\boldsymbol a+\one_r)
	}.
	\]
	The factor \((-1)^{|\boldsymbol n|}\) can be absorbed into \(c\).  The resulting
	rational-Gamma function is the one used in the construction of the
	\(A\)-vector in the proof of
	Theorem~\ref{thm:final-explicit}. The residues of these two functions
	at every node \(t=\kappa_i+k\) are therefore proportional. Since these
	residues are precisely the coefficients of the polynomial components,
	all the coefficients \(p_{i,k}\) are proportional to the corresponding
	coefficients of \(\mathbf A_{\boldsymbol n,\boldsymbol m}\). Hence
	\(\mathbf P\) is a scalar multiple of
	\(\mathbf A_{\boldsymbol n,\boldsymbol m}\).
	
	\medskip
	\noindent\emph{The \(B\)-side.}
	Let
	\(
	\mathbf Q(z)
	=
	\begin{bNiceMatrix}
		Q^{(1)}(z)&\Cdots&Q^{(q)}(z)
	\end{bNiceMatrix}
	\)
	be a solution of \eqref{eq:final-B-orth}. For every
	\(j\in\{1,\ldots,q\}\) with \(m_j\ge1\), write
\(Q^{(j)}(z) = \sum_{d=0}^{m_j-1}q_{j,d}z^d.\)
	If \(m_j=0\), then \(Q^{(j)}\equiv0\), and there are no coefficients
	\(q_{j,d}\) associated with that component.
	
	Associate with \(\mathbf Q\) the polynomial
	\begin{equation}
		\label{eq:general-B-polynomial}
		L_{\mathbf Q}(t)
		\coloneq
		\sum_{j=1}^{q}
		\sum_{d=0}^{m_j-1}
		q_{j,d}
		(t\one_r+\boldsymbol a)_d
		(t\one_q+\boldsymbol b+d\one_q+\boldsymbol e_j)_{
			\boldsymbol m-d\one_q-\boldsymbol e_j}.
	\end{equation}
	First, every term in this expression is shown to be a polynomial.
	For the term indexed by \(j\) and \(d\), one has
	\(d\le m_j-1\), and therefore the \(j\)-th component of
	\(\boldsymbol m-d\one_q-\boldsymbol e_j\) is
\(m_j-d-1\ge0.\)
	If \(h\ne j\), near-diagonality gives
\(m_h\ge m_j-1\ge d,\)
	so the \(h\)-th component \(m_h-d\) is also nonnegative. Thus
\(\boldsymbol m-d\one_q-\boldsymbol e_j\in\N_0^q\)
	for every term occurring in
	\eqref{eq:general-B-polynomial}.
	
	Next, estimate the degree of \(L_{\mathbf Q}\). The first factor in
	the term indexed by \(j\) and \(d\) has degree
\(\deg(t\one_r+\boldsymbol a)_d=rd.\)
	The second factor has degree
\(\left| \boldsymbol m-d\one_q-\boldsymbol e_j \right| = |\boldsymbol m|-qd-1.\)
	Hence the total degree of this term is
\(rd+|\boldsymbol m|-qd-1 = |\boldsymbol m|-1-(q-r)d.\)
	Since \(q>r\), this is at most \(|\boldsymbol m|-1\). For a
	\(B\)-balanced pair, \(|\boldsymbol m|=|\boldsymbol n|+1\), and therefore
\(\deg L_{\mathbf Q} \le |\boldsymbol m|-1 = |\boldsymbol n|.\)
	The \(B\)-orthogonality conditions are translated into zeros of
	\(L_{\mathbf Q}\). Fix \(i\in\{1,\ldots,p\}\), let \(s\ge1\), and set
\(t=s+\kappa_i.\)
	The moment formula \eqref{eq:final-moment-entry} gives
	\begin{equation*}
		\sum_{j=1}^{q}
		\oint_{\Torus}
		z^{s-1}Q^{(j)}(z)C_{j,i}(z)
		\frac{\dz}{2\pi\mathrm i}
 =
		\sum_{j=1}^{q}
		\sum_{d=0}^{m_j-1}
		q_{j,d}
		\frac{
			\Gamma((t+d)\one_r+\boldsymbol a)
		}{
			\Gamma((t+d)\one_q+\boldsymbol b+\boldsymbol e_j)
		}.
	\end{equation*}
	For every term in this sum,
	\[
	\frac{
		\Gamma((t+d)\one_r+\boldsymbol a)
	}{
		\Gamma((t+d)\one_q+\boldsymbol b+\boldsymbol e_j)
	}
	=
	\frac{
		\Gamma(t\one_r+\boldsymbol a)
	}{
		\Gamma(t\one_q+\boldsymbol b+\boldsymbol m)
	}
	(t\one_r+\boldsymbol a)_d
	(t\one_q+\boldsymbol b+d\one_q+\boldsymbol e_j)_{
		\boldsymbol m-d\one_q-\boldsymbol e_j}.
	\]
	Factoring out the Gamma quotient and using
	\eqref{eq:general-B-polynomial} gives
	\begin{equation}
		\label{eq:general-B-moment-identity}
		\sum_{j=1}^{q}
		\oint_{\Torus}
		z^{s-1}Q^{(j)}(z)C_{j,i}(z)
		\frac{\dz}{2\pi\mathrm i}
		=
		\frac{
			\Gamma((s+\kappa_i)\one_r+\boldsymbol a)
		}{
			\Gamma((s+\kappa_i)\one_q+\boldsymbol b+\boldsymbol m)
		}
		L_{\mathbf Q}(s+\kappa_i).
	\end{equation}
	The Gamma quotient in
	\eqref{eq:general-B-moment-identity} is finite and nonzero by
	regularity.
	
	Impose the orthogonality conditions. Take
\(s=\ell+1,\) \(\ell\in\{0,\ldots,n_i-1\}.\)
	The left-hand side of
	\eqref{eq:general-B-moment-identity} then vanishes by
	\eqref{eq:final-B-orth}. Since the Gamma quotient is nonzero, it follows
	that
\(L_{\mathbf Q}(\kappa_i+1+\ell)=0,\) \(\ell\in\{0,\ldots,n_i-1\},\) \(i\in\{1,\ldots,p\}.\)
	These are \(|\boldsymbol n|\) zeros. They are pairwise distinct by the
	noninteger-separation assumptions for the parameters \(\kappa_i\).
	The polynomial having exactly these zeros is
\(\prod_{i=1}^{p} (\kappa_i+1-t)_{n_i}.\)
	Indeed,
\((\kappa_i+1-t)_{n_i} = \prod_{\ell=0}^{n_i-1} (\kappa_i+1+\ell-t).\)
	Since \(\deg L_{\mathbf Q}\le|\boldsymbol n|\), the existence of these
	\(|\boldsymbol n|\) distinct zeros implies that
	\begin{equation}
		\label{eq:general-B-divisibility}
		L_{\mathbf Q}(t)
		=
		-c
		\prod_{i=1}^{p}
		(\kappa_i+1-t)_{n_i}
	\end{equation}
	for some \(c\in\mathbb C\). The sign has been inserted so that the
	convention agrees with
	\eqref{eq:final-B-polynomial-identity}. The case \(c=0\) is allowed and
	corresponds to \(L_{\mathbf Q}\equiv0\).
	
	To conclude that \(\mathbf Q\) is a scalar multiple of the explicit
	vector, it remains to prove that the polynomial \(L_{\mathbf Q}\)
	determines all the coefficients \(q_{j,d}\) uniquely. For this purpose,
	divide \eqref{eq:general-B-polynomial} by
	\((t\one_q+\boldsymbol b)_{\boldsymbol m}\).
	
	For each pair \((j,d)\), one has
	\[
	\frac{
		(t\one_q+\boldsymbol b+d\one_q+\boldsymbol e_j)_{
			\boldsymbol m-d\one_q-\boldsymbol e_j}
	}{
		(t\one_q+\boldsymbol b)_{\boldsymbol m}
	}
	=
	\frac{1}{
		(t\one_q+\boldsymbol b)_d(t+b_j+d)
	}.
	\]
	Indeed, for \(h\ne j\),
\(\frac{ (t+b_h+d)_{m_h-d} }{ (t+b_h)_{m_h} } = \frac{1}{(t+b_h)_d},\)
	while for the \(j\)-th component,
\(\frac{ (t+b_j+d+1)_{m_j-d-1} }{ (t+b_j)_{m_j} } = \frac{1}{ (t+b_j)_d(t+b_j+d) }.\)
	Consequently,
	\begin{equation}
		\label{eq:B-rational-expansion-injective}
		\frac{
			L_{\mathbf Q}(t)
		}{
			(t\one_q+\boldsymbol b)_{\boldsymbol m}
		}
		=
		\sum_{j=1}^{q}
		\sum_{d=0}^{m_j-1}
		q_{j,d}
		\frac{
			(t\one_r+\boldsymbol a)_d
		}{
			(t\one_q+\boldsymbol b)_d(t+b_j+d)
		}.
	\end{equation}
	
The following rational functions are linearly independent:
\(\frac{ (t\one_r+\boldsymbol a)_d }{ (t\one_q+\boldsymbol b)_d(t+b_j+d) },\) \(j\in\{1,\ldots,q\},\) \(d\in\{0,\ldots,m_j-1\},\)
and this will imply that the coefficients
	\(q_{j,d}\), and hence the vector \(\mathbf Q\), are uniquely
	determined by \(L_{\mathbf Q}\).
	
	Suppose, to the contrary, that
	\begin{equation}
		\label{eq:B-rational-linear-relation}
		\sum_{j=1}^{q}
		\sum_{d=0}^{m_j-1}
		q_{j,d}
		\frac{
			(t\one_r+\boldsymbol a)_d
		}{
			(t\one_q+\boldsymbol b)_d(t+b_j+d)
		}
		=
		0
	\end{equation}
	and that not all coefficients \(q_{j,d}\) vanish.
	
	Let \(d_0\) be the largest value of \(d\) for which
	\(q_{j,d_0}\ne0\) for at least one \(j\). Choose
	\(j_0\in\{1,\ldots,q\}\) such that
\(q_{j_0,d_0}\ne0,\)
	and consider the point
\(t_0\coloneq-b_{j_0}-d_0.\)
	The point \(t_0\) is the pole corresponding to the last denominator
	factor \(t+b_{j_0}+d_0\) of the term indexed by
	\((j_0,d_0)\). No other term occurring in
	\eqref{eq:B-rational-linear-relation} has a pole at this point.
	
	Consider first a term whose summation index satisfies \(d<d_0\).
	A pole could arise from its last factor \(t+b_j+d\) only if
\(-b_{j_0}-d_0+b_j+d=0,\)
	or equivalently
\(b_j-b_{j_0}=d_0-d.\)
	The right-hand side is an integer. If \(j\ne j_0\), this contradicts
	the noninteger-separation condition
	\(b_j-b_{j_0}\notin\mathbb Z\). If \(j=j_0\), the equality would imply
	\(d=d_0\), contrary to \(d<d_0\). Thus no factor \(t+b_j+d\) with
	\(d<d_0\) vanishes at \(t=t_0\).
	
	A pole could also arise from a factor of \((t\one_q+\boldsymbol b)_d\). Such a
	factor has the form
\(t+b_h+s,\) \(h\in\{1,\ldots,q\},\) \(s\in\{0,\ldots,d-1\}.\)
	It would vanish at \(t=t_0\) only if
\(b_h-b_{j_0}=d_0-s.\)
	If \(h\ne j_0\), this is again excluded because the right-hand side is
	an integer. If \(h=j_0\), the equality would give \(s=d_0\), which is
	impossible because
\(s\le d-1<d_0.\)
	Therefore no term with \(d<d_0\) has a pole at \(t=t_0\).
	
	Consider next the terms with \(d=d_0\). If \(j\ne j_0\), their last
	factor \(t+b_j+d_0\) would vanish at \(t=t_0\) only if
\(b_j-b_{j_0}=0,\)
	which is excluded by regularity. The factors in
	\((t\one_q+\boldsymbol b)_{d_0}\) are also nonzero at \(t=t_0\). For the
	\(j_0\)-th component, their product is
\((-d_0)_{d_0}=(-1)^{d_0}d_0!\ne0,\)
	while for every \(h\ne j_0\), nonvanishing follows from
	\(b_h-b_{j_0}\notin\mathbb Z\).
	
	It follows that, among all the terms in
	\eqref{eq:B-rational-linear-relation}, only the term indexed by
	\((j_0,d_0)\) has a pole at \(t=t_0\). Its residue is
	\[
	q_{j_0,d_0}
	\frac{
		(\boldsymbol a-b_{j_0}\one_r-d_0\one_r)_{d_0}
	}{
		(-d_0)_{d_0}
		(\boldsymbol b^{\,*j_0}-b_{j_0}\one_{q-1}
		-d_0\one_{q-1})_{d_0}
	}.
	\]
	The factor \((-d_0)_{d_0}\) is nonzero, as observed above. Moreover,
	the standing admissibility convention following
	Definition~\ref{def:final-parameter-regularity} implies
\((\boldsymbol a-b_{j_0}\one_r-d_0\one_r)_{d_0}\ne0\)
	while the noninteger-separation condition for the parameters \(b_h\) gives
\((\boldsymbol b^{\,*j_0}-b_{j_0}\one_{q-1} -d_0\one_{q-1})_{d_0}\ne0.\)
	Hence the residue is a nonzero multiple of \(q_{j_0,d_0}\). Since
	\(q_{j_0,d_0}\ne0\), this residue is nonzero.
	
	This contradicts the assumption that the rational function in
	\eqref{eq:B-rational-linear-relation} vanishes identically. Therefore
	all the coefficients \(q_{j,d}\) must vanish, and the rational
	functions in
	\eqref{eq:B-rational-expansion-injective} are linearly independent.
	Equivalently, the linear map
\(\mathbf Q\longmapsto L_{\mathbf Q}\)
	is injective.
	
	Let \(\mathbf Q_{\mathrm{exp}}\) denote the unnormalized
	\(B\)-polynomial vector constructed in
	Theorem~\ref{thm:final-explicit}. By
	\eqref{eq:final-B-polynomial-identity}, its associated polynomial is
\(L_{\mathbf Q_{\mathrm{exp}}}(t) = - \prod_{i=1}^{p} (\kappa_i+1-t)_{n_i}.\)
	On the other hand, \eqref{eq:general-B-divisibility} gives
\(L_{\mathbf Q}(t) = -c \prod_{i=1}^{p} (\kappa_i+1-t)_{n_i} = cL_{\mathbf Q_{\mathrm{exp}}}(t).\)
	By linearity of the map
	\(\mathbf Q\mapsto L_{\mathbf Q}\),
\(L_{\mathbf Q-c\mathbf Q_{\mathrm{exp}}}(t)=0.\)
	Its injectivity therefore implies
\(\mathbf Q = c\mathbf Q_{\mathrm{exp}}.\)
	Thus every solution of the \(B\)-orthogonality conditions is a scalar
	multiple of the vector constructed in
	Theorem~\ref{thm:final-explicit}.
	
	This proves that both the \(A\)- and \(B\)-solution spaces are
	one-dimensional. This completes the proof.
\end{proof}

\begin{proposition}[Strong normality on the \(A\)-side]
\label{prop:strong-normality-A}
Assume that the Bessel parameters are regular in the sense of
Definition~\ref{def:final-parameter-regularity}.  Then every near-diagonal
\(A\)-balanced index pair covered by Theorem~\ref{thm:final-explicit} is
strongly normal for the \(A\)-polynomial problem.
\end{proposition}

\begin{proof}
The weak normality part has already been proved in
Proposition~\ref{prop:finite-rational-characterization}.  It remains to
check degree attainment on the \(A\)-side.

Let \(n_i\ge1\).  In the explicit expression
\eqref{eq:final-A-components}, the coefficient of \(z^{n_i-1}\) comes from
the highest term of the terminating hypergeometric series.  It is
\[
\nu_{\boldsymbol n,\boldsymbol m}^{A}
C_{\boldsymbol n,\boldsymbol m;i}
\frac{(-n_i+1)_{n_i-1}}{(n_i-1)!}
\frac{
\prod_{j=1}^{q}(\kappa_i+b_j+m_j+1)_{n_i-1}
\prod_{\substack{\ell=1\\\ell\ne i}}^{p}
(\kappa_i-\kappa_\ell-n_\ell+1)_{n_i-1}
}{
\prod_{\rho=1}^{r}(\kappa_i+a_\rho+1)_{n_i-1}
\prod_{\substack{\ell=1\\\ell\ne i}}^{p}
(\kappa_i-\kappa_\ell+1)_{n_i-1}
}.
\]
By the standing admissibility convention following
Definition~\ref{def:final-parameter-regularity}, every Gamma factor entering
the parameter-dependent Pochhammer quotients above is finite; hence none of
those factors vanishes. Moreover,
\((-n_i+1)_{n_i-1}=(-1)^{n_i-1}(n_i-1)!\ne0.\)
Consequently,
\(\deg A_{\boldsymbol n,\boldsymbol m}^{(i)}=n_i-1,\) \(i\in\{1,\ldots,p\},\) \(n_i\ge1.\)
Together with weak normality, this proves strong normality for the
\(A\)-polynomial problem.
\end{proof}

The possible failure of \(B\)-strong normality can only occur in a very
specific way.  Since the multi-index \(\boldsymbol m\) is near-diagonal, its
components take only the two values \(M\) and \(M-1\).  The components with
maximal value \(M\) have a top coefficient which is isolated by a residue and
is nonzero under regularity and the standing admissibility convention.  The
only possible cancellations occur in the
components with \(m_j=M-1\).  The following notation records
these exceptional cancellations explicitly.

\begin{definition}[Exceptional functions for \(B\)-strong normality]
	\label{def:B-strong-normality-exceptional-functions}
	Let \((\boldsymbol n,\boldsymbol m)\) be a near-diagonal \(B\)-balanced pair covered by
	Theorem~\ref{thm:final-explicit}, and set
	\(M\coloneq \max_{1\le h\le q}m_h\).
	Since \(\boldsymbol m\) is near-diagonal, each \(m_j\) is either \(M\) or \(M-1\).
	Write
	\[
	\mathcal J_+
	\coloneq
	\{H\in\{1,\ldots,q\}:m_H=M\},
	\qquad
	\mathcal J_-
	\coloneq
	\{j\in\{1,\ldots,q\}:m_j=M-1,\ m_j\ge1\}.
	\]
	For \(j\in\mathcal J_-\), put
	\(d_j\coloneq m_j-1=M-2\),
	\(t_j\coloneq -b_j-d_j\).
	For \(H\in\mathcal J_+\), let \(\boldsymbol b_+^{\,*H}\) denote the vector obtained
	from \((b_L)_{L\in\mathcal J_+}\) by deleting the component \(b_H\).
	The exceptional function attached to \(j\in\mathcal J_-\) is defined by
	\begin{multline*}
		\mathcal F_j
		\coloneq
		-
		\frac{
			(\boldsymbol\kappa+(b_j+M-1)\one_p)_{\boldsymbol n}
		}{
			(-M+2)_{M-2}
			(\boldsymbol b^{\,*j}-(b_j+M-2)\one_{q-1})_{\boldsymbol m^{\,*j}}
		}
		\\
		+
		\frac{
			(\boldsymbol a-b_j\one_r-(M-2)\one_r)_{M-1}
		}{
			(-M+2)_{M-2}
			(\boldsymbol b^{\,*j}-(b_j+M-2)\one_{q-1})_{(M-1)\one_{q-1}}
		}
		\\
		\times
		\sum_{\substack{H=1\\m_H=M}}^{q}
		\frac{
			(\boldsymbol\kappa+(b_H+M)\one_p)_{\boldsymbol n}
		}{
			(\boldsymbol a-b_H\one_r-(M-1)\one_r)_{M-1}
			(b_H-b_j+1)
			(\boldsymbol b_+^{\,*H}-b_H\one_{|\mathcal J_+|-1})_1
		}.
	\end{multline*}
\end{definition}

The exceptional functions \(\mathcal F_j\) are written as sums over the
indices in \(\mathcal J_+\), corresponding to the components with maximal
value \(m_H=M\).  Although the formula singles out these indices one by
one, the resulting condition depends only on the set of parameters
\(\{b_H:H\in\mathcal J_+\}\), not on the order in which they are listed.  The
next proposition records this invariance and the algebraic nature of the
exceptional equations.

\begin{proposition}[Structure of the exceptional functions]
	\label{prop:B-exceptional-functions-structure}
	Assume that the Bessel parameters are regular in the sense of
	Definition~\ref{def:final-parameter-regularity}. Let
	\((\boldsymbol n,\boldsymbol m)\) be a near-diagonal \(B\)-balanced pair covered by
	Theorem~\ref{thm:final-explicit}. For each \(j\in\mathcal J_-\), the function
	\(\mathcal F_j\) of
	Definition~\ref{def:B-strong-normality-exceptional-functions} is symmetric in
	the parameters \(b_H\), \(H\in\mathcal J_+\).
	
	Equivalently, the sum over the indices \(H\) such that \(m_H=M\) depends only
	on the set of parameters \(\{b_H:m_H=M\}\), and not on any ordering of those
	indices. Consequently, after clearing the nonzero denominators allowed by
	regularity and the standing admissibility convention,
\(\mathcal F_j=0\)
	is an explicit algebraic equation in the parameters. If \(\boldsymbol m\) is
	diagonal, then \(\mathcal J_-=\varnothing\), and no exceptional equation
	appears.
\end{proposition}

\begin{proof}
	Start from the rational reconstruction identity used in the proof of
	Theorem~\ref{thm:final-explicit}. With
\(R(t) \coloneq - \frac{ (\boldsymbol\kappa+\one_p-t\one_p)_{\boldsymbol n} }{ (t\one_q+\boldsymbol b)_{\boldsymbol m} },\)
	one has
\(R(t) = \sum_{j=1}^{q} \sum_{d=0}^{m_j-1} \beta_{j,d} \frac{(t\one_r+\boldsymbol a)_d}{(t\one_q+\boldsymbol b)_d(t+b_j+d)} .\)
	Here
\(\widehat B_{\boldsymbol n,\boldsymbol m}^{(j)}(z) = \sum_{d=0}^{m_j-1}\beta_{j,d}z^d\)
	denotes the unnormalised \(B\)-component.
	
	Let \(H\in\mathcal J_+\). Taking the residue at
	\(t=-b_H-(M-1)\) isolates the coefficient \(\beta_{H,M-1}\), because
	regularity separates the strings \(-b_h-\mathbb N_0\). This gives
	\[
	\beta_{H,M-1}
	=
	-
	\frac{
		(\boldsymbol\kappa+(b_H+M)\one_p)_{\boldsymbol n}
	}{
		(\boldsymbol a-b_H\one_r-(M-1)\one_r)_{M-1}
		\prod_{\substack{L=1\\L\ne H,\ m_L=M}}^{q}
		(b_L-b_H)
	}.
	\]
	Thus every component indexed by \(\mathcal J_+\) has a nonzero top
	coefficient.
	
	Now fix \(j\in\mathcal J_-\). The last possible degree in the \(j\)-th
	component is \(d_j=M-2\), and take the residue at
	\(t_j=-b_j-d_j=-b_j-M+2\).
	This residue receives the contribution of the coefficient \(\beta_{j,M-2}\),
	but also the contributions of the top coefficients \(\beta_{H,M-1}\),
	\(H\in\mathcal J_+\). Indeed, for \(d=M-1\), the common factor
	\((t\one_q+\boldsymbol b)_d\) contains \(t+b_j+M-2\), which vanishes at \(t=t_j\). All
	other terms are regular at \(t=t_j\). Hence
	\[
	\rho_j
	=
	\alpha_{j,j}\beta_{j,M-2}
	+
	\sum_{\substack{H=1\\m_H=M}}^{q}
	\alpha_{H,j}\beta_{H,M-1},
	\]
	where
\begin{align*}
	\rho_j
	&=
	-
	\frac{
		(\boldsymbol\kappa+(b_j+M-1)\one_p)_{\boldsymbol n}
	}{
		(-M+2)_{M-2}
		(\boldsymbol b^{\,*j}-(b_j+M-2)\one_{q-1})_{\boldsymbol m^{\,*j}}
	},
\\
	\alpha_{j,j}
&	=
	\frac{
		(\boldsymbol a-b_j\one_r-(M-2)\one_r)_{M-2}
	}{
		(-M+2)_{M-2}
		(\boldsymbol b^{\,*j}-(b_j+M-2)\one_{q-1})_{(M-2)\one_{q-1}}
	},
\\
	\alpha_{H,j}
&	=
	\frac{
		(\boldsymbol a-b_j\one_r-(M-2)\one_r)_{M-1}
	}{
		(-M+2)_{M-2}
		(\boldsymbol b^{\,*j}-(b_j+M-2)\one_{q-1})_{(M-1)\one_{q-1}}
		(b_H-b_j+1)
	}.
\end{align*}
	Since \(\alpha_{j,j}\ne0\) under regularity and the standing admissibility
	convention, \(\beta_{j,M-2}\) vanishes
	exactly when
	\[
	\rho_j-
	\sum_{\substack{H=1\\m_H=M}}^{q}
	\alpha_{H,j}\beta_{H,M-1}=0.
	\]
	For each \(H\in\mathcal J_+\), substitution of the explicit top
	coefficient gives
	\begin{multline*}
		-\alpha_{H,j}\beta_{H,M-1}
		=
		\frac{
			(\boldsymbol a-b_j\one_r-(M-2)\one_r)_{M-1}
		}{
			(-M+2)_{M-2}
			(\boldsymbol b^{\,*j}-(b_j+M-2)\one_{q-1})_{(M-1)\one_{q-1}}
		}
		\\
		\times
		\frac{
			(\boldsymbol\kappa+(b_H+M)\one_p)_{\boldsymbol n}
		}{
			(\boldsymbol a-b_H\one_r-(M-1)\one_r)_{M-1}
			(b_H-b_j+1)
			(\boldsymbol b_+^{\,*H}-b_H\one_{|\mathcal J_+|-1})_1
		}.
	\end{multline*}
	Consequently,
	\[
		\rho_j-
		\sum_{\substack{H=1\\m_H=M}}^{q}
		\alpha_{H,j}\beta_{H,M-1}
		=
		\rho_j+
		\sum_{\substack{H=1\\m_H=M}}^{q}
		\bigl(-\alpha_{H,j}\beta_{H,M-1}\bigr)
		=
		\mathcal F_j,
	\]
	which gives precisely the formula in
	Definition~\ref{def:B-strong-normality-exceptional-functions}.
	
	It remains to prove the symmetry assertion. The only part of
	\(\mathcal F_j\) depending on the individual parameters \(b_H\) with
	\(m_H=M\) is
	\[
	\sum_{\substack{H=1\\m_H=M}}^{q}
	\dfrac{
		f(b_H)
	}{
		\prod\limits_{\substack{L=1\\L\ne H,\ m_L=M}}^{q}
		(b_L-b_H)
	},
	\]
	where
\(f(x) = \frac{ (\boldsymbol\kappa+(x+M)\one_p)_{\boldsymbol n} }{ (\boldsymbol a-x\one_r-(M-1)\one_r)_{M-1}(x-b_j+1) }.\)
	If the parameters \(b_H\) with \(m_H=M\) are relabelled, the summands are
	relabelled in the same way and the value of the whole sum is unchanged. Thus
	the apparent dependence on the chosen index \(H\) is only notational; the
	expression depends only on the set \(\{b_H:m_H=M\}\). Hence \(\mathcal F_j\)
	is symmetric in the parameters \(b_H\) with \(m_H=M\).
	
	Finally, all factors in the denominators are nonzero on the admissible regular
	parameter domain. Clearing these denominators turns \(\mathcal F_j=0\) into an explicit
	algebraic equation in the parameters. If \(\boldsymbol m\) is diagonal, then
	\(\mathcal J_-=\varnothing\), so no exceptional equation is present.
\end{proof}

These explicit exceptional functions are now related to strong normality.  The
preceding proposition gives the explicit exceptional equations; the next
result shows that they are exactly the equations for the loss of the top
coefficient in the components with \(m_j=M-1\).

\begin{proposition}[Exceptional equations for \(B\)-strong normality]
	\label{prop:B-strong-normality-exceptional-equations}
	Assume that the Bessel parameters are regular in the sense of
	Definition~\ref{def:final-parameter-regularity}. Let
	\((\boldsymbol n,\boldsymbol m)\) be a near-diagonal \(B\)-balanced pair covered by
	Theorem~\ref{thm:final-explicit}, and let \(M\), \(\mathcal J_+\),
	\(\mathcal J_-\), and \(\mathcal F_j\), \(j\in\mathcal J_-\), be as in
	Definition~\ref{def:B-strong-normality-exceptional-functions}.
	
	Then every component with index \(H\in\mathcal J_+\) attains its prescribed
	maximal degree:
	\(\deg B_{\boldsymbol n,\boldsymbol m}^{(H)}=m_H-1\).
	The only components which may lose their highest allowed degree are those
	indexed by \(\mathcal J_-\), and for \(j\in\mathcal J_-\),
	\(\deg B_{\boldsymbol n,\boldsymbol m}^{(j)}<m_j-1\)
	if and only if 
	\(\mathcal F_j=0\).
	Consequently, \(B\)-strong normality holds on the complement, inside the
	admissible regular parameter domain, of the finite union of algebraic
	exceptional sets
	\[
	\mathcal E_{\boldsymbol n,\boldsymbol m}^{B}
	=
	\bigcup_{\substack{j=1\\m_j=M-1,\ m_j\ge1}}^{q}
	\{\mathcal F_j=0\}.
	\]
	In particular, if \(\boldsymbol m\) is diagonal, then \(\mathcal J_-=\varnothing\),
	and the \(B\)-polynomial vector is strongly normal under regularity and the
	standing admissibility convention.
\end{proposition}

\begin{proof}
	Write the unnormalised \(B\)-components as
\(\widehat B_{\boldsymbol n,\boldsymbol m}^{(j)}(z) = \sum_{d=0}^{m_j-1}\beta_{j,d}z^d .\)
	The residue computation carried out in the proof of
	Proposition~\ref{prop:B-exceptional-functions-structure} shows first that,
	for every \(H\in\mathcal J_+\), the top coefficient
	\(\beta_{H,M-1}\) is nonzero under regularity and the standing admissibility
	convention. Hence
\(\deg \widehat B_{\boldsymbol n,\boldsymbol m}^{(H)}=M-1=m_H-1,\) \(H\in\mathcal J_+ .\)
	Let now \(j\in\mathcal J_-\). The last allowed coefficient of the
	\(j\)-th component is \(\beta_{j,M-2}\). By the same residue identity used
	in Proposition~\ref{prop:B-exceptional-functions-structure}, the vanishing of
	this coefficient is equivalent to the exceptional equation
	\(\mathcal F_j=0\).
	Therefore
\(\deg \widehat B_{\boldsymbol n,\boldsymbol m}^{(j)}<M-2\) \(\Longleftrightarrow\) \(\mathcal F_j=0.\)
	Since \(m_j=M-1\), this is exactly
\(\deg \widehat B_{\boldsymbol n,\boldsymbol m}^{(j)}<m_j-1\) \(\Longleftrightarrow\) \(\mathcal F_j=0.\)
	The normalising factor \(\nu_{\boldsymbol n,\boldsymbol m}^{B}\) is nonzero under regularity
	and the standing admissibility convention, so the same degree criterion holds for the normalised
	components \(B_{\boldsymbol n,\boldsymbol m}^{(j)}\). Hence the only possible failures of
	\(B\)-strong normality occur on the zero sets \(\{\mathcal F_j=0\}\), with
	\(j\in\mathcal J_-\). This gives the stated exceptional set. If
	\(\boldsymbol m\) is diagonal, then \(\mathcal J_-=\varnothing\), and there are no
	exceptional equations.
\end{proof}
\begin{example}
	This example shows that \(B\)-strong normality can indeed fail within the regular
	parameter domain.  Take
	\(p=q=2\),
	\(r=1\),
	\(\boldsymbol n=\begin{bNiceMatrix}2&2\end{bNiceMatrix}\),
	\(\boldsymbol m=\begin{bNiceMatrix}3&2\end{bNiceMatrix}\),
	and choose
	\(b_1=\frac{1}{10},\) \(\kappa_1=\frac{1}{5},\) \(\kappa_2=\frac{19}{20},\) \(a_1=\frac{3}{10}.\)
	Let \(b_2=\xi\), where \(\xi\) is the real root in the interval
	\((-0.81,-0.80)\) of
	\(12800x^3+171520x^2-1674878x-1447069=0 .\)
	The pair is \(B\)-balanced and near-diagonal.  The displayed interval, together
	with the rational values of \(b_1\), \(a_1\), \(\kappa_1\), and \(\kappa_2\),
	shows that the row and column strings are separated and that all Gamma and
	Pochhammer denominator factors appearing in the construction are regular.
	Thus these are regular Bessel parameters.
	
	Let \(\widehat{\mathbf B}_{\boldsymbol n,\boldsymbol m}\) denote the unnormalised vector
	before multiplication by \(\nu_{\boldsymbol n,\boldsymbol m}^{B}\), and write
	\[
	\widehat B_{\boldsymbol n,\boldsymbol m}^{(1)}(z)
	=
	\beta_{1,0}+\beta_{1,1}z+\beta_{1,2}z^2,
	\qquad
	\widehat B_{\boldsymbol n,\boldsymbol m}^{(2)}(z)
	=
	\beta_{2,0}+\beta_{2,1}z .
	\]
	Substitution of the above data into the rational identity used in the proof
	of Theorem~\ref{thm:final-explicit} gives
	\(\beta_{1,2}=-\frac{1289871}{6400}\ne0,\)
	whereas
	\(\beta_{2,1} = \frac{ 12800b_2^3+171520b_2^2-1674878b_2-1447069 }{ 128(10b_2-1)(10b_2+7) }=0 .\)
	The normalising constant \(\nu_{\boldsymbol n,\boldsymbol m}^{B}\) is nonzero, since
	\(\beta_{1,2}\ne0\), and it makes the first component monic of degree two.
	However, the second component has degree strictly smaller than
	\(m_2-1=1\).  Hence this regular near-diagonal \(B\)-balanced pair is not
	strongly normal.
	
	The obstruction is the common denominator \((t\one_q+\boldsymbol b)_d\) in the rational
	reconstruction.  For non-diagonal \(\boldsymbol m\), the last pole of a lower string can
	be coupled to the top-degree terms of higher strings, so the corresponding
	leading coefficient need not be isolated by a single residue.
	
\end{example}

\begin{remark}[The \(K=0\) block and the Wolfs Bessel case]
	The formula for the \(B\)-polynomial vector separates the collapsed \(K=0\) block from the
	terminating hypergeometric blocks with \(K\ge1\). This avoids assigning a
	formal hypergeometric meaning to the case \(K=0\) and \(j\ne H\), where
	the finite contribution is zero. For \(K\ge1\), termination is caused by
	the numerator parameter \(1-K-\delta_{j,H}\).
\end{remark}

\section{Rodrigues-type representation}
\label{sec:matrix-differential-B}

The formulae of Aptekarev--Branquinho--Van Assche for multiple Bessel
polynomials are differential formulae for type-II polynomials.  In the present
mixed setting the row-side type-II object is the \(q\)-component vector
\(\mathbf B_{\boldsymbol n,\boldsymbol m}
=
\begin{bNiceMatrix}
	B^{(1)}_{\boldsymbol n,\boldsymbol m}&\Cdots&B^{(q)}_{\boldsymbol n,\boldsymbol m}
\end{bNiceMatrix}\).
The row-vector differential operators and residue functionals needed to formulate the corresponding matrix differential
representation are introduced next.

\begin{definition}[Row-vector differential operators and residue functionals]
	\label{def:Bessel-row-differential-operators}
	The Bessel regime \(0\le r<q\) is retained throughout this section.
	For \(j\in\{1,\ldots,q\}\) and \(h\in\{1,\ldots,p\}\), set
	\[
	C_{j,h}(z)
	=
	\frac1z
	f_0\left(
	\frac1z;
	\boldsymbol a+\kappa_h\one_r,
	\boldsymbol b+\kappa_h\one_q+\boldsymbol e_j
	\right)
	=
	\sum_{t=0}^{\infty}\mu_{j,h}(t)z^{-t-1},
	\]
	where
	\[
	\mu_{j,h}(t)
	=
	\frac{\Gamma((t+\kappa_h)\one_r+\boldsymbol a+\one_r)}
	{\Gamma(t+\kappa_h+b_j+2)
		\prod_{\ell\ne j}\Gamma(t+\kappa_h+b_\ell+1)}.
	\]
	Let
\(D\coloneq\operatorname{diag}(b_1,\ldots,b_q),\) \(\vartheta\coloneq z\frac{\mathrm d}{\mathrm dz}.\)
	Since \(r<q\), the rational function
\(\frac{\prod_{\rho=1}^{r}(x+a_\rho+1)} {\prod_{j=1}^{q}(x+b_j+1)}\)
	is proper.  Denote its partial-fraction coefficients by
	\[
	\frac{\prod_{\rho=1}^{r}(x+a_\rho+1)}
	{\prod_{j=1}^{q}(x+b_j+1)}
	=
	\sum_{j=1}^{q}
	\frac{\lambda^{\boldsymbol a}_j}{x+b_j+1},
	\]
	that is,
\(\lambda^{\boldsymbol a}_j \coloneq \frac{\prod_{\rho=1}^{r}(a_\rho-b_j)} {\prod_{\ell\ne j}(b_\ell-b_j)}.\)
	Write
	\[
	\lambda^{\boldsymbol a}
	=
	\begin{bNiceMatrix}
		\lambda^{\boldsymbol a}_1&\Cdots&\lambda^{\boldsymbol a}_q
	\end{bNiceMatrix}.
	\]
	
	For a row-vector Laurent polynomial
	\[
	\mathbf P(z)
	=
	\begin{bNiceMatrix}
		P_1(z)&\Cdots&P_q(z)
	\end{bNiceMatrix}
	\in \mathbb C[z,z^{-1}]^{1\times q},
	\]
	define
	\[
	\mathcal T^{\boldsymbol a}_{i,s}[\mathbf P]
	\coloneq
	-z\left(
	\vartheta\mathbf P+\mathbf P D+(\kappa_i+s+2)\mathbf P
	\right)
	+
	(\mathbf P\one_q^\top)\lambda^{\boldsymbol a},
	\qquad
	i\in\{1,\ldots,p\},\quad s\in\mathbb N_0.
	\]
	
	For \(h\in\{1,\ldots,p\}\), let
	\[
	\mathbf C_h(z)
	=
	\begin{bNiceMatrix}
		C_{1,h}(z)\\
		\Vdots\\
		C_{q,h}(z)
	\end{bNiceMatrix}
	\]
	be the \(h\)-th column of the matrix weight \(\CB(z)\).  For
	\(k\in\mathbb N_0\), define the residue functional
\(\mathcal L_{h,k}(\mathbf P) \coloneq \operatorname*{Res}_{z=0} \left( z^k\mathbf P(z)\mathbf C_h(z) \right).\)
\end{definition}

\begin{remark}
	The operator \(\mathcal T^{\boldsymbol a}_{i,s}\) acts directly on row-vector Laurent
	polynomials.  The last term is rank one: it depends on the scalar sum of the
	components of \(\mathbf P\) and redistributes it through the vector
	\(\lambda^{\boldsymbol a}\).  With the row-vector convention used here,
\(	\one_q=
	\begin{bNiceMatrix}
		1&\Cdots&1
	\end{bNiceMatrix}\),
	\(\mathbf P(z)\one_q^\top
	=
	\sum_{j=1}^{q}P_j(z)\).
	Thus, when \(q=1\), the vector structure disappears and the construction
	reduces to the scalar multiple Bessel construction.  For \(q>1\), the coupling
	between the components is precisely the rank-one term in
	\(\mathcal T^{\boldsymbol a}_{i,s}\).
\end{remark}

The next lemma states the basic intertwining property of these
operators with the Bessel columns.  It is the row-vector analogue of the
differential identity behind the scalar multiple Bessel formulae.

\begin{lemma}[One-step moment shift]
	\label{lem:matrix-differential-one-step}
	For every row-vector Laurent polynomial \(\mathbf P\), one has
	\[
	\mathcal L_{h,k}
	\left(
	\mathcal T^{\boldsymbol a}_{i,s}[\mathbf P]
	\right)
	=
	(k+\kappa_h-\kappa_i-s)
	\mathcal L_{h,k+1}(\mathbf P).
	\]
\end{lemma}

\begin{proof}
	By linearity it is enough to take \(\mathbf P(z)=\mathbf v z^n\), with
	\(\mathbf v=
	\begin{bNiceMatrix}
		v_1&\Cdots&v_q
	\end{bNiceMatrix}\).
	The differential part contributes
\(-\sum_{j=1}^{q} v_j(n+b_j+\kappa_i+s+2)\mu_{j,h}(k+n+1).\)
	For the coupling term, use the partial fraction identity with
	\(x=k+n+\kappa_h\).  It gives
	\[
	\sum_{j=1}^{q}\lambda^{\boldsymbol a}_j\mu_{j,h}(k+n)
	=
	\frac{\Gamma((k+n+\kappa_h)\one_r+\boldsymbol a+2\one_r)}
	{\prod_{j=1}^{q}\Gamma(k+n+\kappa_h+b_j+2)}.
	\]
	Equivalently, for each fixed \(j\), this common quantity is
\((k+n+\kappa_h+b_j+2)\mu_{j,h}(k+n+1).\)
	Although the expression above contains the index \(j\), it is independent of
	\(j\): the factor \(k+n+\kappa_h+b_j+2\) cancels exactly the corresponding
	shifted factor in the denominator of \(\mu_{j,h}(k+n+1)\).
	Therefore the coupling term contributes
\(\sum_{j=1}^{q} v_j(k+n+\kappa_h+b_j+2)\mu_{j,h}(k+n+1).\)
	Adding both contributions gives
\((k+\kappa_h-\kappa_i-s) \sum_{j=1}^{q}v_j\mu_{j,h}(k+n+1),\)
	which is precisely
\((k+\kappa_h-\kappa_i-s)\mathcal L_{h,k+1}(\mathbf P).\)
\end{proof}

Let \(N=|\boldsymbol n|\), and let
\(\omega=(i_0,\ldots,i_{N-1})\)
be a word in which the index \(i\) occurs exactly \(n_i\) times.  For each
position \(a\), set
\(\rho_a=\#\{b>a:\ i_b=i_a\},\) \(s_a=N-1-a+\rho_a.\)
Define
\[
\mathcal U^{\boldsymbol a}_{\omega}
=
\mathcal T^{\boldsymbol a}_{i_{N-1},s_{N-1}}
\cdots
\mathcal T^{\boldsymbol a}_{i_1,s_1}
\mathcal T^{\boldsymbol a}_{i_0,s_0}.
\]
Iterating Lemma~\ref{lem:matrix-differential-one-step} gives
\begin{equation}
	\label{eq:matrix-differential-iteration}
	\mathcal L_{h,k}
	\left(
	\mathcal U^{\boldsymbol a}_{\omega}[\mathbf P]
	\right)
	=
	\prod_{i=1}^{p}
	\prod_{s=0}^{n_i-1}
	(k+\kappa_h-\kappa_i-s)
	\mathcal L_{h,k+N}(\mathbf P).
\end{equation}
In particular, when \(h=i\), the product contains
\(k(k-1)\cdots(k-n_i+1)\),
and therefore
\(\mathcal L_{i,k} \left( \mathcal U^{\boldsymbol a}_{\omega}[\mathbf P] \right)=0,\) \(k\in\{0,\ldots,n_i-1\}.\)
Thus the matrix differential product automatically generates the
\(B\)-side orthogonality conditions.

The initial Laurent vector in the pure Bessel case is described next.  Let
\((\boldsymbol n,\boldsymbol m)\) be \(B\)-balanced and near the diagonal,
\(|\boldsymbol m|=|\boldsymbol n|+1,\) \(m_j\in\{m_*-1,m_*\},\) \(m_*=\max_j m_j.\)
Set
\(H=\{j:\ m_j=m_*\}\).
The barycentric vector supported on \(H\) is
\[
v^H_j=
\begin{cases}
	\frac{1}{\prod_{\ell\in H,\ \ell\ne j}(b_\ell-b_j)},
	& j\in H,\\[1.2em]
	0,
	& j\notin H.
\end{cases}
\]

\begin{theorem}[Pure Bessel Rodrigues-type formula]
	\label{thm:pure-Bessel-matrix-differential}
	Assume that the Bessel parameters are regular, admissible for the
	formulas in this result, and that \(r=0\). Let
	\((\boldsymbol n,\boldsymbol m)\) be a \(B\)-balanced near-diagonal pair covered by
	Theorem~\textup{\ref{thm:final-explicit}}.
	Then, up to a nonzero normalization constant,
\(\mathbf B_{\boldsymbol n,\boldsymbol m}(z) = \mathcal U^{0}_{\omega} \left[ z^{m_*-1-|\boldsymbol n|}\mathbf v^H \right].\)
	The normalization is fixed by making the component selected by the chosen
	monicity convention monic.
\end{theorem}

\begin{proof}
	The vector on the right-hand side satisfies the \(B\)-orthogonality
	conditions by \eqref{eq:matrix-differential-iteration}.  It remains to check
	that it is a polynomial vector with the prescribed degree bounds.  This is the
	only point where \(r=0\) is used in an essential way.
	
	The identities
\(\sum_{j\in H}v^H_jb_j^\ell=0,\) \(\ell\in\{0,\ldots,|H|-2\},\)
	and, for the full barycentric vector \(\lambda^0\),
\(\sum_{j=1}^{q}\lambda^0_jb_j^\ell=0,\) \(\ell\in\{0,\ldots,q-2\},\)
	are the standard Lagrange interpolation identities at the nodes
	\(b_1,\ldots,b_q\).  Equivalently, the monomial seed satisfies the
	proper partial-fraction identity
	\[
	\sum_{j\in H}
	v_j^H
	\frac{1}{(u+b_j)_{d+1}
		\prod_{\ell\ne j}(u+b_\ell)_d}
	=
	\frac{1}{\prod_{\ell=1}^{q}(u+b_\ell)_{m_\ell}},
	\qquad d=m_*-1.
	\]
	To verify this identity explicitly, multiply both sides by
	\(\prod_{\ell=1}^{q}(u+b_\ell)_{m_\ell}\). Since
	\(m_\ell=d+1\) for \(\ell\in H\) and \(m_\ell=d\) otherwise, the result is
	\[
		\sum_{j\in H}v_j^H
		\prod_{\substack{\ell\in H\\ \ell\ne j}}
		(u+b_\ell+d)
		=1.
	\]
	This is precisely the Lagrange interpolation formula for the constant
	polynomial one at the nodes \(-b_j-d\), \(j\in H\), because
	\[
		\prod_{\substack{\ell\in H\\ \ell\ne j}}
		\bigl((-b_j-d)-(-b_\ell-d)\bigr)
		=
		\prod_{\substack{\ell\in H\\ \ell\ne j}}(b_\ell-b_j).
	\]
	This proves the stated partial-fraction identity.

It remains to justify explicitly that the differential product has no negative
powers.  Write its finite Laurent expansion as
\[
\mathcal U^0_{\omega}
\left[z^{m_*-1-|\boldsymbol n|}\mathbf v^H\right]
=
\sum_{j=1}^{q}\sum_e p_{j,e}z^e\mathbf e_j^{\top}.
\]
After the common Gamma normalization of the moments, a coefficient at the
negative level \(e=-s\) contributes
\(\prod_{h=1}^{q}(t+b_h-s)_s \sum_{j=1}^{q}\frac{p_{j,-s}}{t+b_j-s}.\)
Choose the largest \(s\ge1\) for which some \(p_{j,-s}\ne0\).  This term has
polynomial degree at most \(qs-1\), whereas all smaller negative levels have
degree at most \(q(s-1)-1\), and the nonnegative levels are proper rational
functions.  Since the partial-fraction identity above gives a proper rational
moment function, the highest \(q\) coefficients of the displayed polynomial
part must vanish.  Hence
\(\sum_{j=1}^{q}\frac{p_{j,-s}}{t+b_j-s} = \mathrm{O}(t^{-q-1}),\) \(t\to\infty.\)
Its numerator over the denominator
\(\prod_{j=1}^{q}(t+b_j-s)\) must therefore have degree at most \(-1\), so it
is identically zero.  The poles are distinct, and thus
\(p_{j,-s}=0\) for every \(j\), a contradiction.  Descending on \(s\) proves
that the resulting vector is polynomial.

	The highest possible power is \(z^{m_*-1}\).  Its coefficient is obtained by
	taking the raising term in every factor.  Its \(j\)-th component is
\((-1)^{|\boldsymbol n|} v^H_j \prod_{i=1}^{p}(b_j+\kappa_i+m_*)_{n_i}.\)
	For every \(j\in H\), this coefficient is nonzero by regularity and the
	admissibility convention, so the constructed vector is not identically zero.
	It vanishes for \(j\notin H\). Consequently, the
	components with \(j\in H\) have degree at most \(m_*-1=m_j-1\), whereas the
	components with \(j\notin H\) have degree at most \(m_*-2=m_j-1\).  The
	result follows from weak normality, Proposition~\ref{prop:finite-rational-characterization}.
\end{proof}

\begin{remark}[A degree-three Rodrigues example]
	\label{rem:degree-three-Rodrigues-example}
	Consider the genuinely \(2\times2\) pure Bessel case
	\[
		\begin{gathered}
			p=q=2,\qquad r=0,\qquad
			\kappa_1=0,\qquad \kappa_2=\frac12,\\
			b_1=0,\qquad b_2=-\frac12,\qquad
			\boldsymbol n=(4,3),\qquad \boldsymbol m=(4,4).
		\end{gathered}
	\]
	The explicit formula \eqref{eq:final-B-components}, with the second
	component made monic, gives
	\[
		\mathbf B_{\boldsymbol n,\boldsymbol m}(z)
		=
		\begin{bNiceMatrix}
			-2z^3-\frac{12}{13}z^2-\frac{20}{429}z-\frac{8}{19305}
			&
			z^3+\frac{31}{39}z^2+\frac{94}{2145}z+\frac{1}{2457}
		\end{bNiceMatrix}.
	\]
	The same vector can be recovered from the Rodrigues formula. Here
	\(N=|\boldsymbol n|=7\), \(m_*=4\), \(H=\{1,2\}\), and
	\[
		\mathbf v^H
		=
		\begin{bNiceMatrix}-2&2\end{bNiceMatrix},
		\qquad
		z^{m_*-1-N}\mathbf v^H
		=
		z^{-4}\begin{bNiceMatrix}-2&2\end{bNiceMatrix}.
	\]
	For the grouped word
	\(\omega=(1,1,1,1,2,2,2)\), the shifts are
	\(9,7,5,3,4,2,0\), and hence
	\[
		\mathcal U^0_\omega
		=
		\mathcal T^0_{2,0}
		\mathcal T^0_{2,2}
		\mathcal T^0_{2,4}
		\mathcal T^0_{1,3}
		\mathcal T^0_{1,5}
		\mathcal T^0_{1,7}
		\mathcal T^0_{1,9}.
	\]
	A direct application of these seven first-order factors gives
	\begingroup
	\small
	\[
		\mathcal U^0_\omega
		\left[
			z^{-4}\begin{bNiceMatrix}-2&2\end{bNiceMatrix}
		\right]
		=
		\!\begin{bNiceMatrix}[margin=0pt]
			14(19305z^3+8910z^2+450z+4)
			&
			-135135z^3-107415z^2-5922z-55
		\end{bNiceMatrix}.
	\]
	\endgroup
	Dividing by \(-135135\) reproduces exactly the preceding monic vector.
	Thus the Rodrigues construction and the terminating hypergeometric formula
	agree in this non-scalar example, with both components attaining degree
	three.
\end{remark}

\begin{remark}[Reduction to the ABV Rodrigues formula]
	\label{rem:ABV-Rodrigues-reduction}
	When \(q=1\), the restriction \(0\le r<q\) forces \(r=0\). Hence the
	matrix weight has a single row, \(H=\{1\}\), \(\mathbf v^H=1\), and
	\(\boldsymbol m=(|\boldsymbol n|+1)\). The pure Bessel formula of
	Theorem~\ref{thm:pure-Bessel-matrix-differential} becomes the scalar identity
\(	B_{\boldsymbol n}(z)
	\propto
	\mathcal U^0_{\omega}[1]\).
	Writing \(\alpha_i=b_1+\kappa_i\), the one-step factor reduces to
\(\mathcal T^0_{i,s}[P] = P-z(\vartheta+\alpha_i+s+2)P.\)
	After restoring the scale parameter \(\gamma\) in the ABV weights
	\(z^{\alpha_i}\exp(\gamma/z)\), this factor reads
	\[
	\mathcal T^{(\gamma)}_{i,s}[P]
	=
	\gamma P-z(\vartheta+\alpha_i+s+2)P
	=
	-z^{-\alpha_i-s}\mathrm{e}^{-\gamma/z}
	\frac{\mathrm d}{\mathrm dz}
	\left(
	z^{\alpha_i+s+2}\mathrm{e}^{\gamma/z}P(z)
	\right).
	\]
	Thus the product of the scalar factors in \(\mathcal U^0_{\omega}\) recovers
	the Rodrigues raising formula for the type-II multiple Bessel polynomials
	of Aptekarev--Branquinho--Van Assche \cite{ABV2003}, up to the global
	normalization and a rescaling of the spectral variable.
\end{remark}

The formula for \(r=0\) has a one-term Laurent seed. For \(r>0\),
the numerator parameters \(\boldsymbol a\) deform this monomial
seed into a finite Laurent tail.  Before defining it, its
dependence on the multi-index \(\boldsymbol m\) is made explicit.  This is useful because the notation
\(\mathbf S^{\boldsymbol a}_{\boldsymbol n,\boldsymbol m}\) contains \(\boldsymbol n\), although the
coefficients of the finite tail are determined by \(\boldsymbol a,\boldsymbol b\) and
\(\boldsymbol m\) alone.

For \(r>0\), keep the same \(B\)-balanced near-diagonal pair and the notation
\(m_*\) and \(H\) introduced above.  Since \(\boldsymbol m\) is near-diagonal,
its entries take only the two values \(m_*\) and \(m_*-1\).  Thus the finite
seed is governed by the maximal level \(m_*\) and by the set
\(H=\{j:\ m_j=m_*\}\)
of components which attain that level.  The length of the Laurent tail is
\(d=m_*-1\).
The dependence on \(\boldsymbol n\) enters only through the global Laurent shift
\(z^{-N}\), with \(N=|\boldsymbol n|\), and through the subsequent ordered product
\(\mathcal U^{\boldsymbol a}_{\omega}\).  The word \(\omega\) orders the raising
operators; it does not modify the seed.

On the step-line this dependence becomes especially simple.  Write
\(\boldsymbol m=\boldsymbol\sigma_q(N+1),\) \(N+1=qM+\tau,\) \(\tau\in\{0,\ldots,q-1\},\)
so that \(\tau\) is the remainder when \(N+1\) is divided by \(q\).  Then
\(\boldsymbol m=M\one_q+\sum_{j=1}^{\tau}\boldsymbol e_j\) .
When \(\tau=0\), one has
\(m_*=M\),
\qquad
\(H=\{1,\ldots,q\}\),
\(d=M-1\).
The seed then has the form
\[
\mathbf S^{\boldsymbol a}_{\boldsymbol n,\boldsymbol m}(z)
=
z^{M-1-N}
\left(
\mathbf s_0+\frac{\mathbf s_1}{z}+\cdots+
\frac{\mathbf s_{M-1}}{z^{M-1}}
\right),
\]
	and, for a nonzero scalar \(\gamma\), its coefficients are determined by the
	corresponding specialization of the partial-fraction identity below, namely
\[
\sum_{j=1}^{q}\sum_{a=0}^{M-1}
s_{a,j}
\frac{\prod_{\rho=1}^{r}(u+a_\rho)_{M-1-a}}
{(u+b_j)_{M-a}
	\prod_{\substack{\ell=1\\ \ell\ne j}}^{q}
	(u+b_\ell)_{M-1-a}}
=
\frac{\gamma}{\prod_{\ell=1}^{q}(u+b_\ell)_{M}}.
\]
For \(\tau>0\), one has \(m_*=M+1\),
\(H=\{1,\ldots,\tau\}\), and \(d=M\).  Thus the step-line seed is determined
only by the quotient \(M\) and the remainder \(\tau\), before the raising operators
associated with \(\boldsymbol n\) are applied.

Set
\(N=|\boldsymbol n|\).
\begin{definition}[Finite Laurent seed]
	\label{def:finite-Laurent-seed}
	Assume that the Bessel parameters are regular and that \(0<r<q\).
	Fix \(\gamma\in\mathbb C\setminus\{0\}\).  Define
\(R(u) \coloneq \frac{\gamma} {(u\one_q+\boldsymbol b)_{\boldsymbol m}},\)
	and, for \(a\in\{0,\ldots,d\}\) and \(j\in\{1,\ldots,q\}\),
\(\Phi_{a,j}(u) = \frac{(u\one_r+\boldsymbol a)_{d-a}} {(u+b_j)_{d-a+1} (u\one_{q-1}+\boldsymbol b^{\,*j})_{d-a}} .\)
	For \(a\in\{0,\ldots,d\}\) and \(j\in\{1,\ldots,q\}\), the
	coefficients \(s_{a,j}\) are defined recursively by
	\begin{equation}
		\label{eq:seed-coefficients-residue}
		s_{a,j}
		=
		\frac{
			\operatorname*{Res}_{u=-b_j-d+a} R(u)
			-
			\sum_{\alpha=0}^{a-1}
			\sum_{h=1}^{q}
			s_{\alpha,h}\,
			\operatorname*{Res}_{u=-b_j-d+a}\Phi_{\alpha,h}(u)
		}{
			\operatorname*{Res}_{u=-b_j-d+a}\Phi_{a,j}(u)
		},
	\end{equation}
	where the double sum is empty for \(a=0\).  Define the finite
	Laurent seed by
	\begin{equation}
		\label{eq:seed-ansatz}
		\mathbf S^{\boldsymbol a}_{\boldsymbol n,\boldsymbol m}(z)
		=
		z^{d-N}
		\left(
		\mathbf s_0+\frac{\mathbf s_1}{z}+\cdots+\frac{\mathbf s_d}{z^d}
		\right),
		\qquad
		\mathbf s_a=
		\begin{bNiceMatrix}
			s_{a,1} & \Cdots & s_{a,q}
		\end{bNiceMatrix}.
	\end{equation}
\end{definition}

\begin{proposition}[Properties of the finite Laurent seed]
	\label{prop:finite-Laurent-seed}
	The seed \(\mathbf S^{\boldsymbol a}_{\boldsymbol n,\boldsymbol m}\) of
	Definition~\ref{def:finite-Laurent-seed} is well defined and is unique for the
	prescribed value of \(\gamma\). Its coefficients are equivalently
	characterized by the partial-fraction identity
	\begin{equation}
		\label{eq:seed-partial-fraction}
		\sum_{j=1}^{q}\sum_{a=0}^{d}
		s_{a,j}
		\frac{
			(u\one_r+\boldsymbol a)_{d-a}
		}{
			(u+b_j)_{d-a+1}
			(u\one_{q-1}+\boldsymbol b^{\,*j})_{d-a}
		}
		=
		\frac{\gamma}
		{(u\one_q+\boldsymbol b)_{\boldsymbol m}} .
	\end{equation}
	Moreover, its leading coefficient is
	\begin{equation}
		\label{eq:s0-corrected}
		s_{0,j}
		=
		\begin{cases}
			\frac{\gamma v^H_j}
			{(\boldsymbol a-b_j\one_r-d\one_r)_d},
			& j\in H,\\[1.2em]
			0, & j\notin H.
		\end{cases}
	\end{equation}
	Finally, the image of the seed satisfies
	\begin{equation}
		\label{eq:finite-seed-system}
		\Pi_{<0}
		\mathcal U^{\boldsymbol a}_{\omega}
		\left[
		\mathbf S^{\boldsymbol a}_{\boldsymbol n,\boldsymbol m}
		\right]=0,
		\qquad
		\deg
		\left(
		\mathcal U^{\boldsymbol a}_{\omega}
		\left[
		\mathbf S^{\boldsymbol a}_{\boldsymbol n,\boldsymbol m}
		\right]
		\right)_j
		\le m_j-1,
		\qquad j\in\{1,\ldots,q\}.
	\end{equation}
	Here \(\Pi_{<0}\) denotes projection onto the principal part at the origin,
	that is, onto the negative powers in the Laurent expansion at \(z=0\).
\end{proposition}

\begin{proof}
	First, the recursive definition is shown to be meaningful. The recursion
	\eqref{eq:seed-coefficients-residue} is triangular. Put
	\[
	u_{a,j}\coloneq -b_j-d+a,
	\qquad a\in\{0,\ldots,d\},\quad j\in\{1,\ldots,q\}.
	\]
	At the pole \(u_{a,j}\), the possible contributions are as follows:
	\[
	\begin{array}{c|c|c}
		\text{term} & \text{condition} & \text{behavior at }u_{a,j} \\
		\hline
		\Phi_{\alpha,h} & h\ne j,\ \alpha\ge a & \text{regular} \\
		\Phi_{\alpha,h} & h\ne j,\ \alpha<a & \text{possibly singular, already known} \\
		\Phi_{\alpha,j} & \alpha>a & \text{regular} \\
		\Phi_{\alpha,j} & \alpha<a & \text{possibly singular, already known} \\
		\Phi_{a,j} &  & \text{simple pole with nonzero residue}
	\end{array}
	\]
	Indeed, for \(h=j\) the denominator \((u+b_j)_{d-\alpha+1}\) reaches the
	point \(u_{a,j}\) exactly when \(\alpha\le a\), while for \(h\ne j\) the
	factor \((u+b_j)_{d-\alpha}\) contained in
	\((u\one_{q-1}+\boldsymbol b^{\,*h})_{d-\alpha}\) reaches it exactly when
	\(\alpha<a\); the factors \((u+b_h)_{d-\alpha+1}\) with \(h\ne j\) are
	regular there, since \(b_h-b_j\notin\mathbb Z\). Thus, after subtracting the
	contributions from all already determined lower levels \(\alpha<a\), the
	residue at \(u_{a,j}\) determines the new coefficient \(s_{a,j}\) linearly,
	with a nonzero coefficient. Hence the coefficients are obtained
	successively, starting from \(a=0\) and ending at \(a=d\), and the seed is
	unique for the chosen value of \(\gamma\).
	
	Next, the partial-fraction identity \eqref{eq:seed-partial-fraction}
	is proved directly from this triangular residue matching.  At the deepest poles
	\(u=-b_j-d\), the residue comparison fixes the coefficients \(s_{0,j}\).
	After subtracting the corresponding contributions, the next poles
	\(u=-b_j-d+1\) fix the coefficients \(s_{1,j}\), and so on. Since the
	right-hand side has poles only at
	\(u=-b_j-\nu\),
	\(\nu\in\{0,\ldots,m_j-1\}\subseteq\{0,\ldots,d\}\),
	the process stops at \(a=d\). At each step the residue of the difference
	between the two sides of \eqref{eq:seed-partial-fraction} at the corresponding
	pole is zero. Since these are all possible poles, the difference is a rational
	function without poles and vanishing at infinity, hence it is identically
	zero. Conversely, reading the residues of \eqref{eq:seed-partial-fraction}
	level by level gives exactly \eqref{eq:seed-coefficients-residue}.
	
	The leading coefficient follows from the deepest pole \(u=-b_j-d\). On the
	left-hand side of \eqref{eq:seed-partial-fraction}, only the term
	\((a,j)=(0,j)\) contributes there. On the right-hand side such a pole occurs
	if and only if \(m_j=m_*\), that is, if and only if \(j\in H\). Hence
	\(s_{0,j}=0\) for \(j\notin H\). For \(j\in H\), comparison of residues gives
\(s_{0,j} = \frac{\gamma v^H_j} {(\boldsymbol a-b_j\one_r-d\one_r)_d}.\)
	Indeed, the factors coming from the other \(b\)-nodes reduce to the
	barycentric factors defining \(\mathbf v^H\): after cancelling the common
	factors at the pole \(u=-b_j-d\), one uses
\(\frac{(b_\ell-b_j-d)_{m_\ell}} {(b_\ell-b_j-d)_d} = (b_\ell-b_j)_{m_\ell-d},\)
	which equals \(b_\ell-b_j\) for \(\ell\in H\), \(\ell\ne j\), and equals \(1\)
	for \(\ell\notin H\). This proves \eqref{eq:s0-corrected}.
	
	It remains to verify the two properties in \eqref{eq:finite-seed-system}. By
	inserting
\(\mathbf S^{(j)}(z) = \sum_{a=0}^{d}s_{a,j}z^{d-N-a}\)
	into the moment functionals, and using
\(\frac{\Gamma(u+m+\alpha)}{\Gamma(u+\alpha)} = (u+\alpha)_m,\)
	the identity \eqref{eq:seed-partial-fraction}, evaluated at
	\(u=k+\kappa_i+1\), is equivalent to
	\begin{equation}
		\label{eq:seed-moment-identity}
		\mathcal L_{i,k+N}
		\left(
		\mathbf S^{\boldsymbol a}_{\boldsymbol n,\boldsymbol m}
		\right)
		=
		\gamma
		\frac{\Gamma((k+1+\kappa_i)\one_r+\boldsymbol a)}
		{\Gamma((k+1+\kappa_i)\one_q+\boldsymbol b+\boldsymbol m)},
		\qquad
		k\in\N_0,
		\quad i\in\{1,\ldots,p\}.
	\end{equation}
	Combining this identity with the iterated moment shift
	\eqref{eq:matrix-differential-iteration} gives the \(B\)-side orthogonality
	of
	\(\mathcal U^{\boldsymbol a}_{\omega}
	\left[
	\mathbf S^{\boldsymbol a}_{\boldsymbol n,\boldsymbol m}
	\right]\).
	It remains to justify the Laurent support.  Let
	\[
	\mathbf P(z)
	=
	\mathcal U^{\boldsymbol a}_{\omega}
	[\mathbf S^{\boldsymbol a}_{\boldsymbol n,\boldsymbol m}](z)
	=
	\sum_{j=1}^{q}\sum_e p_{j,e}z^e\mathbf e_j^{\top}
	\]
	be its finite Laurent expansion.  After multiplying the moment sequence
	\(\mathcal L_{i,k}(\mathbf P)\) by the common Gamma denominator
	corresponding to \(\boldsymbol m\), each coefficient \(p_{j,e}\) contributes the
	rational factor
	\begin{equation}
		\label{eq:Rodrigues-Laurent-rational-factor}
		p_{j,e}
		\frac{(t\one_r+\boldsymbol a)_e}
		{(t\one_q+\boldsymbol b)_e(t+b_j+e)},
		\qquad t=k+\kappa_i+1,
	\end{equation}
	with the usual Gamma interpretation for negative \(e\).  Both sides of the
	resulting identity are rational functions of \(t\) which agree at the
	infinitely many points \(t=k+\kappa_i+1\), and hence coincide identically.
	The right-hand side obtained from \eqref{eq:seed-moment-identity} is a proper
	rational function with denominator \((t\one_q+\boldsymbol b)_{\boldsymbol m}\); in particular,
	it has neither poles on the \(a\)-strings nor a polynomial part at infinity.

First, negative Laurent levels are excluded. Since the smallest exponent in the
seed is \(-N\) and the operators do not lower Laurent degree, it is enough to
consider \(1\le s\le N\). First, the argument is given on the dense parameter domain on which
\begin{equation}
	\label{eq:Rodrigues-temporary-nonresonance}
	\begin{aligned}
		a_\rho-b_h
		&\notin
		\{-(N-1),\ldots,N-1\},
		&&
		\rho\in\{1,\ldots,r\},
		\quad
		h\in\{1,\ldots,q\},
		\\
		a_\rho-a_\sigma
		&\notin
		\{-(N-1),\ldots,N-1\},
		&&
		\rho,\sigma\in\{1,\ldots,r\},
		\quad
		\rho\ne\sigma.
	\end{aligned}
\end{equation}
This temporary restriction ensures that, for every \(s\le N\), the zero sets
of \(A_s\) and \(B_s\) below are disjoint, and the new factors
\(t+a_\rho-s\) are pairwise distinct and coprime to \(A_{s-1}\). It will be
removed after the argument.

Suppose that \(e=-s<0\), and let \(s\ge1\) be the largest integer for which
some coefficient \(p_{j,-s}\) is nonzero. Put
\[
A_s(t)\coloneq
\prod_{\rho=1}^{r}(t+a_\rho-s)_s,
\qquad
B_s(t)\coloneq
\prod_{h=1}^{q}(t+b_h-s)_s,
\qquad
F_s(t)
\coloneq
\sum_{j=1}^{q}\frac{p_{j,-s}}{t+b_j-s}.
\]
The contribution of the level \(-s\) to
\eqref{eq:Rodrigues-Laurent-rational-factor} is then
\(R_s(t)=\frac{B_s(t)}{A_s(t)}F_s(t)\).
Since
\(A_s(t)=A_{s-1}(t) \prod_{\rho=1}^{r}(t+a_\rho-s),\)
the terms belonging to smaller negative levels have denominator dividing
\(A_{s-1}(t)\).  The full rational identity has no additional principal
parts on the \(a\)-strings.  Hence the numerator of \(F_s\), written over
the common denominator
\(\prod_{j=1}^{q}(t+b_j-s)\), must be divisible by
\(\prod_{\rho=1}^{r}(t+a_\rho-s)\).

On the other hand, \(B_s(t)/A_s(t)=\mathrm{O}(t^{(q-r)s})\) at infinity.
Terms from the level \(-(s-1)\) have polynomial degree at most
\((q-r)(s-1)-1\).  Since the full rational identity has no polynomial part,
the highest \(q-r\) coefficients contributed by \(R_s\) must vanish.  This is
equivalent to
\(F_s(t)=\mathrm{O}\left(t^{-(q-r+1)}\right),\) \(t\to\infty.\)
Therefore, if
\(F_s(t) = \frac{Q_s(t)}{\prod_{j=1}^{q}(t+b_j-s)},\)
then \(\deg Q_s\le r-1\).  But the preceding principal-part comparison shows
that the degree-\(r\) polynomial
\(\prod_{\rho=1}^{r}(t+a_\rho-s)\) divides \(Q_s\).  Consequently
\(Q_s\equiv0\), and hence \(F_s\equiv0\).  Since the poles
\(-b_j+s\) are pairwise distinct, this implies
\(p_{j,-s}=0\) for every \(j\), contradicting the choice of \(s\).
Descending on \(s\) proves that no negative Laurent coefficient is present.

It remains to remove the temporary restriction
\eqref{eq:Rodrigues-temporary-nonresonance}. By the triangular residue
construction \eqref{eq:seed-coefficients-residue}, every coefficient
\(s_{a,j}\) is a rational function of the parameters on its admissible domain.
The same is therefore true of every Laurent coefficient \(p_{j,e}\), since it
is obtained by applying finitely many operators
\(\mathcal T^{\boldsymbol a}_{i,s}\). This proves that each coefficient with
\(e<0\) vanishes on the dense open set
\eqref{eq:Rodrigues-temporary-nonresonance}; hence it vanishes identically as
a rational function. Thus the absence of negative Laurent powers extends to
the entire admissible parameter domain, including the resonant cases and
coincident numerator parameters.

	It remains to verify the prescribed upper degree bounds.  Each operator
\(\mathcal T^{\boldsymbol a}_{i,s}\) raises the Laurent degree by at most one: for a
monomial row vector \(\mathbf v z^e\), the only contribution of degree
\(e+1\) is
\(-z^{e+1}\mathbf v \bigl(eI_q+D+(\kappa_i+s+2)I_q\bigr),\)
whereas the rank-one coupling term has degree \(e\).  Since the largest
exponent in the seed is \(d-N\), after the \(N\) factors in
\(\mathcal U^{\boldsymbol a}_{\omega}\) the largest possible exponent is \(d\).

Moreover, the coefficient of \(z^d\) can only arise from the leading seed
term \(z^{d-N}\mathbf s_0\) by taking the degree-raising part in every
operator.  Its \(j\)-th component is therefore
\((-1)^N s_{0,j} \prod_{i=1}^{p}(b_j+\kappa_i+m_*)_{n_i}.\)
By \eqref{eq:s0-corrected}, one has \(s_{0,j}=0\) whenever \(j\notin H\).
For \(j\in H\), both \(s_{0,j}\) and the displayed product are nonzero by
the admissibility convention. Hence \(\mathbf P\) is not identically zero.
Consequently, for \(j\in H\),
\(\deg P^{(j)}\le d=m_*-1=m_j-1,\)
whereas, for \(j\notin H\), the coefficient of \(z^d\) vanishes and hence
\(\deg P^{(j)}\le d-1=m_*-2=m_j-1.\)
Therefore \(\mathbf P\) has no principal part at the origin and satisfies
the required componentwise degree bounds.  This is precisely
\eqref{eq:finite-seed-system}.
\end{proof}

\begin{theorem}[Bessel Rodrigues-type formula]
	\label{thm:Bessel-Rodrigues-general}
	Assume that the Bessel parameters are regular, admissible for the
	formulas in this result, and that \(0<r<q\). Let
	\((\boldsymbol n,\boldsymbol m)\) be a \(B\)-balanced near-diagonal pair covered by
	Theorem~\textup{\ref{thm:final-explicit}}. Let
	\(\mathbf S^{\boldsymbol a}_{\boldsymbol n,\boldsymbol m}\) be the finite Laurent seed of
	Definition~\textup{\ref{def:finite-Laurent-seed}}. Then, up to a nonzero
	normalization constant,
\(\mathbf B_{\boldsymbol n,\boldsymbol m}(z) = \mathcal U^{\boldsymbol a}_{\omega} \left[ \mathbf S^{\boldsymbol a}_{\boldsymbol n,\boldsymbol m}(z) \right].\)
	The remaining scalar normalization is fixed by making the component selected
	by the chosen monicity convention monic.
\end{theorem}

\begin{proof}
	By Proposition~\ref{prop:finite-Laurent-seed}, the vector
	\(\mathcal U^{\boldsymbol a}_{\omega}
	\left[
	\mathbf S^{\boldsymbol a}_{\boldsymbol n,\boldsymbol m}
	\right]\)
	has no principal part at the origin and satisfies the componentwise degree
	bounds
	\[
	\deg
	\left(
	\mathcal U^{\boldsymbol a}_{\omega}
	\left[
	\mathbf S^{\boldsymbol a}_{\boldsymbol n,\boldsymbol m}
	\right]
	\right)_j
	\le m_j-1,
	\qquad j\in\{1,\ldots,q\}.
	\]
	By \eqref{eq:matrix-differential-iteration}, it also satisfies the
	\(B\)-side orthogonality conditions associated with \((\boldsymbol n,\boldsymbol m)\).
	The final coefficient calculation in the proof of
	Proposition~\ref{prop:finite-Laurent-seed} shows that this vector is nonzero.
	Therefore it belongs to the one-dimensional solution space characterized by
	Proposition~\ref{prop:finite-rational-characterization}, and hence it is
	proportional to \(\mathbf B_{\boldsymbol n,\boldsymbol m}\). The proportionality constant is
	fixed by the chosen monicity condition.
\end{proof}

\begin{remark}[Dependence on the word]
	The finite Laurent seed \(\mathbf S^{\boldsymbol a}_{\boldsymbol n,\boldsymbol m}\) does not depend
	on the word \(\omega\); only the ordered product
	\(\mathcal U^{\boldsymbol a}_{\omega}\) does.  The iterated moment identity
	\eqref{eq:matrix-differential-iteration} depends only on the multiplicities
	\(n_i\), and weak normality shows that the outputs associated with any two
	such words are proportional. Their leading coefficients coincide by the
	final coefficient calculation in the proof of
	Proposition~\ref{prop:finite-Laurent-seed}; hence the proportionality
	constant is one. Therefore
	\[
		\mathcal U^{\boldsymbol a}_{\omega}
		\left[\mathbf S^{\boldsymbol a}_{\boldsymbol n,\boldsymbol m}\right]
		=
		\mathcal U^{\boldsymbol a}_{\widetilde\omega}
		\left[\mathbf S^{\boldsymbol a}_{\boldsymbol n,\boldsymbol m}\right]
	\]
	for any two words \(\omega\) and \(\widetilde\omega\) with the same
	multiplicities.
\end{remark}

\begin{remark}[The leading coefficient of the seed]
	For \(r=0\), the product in \eqref{eq:s0-corrected} is empty and
	\(\mathbf s_0=\gamma\mathbf v^H\), recovering the one-term seed of
	Theorem~\ref{thm:pure-Bessel-matrix-differential}. If \(|H|=1\), the single
	row-dependent factor can also be absorbed into the normalization. For \(r>0\) and
	\(|H|\ge2\), however, the factor
	\(\prod_{\rho=1}^r(a_\rho-b_j-d)_d\)
	depends on the row index \(j\). Thus the leading coefficient is generally
	not proportional to \(\mathbf v^H\), but to the \(\boldsymbol a\)-reweighted
	barycentric vector \eqref{eq:s0-corrected}.
\end{remark}

\begin{remark}[Relation with \(B\)-strong normality]
	The construction of the finite Laurent seed uses only regularity of the
	parameters. The formula-wise admissibility domain of the Rodrigues
	construction may be strictly larger than that of the explicit hypergeometric
	reconstruction. In particular, the finite Laurent seed and the differential
	operators may remain well defined at parameter values where some Gamma
	quotients occurring in the latter representation cease to be admissible.
	The \(B\)-strong normality condition has a different role: it decides whether
	the components not selected by the monicity convention, with
	\(m_j=m_*-1\), attain their maximal allowed degree. On the exceptional
	hypersurfaces for \(B\)-strong normality, the same seed still gives the
	weakly normal \(B\)-vector, but one of these components may have degree
	strictly smaller than its allowed bound.
\end{remark}

\section{Non-mixed-type reductions}
\label{sec:final-reductions}

The two distinguished reductions of the \(q\times p\) construction are recorded next.
The specialization \(p=1\) gives the Bessel-like circle family of Wolfs,
whereas the specialization \(q=1\) gives the multiple Bessel moment system.

\begin{proposition}[Reduction to Wolfs' Bessel-like family]
	\label{prop:Wolfs-case}
	If \(p=1\) and \(\kappa_1=0\), then the direct circle weights are
	\[
	\WB(z)
	=
	\begin{bNiceMatrix}
		f_0\left(z;\boldsymbol a,\boldsymbol b+\boldsymbol e_1\right)
		\\\Vdots\\
		f_0\left(z;\boldsymbol a,\boldsymbol b+\boldsymbol e_q\right)
	\end{bNiceMatrix},
	\]
	and their reciprocal representatives are
	\[
	\CB(z)
	=
	\begin{bNiceMatrix}
		\frac1z f_0\left(\frac1z;\boldsymbol a,\boldsymbol b+\boldsymbol e_1\right)
		\\\Vdots\\
		\frac1z f_0\left(\frac1z;\boldsymbol a,\boldsymbol b+\boldsymbol e_q\right)
	\end{bNiceMatrix}.
	\]
\end{proposition}

\begin{proof}
	The first displayed identity follows from \eqref{eq:final-weight} after
	setting \(p=1\) and \(\kappa_1=0\). The second displayed identity follows
	from \eqref{eq:final-cauchy-representative} under the same specialization.
\end{proof}

\begin{remark}
	The first displayed vector is the Bessel-like multiple orthogonality system
	of Wolfs on the circle. The second displayed vector is only its reciprocal
	Laurent presentation, used here to write the same moments in the mixed
	Gauss--Borel formalism. Under this reduction, \(\mathbf B\) is
	the type-I vector of Wolfs and \(\mathbf A\) is the scalar
	\(A\)-polynomial vector, equivalently the scalar type-II polynomial in the
	usual terminology. Thus the recurrence operator defined below, which acts
	on the \(B\)-side, acts directly on the type-I Wolfs vectors,
	while the type-II sequence satisfies the dual relation.
\end{remark}

Next, consider the opposite distinguished reduction. If \(q=1\), then the
Bessel constraint \(0\le r<q\) forces \(r=0\). Put
\(\alpha_j\coloneq b_1+\kappa_j\),
 \(j\in\{1,\ldots,p\}\).
The moment functionals of the single row are
\begin{equation}
	\label{eq:multiple-bessel-moments}
	\mathcal L_j[z^n]
	=
	\frac{1}{\Gamma(n+\alpha_j+2)},
	\qquad n\in\N_0.
\end{equation}

\begin{proposition}[The multiple Bessel case]
	\label{prop:multiple-bessel-case}
	At \(q=1\), the mixed \(q\times p\) system reduces to the row vector of
	moment functionals \eqref{eq:multiple-bessel-moments}. The scalar
	\(B\)-component is the type-II multiple Bessel polynomial for
	these functionals, while the \(p\)-component \(A\)-polynomial vector is the
	corresponding type-I vector for the same moment system.
\end{proposition}

\begin{proof}
	When \(q=1\), the inequality \(0\le r<q\) gives \(r=0\). In the
	moment formula for the Bessel matrix, all numerator Gamma factors
	disappear and the only row index is \(j=1\). Therefore
\(\mu_i(n) = \frac{1}{\Gamma(n+\kappa_i+b_1+2)} = \frac{1}{\Gamma(n+\alpha_i+2)}.\)
	Since the measure matrix has one row, the \(B\)-problem is
	scalar and is precisely the type-II problem for the row of functionals;
	the \(A\)-problem is the dual type-I vector problem.
\end{proof}

\begin{remark}[On the unit-circle representation in \cite{ABV2003}]
	Aptekarev--Branquinho--Van Assche write the Bessel weights in the form
	\(z^{\alpha_j}\exp(\gamma/z)\) and list a circle around the origin as the
	classical path for the Bessel case. When \(\alpha_j\notin\mathbb Z\), this
	expression is multivalued along a closed circle around the branch point.
	Already in the one-weight, degree-one case, the Rodrigues polynomial is
	\(Q_1(z)=(\alpha+2)z-\gamma\), and
	\[
		Q_1(z)z^\alpha\exp(\gamma/z)
		=
		\frac{\mathrm d}{\mathrm dz}
		\left(z^{\alpha+2}\exp(\gamma/z)\right).
	\]
	Thus the integral obtained by parametrizing the circle with
	\(z^\alpha=\exp(\mathrm i\alpha t)\),
	\(z=\exp(\mathrm i t)\), \(t\in[0,2\pi]\), gives the endpoint contribution
	\[
		\left[z^{\alpha+2}\exp(\gamma/z)\right]_{t=0}^{t=2\pi}
		=
		(\mathrm e^{2\pi\mathrm i\alpha}-1)\mathrm e^\gamma,
	\]
	which is not zero in general. Hence the naive closed-circle reading does not
	produce the stated orthogonality for nonintegral \(\alpha\); the functional
	must instead be supplied with a branch-cut or jump interpretation.
	Consequently, the present paper defines the \(q=1\) case by the
	single-valued reciprocal-Gamma moments
	\eqref{eq:multiple-bessel-moments}. The next proposition shows that these
	moments are exactly those obtained from the normalized Exton-type cut
	functional.
\end{remark}

\begin{proposition}[Exton-type realization of the multiple Bessel functionals]
	\label{prop:exton-realization-multiple-bessel}
	Let \(\lambda\in\mathbb C^\ast\). Fix a simple branch cut \(\Gamma\)
	issuing from the origin and unbounded at infinity, choose a branch of
	\(\log z\) in \(\mathbb C\setminus\Gamma\), and let \(0^{(\infty+)}\)
	denote the associated Exton contour, equivalently the contour functional
	obtained by taking the jump of the two boundary values across \(\Gamma\).
	With the branch convention used in Exton's formula,
	\begin{equation}
		\label{eq:exton-basic-moment}
		\int_{0}^{(\infty+)}
		z^{a+k-2}\exp(-1/z)\,\dz
		=
		\frac{2\pi\mathrm i\,(-1)^{a+k}}{\Gamma(a+k)},
		\qquad k\in\N_0.
	\end{equation}
	Let \(\alpha+2\notin-\N_0\), and define
	\begin{equation}
		\label{eq:exton-normalized-functional}
		\mathcal E_{\alpha,\lambda}[P]
		\coloneq
		\frac{1}{2\pi\mathrm i\,\mathrm e^{\pi\mathrm i\alpha}(-\lambda)^{\alpha+1}}
		\int_{0}^{(\infty+)}
		P\left(\frac{z}{\lambda}\right)z^\alpha\exp(\lambda/z)\,\dz,
	\end{equation}
	where \((-\lambda)^{\alpha+1}\) is taken with the branch induced by the
	corresponding scaling of the Exton contour. Then
	\begin{equation}
		\label{eq:exton-normalized-moments}
		\mathcal E_{\alpha,\lambda}[z^n]
		=
		\frac{1}{\Gamma(n+\alpha+2)},
		\qquad n\in\N_0.
	\end{equation}
	Thus, for each \(j\in\{1,\ldots,p\}\), the reciprocal-Gamma functional
	\(\mathcal L_j\) in \eqref{eq:multiple-bessel-moments} is exactly
	the normalized Exton-type functional \(\mathcal E_{\alpha_j,\lambda}\).
\end{proposition}

\begin{proof}
	Exton's moment evaluation \eqref{eq:exton-basic-moment}, with
	\(a=\alpha+2\), gives
	\[
	\int_{0}^{(\infty+)}
	z^{n+\alpha}\exp(-1/z)\,\dz
	=
	\frac{2\pi\mathrm i\,\mathrm e^{\pi\mathrm i\alpha}(-1)^n}
	{\Gamma(n+\alpha+2)}.
	\]
	After the change of variables \(z=-\lambda t\), this becomes
	\[
	\int_{0}^{(\infty+)}
	z^{n+\alpha}\exp(\lambda/z)\,\dz
	=
	2\pi\mathrm i\,\mathrm e^{\pi\mathrm i\alpha}(-\lambda)^{\alpha+1}
	\frac{\lambda^n}{\Gamma(n+\alpha+2)}.
	\]
	Substituting \(P(z)=z^n\) in \eqref{eq:exton-normalized-functional} cancels
	the factor \(\lambda^n\) and gives \eqref{eq:exton-normalized-moments}.
	The last assertion follows by taking \(\alpha=\alpha_j\).
\end{proof}

\begin{remark}
	Since \(q=1\), the near-diagonal condition in
	Theorem~\ref{thm:final-explicit}, which constrains \(\boldsymbol m\in\N_0^q\),
	is automatic. Thus the scalar type-II polynomial in this case, namely
	the \(B\)-component, is constructed for arbitrary multiple Bessel
	multi-indices \(\boldsymbol n\), not merely along the step-line.
\end{remark}

\begin{corollary}[Common scalar case]
	If \(q=p=1\), the two reductions meet in the ordinary Bessel
	polynomial system. In Wolfs' notation the corresponding polynomial is
	\[
	B_n^{(b_1)}(z)
	=
	\pFq{2}{0}{-n,\ n+b_1+1}{-}{z},
	\]
	whereas, with the monic normalization used here,
	\[
	B_n(z)
	=
	\frac{(-1)^n}{(n+b_1+1)_n}
	B_n^{(b_1)}(z).
	\]
	In the multiple-Bessel notation the corresponding parameter is
	\(\alpha=b_1\).
\end{corollary}

The specialization \(q=1\) is refined next. In this case the
Kamp\'e de F\'eriet block of
Corollary~\ref{cor:final-B-KdF} admits a one-variable generalized
hypergeometric reduction: its coupled numerator and denominator parameters
differ by the nonnegative integer shifts \(n_i\), and coefficient extraction
leads to terminating Karlsson--Minton-type sums. The following direct
finite-difference argument reduces this block to a single generalized
hypergeometric polynomial. The roots of an auxiliary characteristic
polynomial enter as unit-difference parameter pairs, as in transformations
of Miller--Paris type; see
\cite{MillerParis2013,KarpPrilepkina2018}. The proof itself is direct and
does not require the application of a Miller--Paris transformation.

\begin{corollary}[One-row reduction to a single generalized hypergeometric polynomial]
	\label{cor:final-B-q-one-hypergeometric}
	Under the assumptions of
	Corollary~\ref{cor:final-B-KdF}, let \(q=1\). Then \(r=0\), and a
	\(B\)-balanced pair has the form \(\boldsymbol m=(|\boldsymbol n|+1)\). Write
	\(b\coloneq b_1\), set
	\(c_i\coloneq\kappa_i+b+1\) for \(i\in\{1,\ldots,p\}\), and let
	\(\boldsymbol c\coloneq\boldsymbol\kappa+(b+1)\one_p\) and
	\(P(t)\coloneq\prod_{i=1}^{p}(c_i+t)_{n_i}\).
	The sole Kamp\'e de F\'eriet block occurring in
	\eqref{eq:final-B-KdF-representation} is
	\begin{equation}
		\label{eq:q-one-KdF-block}
		\mathcal K_{\boldsymbol n}(z)
		\coloneq
		F_{p:0;0}^{p+1:1;0}
		\left[
		\begin{array}{c}
			-|\boldsymbol n|,\boldsymbol c+\boldsymbol n:1;\text{---}\\
			\boldsymbol c:\text{---};\text{---}
		\end{array}
		\middle|-z,1
		\right].
	\end{equation}
	With the monic normalization \(\nu=\nu_{\boldsymbol n,(|\boldsymbol n|+1)}^{B}\), the
	scalar \(B\)-polynomial
	\(B_{\boldsymbol n}(z)\coloneq B_{\boldsymbol n,(|\boldsymbol n|+1)}^{(1)}(z)\) is given in
	terms of this Kamp\'e de F\'eriet polynomial by
	\begin{equation}
		\label{eq:q-one-B-KdF}
		B_{\boldsymbol n}(z)
		=
		\frac{P(0)}{|\boldsymbol n|!\,P(|\boldsymbol n|)}
		\mathcal K_{\boldsymbol n}(z).
	\end{equation}

	Choose \(i_*\in\{1,\ldots,p\}\) such that
	\(n_*\coloneq n_{i_*}=\max_{1\le i\le p}n_i\), and set
	\(c_*\coloneq c_{i_*}\). Define
	\(S_*(t)\coloneq\prod_{i\ne i_*}(c_i+t)_{n_i}\), so that
	\(P(t)=(c_*+t)_{n_*}S_*(t)\) and \(\deg S_*=|\boldsymbol n|-n_*\). Introduce the
	characteristic polynomial
	\begin{equation}
		\label{eq:q-one-characteristic-polynomial}
		\begin{aligned}
		\mathcal Q_{\boldsymbol n}(x)
		&\coloneq
		\sum_{j=0}^{|\boldsymbol n|-n_*}
		\frac{
			(-1)^j(-x)_j
			(n_*-x+j+1)_{|\boldsymbol n|-n_*-j}
		}{
			(c_*+n_*)_j(|\boldsymbol n|-n_*-j)!
		}
		\\
		&\quad\times
		\Delta^{|\boldsymbol n|-n_*-j}S_*(n_*+j).
		\end{aligned}
	\end{equation}
	Regularity together with the standing admissibility convention ensures that the
	parameter-dependent Pochhammer symbols occurring in the denominators
	of \eqref{eq:q-one-characteristic-polynomial}, as well as \(P(0)\)
	and \(P(|\boldsymbol n|)\), are nonzero. One has
	\(\mathcal Q_{\boldsymbol n}(0)=|\boldsymbol n|!/n_*!\),
	\(\deg\mathcal Q_{\boldsymbol n}=|\boldsymbol n|-n_*\), and
	\begin{equation}
		\label{eq:q-one-characteristic-leading}
		[x^{|\boldsymbol n|-n_*}]\mathcal Q_{\boldsymbol n}(x)
		=
		\frac{S_*(1-c_*)}{(c_*+n_*)_{|\boldsymbol n|-n_*}}
		=
		\frac{1}{(\kappa_{i_*}+b+n_*+1)_{|\boldsymbol n|-n_*}}
		\prod_{\substack{i=1\\i\ne i_*}}^{p}
		(\kappa_i-\kappa_{i_*}+1)_{n_i}.
	\end{equation}

	Let \(\xi_1,\ldots,\xi_{|\boldsymbol n|-n_*}\) be the zeros of
	\(\mathcal Q_{\boldsymbol n}\), counted with multiplicity. Then the
	Kamp\'e de F\'eriet polynomial reduces to
	\begin{equation}
		\label{eq:q-one-KdF-to-pFq}
		\mathcal K_{\boldsymbol n}(z)
		=
		(-1)^{|\boldsymbol n|}\frac{|\boldsymbol n|!}{P(0)}
		\pFq{|\boldsymbol n|-n_*+2}{|\boldsymbol n|-n_*}
		{
			-|\boldsymbol n|,c_*+n_*,
			1-\xi_1,\ldots,1-\xi_{|\boldsymbol n|-n_*}
		}
		{
			-\xi_1,\ldots,-\xi_{|\boldsymbol n|-n_*}
		}
		{z}.
	\end{equation}
	Consequently,
	\begin{equation}
		\label{eq:q-one-B-single-pFq}
		B_{\boldsymbol n}(z)
		=
		\frac{(-1)^{|\boldsymbol n|}}{P(|\boldsymbol n|)}
		\pFq{|\boldsymbol n|-n_*+2}{|\boldsymbol n|-n_*}
		{
			-|\boldsymbol n|,c_*+n_*,
			1-\xi_1,\ldots,1-\xi_{|\boldsymbol n|-n_*}
		}
		{
			-\xi_1,\ldots,-\xi_{|\boldsymbol n|-n_*}
		}
		{z}.
	\end{equation}
	Here
	\(|\boldsymbol n|-n_*=|\boldsymbol n|-\max_{1\le i\le p}n_i\).
	The parameter strings involving the \(\xi_\ell\) are omitted when
	\(|\boldsymbol n|-n_*=0\). Since \(\mathcal Q_{\boldsymbol n}(0)\ne0\), none of the
	\(\xi_\ell\) is zero. If some \(\xi_\ell\in\{1,\ldots,|\boldsymbol n|\}\), the
	generalized hypergeometric expressions in
	\eqref{eq:q-one-KdF-to-pFq} and
	\eqref{eq:q-one-B-single-pFq} are understood coefficientwise, after
	replacing the formally singular unit-difference quotient by its
	polynomial continuation
\(\frac{(1-\xi_\ell)_d}{(-\xi_\ell)_d} = \frac{d-\xi_\ell}{-\xi_\ell},\) \(d\in\{0,\ldots,|\boldsymbol n|\}.\)
\end{corollary}

\begin{proof}
	When \(q=1\), the inequality \(0\le r<q\) gives \(r=0\), while
	\(B\)-balancedness gives \(m_1=|\boldsymbol n|+1\). There is only one row index, so
	the sole contribution in
	Corollary~\ref{cor:final-B-KdF} is the diagonal block \(j=H=1\).
	All parameter strings of lengths \(q-1\) and \(r\) are empty, and the
	resulting block is exactly \(\mathcal K_{\boldsymbol n}\) in
	\eqref{eq:q-one-KdF-block}. Furthermore,
	\(\pi_{1,0}=-P(0)/|\boldsymbol n|!\) and
	\(\nu_{\boldsymbol n,(|\boldsymbol n|+1)}^{B}=-1/P(|\boldsymbol n|)\). Hence
	\(B_{\boldsymbol n}=\nu_{\boldsymbol n,(|\boldsymbol n|+1)}^{B}\pi_{1,0}\mathcal K_{\boldsymbol n}\),
	which proves \eqref{eq:q-one-B-KdF}.

	By the definition of the Kamp\'e de F\'eriet series,
	\begin{equation}
		\label{eq:q-one-KdF-double-sum}
		\mathcal K_{\boldsymbol n}(z)
		=
		\sum_{\substack{u,\lambda\ge0\\u+\lambda\le |\boldsymbol n|}}
		\frac{
			(-|\boldsymbol n|)_{u+\lambda}
			(\boldsymbol c+\boldsymbol n)_{u+\lambda}
			(1)_u
		}{
			(\boldsymbol c)_{u+\lambda}
		}
		\frac{(-z)^u}{u!}
		\frac{1}{\lambda!}.
	\end{equation}
	The numerator parameter \(-|\boldsymbol n|\) imposes the triangular termination,
	and \((1)_u/u!=1\). For every \(k\in\N_0\),
	\[
		\frac{(\boldsymbol c+\boldsymbol n)_k}{(\boldsymbol c)_k}
		=
		\prod_{i=1}^{p}
		\frac{(c_i+n_i)_k}{(c_i)_k}
		=
		\prod_{i=1}^{p}
		\frac{(c_i+k)_{n_i}}{(c_i)_{n_i}}
		=
		\frac{P(k)}{P(0)}.
	\]
	The remaining \(\lambda\)-sum obtained after fixing \(u=d\) is
	therefore a terminating \({}_{p+1}F_p(1)\) with nonnegative integral
	parameter differences \(n_i\); pairs corresponding to \(n_i=0\)
	cancel.

	For \(d\in\{0,\ldots,|\boldsymbol n|\}\), the coefficient of \(z^d\) is obtained
	by setting \(u=d\) in \eqref{eq:q-one-KdF-double-sum}. Using
	\((-|\boldsymbol n|)_{d+\lambda}=(-|\boldsymbol n|)_d(-|\boldsymbol n|+d)_\lambda\) gives
	\begin{equation}
		\label{eq:q-one-KdF-coefficient-first}
		[z^d]\mathcal K_{\boldsymbol n}(z)
		=
		(-1)^d\frac{(-|\boldsymbol n|)_d}{P(0)}
		\sum_{\lambda=0}^{|\boldsymbol n|-d}
		\frac{(-|\boldsymbol n|+d)_\lambda}{\lambda!}
		P(d+\lambda).
	\end{equation}
	Since
	\((-|\boldsymbol n|+d)_\lambda/\lambda!=(-1)^\lambda
	\binom{|\boldsymbol n|-d}{\lambda}\), the forward-difference identity
\(\Delta^mP(d) = \sum_{\lambda=0}^{m} (-1)^{m-\lambda} \binom{m}{\lambda}P(d+\lambda)\)
	transforms \eqref{eq:q-one-KdF-coefficient-first} into
	\begin{equation}
		\label{eq:q-one-KdF-coefficient-difference}
		[z^d]\mathcal K_{\boldsymbol n}(z)
		=
		(-1)^{|\boldsymbol n|}\frac{(-|\boldsymbol n|)_d}{P(0)}
		\Delta^{|\boldsymbol n|-d}P(d).
	\end{equation}

	Factor the finite difference in
	\eqref{eq:q-one-KdF-coefficient-difference}. Since
	\(P(t)=(c_*+t)_{n_*}S_*(t)\), the discrete Leibniz rule gives
\(\Delta^k(fg)(t) = \sum_{\ell=0}^{k} \binom{k}{\ell} \Delta^\ell f(t) \Delta^{k-\ell}g(t+\ell).\)
	For the first factor,
\(\Delta^\ell(c_*+t)_{n_*} = \frac{n_*!}{(n_*-\ell)!} (c_*+t+\ell)_{n_*-\ell},\) \(0\le\ell\le n_*,\)
	and this finite difference vanishes for \(\ell>n_*\).

	Taking \(k=|\boldsymbol n|-d\), \(t=d\), and writing
	\(\ell=n_*-d+j\), one has
	\(k-\ell=|\boldsymbol n|-n_*-j\) and \(t+\ell=n_*+j\). The corresponding
	coefficient satisfies
	\begin{equation*}
		\binom{|\boldsymbol n|-d}{n_*-d+j}
		\frac{n_*!}{(d-j)!}
		(c_*+n_*+j)_{d-j}
		=
		\frac{n_*!}{d!}
		(c_*+n_*)_d
		\frac{
			(-1)^j(-d)_j
			(n_*-d+j+1)_{|\boldsymbol n|-n_*-j}
		}{
			(c_*+n_*)_j(|\boldsymbol n|-n_*-j)!
		}.
	\end{equation*}
	Indeed, \(d!/(d-j)!=(-1)^j(-d)_j\),
\(\frac{(c_*+n_*+j)_{d-j}} {(c_*+n_*)_d} = \frac{1}{(c_*+n_*)_j},\)
	and
	\((|\boldsymbol n|-d)!/(n_*-d+j)!=(n_*-d+j+1)_{|\boldsymbol n|-n_*-j}\). Substitution in the
	discrete Leibniz formula yields
	\begin{align*}
		\Delta^{|\boldsymbol n|-d}P(d)
		&=
		\frac{n_*!}{d!}
		(c_*+n_*)_d
		\sum_{j=0}^{|\boldsymbol n|-n_*}
		\frac{
			(-1)^j(-d)_j
			(n_*-d+j+1)_{|\boldsymbol n|-n_*-j}
		}{
			(c_*+n_*)_j(|\boldsymbol n|-n_*-j)!
		}
		\\
		&\quad\times
		\Delta^{|\boldsymbol n|-n_*-j}S_*(n_*+j).
	\end{align*}
	The sum may be extended to the full range
	\(j\in\{0,\ldots,|\boldsymbol n|-n_*\}\). The factor \((-d)_j\) removes the
	terms with \(j>d\). Moreover, if \(\ell=n_*-d+j<0\), then
	\((n_*-d+j+1)_{|\boldsymbol n|-n_*-j}=(\ell+1)_{|\boldsymbol n|-n_*-j}=0\). Indeed,
	\(\ell\) is a negative integer and \(|\boldsymbol n|-n_*-j\ge-\ell\), the latter
	inequality being equivalent to \(d\le |\boldsymbol n|\); hence the
	product defining \((\ell+1)_{|\boldsymbol n|-n_*-j}\) contains the factor zero.
	Comparing with \eqref{eq:q-one-characteristic-polynomial} evaluated at
	\(x=d\) gives
	\begin{equation}
		\label{eq:q-one-difference-factorization}
		\Delta^{|\boldsymbol n|-d}P(d)
		=
		\frac{n_*!}{d!}
		(c_*+n_*)_d
		\mathcal Q_{\boldsymbol n}(d).
	\end{equation}

	At \(x=0\), all terms with \(j\ge1\) vanish because of the factor
	\((-x)_j\). Since \(S_*\) is monic of degree \(|\boldsymbol n|-n_*\), one has
	\(\Delta^{|\boldsymbol n|-n_*} S_*(n_*)=(|\boldsymbol n|-n_*)!\), and therefore
	\(\mathcal Q_{\boldsymbol n}(0)=(n_*+1)_{|\boldsymbol n|-n_*}=|\boldsymbol n|!/n_*!\).

	Each summand in \eqref{eq:q-one-characteristic-polynomial} has degree
	at most \(|\boldsymbol n|-n_*\). Extracting the coefficient of \(x^{|\boldsymbol n|-n_*}\) and
	setting \(k=|\boldsymbol n|-n_*-j\) gives
	\[
		[x^{|\boldsymbol n|-n_*}]\mathcal Q_{\boldsymbol n}(x)
		=
		\sum_{k=0}^{|\boldsymbol n|-n_*}
		\frac{(-1)^k}{(c_*+n_*)_{|\boldsymbol n|-n_*-k}k!}
		\Delta^kS_*(|\boldsymbol n|-k).
	\]
	The elementary reflection identity
\(\frac{(-1)^k}{(c_*+n_*)_{|\boldsymbol n|-n_*-k}} = \frac{(1-c_*-n_*-(|\boldsymbol n|-n_*))_k}{(c_*+n_*)_{|\boldsymbol n|-n_*}}\)
	therefore gives
	\[
		[x^{|\boldsymbol n|-n_*}]\mathcal Q_{\boldsymbol n}(x)
		=
		\frac{1}{(c_*+n_*)_{|\boldsymbol n|-n_*}}
		\sum_{k=0}^{|\boldsymbol n|-n_*}
		\frac{(1-c_*-n_*-(|\boldsymbol n|-n_*))_k}{k!}
		\Delta^kS_*(|\boldsymbol n|-k).
	\]
	The backward Newton interpolation formula
\(S_*(x) = \sum_{k=0}^{|\boldsymbol n|-n_*} \frac{(x-n_*-(|\boldsymbol n|-n_*))_k}{k!} \Delta^kS_*(|\boldsymbol n|-k),\)
	evaluated at \(x=1-c_*\), now gives
	\([x^{|\boldsymbol n|-n_*}]\mathcal Q_{\boldsymbol n}(x)
	=S_*(1-c_*)/(c_*+n_*)_{|\boldsymbol n|-n_*}\). Finally,
	\[
		S_*(1-c_*)
		=
		\prod_{\substack{i=1\\i\ne i_*}}^{p}
		(c_i-c_*+1)_{n_i}
		=
		\prod_{\substack{i=1\\i\ne i_*}}^{p}
		(\kappa_i-\kappa_{i_*}+1)_{n_i},
	\]
	which is nonzero by the noninteger-separation assumptions. The
	denominator is nonzero by the standing admissibility convention, so
	\(\deg\mathcal Q_{\boldsymbol n}=|\boldsymbol n|-n_*\).

	Let \(\xi_1,\ldots,\xi_{|\boldsymbol n|-n_*}\) be the zeros of
	\(\mathcal Q_{\boldsymbol n}\). Since
	\(\mathcal Q_{\boldsymbol n}(0)\ne0\), none of them is zero, and, for
	\(d\in\{0,\ldots,|\boldsymbol n|\}\),
	\[
		\frac{\mathcal Q_{\boldsymbol n}(d)}{\mathcal Q_{\boldsymbol n}(0)}
		=
		\prod_{\ell=1}^{|\boldsymbol n|-n_*}
		\frac{d-\xi_\ell}{-\xi_\ell}
		=
		\prod_{\ell=1}^{|\boldsymbol n|-n_*}
		\frac{(1-\xi_\ell)_d}{(-\xi_\ell)_d},
	\]
	where the last quotient is interpreted coefficientwise by polynomial
	continuation when one of the parameters \(-\xi_\ell\) is a
	nonpositive integer. Combining this identity with
	\eqref{eq:q-one-KdF-coefficient-difference} and
	\eqref{eq:q-one-difference-factorization}, and using
	\(\mathcal Q_{\boldsymbol n}(0)=|\boldsymbol n|!/n_*!\) gives
	\[
		[z^d]\mathcal K_{\boldsymbol n}(z)
		=
		(-1)^{|\boldsymbol n|}\frac{|\boldsymbol n|!}{P(0)}
		\frac{(-|\boldsymbol n|)_d(c_*+n_*)_d}{d!}
		\prod_{\ell=1}^{|\boldsymbol n|-n_*}
		\frac{(1-\xi_\ell)_d}{(-\xi_\ell)_d}.
	\]
	This is precisely the coefficient of \(z^d\) in the terminating
	\({}_{|\boldsymbol n|-n_*+2}F_{|\boldsymbol n|-n_*}\) in
	\eqref{eq:q-one-KdF-to-pFq}. Formula
	\eqref{eq:q-one-B-single-pFq} follows from
	\eqref{eq:q-one-B-KdF}.
\end{proof}

The characteristic polynomial just introduced has an intrinsic discrete
orthogonality interpretation. In fact, after reflection it is a type-II
multiple Hahn polynomial for the \(p-1\) columns different from \(i_*\).
The scalar Hahn polynomials and their finite-difference structure belong to
the classical theory of orthogonal polynomials of a discrete variable; see
\cite{NikiforovSuslovUvarov1991}. Their multiple extension was developed by
Arves\'u, Coussement, and Van Assche, who obtained the corresponding
type-II systems for the classical discrete families, including the Hahn
family \cite{ArvesuCoussementVanAssche2003}. Explicit hypergeometric
representations for the type-I multiple Hahn polynomials in the two-weight
case were subsequently derived in
\cite{BranquinhoDiazFoulquieManas2023Hahn}, while integral and
hypergeometric representations for both types and an arbitrary number of
weights were obtained in
\cite{BranquinhoDiazFoulquieManasWolfs2025}. The proposition below shows
that the auxiliary polynomial arising in the present Bessel reduction fits
intrinsically into this multiple Hahn framework.

\begin{proposition}[Multiple Hahn structure of the characteristic polynomial]
	\label{prop:q-one-characteristic-multiple-Hahn}
	Retain the notation of
	Corollary~\ref{cor:final-B-q-one-hypergeometric}. Let
	\(\boldsymbol n^{\,*i_*}\coloneq(n_h)_{h\ne i_*}\), so that
	\(|\boldsymbol n^{\,*i_*}|=|\boldsymbol n|-n_*\). For every \(h\ne i_*\), define
	\(\alpha_h\coloneq-c_h-|\boldsymbol n|-n_h =-\kappa_h-b-|\boldsymbol n|-n_h-1,\) \(\beta\coloneq c_*+n_*-1 =\kappa_{i_*}+b+n_*,\)
	and consider on the lattice \(\{0,\ldots,|\boldsymbol n|\}\) the \(p-1\) Hahn
	weights
	\begin{equation}
		\label{eq:q-one-multiple-Hahn-weights}
		\omega_h(y)
		\coloneq
		\frac{(\alpha_h+1)_y}{y!}
		\frac{(\beta+1)_{|\boldsymbol n|-y}}{(|\boldsymbol n|-y)!},
		\qquad
		y\in\{0,\ldots,|\boldsymbol n|\},
		\quad h\ne i_*.
	\end{equation}
	Then the reflected characteristic polynomial
	\(\mathcal Q_{\boldsymbol n}(|\boldsymbol n|-y)\) satisfies the type-II multiple Hahn
	orthogonality conditions
	\begin{equation}
		\label{eq:q-one-characteristic-Hahn-orthogonality}
		\sum_{y=0}^{|\boldsymbol n|}
		\mathcal Q_{\boldsymbol n}(|\boldsymbol n|-y)(-y)_k\omega_h(y)
		=0,
		\qquad
		k\in\{0,\ldots,n_h-1\},
		\quad h\ne i_*.
	\end{equation}
	Since the factorial polynomials \((-y)_k\) span the polynomials of
	degree at most \(k\), these are precisely the type-II multiple Hahn
	orthogonality conditions for the multi-index \(\boldsymbol n^{\,*i_*}\).
	
	Whenever this multiple Hahn index is normal, let
	\(H_{\boldsymbol n^{\,*i_*}}(y)\) denote the corresponding monic type-II
	multiple Hahn polynomial. Then
	\begin{equation}
		\label{eq:q-one-characteristic-multiple-Hahn}
		\mathcal Q_{\boldsymbol n}(x)
		=
		(-1)^{|\boldsymbol n|-n_*} L_{\boldsymbol n}\,
		H_{\boldsymbol n^{\,*i_*}}(|\boldsymbol n|-x),
	\end{equation}
	where
	\begin{equation}
		\label{eq:q-one-characteristic-leading-Hahn}
		L_{\boldsymbol n}
		\coloneq
		[x^{|\boldsymbol n|-n_*}]\mathcal Q_{\boldsymbol n}(x)
		=
		\frac{
			\prod_{\substack{h=1\\h\ne i_*}}^{p}
			(\kappa_h-\kappa_{i_*}+1)_{n_h}
		}{
			(\kappa_{i_*}+b+n_*+1)_{|\boldsymbol n|-n_*}		}.
	\end{equation}
	Consequently, if \(\zeta_1,\ldots,\zeta_{|\boldsymbol n|-n_*}\) are the zeros of
	\(H_{\boldsymbol n^{\,*i_*}}\), counted with multiplicity, then the zeros
	of the characteristic polynomial are
	\(\xi_\ell=|\boldsymbol n|-\zeta_\ell\),
	\(\ell\in\{1,\ldots,|\boldsymbol n|-n_*\}\).
\end{proposition}

\begin{proof}
	The finite-difference factorization obtained in the proof of
	Corollary~\ref{cor:final-B-q-one-hypergeometric} gives, for
	\(d\in\{0,\ldots,|\boldsymbol n|\}\),
	\[
	\mathcal Q_{\boldsymbol n}(d)
	=
	\frac{d!}{n_*!\,(c_*+n_*)_d}
	\Delta^{|\boldsymbol n|-d}P(d).
	\]
	Setting \(d=|\boldsymbol n|-y\) gives
	\begin{equation}
		\label{eq:q-one-Q-reflected-difference}
		\mathcal Q_{\boldsymbol n}(|\boldsymbol n|-y)
		=
		\frac{(|\boldsymbol n|-y)!}
		{n_*!\,(c_*+n_*)_{|\boldsymbol n|-y}}
		\Delta^yP(|\boldsymbol n|-y).
	\end{equation}
	Since \(\beta+1=c_*+n_*\), multiplication by the weight
	\eqref{eq:q-one-multiple-Hahn-weights} cancels the factorial and
	Pochhammer factors in \eqref{eq:q-one-Q-reflected-difference}, giving
	\begin{equation}
		\label{eq:q-one-Q-times-Hahn-weight}
		\mathcal Q_{\boldsymbol n}(|\boldsymbol n|-y)\omega_h(y)
		=
		\frac{1}{n_*!}
		\frac{(\alpha_h+1)_y}{y!}
		\Delta^yP(|\boldsymbol n|-y).
	\end{equation}
	
	The following elementary finite-difference identity will be used. For every
	polynomial \(R\) of degree at most \(|\boldsymbol n|\), every
	\(k\in\{0,\ldots,|\boldsymbol n|\}\), and every parameter \(a\),
	\begin{equation}
		\label{eq:q-one-Newton-difference-identity}
		\sum_{y=0}^{|\boldsymbol n|}
		(-y)_k\frac{(a)_y}{y!}
		\Delta^yR(|\boldsymbol n|-y)
		=
		(-1)^k(a)_k\Delta^kR(|\boldsymbol n|+a).
	\end{equation}
	Indeed, the terms with \(y<k\) vanish, and
	\[
	(-y)_k\frac{(a)_y}{y!}
	=
	(-1)^k(a)_k
	\frac{(a+k)_{y-k}}{(y-k)!}.
	\]
	After setting \(s=y-k\), the remaining sum is
	\[
	(-1)^k(a)_k
	\sum_{s=0}^{|\boldsymbol n|-k}
	\frac{(a+k)_s}{s!}
	\Delta^s\bigl(\Delta^kR\bigr)(|\boldsymbol n|-k-s).
	\]
	By the backward Newton expansion, this equals
	\[
	(-1)^k(a)_k
	\Delta^kR\bigl(|\boldsymbol n|-k+a+k\bigr)
	=
	(-1)^k(a)_k\Delta^kR(|\boldsymbol n|+a),
	\]
	which proves \eqref{eq:q-one-Newton-difference-identity}.
	
	Fix \(h\ne i_*\) and apply
	\eqref{eq:q-one-Newton-difference-identity} with
	\(R=P\) and \(a=\alpha_h+1\). By definition,
	\(|\boldsymbol n|+\alpha_h+1=1-c_h-n_h\). Hence, for
	\(k\in\{0,\ldots,n_h-1\}\),
	\(\Delta^kP(1-c_h-n_h)=0\). Indeed, this difference is a linear
	combination of the values \(P(1-c_h-n_h+s)\), with
	\(s\in\{0,\ldots,k\}\), and each of them contains the vanishing
	factor
	\[
	(c_h+1-c_h-n_h+s)_{n_h}
	=
	(1-n_h+s)_{n_h}
	=
	0,
	\]
	because \(s<n_h\). Combining this vanishing with
	\eqref{eq:q-one-Q-times-Hahn-weight} and
	\eqref{eq:q-one-Newton-difference-identity} proves
	\eqref{eq:q-one-characteristic-Hahn-orthogonality}.
	
	The total number of orthogonality conditions is
	\(\sum_{h\ne i_*}n_h=|\boldsymbol n|-n_*\), which is the degree of
	\(\mathcal Q_{\boldsymbol n}\). Therefore, whenever the multiple Hahn
	index is normal, \(\mathcal Q_{\boldsymbol n}(|\boldsymbol n|-y)\) is proportional
	to the monic type-II multiple Hahn polynomial
	\(H_{\boldsymbol n^{\,*i_*}}(y)\). Its leading coefficient as a polynomial
	in \(y\) is \((-1)^{|\boldsymbol n|-n_*}L_{\boldsymbol n}\), which proves
	\eqref{eq:q-one-characteristic-multiple-Hahn}. The relation between
	the zeros follows immediately.
\end{proof}

\begin{remark}[Why the multiple Hahn polynomial appears]
	\label{rem:q-one-why-multiple-Hahn}
	The multiple Hahn structure follows directly from the finite-difference
	construction of \(\mathcal Q_{\boldsymbol n}\). For each
	\(h\ne i_*\), the factor \((c_h+t)_{n_h}\) in \(P(t)\) has the
	\(n_h\) consecutive zeros
	\(-c_h,-c_h-1,\ldots,-c_h-n_h+1\). After the reflection
	\(y=|\boldsymbol n|-d\), the complementary difference \(\Delta^{|\boldsymbol n|-d}P(d)\)
	becomes the backward difference \(\nabla^yP(|\boldsymbol n|)\). Multiplication by
	the Hahn weight cancels exactly the factorial and Pochhammer factors in
	\eqref{eq:q-one-Q-reflected-difference}, while the block of consecutive
	zeros of \((c_h+t)_{n_h}\) produces the \(n_h\) orthogonality
	conditions associated with the weight \(\omega_h\).
	
	Thus each column \(h\ne i_*\) contributes one Hahn weight and exactly
	\(n_h\) orthogonality conditions. Their total number is
	\(\sum_{h\ne i_*}n_h=|\boldsymbol n|-n_*\), the degree of the characteristic
	polynomial. In particular, \(p=2\) gives an ordinary Hahn polynomial,
	whereas \(p\ge3\) gives a type-II multiple Hahn polynomial with
	\(p-1\) weights.
\end{remark}

\begin{remark}[Relation with the ABV formula and the two-index reduction]
	\label{rem:q-one-ABV-double-sum}
	For \(p=2\), let \(\boldsymbol n=(n_1,n_2)\in\N_0^2\). Aptekarev,
	Branquinho, and Van Assche give the type-II multiple Bessel polynomial
	associated with the weights \(z^{\alpha_i}\exp(\gamma/z)\) as the
	explicit double sum \cite[Section~3.4]{ABV2003}
	\begin{equation}
		\label{eq:q-one-ABV-double-sum}
		Q_{(n_1,n_2)}^{\mathrm{ABV}}(z)
		=
		\sum_{k=0}^{n_1}
		\sum_{j=0}^{n_2}
		\binom{n_1}{k}
		\binom{n_2}{j}
		(|\boldsymbol n|+\alpha_1+1)_k
		(n_2+\alpha_2+k+1)_j
		(-\gamma)^{|\boldsymbol n|-k-j}
		z^{k+j}.
	\end{equation}
	Under the identification
	\(\alpha_i=b+\kappa_i\), \(i\in\{1,2\}\), and for
	\(\gamma\ne0\), direct coefficient comparison gives the exact
	correspondence
	\begin{equation}
		\label{eq:q-one-ABV-exact-correspondence}
		Q_{(n_1,n_2)}^{\mathrm{ABV}}(\gamma z)
		=
		\gamma^{|\boldsymbol n|}
		(\kappa_1+b+|\boldsymbol n|+1)_{n_1}
		(\kappa_2+b+|\boldsymbol n|+1)_{n_2}
		B_{(n_1,n_2)}(z).
	\end{equation}
	Thus the ABV double sum is exactly the \(p=2\) specialization of the
	Kamp\'e de F\'eriet representation obtained above, after monic
	normalization and rescaling.

	If \(i_*\in\{1,2\}\) is chosen so that
	\(n_{i_*}=\max\{n_1,n_2\}\), Proposition~
	\ref{prop:q-one-characteristic-multiple-Hahn} identifies the characteristic
	polynomial with the reflection of an ordinary Hahn polynomial of degree
	\(\min\{n_1,n_2\}\). Consequently,
	Corollary~\ref{cor:final-B-q-one-hypergeometric} reduces in this case to a
	terminating
	\({}_{\min\{n_1,n_2\}+2}F_{\min\{n_1,n_2\}}\) representation whose
	unit-difference parameters are the reflected Hahn zeros. Hence the ABV
	double sum, the Kamp\'e de F\'eriet polynomial, and this single
	generalized hypergeometric polynomial are three representations of the same
	monic scalar \(B\)-polynomial. If \(n_1=n_2\), either maximizing index may
	be chosen; the resulting parameter lists may differ, but the represented
	polynomial is the same.
\end{remark}

\section{Recurrence relations}
\label{sec:final-recurrences}

The recurrence relations for the explicit Bessel system are derived next.
The pairing used throughout this section is the bilinear contour pairing
\eqref{eq:general-pairing}, specialized to the Laurent representative \(\CB\).
Thus all pairings below are contour pairings on \(\Torus\), evaluated
by the moment formula \eqref{eq:final-moment-entry}; no new bilinear form is
introduced.

The abstract step-line recurrence matrix \(T\), together with the dual action
of \(T^\top\), was established in
Proposition~\ref{prop:general-step-recurrence}.  In the present section, the
step-line enters only through the explicit evaluation of its entries
\(t_{N,k}\) by means of the finite Gamma--Pochhammer pairing formulas developed
below.

The main purpose is to construct recurrences throughout the near-diagonal
range.  Away from the step-line, permuting the components of a near-diagonal
multi-index gives several equally natural orderings, and those orderings need
not select the same neighboring indices.  The construction therefore begins with a chosen
near-diagonal index pair, whose neighbors are specified explicitly.

Multiplication by \(z\) places the central vector in a finite-dimensional
space determined by the enlarged degree bounds and the remaining moment
conditions.  For interior indices, this local recurrence space has dimension
\(p+q+1\).  A set of \(p+q+1\) admissible near-diagonal neighbors is
chosen, and a diagonal pairing argument proves that they form a basis of that
space.
This yields a unique local recurrence.  At the boundary, the decreasing moves
which would produce negative components are omitted, and the same argument
gives a shorter recurrence.

The coefficients of both the local recurrences and the step-line recurrence
matrix are obtained from the same finite Gamma--Pochhammer pairing block.

\subsection{Explicit near-diagonal finite Gamma--Pochhammer pairing block}
\label{sec:final-pairing-blocks}

The recurrence coefficients below are expressed through the two pairings
\(\langle \mathbf B,\mathbf A\rangle\) and
\(\langle z\mathbf B,\mathbf A\rangle\).  They are evaluated only in the range
where the explicit Bessel representatives of
Theorem~\ref{thm:final-explicit} are available.  Thus the following block is
not an arbitrary-index formula; it is a near-diagonal finite formula, and it is
used only with admissible index pairs.

Recall that for a vector \(\boldsymbol c\), \(\boldsymbol c^{\,*u}\) denotes the
vector obtained by removing its \(u\)-th component.  When \(H\ne j\),
\(\boldsymbol e_j^{\,*H}\in\mathbb N_0^{q-1}\) denotes the vector obtained from
\(\boldsymbol e_j\) by deleting its \(H\)-th component.

Let \((\boldsymbol\nu,\boldsymbol\mu)\) be the near-diagonal index of a \(B\)-vector and
let \((\boldsymbol\lambda,\boldsymbol\eta)\) be the near-diagonal index of an \(A\)-vector,
both covered by Theorem~\ref{thm:final-explicit}.  Define
\(\eta_-\coloneq\min_{1\le h\le q}\eta_h,\) \(j_-\coloneq\min\{h:\eta_h=\eta_-\},\)
and
\(\mu_+\coloneq\max_{1\le h\le q}\mu_h,\) \(j_+\coloneq\max\{h:\mu_h=\mu_+\},\) \(K_+\coloneq\mu_+-1.\)
The product of the two normalizing constants appearing in
\eqref{eq:final-A-nu} and \eqref{eq:final-B-nu} is written as a single
normalizing prefactor, with the denominator depending on \(\boldsymbol\nu\)
left to be absorbed into the summand:
\begin{equation}
	\label{eq:nd-AB-normalization-prefactor}
	\mathcal N_{(\boldsymbol\nu,\boldsymbol\mu)}^{(\boldsymbol\lambda,\boldsymbol\eta)}
	\coloneq
	\frac{
		\bigl(\boldsymbol\kappa+(b_{j_-}+\eta_-+1)\one_p\bigr)_{\boldsymbol\lambda}
		\bigl(\boldsymbol a-(b_{j_+}+K_+)\one_r\bigr)_{K_+}
	}{
		\bigl(\boldsymbol a-(b_{j_-}+\eta_-)\one_r\bigr)_{\eta_-}
		\prod_{\substack{h=1\\ \eta_h=\eta_-+1}}^{q}(b_h-b_{j_-})
	}
	\prod_{\substack{h=1\\ h\ne j_+,\ \mu_h=\mu_+}}^{q}
	(b_h-b_{j_+}).
\end{equation}
For
\(i\in\{1,\ldots,p\},\) \(j,H\in\{1,\ldots,q\},\) \(\alpha\in\{0,\ldots,\lambda_i-1\},\)
Define the completely canceled Pochhammer summand
\(\mathcal S_{d;i,j,H,K} (\boldsymbol\nu,\boldsymbol\mu;\boldsymbol\lambda,\boldsymbol\eta;\alpha,\beta),\)
with \(d\in\{0,1\}\), as follows.  Empty products are understood as one.
Throughout this finite block, Pochhammer symbols with integer, possibly
negative, exponents are interpreted componentwise through the Gamma quotient
\((x)_n=\frac{\Gamma(x+n)}{\Gamma(x)},\) \(n\in\mathbb Z,\)
whenever the two Gamma factors are finite.  Equivalently,
\((x)_{-n}=\frac{1}{(x-n)_n},\) \(n\in\N.\)
The regularity assumptions and the standing admissibility convention exclude
the poles which could occur in the applications below.

First suppose that \(H=j\),
\(K\in\{0,\ldots,\mu_H-1\},\) \(\beta\in\{0,\ldots,K\}.\)
Then
{\scriptsize\begin{multline}
		\label{eq:nd-fused-summand-diagonal}
		\mathcal S_{d;i,H,H,K}
		(\boldsymbol\nu,\boldsymbol\mu;\boldsymbol\lambda,\boldsymbol\eta;\alpha,\beta)
		\\=
		\frac{(-1)^{|\boldsymbol\lambda|-\lambda_i+\alpha+\beta-K}}{
			\alpha!(\lambda_i-1-\alpha)!(K-\beta)!(\mu_H-K-1)!}
		\\
		\times
		\frac{
			\bigl(\kappa_i\one_r+\boldsymbol a+(\alpha+1)\one_r\bigr)_{d+\beta}
			\bigl(\boldsymbol a-(b_H+K)\one_r+(\beta+1)\one_r\bigr)_{K-\beta}
		}{
			\bigl(\kappa_i\one_{p-1}-\boldsymbol\kappa^{\,*i}-\boldsymbol\lambda^{\,*i}
			+(\alpha+1)\one_{p-1}\bigr)_{\boldsymbol\lambda^{\,*i}}
			\bigl(\boldsymbol a-(b_H+K)\one_r\bigr)_K
		}
		\\
		\times
		\frac{\bigl(\boldsymbol\kappa+(b_H+K+1)\one_p\bigr)_{\boldsymbol\nu}}
		{\bigl(\boldsymbol\kappa+(b_{j_+}+K_++1)\one_p\bigr)_{\boldsymbol\nu}}
		\\
		\times
		\frac{
			\bigl(\kappa_i\one_q+\boldsymbol b+(d+\alpha+\beta+1)\one_q+\boldsymbol e_H\bigr)_{
				\boldsymbol\eta-(d+\beta)\one_q-\boldsymbol e_H}
		}{
			\bigl(\boldsymbol b^{\,*H}-b_H\one_{q-1}\bigr)_1
		}
		\\
		\times
		\frac{1}{
			\bigl(\boldsymbol b^{\,*H}-(b_H+K)\one_{q-1}+(\beta+1)\one_{q-1}\bigr)_{
				\boldsymbol\mu^{\,*H}-(\beta+1)\one_{q-1}}
		}.
\end{multline}}
This formula includes the term \(K=0\), for which necessarily \(\beta=0\).

Now suppose that \(H\ne j\),
\(K\in\{1,\ldots,\mu_H-1\},\) \(\beta\in\{0,\ldots,K-1\}.\)
Then
{\scriptsize\begin{multline}
		\label{eq:nd-fused-summand-offdiagonal}
		\mathcal S_{d;i,j,H,K}
		(\boldsymbol\nu,\boldsymbol\mu;\boldsymbol\lambda,\boldsymbol\eta;\alpha,\beta)
		\\
		=
		\frac{(-1)^{|\boldsymbol\lambda|-\lambda_i+1+\alpha+K+\beta}}{
			\alpha!(\lambda_i-1-\alpha)!(K-1-\beta)!(\mu_H-K-1)!}
		\\
		\times
		\frac{
			\bigl(\boldsymbol a-b_j\one_r\bigr)_1
			\bigl(\kappa_i\one_r+\boldsymbol a+(\alpha+1)\one_r\bigr)_{d+\beta}
		}{\bigl(\kappa_i\one_{p-1}-\boldsymbol\kappa^{\,*i}-\boldsymbol\lambda^{\,*i}
			+(\alpha+1)\one_{p-1}\bigr)_{\boldsymbol\lambda^{\,*i}}
			\bigl(\boldsymbol a-(b_H+K)\one_r\bigr)_{\beta+1}
		}
		\\
		\times
		\frac{\bigl(\boldsymbol\kappa+(b_H+K+1)\one_p\bigr)_{\boldsymbol\nu}}
		{\bigl(\boldsymbol\kappa+(b_{j_+}+K_++1)\one_p\bigr)_{\boldsymbol\nu}}
		\\
		\times
		\frac{
			\bigl(\kappa_i\one_q+\boldsymbol b+(d+\alpha+\beta+1)\one_q+\boldsymbol e_j\bigr)_{
				\boldsymbol\eta-(d+\beta)\one_q-\boldsymbol e_j}
		}{
			\bigl(\boldsymbol b^{\,*j}-b_j\one_{q-1}\bigr)_1
		}
		\\
		\times
		\frac{1}{
			\bigl(\boldsymbol b^{\,*H}-(b_H+K)\one_{q-1}+(\beta+1)\one_{q-1}
			-\boldsymbol e_j^{\,*H}\bigr)_{
				\boldsymbol\mu^{\,*H}-(\beta+1)\one_{q-1}+\boldsymbol e_j^{\,*H}}
		}.
\end{multline}}
In all other cases the summand is defined to be zero.

\begin{definition}[Near-diagonal finite Pochhammer pairing block]
	\label{def:final-pairing-block}
	For two index pairs \((\boldsymbol\nu,\boldsymbol\mu)\) and
	\((\boldsymbol\lambda,\boldsymbol\eta)\) covered by Theorem~\ref{thm:final-explicit}, set,
	for \(d\in\{0,1\}\),
	\begin{equation}
		\label{eq:final-pairing-block}
		\mathscr P
		_{d,(\boldsymbol\nu,\boldsymbol\mu)}^{(\boldsymbol\lambda,\boldsymbol\eta)}
		\coloneq
		\mathcal N_{(\boldsymbol\nu,\boldsymbol\mu)}^{(\boldsymbol\lambda,\boldsymbol\eta)}
		\sum_{j=1}^{q}
		\sum_{i=1}^{p}
		\sum_{\beta=0}^{\mu_j-1}
		\sum_{\alpha=0}^{\lambda_i-1}
		\sum_{H=1}^{q}
		\sum_{K=0}^{\mu_H-1}
		\mathcal S_{d;i,j,H,K}
		(\boldsymbol\nu,\boldsymbol\mu;\boldsymbol\lambda,\boldsymbol\eta;\alpha,\beta).
	\end{equation}
	This is a finite expression involving only Pochhammer symbols, factorials and
	finite sums.  The prefactor
	\(\mathcal N_{(\boldsymbol\nu,\boldsymbol\mu)}^{(\boldsymbol\lambda,\boldsymbol\eta)}\) contains the
	explicit normalizations of the \(A\)- and \(B\)-vectors, and the remaining
	\(B\)-normalizing denominator has been absorbed directly into each
	\(H,K\) residue contribution.  The block is defined only in the near-diagonal
	regime where the explicit Bessel formulas have been proved.
\end{definition}

\begin{proposition}[Evaluation of the near-diagonal pairing block]
	\label{prop:final-pairing-block-interpretation}
	For every \(d\in\{0,1\}\),
	\begin{equation}
		\label{eq:final-pairing-block-interpretation}
		\mathscr P
		_{d,(\boldsymbol\nu,\boldsymbol\mu)}^{(\boldsymbol\lambda,\boldsymbol\eta)}
		=
		\left\langle
		z^d\mathbf B_{\boldsymbol\nu,\boldsymbol\mu},
		\mathbf A_{\boldsymbol\lambda,\boldsymbol\eta}
		\right\rangle .
	\end{equation}
\end{proposition}

\begin{proof}
	Expand \(\mathbf B_{\boldsymbol\nu,\boldsymbol\mu}\) using
	\eqref{eq:final-B-components} and expand \(\mathbf A_{\boldsymbol\lambda,\boldsymbol\eta}\)
	using \eqref{eq:final-A-components}.  The explicit normalizations
	\eqref{eq:final-A-nu} and \eqref{eq:final-B-nu} give the prefactor
	\eqref{eq:nd-AB-normalization-prefactor}; the denominator in
	\(\nu^B_{\boldsymbol\nu,\boldsymbol\mu}\) combines with the residue factor
	\(\bigl(\boldsymbol\kappa+(b_H+K+1)\one_p\bigr)_{\boldsymbol\nu}\), giving the
	vector-Pochhammer quotient displayed directly in the summand.  Before the
	remaining cancellations, the pairing is a finite sum of terms of the form
	\[
	(\text{coefficient of }z^\beta\text{ in }B^{(j)}_{\boldsymbol\nu,\boldsymbol\mu})
	(\text{coefficient of }z^\alpha\text{ in }A^{(i)}_{\boldsymbol\lambda,\boldsymbol\eta})
	\frac{
		\Gamma((d+\alpha+\beta+1+\kappa_i)\one_r+\boldsymbol a)
	}{
		\Gamma((d+\alpha+\beta+1+\kappa_i)\one_q+\boldsymbol b+\boldsymbol e_j)
	}.
	\]
	Substitution of the explicit coefficients, followed by the elementary
	cancellations of common Pochhammer factors, gives exactly the two summands
	\eqref{eq:nd-fused-summand-diagonal} and
	\eqref{eq:nd-fused-summand-offdiagonal}.  The diagonal formula also includes
	the collapsed \(K=0\) term.  Summing over the displayed finite ranges gives
	\eqref{eq:final-pairing-block} and hence
	\eqref{eq:final-pairing-block-interpretation}.
\end{proof}

\begin{corollary}[Explicit step-line recurrence coefficients]
	\label{cor:final-step-recurrence-explicit}
	Let \(T\) be the step-line recurrence matrix of
	Proposition~\ref{prop:general-step-recurrence}.  Then, for
	\(k\in\{-p,\ldots,q\}\) and \(N+k\ge0\), its band entries are the finite
	Pochhammer sums
	\begin{equation}
		\label{eq:final-step-recurrence-coefficients}
		t_{N,k}
		=
		\mathscr P_{1,\,(\boldsymbol n_N,\boldsymbol m_N)}^{(\boldsymbol n_{N+k}^{\,*},\boldsymbol m_{N+k}^{\,*})}.
	\end{equation}
	Consequently, the entries appearing in the dual recurrence are
	\begin{equation}
		\label{eq:final-step-A-recurrence-coefficients}
		t_{N+k,-k}
		=
		\mathscr P_{1,\,(\boldsymbol n_{N+k},\boldsymbol m_{N+k})}^{(\boldsymbol n_N^{\,*},\boldsymbol m_N^{\,*})},
		\qquad
		k\in\{-q,\ldots,p\},
		\qquad N+k\ge0.
	\end{equation}
\end{corollary}

\begin{proof}
	By Proposition~\ref{prop:general-step-recurrence},
\(t_{N,k} = \left\langle z\mathbf B_N,\mathbf A_{N+k}\right\rangle.\)
	Proposition~\ref{prop:final-pairing-block-interpretation} evaluates this
	pairing as the finite sum in
	\eqref{eq:final-step-recurrence-coefficients}.  Applying the same evaluation
	to
\(t_{N+k,-k} = \left\langle z\mathbf B_{N+k},\mathbf A_N\right\rangle\)
	gives \eqref{eq:final-step-A-recurrence-coefficients}.
\end{proof}

\begin{remark}[Local recurrences and the step-line]
	The local near-diagonal coefficients below are quotients of
	\(\mathscr P_{1,\,\cdot}^{\cdot}\) and
	\(\mathscr P_{0,\,\cdot}^{\cdot}\) because the neighboring local vectors are
	not a single normalized biorthogonal sequence.  On the step-line, the
	normalization \(\langle\mathbf B_N,\mathbf A_M\rangle=\delta_{N,M}\) removes
	these denominators.  The same near-diagonal finite Pochhammer block therefore
	gives both the local coefficients and the step-line band entries.
\end{remark}

The step-line therefore provides a canonical recurrence whose coefficients are
explicitly evaluated by the finite pairing block.  The question is whether analogous
recurrences can be constructed beyond the step-line, while remaining in the
near-diagonal regime where the explicit Bessel formulas are available.

Outside the step-line there is no distinguished total ordering of the
multi-indices, and hence no canonical band range selecting the neighboring
vectors.  The problem must instead be formulated locally: starting from a
central near-diagonal index pair, the aim is to find \(p+q+1\) admissible
neighboring vectors which span the space containing the multiplied central vector.  The
following subsections construct these neighbors, determine the exact
dimension of the corresponding local recurrence spaces, and evaluate the
resulting coefficients through the same pairing block used above.

\subsection{Admissible near-diagonal chains}
\label{subsec:cyclic-balanced-indices}

The neighboring indices entering the local recurrences are constructed from
chains which preserve the near-diagonal row-index range.

For \(d\in\N\) and \(N\in\N_0\), set
\begin{equation}
	\label{eq:near-diagonal-orbit}
	\mathcal N_d(N)
	\coloneq
	\left\{
	\boldsymbol\nu\in\N_0^d:
	|\boldsymbol\nu|=N,
	|\nu_\alpha-\nu_\beta|\le1,
	\quad
	\alpha,\beta\in\{1,\ldots,d\}
	\right\}.
\end{equation}
Equivalently, \(\mathcal N_d(N)\) consists of all vectors obtained by permuting
the components of \(\boldsymbol\sigma_d(N)\).

\begin{definition}[Admissible near-diagonal chains]
	\label{def:admissible-neardiagonal-chains}
	Let \(\boldsymbol\mu\in\mathcal N_d(M)\).  An increasing admissible chain from
	\(\boldsymbol\mu\) is a sequence
\(\boldsymbol\mu^{[0]},\boldsymbol\mu^{[1]},\ldots,\boldsymbol\mu^{[r]},\) \(\boldsymbol\mu^{[0]}=\boldsymbol\mu,\)
	such that \(\boldsymbol\mu^{[k+1]}\) is obtained from \(\boldsymbol\mu^{[k]}\) by adding
	one unit to a minimal component of \(\boldsymbol\mu^{[k]}\).
	
	A decreasing admissible chain from \(\boldsymbol\mu\) is a sequence
\(\boldsymbol\mu^{[0]},\boldsymbol\mu^{[-1]},\ldots,\boldsymbol\mu^{[-r]},\) \(\boldsymbol\mu^{[0]}=\boldsymbol\mu,\)
	such that \(\boldsymbol\mu^{[-k-1]}\) is obtained from \(\boldsymbol\mu^{[-k]}\) by
	subtracting one unit from a maximal component of \(\boldsymbol\mu^{[-k]}\).
	All entries are required to remain nonnegative.
\end{definition}

\begin{proposition}[Stability of admissible chains]
	\label{prop:admissible-chain-stability}
	If \(\boldsymbol\mu\in\mathcal N_d(M)\), then every increasing admissible chain
	satisfies
\(\boldsymbol\mu^{[k]}\in\mathcal N_d(M+k),\) \(k\in\{0,\ldots,r\},\)
	and every decreasing admissible chain satisfies
\(\boldsymbol\mu^{[-k]}\in\mathcal N_d(M-k),\) \(k\in\{0,\ldots,r\}.\)
\end{proposition}

\begin{proof}
	A vector in \(\mathcal N_d(M)\) has at most two consecutive component
	values.  Adding one unit to a minimal component, or subtracting one unit
	from a maximal component, preserves the property that any two components
	differ by at most one.  The nonnegativity of decreasing chains is part of
	the definition.
\end{proof}

The row-index chains used below are chosen according to this rule and
therefore remain in the near-diagonal range by
Proposition~\ref{prop:admissible-chain-stability}.  The column-index chains
need only satisfy the stated nonnegativity and first-cycle conditions, since
the explicit Bessel formulas impose the near-diagonal restriction on the
row multi-index.

\subsection{Admissible local data for the two local recurrences}
\label{subsec:near-diagonal-row-indices}

The neighboring vectors and test vectors are recorded before stating the
local recurrences.  These definitions ensure that the explicit Bessel
formulas apply to every pairing used to calculate the coefficients.  Starting
from a central near-diagonal vector, they construct \(p+q+1\) neighboring
near-diagonal vectors which form a local basis
for the multiplied central vector.

\begin{definition}[\(A\)-side admissible local data]
	\label{def:A-side-admissible-data}
	Let \((\boldsymbol n,\boldsymbol m)\) be \(A\)-balanced, with
	\(|\boldsymbol n|=|\boldsymbol m|+1\) and
	\(\boldsymbol m\in\mathcal N_q(|\boldsymbol m|)\).
	An \(A\)-side local data set associated with \((\boldsymbol n,\boldsymbol m)\) consists
	of four sequences of cumulative shifts
\(\boldsymbol t_k,\boldsymbol{\mathfrak t}_k\in\N_0^p,\) \(\boldsymbol s_k,\boldsymbol{\mathfrak s}_k\in\N_0^q,\) \(k\in\N_0,\)
	all starting at zero and obtained by adding one standard basis vector at
	each step.  During the first \(p\) steps of \((\boldsymbol t_k)\), no component is
	used more than once, so that
	\(\boldsymbol t_i\le\one_p\) for \(i\in\{0,\ldots,p\}\).  Likewise, during the
	first \(q\) steps of \((\boldsymbol s_k)\), no component is used more than once,
	so that \(\boldsymbol s_j\le\one_q\) for \(j\in\{0,\ldots,q\}\).
	
	The row shifts are required to satisfy
	\begin{equation}
		\label{eq:A-side-neardiagonal-condition}
		\boldsymbol m+\boldsymbol{\mathfrak s}_k
		\in\mathcal N_q(|\boldsymbol m|+k),
		\quad k\in\{0,\ldots,p+1\},
		\qquad
		\boldsymbol m-\boldsymbol s_k
		\in\mathcal N_q(|\boldsymbol m|-k),
		\quad k\in\{0,\ldots,q\}.
	\end{equation}
	
	The recurrence vectors associated with these shifts are
	\[
	\begin{aligned}
		\mathbf V_i^+
		&\coloneq
		\mathbf A_{\boldsymbol n+\boldsymbol t_i,\,\boldsymbol m+\boldsymbol{\mathfrak s}_i},
		&i&\in\{1,\ldots,p\},
		\\
		\mathbf V^0
		&\coloneq
		\mathbf A_{\boldsymbol n,\boldsymbol m},
		\\
		\mathbf V_j^-
		&\coloneq
		\mathbf A_{\boldsymbol n-\boldsymbol{\mathfrak t}_j,\,\boldsymbol m-\boldsymbol s_j},
		&j&\in\{1,\ldots,q\},
	\end{aligned}
	\]
	and the corresponding test vectors are
	\[
	\begin{aligned}
		\mathbf T_i^+
		&\coloneq
		\mathbf B_{\boldsymbol n+\boldsymbol t_{i-1},\,\boldsymbol m+\boldsymbol{\mathfrak s}_{i+1}},
		&i&\in\{1,\ldots,p\},
		\\
		\mathbf T^0
		&\coloneq
		\mathbf B_{\boldsymbol n-\boldsymbol{\mathfrak t}_1,\,\boldsymbol m+\boldsymbol{\mathfrak s}_1},
		\\
		\mathbf T_j^-
		&\coloneq
		\mathbf B_{\boldsymbol n-\boldsymbol{\mathfrak t}_{j+1},\,\boldsymbol m-\boldsymbol s_{j-1}},
		&j&\in\{1,\ldots,q\}.
	\end{aligned}
	\]
	
	The \(A\)-side local data set is called admissible if every index pair
	appearing in the associated recurrence and test vectors is nonnegative componentwise 
	and lies in the range covered by
	Theorem~\ref{thm:final-explicit}.
\end{definition}

For later use, note that the row norms of the recurrence vectors range from
\(|\boldsymbol m|-q\) to \(|\boldsymbol m|+p\).  The test vectors lie in the same
near-diagonal range, except that the positive test vectors reach the additional
level \(|\boldsymbol m|+p+1\).  Thus all recurrence and test vectors have row
multi-indices in
\(\bigcup_{M=|\boldsymbol m|-q}^{|\boldsymbol m|+p+1}\mathcal N_q(M).\)
\begin{definition}[\(B\)-side admissible local data]
	\label{def:B-side-admissible-data}
	Let \((\boldsymbol n,\boldsymbol m)\) be \(B\)-balanced, with
	\(|\boldsymbol m|=|\boldsymbol n|+1\) and
	\(\boldsymbol m\in\mathcal N_q(|\boldsymbol m|)\).
	A \(B\)-side local data set associated with \((\boldsymbol n,\boldsymbol m)\) consists
	of four sequences of cumulative shifts
\(\boldsymbol u_k,\boldsymbol{\mathfrak u}_k\in\N_0^p,\) \(\boldsymbol v_k,\boldsymbol{\mathfrak v}_k\in\N_0^q,\) \(k\in\N_0,\)
	all starting at zero and obtained by adding one standard basis vector at
	each step.  During the first \(p\) steps of \((\boldsymbol u_k)\), no component is
	used more than once, so that
	\(\boldsymbol u_i\le\one_p\) for \(i\in\{0,\ldots,p\}\).  Likewise, during the
	first \(q\) steps of \((\boldsymbol v_k)\), no component is used more than once,
	so that \(\boldsymbol v_j\le\one_q\) for \(j\in\{0,\ldots,q\}\).
	
	The row shifts are required to satisfy
	\begin{equation}
		\label{eq:B-side-neardiagonal-condition}
		\boldsymbol m+\boldsymbol v_k
		\in\mathcal N_q(|\boldsymbol m|+k),
		\quad k\in\{0,\ldots,q\},
		\qquad
		\boldsymbol m-\boldsymbol{\mathfrak v}_k
		\in\mathcal N_q(|\boldsymbol m|-k),
		\quad k\in\{0,\ldots,p+1\}.
	\end{equation}
	
	The recurrence vectors associated with these shifts are
	\[
	\begin{aligned}
		\mathbf W_j^+
		&\coloneq
		\mathbf B_{\boldsymbol n+\boldsymbol{\mathfrak u}_j,\,\boldsymbol m+\boldsymbol v_j},
		&j&\in\{1,\ldots,q\},
		\\
		\mathbf W^0
		&\coloneq
		\mathbf B_{\boldsymbol n,\boldsymbol m},
		\\
		\mathbf W_i^-
		&\coloneq
		\mathbf B_{\boldsymbol n-\boldsymbol u_i,\,\boldsymbol m-\boldsymbol{\mathfrak v}_i},
		&i&\in\{1,\ldots,p\},
	\end{aligned}
	\]
	and the corresponding test vectors are
	\[
	\begin{aligned}
		\mathbf U_j^+
		&\coloneq
		\mathbf A_{\boldsymbol n+\boldsymbol{\mathfrak u}_{j+1},\,\boldsymbol m+\boldsymbol v_{j-1}},
		&j&\in\{1,\ldots,q\},
		\\
		\mathbf U^0
		&\coloneq
		\mathbf A_{\boldsymbol n+\boldsymbol{\mathfrak u}_1,\,\boldsymbol m-\boldsymbol{\mathfrak v}_1},
		\\
		\mathbf U_i^-
		&\coloneq
		\mathbf A_{\boldsymbol n-\boldsymbol u_{i-1},\,\boldsymbol m-\boldsymbol{\mathfrak v}_{i+1}},
		&i&\in\{1,\ldots,p\}.
	\end{aligned}
	\]
	
	The \(B\)-side local data set is called admissible if every index pair
	appearing in the associated recurrence and test vectors is nonnegative
	componentwise and lies in the range covered by
	Theorem~\ref{thm:final-explicit}.
\end{definition}

The only possible obstruction to constructing the admissible local data sets comes
from the decreasing shifts, which may reach the lower boundary of the
nonnegative multi-index lattice before all required steps have been performed.
The near-diagonal condition itself is preserved by adding to minimal
components and subtracting from maximal ones.

The next lemma gives sufficient conditions ensuring that none of the required
decreasing chains is cut off at the boundary.  Under these conditions, the
complete collections of recurrence and test vectors can be constructed on both
the \(A\)- and \(B\)-sides.

\begin{lemma}[Existence of admissible local data sets]
	\label{lem:existence-interior-local-windows}
	Let \((\boldsymbol n,\boldsymbol m)\) be a near-diagonal index pair covered by
	Theorem~\ref{thm:final-explicit}.
	\begin{enumerate}[label=\textnormal{(\roman*)}]
		\item If \((\boldsymbol n,\boldsymbol m)\) is \(A\)-balanced,
		\(|\boldsymbol n|=|\boldsymbol m|+1\),
		\(\boldsymbol m\in\mathcal N_q(|\boldsymbol m|)\), and
		\(\boldsymbol m-\one_q\in\N_0^q\),
		then an admissible \(A\)-side local data set in the sense of
		Definition~\ref{def:A-side-admissible-data} exists.
		
		\item If \((\boldsymbol n,\boldsymbol m)\) is \(B\)-balanced,
		\(|\boldsymbol m|=|\boldsymbol n|+1\),
		\(\boldsymbol m\in\mathcal N_q(|\boldsymbol m|)\), and
		\(\boldsymbol n-\one_p\in\N_0^p\),
		then an admissible \(B\)-side local data set in the sense of
		Definition~\ref{def:B-side-admissible-data} exists.
	\end{enumerate}
	In both cases, the admissible local data set determines \(p+q+1\)
	neighboring vectors, all lying in the near-diagonal range where
	Theorem~\ref{thm:final-explicit} applies.
\end{lemma}

\begin{proof}
	The required shift sequences are constructed on the two sides.
	
	\medskip
	\noindent
	\emph{\(A\)-side.}
	Starting from \(\boldsymbol m\), choose the increasing row chain
\(\boldsymbol m,\, \boldsymbol m+\boldsymbol{\mathfrak s}_1,\, \ldots,\, \boldsymbol m+\boldsymbol{\mathfrak s}_{p+1}\)
	by adding one unit at each step to a minimal component.  This chain may be
	continued for arbitrary length, and
	Proposition~\ref{prop:admissible-chain-stability} gives
	\(\boldsymbol m+\boldsymbol{\mathfrak s}_i\in
	\mathcal N_q(|\boldsymbol m|+i)\) for
	\(i\in\{0,\ldots,p+1\}\).
	
	Choose the decreasing row chain
\(\boldsymbol m,\, \boldsymbol m-\boldsymbol s_1,\, \ldots,\, \boldsymbol m-\boldsymbol s_q\)
	by subtracting from maximal components.  Since
	\(\boldsymbol m-\one_q\in\N_0^q\), every component of \(\boldsymbol m\) is positive, so
	the first \(q\) steps may be chosen using each component once.  Hence
	\(\boldsymbol s_j\le\one_q\), all shifted indices remain nonnegative, and
	Proposition~\ref{prop:admissible-chain-stability} gives
	\(\boldsymbol m-\boldsymbol s_j\in\mathcal N_q(|\boldsymbol m|-j)\).
	
	Choose \(\boldsymbol t_0=\boldsymbol0,\ldots,\boldsymbol t_p\) by using each of the \(p\)
	components once, so that \(\boldsymbol t_i\le\one_p\).  Finally, the sequence
	\(\boldsymbol{\mathfrak t}_0=\boldsymbol0,\ldots,\boldsymbol{\mathfrak t}_{q+1}\) can be chosen
	with \(\boldsymbol{\mathfrak t}_j\le\boldsymbol n\), because
	\(\boldsymbol m-\one_q\in\N_0^q\) implies \(|\boldsymbol m|\ge q\), and therefore
	\(|\boldsymbol n|=|\boldsymbol m|+1\ge q+1\).
	
	Thus all index pairs in
	Definition~\ref{def:A-side-admissible-data} are nonnegative and covered by
	Theorem~\ref{thm:final-explicit}.
	
	\medskip
	\noindent
	\emph{\(B\)-side.}
	Choose the increasing row chain
\(\boldsymbol m,\, \boldsymbol m+\boldsymbol v_1,\, \ldots,\, \boldsymbol m+\boldsymbol v_q\)
	by adding to minimal components and using each component once during the
	first \(q\) steps.  Then \(\boldsymbol v_j\le\one_q\), and
	Proposition~\ref{prop:admissible-chain-stability} places all these indices in
	the corresponding near-diagonal sets \(\mathcal N_q(\cdot)\).
	
	Choose the decreasing row chain
\(\boldsymbol m,\, \boldsymbol m-\boldsymbol{\mathfrak v}_1,\, \ldots,\, \boldsymbol m-\boldsymbol{\mathfrak v}_{p+1}\)
	by subtracting from maximal components.  Since
	\(\boldsymbol n-\one_p\in\N_0^p\), one has \(|\boldsymbol n|\ge p\), and hence
	\(|\boldsymbol m|=|\boldsymbol n|+1\ge p+1\).  Therefore the chain can be continued for
	\(p+1\) steps without leaving \(\N_0^q\), and its indices remain
	near-diagonal by
	Proposition~\ref{prop:admissible-chain-stability}.
	
	Choose \(\boldsymbol u_0=\boldsymbol0,\ldots,\boldsymbol u_p\) by using each component of
	\(\boldsymbol n\) once; this is possible because
	\(\boldsymbol n-\one_p\in\N_0^p\), and gives \(\boldsymbol u_i\le\one_p\).
	The increasing sequence
	\(\boldsymbol{\mathfrak u}_0=\boldsymbol0,\ldots,\boldsymbol{\mathfrak u}_{q+1}\) has no
	lower-bound obstruction.
	
	Thus all index pairs in
	Definition~\ref{def:B-side-admissible-data} are nonnegative and covered by
	Theorem~\ref{thm:final-explicit}.  This completes the construction of
	admissible local data sets on both sides.
\end{proof}

The admissibility requirement must be checked for both the recurrence and test
vectors, since the latter may reach row levels not attained by the visible
recurrence neighbors.

\subsection{Local recurrence spaces}
\label{subsec:local-recurrence-spaces}

The two linear spaces in which the multiplied central vectors
\(z\mathbf V^0\) and \(z\mathbf W^0\) will be expanded are introduced next.  The exact dimension
statement below is the reason why the local recurrences have length
\(p+q+1\).

\begin{definition}[Local recurrence spaces]
	\label{def:local-recurrence-spaces}
	For \(\boldsymbol N\in\N_0^p\) and \(\boldsymbol M\in\N_0^q\), the right and left local
	recurrence spaces are
	{\small \begin{align*}
			\mathscr R(\boldsymbol N,\boldsymbol M)
			&\coloneq
			\left\{
			\begin{aligned}
			&\mathbf P\in\mathbb C^{p\times1}[z]:
			P^{(i)}\in\Poly_{N_i-1},\\
			&\sum_{i=1}^{p}\oint_{\Torus}
			z^\ell C_{j,i}(z)P^{(i)}(z)\frac{\dz}{2\pi\mathrm i}=0,\\
			&\ell\in\{0,\ldots,M_j-1\},\quad
			j\in\{1,\ldots,q\}
			\end{aligned}
			\right\},
			\\
			\mathscr L(\boldsymbol N,\boldsymbol M)
			&\coloneq
			\left\{
			\begin{aligned}
			&\mathbf Q\in\mathbb C^{1\times q}[z]:
			Q^{(j)}\in\Poly_{M_j-1},\\
			&\sum_{j=1}^{q}\oint_{\Torus}
			z^\ell Q^{(j)}(z)C_{j,i}(z)\frac{\dz}{2\pi\mathrm i}=0,\\
			&\ell\in\{0,\ldots,N_i-1\},\quad
			i\in\{1,\ldots,p\}
			\end{aligned}
			\right\}.
	\end{align*}}
	Both are finite-dimensional linear spaces.  
\end{definition}

\begin{remark}
	Counting coefficients and moment conditions gives
\[
\dim\mathscr R(\boldsymbol N,\boldsymbol M)\ge |\boldsymbol N|-|\boldsymbol M|, \qquad \dim\mathscr L(\boldsymbol N,\boldsymbol M)\ge |\boldsymbol M|-|\boldsymbol N|,
\]
	whenever the right-hand sides are nonnegative.  Equality holds precisely
	when the corresponding moment conditions are linearly independent.  In the
	Bessel near-diagonal regime, this independence is established in the
	next proposition.
\end{remark}

\begin{proposition}[Exact dimension of the local recurrence spaces]
	\label{prop:local-recurrence-space-dimension}
	Assume that the Bessel parameters are regular in the sense of
	Definition~\ref{def:final-parameter-regularity} and admissible for the
	rational functions used below. Let \(\boldsymbol N\in\N_0^p\) and
	\(\boldsymbol M\in\N_0^q\), with \(\boldsymbol M\) near the diagonal in the sense of
	\eqref{eq:final-near-diagonal-row-index}.
	\begin{enumerate}[label=\textnormal{(\roman*)}]
		\item If \(|\boldsymbol N|\ge|\boldsymbol M|\), then \(\mathscr L(\boldsymbol N,\boldsymbol M)=\{0\}\) and
\(\dim\mathscr R(\boldsymbol N,\boldsymbol M)=|\boldsymbol N|-|\boldsymbol M|.\)
		\item If \(|\boldsymbol M|\ge|\boldsymbol N|\), then \(\mathscr R(\boldsymbol N,\boldsymbol M)=\{0\}\) and
\(\dim\mathscr L(\boldsymbol N,\boldsymbol M)=|\boldsymbol M|-|\boldsymbol N|.\)
	\end{enumerate}
	In particular, for an \(A\)-balanced pair \((\boldsymbol n,\boldsymbol m)\) with \(\boldsymbol m\)
	near the diagonal and \(\boldsymbol m-\one_q\in\N_0^q\),
\(\dim\mathscr R(\boldsymbol n+\one_p,\boldsymbol m-\one_q)=p+q+1,\)
	and for a \(B\)-balanced pair \((\boldsymbol n,\boldsymbol m)\) with \(\boldsymbol m\) near the
	diagonal and \(\boldsymbol n-\one_p\in\N_0^p\),
\(\dim\mathscr L(\boldsymbol n-\one_p,\boldsymbol m+\one_q)=p+q+1.\)
\end{proposition}

\begin{proof}
	First, the two vanishing statements are proved; the dimension formulas then
	follow by a duality of the defining conditions, explained at the end.
	
 \smallskip
 \noindent\emph{Vanishing in \textnormal{(i)}.}
 Let \(\mathbf Q\in\mathscr L(\boldsymbol N,\boldsymbol M)\), with components
 \(Q^{(j)}(z)=\sum_{d=0}^{M_j-1}q_{j,d}z^d\).  Associate with
 \(\mathbf Q\) the polynomial
 \begin{equation}
 	\label{eq:local-B-polynomial}
 	L_{\mathbf Q}(t)
 	\coloneq
 	\sum_{j=1}^{q}
 	\sum_{d=0}^{M_j-1}
 	q_{j,d}
 	(t\one_r+\boldsymbol a)_d
 	\bigl(t\one_q+\boldsymbol b+d\one_q+\boldsymbol e_j\bigr)_{
 		\boldsymbol M-d\one_q-\boldsymbol e_j}.
 \end{equation}
 This is the polynomial of \eqref{eq:general-B-polynomial}, with
 \((\boldsymbol N,\boldsymbol M)\) in place of \((\boldsymbol n,\boldsymbol m)\).
 Near-diagonality of \(\boldsymbol M\) ensures that all multi-indices occurring in
 \eqref{eq:local-B-polynomial} are nonnegative.  The summand corresponding to
 \(j\) and \(d\) has degree
 \(|\boldsymbol M|-1-(q-r)d\), and therefore
 \(\deg L_{\mathbf Q}\le|\boldsymbol M|-1\).
 
 The moment identity \eqref{eq:general-B-moment-identity}, with
 \((\boldsymbol N,\boldsymbol M)\) in place of \((\boldsymbol n,\boldsymbol m)\), shows that membership in
 \(\mathscr L(\boldsymbol N,\boldsymbol M)\) implies
\(L_{\mathbf Q}(\kappa_i+1+\ell)=0,\) \(\ell\in\{0,\ldots,N_i-1\},\) \(i\in\{1,\ldots,p\}.\)
 There are \(\sum_{i=1}^{p}N_i=|\boldsymbol N|\) such nodes.  They are pairwise
 distinct: for fixed \(i\), distinct values of \(\ell\) give distinct nodes,
 whereas an equality
 \(\kappa_i+1+\ell=\kappa_{i'}+1+\ell'\) with \(i\ne i'\) would imply
 \(\kappa_i-\kappa_{i'}=\ell'-\ell\in\mathbb Z\), contrary to the
 noninteger-separation assumption.
 
 Since \(|\boldsymbol N|\ge|\boldsymbol M|\), one has
 \(\deg L_{\mathbf Q}\le|\boldsymbol M|-1\le|\boldsymbol N|-1\).  Thus
 \(L_{\mathbf Q}\) has more distinct zeros than its degree, and therefore
 \(L_{\mathbf Q}\equiv0\).  Finally, the injectivity of the map
 \(\mathbf Q\mapsto L_{\mathbf Q}\), established through
 \eqref{eq:B-rational-expansion-injective} in the proof of
 Proposition~\ref{prop:finite-rational-characterization}, gives
 \(\mathbf Q=0\).
	
	\smallskip
	\noindent\emph{Vanishing in \textnormal{(ii)}.}
	Let \(\mathbf P\in\mathscr R(\boldsymbol N,\boldsymbol M)\), with components
	\(P^{(i)}(z)=\sum_{k=0}^{N_i-1}p_{i,k}z^k\).  Associate with
	\(\mathbf P\) the interpolation polynomial
	\begin{equation}
		\label{eq:local-A-interpolation-polynomial}
		Q_{\mathbf P}(t)
		\coloneq
		\sum_{i=1}^{p}
		\sum_{k=0}^{N_i-1}
		p_{i,k}
		\frac{
			\prod\limits_{\substack{h=1\\ h\ne i}}^{p}
			(t-\kappa_h)_{N_h}
			(t-\kappa_i)_k
			(t-\kappa_i-k-1)_{N_i-k-1}
		}{
			\prod\limits_{\substack{h=1\\ h\ne i}}^{p}
			(\kappa_i+k-\kappa_h)_{N_h}
			k!\,(-1)^{N_i-k-1}(N_i-k-1)!
		}.
	\end{equation}
	This is the interpolation polynomial of
	\eqref{eq:general-A-interpolation}, with \(\boldsymbol N\) in place of
	\(\boldsymbol n\).  It satisfies \(\deg Q_{\mathbf P}\le|\boldsymbol N|-1\), and the
	residues of the rational-Gamma function
	\eqref{eq:general-A-rational-gamma} at the nodes \(\kappa_i+k\) recover the
	coefficients \(p_{i,k}\).
	
	The tested rational function
	\eqref{eq:general-A-tested-rational-expanded} is
	\(\mathrm{O}(t^{-2})\) at infinity, because its denominator degree exceeds
	its numerator degree by at least \(2+(q-r)\ell\ge2\).  Hence the same
	induction used in
	Proposition~\ref{prop:finite-rational-characterization} applies without any
	balancedness assumption.  It uses only the decay at infinity, the residue
	theorem, regularity of the parameter strings, and near-diagonality of
	\(\boldsymbol M\).  Membership in \(\mathscr R(\boldsymbol N,\boldsymbol M)\) therefore implies
\(Q_{\mathbf P}(-b_j-s-1)=0,\) \(s\in\{0,\ldots,M_j-1\},\) \(j\in\{1,\ldots,q\}.\)
	There are \(\sum_{j=1}^{q}M_j=|\boldsymbol M|\) such points.  They are pairwise
	distinct: for fixed \(j\), distinct values of \(s\) give distinct points,
	whereas an equality
	\(-b_j-s-1=-b_{j'}-s'-1\) with \(j\ne j'\) would imply
	\(b_j-b_{j'}=s'-s\in\mathbb Z\), contrary to the noninteger-separation
	assumption.
	
	Since \(|\boldsymbol M|\ge|\boldsymbol N|\), one has
	\(\deg Q_{\mathbf P}\le|\boldsymbol N|-1\le|\boldsymbol M|-1\).  Thus
	\(Q_{\mathbf P}\) has more distinct zeros than its degree, and therefore
	\(Q_{\mathbf P}\equiv0\).  Finally, every coefficient \(p_{i,k}\) is recovered
	as a residue of \eqref{eq:general-A-rational-gamma}; hence all \(p_{i,k}\)
	vanish and \(\mathbf P=0\).
	
	\smallskip
	\noindent\emph{Dimension formulas.}
	Consider the \(|\boldsymbol M|\) linear functionals defining
	\(\mathscr R(\boldsymbol N,\boldsymbol M)\) on the component space
	\(\bigoplus_{i=1}^{p}\Poly_{N_i-1}\), namely
	\[
	\varphi_{j,\ell}(\mathbf P)
	=
	\sum_{i=1}^{p}
	\oint_{\Torus}
	z^{\ell}C_{j,i}(z)P^{(i)}(z)\frac{\dz}{2\pi\mathrm i},
	\qquad
	\ell\in\{0,\ldots,M_j-1\},
	\quad
	j\in\{1,\ldots,q\}.
	\]
	A vanishing linear combination
	\(\sum_{j,\ell}c_{j,\ell}\varphi_{j,\ell}=0\), tested on the column vectors
	whose only nonzero component is \(P^{(i)}=z^{\alpha}\) with
	\(\alpha\in\{0,\ldots,N_i-1\}\), states precisely that the row vector with
	components \(\sum_{\ell=0}^{M_j-1}c_{j,\ell}z^{\ell}\) belongs to
	\(\mathscr L(\boldsymbol N,\boldsymbol M)\).  Under \textnormal{(i)} this space is trivial,
	so the functionals are linearly independent and
	\(\dim\mathscr R(\boldsymbol N,\boldsymbol M)=|\boldsymbol N|-|\boldsymbol M|\).  The dual argument under
	\textnormal{(ii)} gives
	\(\dim\mathscr L(\boldsymbol N,\boldsymbol M)=|\boldsymbol M|-|\boldsymbol N|\).
	
	The two particular statements follow with
	\((\boldsymbol N,\boldsymbol M)=(\boldsymbol n+\one_p,\boldsymbol m-\one_q)\) and
	\((\boldsymbol N,\boldsymbol M)=(\boldsymbol n-\one_p,\boldsymbol m+\one_q)\): in both cases the row
	multi-index remains near the diagonal, since shifting all components by the
	same unit does not change their differences, and the balancedness relations
	give \(|\boldsymbol N|-|\boldsymbol M|=p+q+1\) and \(|\boldsymbol M|-|\boldsymbol N|=p+q+1\),
	respectively.
\end{proof}

\subsection{Local near-diagonal recurrences of length \texorpdfstring{\(p+q+1\)}{p+q+1}}
\label{sec:final-general-recurrence}

Combining the existence of admissible local data sets with the exact
dimension of the corresponding recurrence spaces gives a local
recurrence involving \(p+q+1\) vectors associated with the chosen
near-diagonal index pair.  This extends the canonical step-line recurrence
within the near-diagonal regime: the recurrence vectors are not selected by a
global ordering, but are constructed from admissible local chains as in
Lemma~\ref{lem:existence-interior-local-windows}.

\begin{theorem}[Local \(A\)-recurrence of length \(p+q+1\)]
	\label{thm:final-general-recurrence}
	Assume that the Bessel parameters are regular in the sense of
	Definition~\ref{def:final-parameter-regularity}. Let \((\boldsymbol n,\boldsymbol m)\) be
	an \(A\)-balanced index pair covered by
	Theorem~\ref{thm:final-explicit}, with \(|\boldsymbol n|=|\boldsymbol m|+1\),
	\(\boldsymbol m\in\mathcal N_q(|\boldsymbol m|)\), and
	\(\boldsymbol m-\one_q\in\N_0^q\).  Choose admissible \(A\)-side local data supplied by
	Lemma~\ref{lem:existence-interior-local-windows}, and assume that the
	denominators displayed below do not vanish.  Then these data determine
	\(p+q+1\) near-diagonal vectors
\(\mathbf V_1^+,\ldots,\mathbf V_p^+,\mathbf V^0, \mathbf V_1^-,\ldots,\mathbf V_q^-,\)
	and multiplication by \(z\) of the central vector \(\mathbf V^0\) expands in
	this local basis as
	\begin{equation}
		\label{eq:final-general-A-recurrence}
		z\mathbf V^0(z)
		=
		\sum_{i=1}^{p}
		a_{\boldsymbol n,\boldsymbol m}^{+i}\mathbf V_i^+(z)
		+
		a_{\boldsymbol n,\boldsymbol m}^{0}\mathbf V^0(z)
		+
		\sum_{j=1}^{q}
		a_{\boldsymbol n,\boldsymbol m}^{-j}\mathbf V_j^-(z).
	\end{equation}
	For \(i\in\{1,\ldots,p\}\) and \(j\in\{1,\ldots,q\}\), the coefficients are
	the explicit finite Pochhammer quotients
	\begin{align}
		\label{eq:final-general-recurrence-coefficients}
		a_{\boldsymbol n,\boldsymbol m}^{+i}
		&=
		\frac{
			\mathscr P_{1,\,(\boldsymbol n+\boldsymbol t_{i-1},\boldsymbol m+\boldsymbol{\mathfrak s}_{i+1})}^{(\boldsymbol n,\boldsymbol m)}
		}{
			\mathscr P_{0,\,(\boldsymbol n+\boldsymbol t_{i-1},\boldsymbol m+\boldsymbol{\mathfrak s}_{i+1})}^{(\boldsymbol n+\boldsymbol t_i,\boldsymbol m+\boldsymbol{\mathfrak s}_i)}
		},
		&
		a_{\boldsymbol n,\boldsymbol m}^{0}
		&=
		\frac{
			\mathscr P_{1,\,(\boldsymbol n-\boldsymbol{\mathfrak t}_1,\boldsymbol m+\boldsymbol{\mathfrak s}_1)}^{(\boldsymbol n,\boldsymbol m)}
		}{
			\mathscr P_{0,\,(\boldsymbol n-\boldsymbol{\mathfrak t}_1,\boldsymbol m+\boldsymbol{\mathfrak s}_1)}^{(\boldsymbol n,\boldsymbol m)}
		},
		&
		a_{\boldsymbol n,\boldsymbol m}^{-j}
		&=
		\frac{
			\mathscr P_{1,\,(\boldsymbol n-\boldsymbol{\mathfrak t}_{j+1},\boldsymbol m-\boldsymbol s_{j-1})}^{(\boldsymbol n,\boldsymbol m)}
		}{
			\mathscr P_{0,\,(\boldsymbol n-\boldsymbol{\mathfrak t}_{j+1},\boldsymbol m-\boldsymbol s_{j-1})}^{(\boldsymbol n-\boldsymbol{\mathfrak t}_j,\boldsymbol m-\boldsymbol s_j)}
		}.
	\end{align}
\end{theorem}

\begin{proof}
	Multiplication by \(z\) sends \(\mathbf V^0=\mathbf A_{\boldsymbol n,\boldsymbol m}\) into
	\(\mathscr R(\boldsymbol n+\one_p,\boldsymbol m-\one_q)\) of
	Definition~\ref{def:local-recurrence-spaces}: the component degrees increase by
	one, and the row moments with exponents \(0,\ldots,m_j-2\) become the imposed
	moments with exponents \(1,\ldots,m_j-1\).  The vectors
	\(\mathbf V_i^+\), \(\mathbf V^0\), and \(\mathbf V_j^-\) also belong to this
	space.  By Proposition~\ref{prop:local-recurrence-space-dimension}, the
	space has dimension \(p+q+1\), so it is enough to prove that they are
	independent.
	
	Pair the ordered list
\[
\mathbf V_1^+,\ldots,\mathbf V_p^+,\mathbf V^0, \mathbf V_1^-,\ldots,\mathbf V_q^-
\]
	with the ordered test list
\[
\mathbf T_1^+,\ldots,\mathbf T_p^+,\mathbf T^0, \mathbf T_1^-,\ldots,\mathbf T_q^-.
\]
	For arbitrary balanced indices, the mixed orthogonality relations give
	\[
	\left\langle
	\mathbf B_{\boldsymbol\nu,\boldsymbol\mu},
	\mathbf A_{\boldsymbol\lambda,\boldsymbol\eta}
	\right\rangle=0
	\qquad\text{whenever}\qquad
	\boldsymbol\mu\le\boldsymbol\eta
	\quad\text{or}\quad
	\boldsymbol\lambda\le\boldsymbol\nu.
	\]
	Indeed, the first componentwise inequality allows one to apply the
	orthogonality of the \(A\)-vector after expanding the components of the
	\(B\)-vector, whereas the second allows one to apply the orthogonality of the
	\(B\)-vector after expanding the components of the \(A\)-vector.

	The cumulative-chain ordering now makes the pairing matrix diagonal. The
	test \(\mathbf T_i^+\) annihilates \(\mathbf V_\ell^+\) for \(\ell<i\)
	through the second inequality and for \(\ell>i\) through the first; it also
	annihilates \(\mathbf V^0\) and all the negative vectors through the second
	inequality. The test \(\mathbf T^0\) annihilates the positive vectors through
	the first inequality and the negative vectors through the second. Finally,
	\(\mathbf T_j^-\) annihilates \(\mathbf V_\ell^-\) for \(\ell<j\) through the
	first inequality and for \(\ell>j\) through the second, while it annihilates
	\(\mathbf V^0\) and all the positive vectors through the first. Thus each test
	vector kills all recurrence vectors except the one with the same position in
	the list. The nonzero diagonal entries are the denominators in
	\eqref{eq:final-general-recurrence-coefficients}.  Therefore the recurrence
	vectors form a basis, and the expansion \eqref{eq:final-general-A-recurrence}
	exists and is unique.
	
	Pairing the expansion with the corresponding test vector isolates the
	coefficient as a quotient of contour pairings.  Since the pairing is
	bilinear, \(\langle \mathbf T,z\mathbf V^0\rangle
	=\langle z\mathbf T,\mathbf V^0\rangle\).
	Proposition~\ref{prop:final-pairing-block-interpretation} evaluates the
	resulting pairings as the blocks \(\mathscr P_{1,\,\cdot}^{\cdot}\) and
	\(\mathscr P_{0,\,\cdot}^{\cdot}\), giving
	\eqref{eq:final-general-recurrence-coefficients}.
\end{proof}

\begin{theorem}[Local \(B\)-recurrence of length \(p+q+1\)]
	\label{thm:final-general-B-recurrence}
	Assume that the Bessel parameters are regular in the sense of
	Definition~\ref{def:final-parameter-regularity}. Let \((\boldsymbol n,\boldsymbol m)\) be
	a \(B\)-balanced index pair covered by
	Theorem~\ref{thm:final-explicit}, with \(|\boldsymbol m|=|\boldsymbol n|+1\),
	\(\boldsymbol m\in\mathcal N_q(|\boldsymbol m|)\), and
	\(\boldsymbol n-\one_p\in\N_0^p\).  Choose admissible \(B\)-side local data supplied by
	Lemma~\ref{lem:existence-interior-local-windows}, and assume that the
	denominators displayed below do not vanish.  Then these data determine
	\(p+q+1\) near-diagonal vectors
\(\mathbf W_1^+,\ldots,\mathbf W_q^+,\mathbf W^0, \mathbf W_1^-,\ldots,\mathbf W_p^-,\)
	and multiplication by \(z\) of the central vector \(\mathbf W^0\) expands in
	this local basis as
	\begin{equation}
		\label{eq:final-general-B-recurrence}
		z\mathbf W^0(z)
		=
		\sum_{j=1}^{q}
		c_{\boldsymbol n,\boldsymbol m}^{B,+j}\mathbf W_j^+(z)
		+
		c_{\boldsymbol n,\boldsymbol m}^{B,0}\mathbf W^0(z)
		+
		\sum_{i=1}^{p}
		c_{\boldsymbol n,\boldsymbol m}^{B,-i}\mathbf W_i^-(z).
	\end{equation}
	For \(i\in\{1,\ldots,p\}\) and \(j\in\{1,\ldots,q\}\), the coefficients are
	the explicit finite Pochhammer quotients
	\begin{align}
		\label{eq:final-general-B-recurrence-coefficients}
		c_{\boldsymbol n,\boldsymbol m}^{B,+j}
		&=
		\frac{
			\mathscr P_{1,\,(\boldsymbol n,\boldsymbol m)}^{(\boldsymbol n+\boldsymbol{\mathfrak u}_{j+1},\boldsymbol m+\boldsymbol v_{j-1})}
		}{
			\mathscr P_{0,\,(\boldsymbol n+\boldsymbol{\mathfrak u}_j,\boldsymbol m+\boldsymbol v_j)}^{(\boldsymbol n+\boldsymbol{\mathfrak u}_{j+1},\boldsymbol m+\boldsymbol v_{j-1})}
		},
		&
		c_{\boldsymbol n,\boldsymbol m}^{B,0}
		&=
		\frac{
			\mathscr P_{1,\,(\boldsymbol n,\boldsymbol m)}^{(\boldsymbol n+\boldsymbol{\mathfrak u}_1,\boldsymbol m-\boldsymbol{\mathfrak v}_1)}
		}{
			\mathscr P_{0,\,(\boldsymbol n,\boldsymbol m)}^{(\boldsymbol n+\boldsymbol{\mathfrak u}_1,\boldsymbol m-\boldsymbol{\mathfrak v}_1)}
		},
		&
		c_{\boldsymbol n,\boldsymbol m}^{B,-i}
		&=
		\frac{
			\mathscr P_{1,\,(\boldsymbol n,\boldsymbol m)}^{(\boldsymbol n-\boldsymbol u_{i-1},\boldsymbol m-\boldsymbol{\mathfrak v}_{i+1})}
		}{
			\mathscr P_{0,\,(\boldsymbol n-\boldsymbol u_i,\boldsymbol m-\boldsymbol{\mathfrak v}_i)}^{(\boldsymbol n-\boldsymbol u_{i-1},\boldsymbol m-\boldsymbol{\mathfrak v}_{i+1})}
		}.
	\end{align}
\end{theorem}

\begin{proof}
	The proof is the left-right dual of the preceding one.  Multiplication by
	\(z\) sends \(\mathbf W^0=\mathbf B_{\boldsymbol n,\boldsymbol m}\) into
	\(\mathscr L(\boldsymbol n-\one_p,\boldsymbol m+\one_q)\) of
	Definition~\ref{def:local-recurrence-spaces}: the component degrees increase
	by one, and the column moments with exponents \(0,\ldots,n_i-2\) become the
	imposed moments with exponents \(1,\ldots,n_i-1\).  The vectors
	\(\mathbf W_j^+\), \(\mathbf W^0\), and \(\mathbf W_i^-\) belong to that
	space, which has dimension \(p+q+1\) by
	Proposition~\ref{prop:local-recurrence-space-dimension}. Pair them with the
	tests \(\mathbf U_j^+\), \(\mathbf U^0\), and \(\mathbf U_i^-\). The vanishing
	criterion established in the proof of
	Theorem~\ref{thm:final-general-recurrence} makes the resulting pairing matrix
	diagonal. More explicitly, \(\mathbf U_j^+\) annihilates
	\(\mathbf W_\ell^+\) for \(\ell<j\) through the row-index inequality and for
	\(\ell>j\) through the column-index inequality; it also annihilates
	\(\mathbf W^0\) and all the negative vectors through the row-index inequality.
	The test \(\mathbf U^0\) annihilates the positive vectors through the
	column-index inequality and the negative vectors through the row-index
	inequality. Finally, \(\mathbf U_i^-\) annihilates
	\(\mathbf W_\ell^-\) for \(\ell<i\) through the column-index inequality and
	for \(\ell>i\) through the row-index inequality, while it annihilates
	\(\mathbf W^0\) and all the positive vectors through the column-index
	inequality. Hence each test pairs only with its matching recurrence vector.
	The diagonal entries are the denominators in
	\eqref{eq:final-general-B-recurrence-coefficients}; by hypothesis they do not
	vanish.  Hence the recurrence vectors form a basis.  Pairing the expansion
	with the corresponding test vector and using
	Proposition~\ref{prop:final-pairing-block-interpretation} gives the displayed
	quotients in terms of the blocks \(\mathscr P_{1,\,\cdot}^{\cdot}\) and
	\(\mathscr P_{0,\,\cdot}^{\cdot}\).
\end{proof}

\begin{remark}[Boundary local recurrences]
	The interior assumptions in Theorems~\ref{thm:final-general-recurrence} and
	\ref{thm:final-general-B-recurrence} are imposed only to guarantee the full
	\(p+q+1\)-term windows used in the statements.  At the boundary, the moment
	conditions satisfied by the multiplied central vector are obtained by replacing
	\(\boldsymbol m-\one_q\) with \((\boldsymbol m-\one_q)_+\) on the \(A\)-side, or
	\(\boldsymbol n-\one_p\) with \((\boldsymbol n-\one_p)_+\) on the \(B\)-side, where the
	positive part is taken componentwise.  Proposition~\ref{prop:local-recurrence-space-dimension}
	then gives the dimension of the corresponding truncated local space. If one can
	select exactly that number of admissible neighbors, together with matching test
	vectors whose diagonal pairings are nonzero, the same diagonal-pairing argument
	yields a unique shorter boundary recurrence. Its coefficients are given by the
	same quotients of \(\mathscr P_{1,\,\cdot}^{\cdot}\) and
	\(\mathscr P_{0,\,\cdot}^{\cdot}\). Since a systematic selection of boundary
	neighbors is not needed below, those recurrences are not recorded explicitly.
\end{remark}

\section{Bidiagonal factorizations}
\label{sec:final-bidiagonal-factorizations}

This section applies the mixed step-line Christoffel--Gauss--Borel
factorization scheme of
\cite{BranquinhoFoulquieManas2026} to the
Bessel family.  The step-line recurrence matrix constructed in
Section~\ref{sec:final-recurrences} is \((p,q)\)-banded.  The purpose of this
section is to factor its finite principal truncations into elementary
bidiagonal Christoffel factors.

The Christoffel factors used below are not introduced as new orthogonal
systems.  They are obtained by multiplying the matrix of measures on the
column side or on the row side by the cyclic companion matrices
\(\mathfrak X_{[p]}(z)\) and \(\mathfrak X_{[q]}(z)\).  The lower bidiagonal
factors correspond to successive column Christoffel steps, whereas the upper
bidiagonal factors correspond to successive row Christoffel steps.

If a Christoffel perturbation remains inside the same explicit family, its
bidiagonal factor can be evaluated from the transformed polynomial vectors.
If it does not, the finite tau-determinants supplied by the general
Christoffel theory are used.

In the present \(q\times p\) Bessel system the column Christoffel chain is
closed at the level of bimoments, as proved below in
Lemma~\ref{lem:column-Christoffel-chain-closure}: after each elementary column
transformation one obtains again a Bessel moment matrix with shifted
column parameters.  Hence the lower bidiagonal factors are evaluated from the
transformed \(A\)-polynomial vectors.
The intermediate row Christoffel perturbations are not, in general, members of
the original Bessel family; the corresponding upper factors are obtained from the
finite \(\tau^B\)-determinants.  This is the point at which the mixed
\(q\times p\) case differs from the one-row multiple Bessel reduction: the
column transformations remain inside the Bessel class, whereas the
intermediate row transformations do not.  The one-row case \(q=1\) is
exceptional: the unique row step is the same as one full column-Christoffel
cycle, and the multiple Bessel family is invariant.  Therefore the complete
factorization in that case is written only in terms of quotients of
Christoffel-transformed monic type-II polynomials evaluated at the origin.

\subsection{Bessel data for the Christoffel formulas}
\label{sec:final-bidiagonal-bessel-data}

First, the explicit coefficients of the step-line vectors are recorded.  For $d\in\N_0$, the
coefficient of \(z^d\) in the \(B\)-component \eqref{eq:final-B-components},
with the normalization constant \(\nu^B_{\boldsymbol n,\boldsymbol m}\) defined in
\eqref{eq:final-B-nu}, is
\begin{multline}
	\label{eq:explicit-B-coefficients}
	\mathsf b_j(\boldsymbol n,\boldsymbol m;d\mid\boldsymbol\kappa)
	=
	\nu^B_{\boldsymbol n,\boldsymbol m}\,\frac{(\boldsymbol a-b_j\one_r)_1}{(\boldsymbol b^{\,*j}-b_j\one_{q-1})_1}
	\sum_{H=1}^{q}
	\sum_{\substack{K=0\\ d\le K-1+\delta_{j,H}}}^{m_H-1}
	\pi_{H,K}
	\frac{(\boldsymbol b^{\,*j}-(b_H+K)\one_{q-1})_1}
	{(\boldsymbol a-(b_H+K)\one_r)_1}
	\\
	\times
	\frac{
		(\boldsymbol b-(b_H+K)\one_q+\one_q-\boldsymbol e_j)_d
	}{
		(\boldsymbol a-(b_H+K)\one_r+\one_r)_d
	}.
\end{multline}
The dependence on \(\boldsymbol\kappa\) is through the partial-fraction data
\(\pi_{H,K}\) in \eqref{eq:final-B-pi} and the normalization constant in
\eqref{eq:final-B-nu}.  Similarly, for $d\in\N_0$, the coefficient of \(z^d\) in the
\(A\)-component \eqref{eq:final-A-components}, with \(\nu^A_{\boldsymbol n,\boldsymbol m}\) from \eqref{eq:final-A-nu}, is
\begin{equation}
	\label{eq:explicit-A-coefficients}
	\begin{aligned}
	\mathsf a_i(\boldsymbol n,\boldsymbol m;d\mid\boldsymbol\kappa)
	&=
	\nu^A_{\boldsymbol n,\boldsymbol m}\,C_{\boldsymbol n,\boldsymbol m;i}
	\frac{(-n_i+1)_d}{d!}
	\frac{
		(\kappa_i\one_q+\boldsymbol b+\boldsymbol m+\one_q)_d
	}{
		(\kappa_i\one_r+\boldsymbol a+\one_r)_d
	}\\
	&\quad\times
	\frac{
		(\kappa_i\one_{p-1}-\boldsymbol\kappa^{\,*i}-\boldsymbol n^{\,*i}+
		\one_{p-1})_d
	}{
		(\kappa_i\one_{p-1}-\boldsymbol\kappa^{\,*i}+\one_{p-1})_d
	}.
	\end{aligned}
\end{equation}
These coefficients give the values and leading coefficients of the normalized
solutions of Theorem~\ref{thm:final-explicit}.  The constants used in their
normalization are those in \eqref{eq:final-A-nu} and \eqref{eq:final-B-nu}; no
additional step-line rescaling is applied in the recurrence, the lower factors,
or the Christoffel tau-determinants.

Let
\(\mathfrak X_{[1]}(z)\coloneq z,\)
and, for \(d\ge2\), let
\[
\mathfrak X_{[d]}(z)
\coloneq
\begin{bNiceMatrix}
	0&1&0&\Cdots&0\\
	0&0&1&\Ddots&\Vdots\\
	\Vdots&\Ddots[shorten-end=-2pt]&\Ddots[shorten-end=2pt]&\Ddots[shorten-end=-2pt]&0\\[8pt]
	0&\Cdots&0&0&1\\
	z&0&\Cdots&0&0
\end{bNiceMatrix},
\qquad
\mathfrak X_{[d]}(z)^d=zI_d.
\]
The elementary column and row Christoffel perturbations are
\[
\mathrm d\boldsymbol\mu_{\mathrm L}^{(a)}
\coloneq
\mathrm d\boldsymbol\mu
\left(\mathfrak X_{[p]}(z)^a\right)^{\top},
\qquad
\mathrm d\boldsymbol\mu_{\mathrm R}^{(b)}
\coloneq
\mathfrak X_{[q]}(z)^b\mathrm d\boldsymbol\mu.
\]
With this convention the column cycle used below is the cyclic Christoffel
cycle compatible with the ordering of the step-line recurrence matrix.  If
\(a\in\{0,\ldots,p\}\), the column parameters of
\(\mathrm d\boldsymbol\mu_{\mathrm L}^{(a)}\) are
\(\boldsymbol\kappa^{\langle a\rangle}\), where
\[
\boldsymbol\kappa^{\langle0\rangle}
\coloneq
\begin{bNiceMatrix}\kappa_1&\Cdots&\kappa_p\end{bNiceMatrix},
\]
and, for \(a\in\{1,\ldots,p\}\),
\[
\boldsymbol\kappa^{\langle a\rangle}
\coloneq
\begin{bNiceMatrix}
\kappa_{a+1}&\Cdots&\kappa_p&
\kappa_1+1&\Cdots&\kappa_a+1
\end{bNiceMatrix}.
\]
Thus one elementary column step sends
\[
\begin{bNiceMatrix}\kappa_1&\Cdots&\kappa_p\end{bNiceMatrix}
\mapsto
\begin{bNiceMatrix}
\kappa_2&\Cdots&\kappa_p&\kappa_1+1
\end{bNiceMatrix},
\]
and one full cycle gives
\(\boldsymbol\kappa^{\langle p\rangle}
=\begin{bNiceMatrix}\kappa_1+1&\Cdots&\kappa_p+1\end{bNiceMatrix}\).
This is the ordering used in the leading-coefficient quotients for the lower
factors.

\begin{lemma}[Closure of the column Christoffel chain]
\label{lem:column-Christoffel-chain-closure}
Let \(\MB^{\langle a\rangle}\) denote the Bessel bimoment array obtained
from \eqref{eq:final-moment-entry} by replacing \(\boldsymbol\kappa\) with
\(\boldsymbol\kappa^{\langle a\rangle}\).  Then, for every
\(a\in\{0,\ldots,p\}\),
\begin{equation}
\label{eq:column-Christoffel-chain-closure}
\oint_{\Torus}
z^{u+v}
\left(
\CB(z)\bigl(\mathfrak X_{[p]}(z)^a\bigr)^{\top}
\right)_{j,i}
\frac{\dz}{2\pi\mathrm i}
=
\MB^{\langle a\rangle}_{(u,j),(v,i)},
\qquad
u,v\in\N_0 .
\end{equation}
In particular, the elementary column Christoffel chain stays inside the
Bessel moment family, with the cyclic parameter shift
\[
\boldsymbol\kappa
\mapsto
\boldsymbol\kappa^{\langle1\rangle}
=
\begin{bNiceMatrix}
	\kappa_2&\Cdots&\kappa_p&\kappa_1+1
\end{bNiceMatrix},
\]
and one complete cycle gives
\(\boldsymbol\kappa^{\langle p\rangle} = \boldsymbol\kappa+\one_p .\)
\end{lemma}

\begin{proof}
It is enough to prove the assertion for one elementary column step, since the
general statement follows by iteration.  From the definition of
\(\mathfrak X_{[p]}(z)\), multiplication on the right by
\(\mathfrak X_{[p]}(z)^{\top}\) sends the columns of \(\CB\) to
\(\mathbf C_2,\ldots,\mathbf C_p,z\mathbf C_1.\)
Therefore, for \(i\in\{1,\ldots,p-1\}\), formula
\eqref{eq:final-moment-entry} gives
\[
\oint_{\Torus}
z^{u+v}
C_{j,i+1}(z)
\frac{\dz}{2\pi\mathrm i}
=
\frac{
\Gamma((u+v+1+\kappa_{i+1})\one_r+\boldsymbol a)
}{
\Gamma((u+v+1+\kappa_{i+1})\one_q+\boldsymbol b+\boldsymbol e_j)
},
\]
which is precisely the Bessel moment with the new \(i\)-th column
parameter \(\kappa_{i+1}\).  For the last column one obtains
\[
\oint_{\Torus}
z^{u+v}
zC_{j,1}(z)
\frac{\dz}{2\pi\mathrm i}
=
\frac{
\Gamma((u+v+2+\kappa_1)\one_r+\boldsymbol a)
}{
\Gamma((u+v+2+\kappa_1)\one_q+\boldsymbol b+\boldsymbol e_j)
},
\]
which is the Bessel moment with last column parameter \(\kappa_1+1\).
Thus the first column Christoffel step replaces
\[
\begin{bNiceMatrix}
	\kappa_1&\Cdots&\kappa_p
\end{bNiceMatrix},
\quad\text{by}\quad
\begin{bNiceMatrix}
	\kappa_2&\Cdots&\kappa_p&\kappa_1+1
\end{bNiceMatrix}.
\]
Iterating this identity gives \eqref{eq:column-Christoffel-chain-closure} for
all \(a\in\{0,\ldots,p\}\), and after \(p\) steps every parameter has been
shifted by one.
\end{proof}

For the unperturbed step-line, the canonical near-diagonal normalization gives
\begin{equation}
	\label{eq:Bessel-B-active-leading}
	\LC\left(\mathbf B_N^{(r_N)}\right)
	=
	\mathsf b_{r_N}
	\left(
	\boldsymbol n_N,\boldsymbol m_N;m_{N,r_N}-1\mid\boldsymbol\kappa
	\right)
	=1.
\end{equation}
Thus the \(B\)-side is monic in the row selected by the step-line, and the diagonal normalization
is already carried by the dual \(A\)-vectors.  No additional Gauss--Borel
superscript or later rescaling is introduced: every vector used below is the
near-diagonal vector in the compatible normalization.  The same scalar
convention is used after each Christoffel transformation, so that a transformed
step-line vector is obtained either by specializing the near-diagonal formula
with shifted parameters or by applying the Christoffel factorization.

By Lemma~\ref{lem:column-Christoffel-chain-closure}, the column chain is
closed in the present family at the level of bimoments.  Thus the transformed
step-line vectors associated with \(\mathrm d\boldsymbol\mu_{\mathrm L}^{(a)}\)
are obtained from Theorem~\ref{thm:final-explicit} by replacing
\(\boldsymbol\kappa\) with the cyclic shift
\(\boldsymbol\kappa^{\langle a\rangle}\) in the compatible normalization.  The
lower factors are therefore read directly from the leading coefficients of
these transformed step-line \(A\)-polynomials, as in the general Christoffel
formulas of \cite{BranquinhoFoulquieManas2026}.

The row chain is not closed for intermediate values of \(b\) in the full
\(q\times p\) system.  Indeed, an intermediate row step multiplies only the
rows that have completed the cyclic passage by \(z\), producing nonuniform
moment shifts that cannot be encoded by a single common vector of column
parameters \(\boldsymbol\kappa\).  After a full row cycle all rows have acquired the
same factor \(z\), and the common parameter structure is restored.  For the
upper factors, the finite Christoffel \(\tau\)-determinants formed with
the monic-normalized \(B\)-vectors are used.
Put
\begin{equation}
	\label{eq:Bessel-B-values-at-zero}
	\mathfrak b_{N,j}^{[0]}
	\coloneq
	\left(\mathbf B_N\right)^{(j)}(0)
	=
	\mathsf b_j(\boldsymbol n_N,\boldsymbol m_N;0\mid\boldsymbol\kappa),
	\qquad
	j\in\{1,\ldots,q\}.
\end{equation}
The values in \eqref{eq:Bessel-B-values-at-zero} enter the finite \(\tau\)-determinants.  Following
\cite{BranquinhoFoulquieManas2026}, define
\begin{equation}
	\label{eq:Bessel-tau-B}
	\tau^B_{b,n}
	\coloneq
	\det\left[
	\mathfrak b_{n+\mu-1,\nu}^{[0]}
	\right]_{\mu,\nu=1}^{b},
	\qquad
	b\in\{1,\ldots,q\},
\end{equation}
and put \(\tau^B_{0,n}=1\).  These are the Christoffel tau-determinants of the
row side, formed from the values at the Christoffel point of the normalized mixed \(B\)-polynomial vectors and their truncations.

\subsection{Bessel factorization of mixed type}
\label{sec:final-bidiagonal-mixed-factorization}

Let \(\mathscr T_N\) be the \(N\times N\) principal truncation of the step-line
recurrence matrix.  Write \(L_a\) and \(U_b\) for the corresponding infinite
bidiagonal factors, and \(L_a^{[N]}\) and \(U_b^{[N]}\) for their
\(N\times N\) principal truncations.  Assume that the moment matrix and the
Christoffel-perturbed moment matrices used below admit Gauss--Borel
factorizations.  By \cite{ManasRojas2026}, this is equivalent to the existence
of the corresponding perturbed mixed-type orthogonality and to the
nonvanishing of the relevant Christoffel determinants.  The general mixed
Christoffel--Gauss--Borel factorization of
\cite{BranquinhoFoulquieManas2026} then gives
\begin{equation}
	\label{eq:final-bidiagonal-factorization}
	\mathscr T_N
	=
	L_1^{[N]}\cdots L_p^{[N]}
	U_q^{[N]}\cdots U_1^{[N]}.
\end{equation}
The lower factors \(L_a^{[N]}\) have unit diagonal.  Since the column chain is
closed, no auxiliary family is needed.  Their subdiagonal entries are read from
the normalized step-line \(A\)-polynomials of the Christoffel-transformed
systems.  These polynomials may be computed either from the general
near-diagonal formula, after the parameter shift
\(\boldsymbol\kappa\mapsto\boldsymbol\kappa^{\langle a\rangle}\), or from the
Christoffel factorization; the monic-\(B\), dual-\(A\) normalization is fixed
from the beginning precisely so that the two computations give the same vector.
Equivalently, if \(A_{\mathrm L,n}^{[a]}\) denotes this normalized step-line
\(A\)-vector for \(\mathrm d\boldsymbol\mu_{\mathrm L}^{(a)}\), computed with
\(\boldsymbol\kappa^{\langle a\rangle}\), then
\begin{equation}
	\label{eq:Bessel-L-leading-coefficients}
	(L_a)_{n+1,n}
	=
	\frac{
		\LC\left(
		(A_{\mathrm L,n}^{[a]})^{(c_n)}
		\right)
	}{
		\LC\left(
		(A_{\mathrm L,n+1}^{[a-1]})^{(c_{n+1})}
		\right)
	},
\end{equation}
with the compatible near-diagonal normalization fixed above.  Here \(c_n\)
denotes the column selected by the step-line ordering of the transformed system:
the cyclic shift of \(\boldsymbol\kappa\) changes the parameters attached to the
columns, but not the step-line rule defining the selected column.  The opposite
cyclic order gives the transposed column ordering and is not the
convention used in \eqref{eq:final-bidiagonal-factorization}.

The upper factors \(U_b^{[N]}\) have unit superdiagonal.  Since the row chain is
not generally closed in the Bessel family, their diagonal entries are kept
in the finite \(\tau\)-determinant form
\begin{equation}
	\label{eq:Bessel-finite-U-tau}
	(U_b)_{n,n}
	=
	-
	\frac{
		\tau^B_{b-1,n}\tau^B_{b,n+1}
	}{
		\tau^B_{b-1,n+1}\tau^B_{b,n}
	},
	\qquad
	b\in\{1,\ldots,q\}.
\end{equation}
Formula \eqref{eq:Bessel-finite-U-tau} is the finite Christoffel formula of
\cite{BranquinhoFoulquieManas2026} specialized to the non-closed row chain.

\subsection{The one-row multiple Bessel reduction \texorpdfstring{\(q=1\)}{q=1}}
\label{sec:final-multiple-bessel-bidiagonal-factorization}

The one-row case \(q=1\) must be treated differently from the non-closed row
case.  Now the unique row-Christoffel step coincides with one full
column-Christoffel cycle.  Hence the multiple Bessel family is invariant under
all the Christoffel transformations needed for the factorization, and the
entries can be written directly as quotients of transformed type-II polynomials
evaluated at the origin.

Write \(\boldsymbol\alpha=\begin{bNiceMatrix}\alpha_1&\Cdots&\alpha_p\end{bNiceMatrix}\), with
\(\alpha_j=b_1+\kappa_j\), and extend the parameters cyclically by
\(\alpha_{sp+j}\coloneq\alpha_j+s,\) \(s\in\N_0,\) \(j\in\{1,\ldots,p\}.\)
Set
\(\theta_n\coloneq\alpha_{n+1},\) \(n\in\N_0,\)
so that \(\theta_{pu+j-1}=\alpha_j+u\).  For \(b\in\{0,\ldots,p\}\), define the
shifted sequence
\(\boldsymbol\theta^{[b]} \coloneq (\theta_b,\theta_{b+1},\theta_{b+2},\ldots).\)
Let \(P_n^{[b]}\) denote the monic type-II multiple Bessel polynomial of degree \(n\) for the shifted scalar moment system
\(\mu_r^{[b]}(s) \coloneq \frac{1}{\Gamma(\theta_{b+r}+s+2)},\) \(r,s\in\N_0.\)
Thus
\(P_n^{[b]}(z) = z^n+ \sum_{s=0}^{n-1}p_{n,s}^{[b]}z^s\)
is characterized by
\(\sum_{s=0}^{n}p_{n,s}^{[b]} \frac{1}{\Gamma(\theta_{b+r}+s+2)} = 0,\) \(r\in\{0,\ldots,n-1\},\)
where \(p_{n,n}^{[b]}=1\).  Thus \(P_n^{[b]}\) is not a new family of polynomials; it is the usual monic type-II multiple Bessel polynomial associated with the shifted sequence \(\boldsymbol\theta^{[b]}\).
The equality
\(\boldsymbol\theta^{[p]}=(\theta_0+1,\theta_1+1,\ldots)\)
expresses the fact that the unique row-Christoffel step is one full closed
column cycle.

The principal scalar moment determinants of the shifted systems are
\[
\det\left[ \frac{1}{\Gamma(\theta_r^{[b]}+s+2)} \right]_{r,s=0}^{n-1}, \qquad \theta_r^{[b]}\coloneq\theta_{b+r}.
\]
\begin{proposition}[Closed Gamma--Vandermonde determinant]
	\label{prop:ABV-determinant}
	For a sequence \(\boldsymbol\vartheta=(\vartheta_0,\vartheta_1,\ldots)\), put
	\[
	\mathscr D_n(\boldsymbol\vartheta)
	\coloneq
	(-1)^{n(n-1)/2}
	\frac{
		\prod_{0\le r<s<n}(\vartheta_s-\vartheta_r)
	}{
		\prod_{r=0}^{n-1}\Gamma(\vartheta_r+n+1)
	},
	\qquad
	\mathscr D_0(\boldsymbol\vartheta)\coloneq1.
	\]
	Then
\(\det\left[ \frac{1}{\Gamma(\vartheta_r+s+2)} \right]_{r,s=0}^{n-1} = \mathscr D_n(\boldsymbol\vartheta).\)
\end{proposition}

\begin{proof}
	Assume first that the Gamma factors used as row multipliers below are finite.
	Multiplying the \(r\)-th row by \(\Gamma(\vartheta_r+n+1)\) transforms the entry
	in column \(s\) into
\(\frac{\Gamma(\vartheta_r+n+1)}{\Gamma(\vartheta_r+s+2)} = (\vartheta_r+s+2)_{n-s-1}.\)
	For fixed \(s\), this is a monic polynomial in \(\vartheta_r\) of degree
	\(n-s-1\).  Reversing the column order gives the Vandermonde determinant in
	\(\vartheta_0,\ldots,\vartheta_{n-1}\), with sign \((-1)^{n(n-1)/2}\).  Dividing
	back by the row factors gives the formula on this domain.  Since both sides
	are entire functions of \(\vartheta_0,\ldots,\vartheta_{n-1}\), the identity
	extends to arbitrary parameter values by analytic continuation.
\end{proof}

\begin{proposition}[Values of the transformed type-II polynomials]
	\label{prop:multiple-bessel-transformed-values}
	For \(b\in\{0,\ldots,p\}\) and \(n\in\N_0\), the monic transformed type-II
	polynomials satisfy
	\begin{equation}
		\label{eq:multiple-bessel-polynomial-value}
		P_n^{[b]}(0)
		=
		(-1)^n
		\frac{
			\mathscr D_n(\boldsymbol\theta^{[b+p]})
		}{
			\mathscr D_n(\boldsymbol\theta^{[b]})
		}.
	\end{equation}
\end{proposition}

\begin{proof}
	With the preceding normalization, this is the standard determinant formula for
	the monic type-II polynomial of the shifted multiple Bessel moment matrix.  In
	the denominator one has the principal determinant with columns corresponding to
	the moments
	\(0,\ldots,n-1\).  Evaluating the monic determinant at \(z=0\) replaces the
	polynomial row by the first basis vector and leaves the determinant built from
	the shifted moment columns \(1,\ldots,n\).  Equivalently,
	\[
	P_n^{[b]}(0)
	=
	(-1)^n
	\frac{
	\det\left[
	\dfrac{1}{\Gamma(\theta_r^{[b]}+s+3)}
	\right]_{r,s=0}^{n-1}
	}{
	\det\left[
	\dfrac{1}{\Gamma(\theta_r^{[b]}+s+2)}
	\right]_{r,s=0}^{n-1}
	}.
	\]
	For the Bessel moment sequence this column shift is exactly the full
	Christoffel shift
	\(\boldsymbol\theta^{[b]}\mapsto\boldsymbol\theta^{[b+p]}\), because
	\(\theta_r^{[b+p]}=\theta_r^{[b]}+1\).  Applying
	Proposition~\ref{prop:ABV-determinant} to the numerator and denominator gives
	\eqref{eq:multiple-bessel-polynomial-value}.
\end{proof}

For the lower factors it is useful to write the answer directly in terms of
the cyclic sequence.  If the lower factor is the one associated with the
transition
\(\boldsymbol\theta^{[b-1]}\mapsto\boldsymbol\theta^{[b]}\), put
\(a=b-1\).  The relevant consecutive parameters are then
\(\theta_{a+1},\ldots,\theta_{a+n}\).

\begin{theorem}[Complete Christoffel factorization in the multiple Bessel case]
	\label{thm:multiple-bessel-factorization}
	For \(q=1\), let \(T_{\mathrm{II}}^{\mathrm{MB}}\) be the recurrence matrix of
	the monic step-line type-II multiple Bessel polynomial sequence. Then
	\[
	T_{\mathrm{II}}^{\mathrm{MB}}
	=
	\widetilde L_1^{\mathrm{MB}}\cdots
	\widetilde L_p^{\mathrm{MB}}
	\widetilde U_1^{\mathrm{MB}}.
	\]
	The lower bidiagonal factors have unit diagonal and subdiagonal entries
	\begin{equation}
		\label{eq:multiple-bessel-L-typeII-poly}
		\left(\widetilde L_b^{\mathrm{MB}}\right)_{n+1,n}
		=
		\frac{
			P_{n+1}^{[b-1]}(0)-P_{n+1}^{[b]}(0)
		}{
			P_n^{[b]}(0)
		},
		\qquad
		b\in\{1,\ldots,p\}.
	\end{equation}
	Equivalently, with \(a=b-1\),
	\begin{equation}
		\label{eq:multiple-bessel-L-typeII-product}
		\left(\widetilde L_b^{\mathrm{MB}}\right)_{n+1,n}
		=
		\left(
		\frac{1}{\theta_{a+n+1}+n+2}
		-
		\frac{1}{\theta_a+n+2}
		\right)
		\prod_{r=1}^{n}
		\frac{\theta_{a+r}+n+1}{\theta_{a+r}+n+2}.
	\end{equation}
	The unique upper factor has unit superdiagonal and diagonal entries
	\begin{equation}
		\label{eq:multiple-bessel-U-typeII-poly}
		\left(\widetilde U_1^{\mathrm{MB}}\right)_{n,n}
		=
		-
		\frac{
			P_{n+1}^{[0]}(0)
		}{
			P_n^{[0]}(0)
		}.
	\end{equation}
	Equivalently, using \eqref{eq:multiple-bessel-polynomial-value}, the unique
	upper factor can be written as
	\begin{equation}
		\label{eq:multiple-bessel-U-typeII}
		\left(\widetilde U_1^{\mathrm{MB}}\right)_{n,n}
		=
		\frac{
			\mathscr D_n(\boldsymbol\theta)
			\mathscr D_{n+1}(\boldsymbol\theta^{[p]})
		}{
			\mathscr D_{n+1}(\boldsymbol\theta)
			\mathscr D_n(\boldsymbol\theta^{[p]})
		}.
	\end{equation}
\end{theorem}

\begin{proof}
	Since \(q=1\), every Christoffel transformation appearing in the factorization
	remains in the multiple Bessel family.  Therefore one uses the connection
	formulas for the transformed monic type-II polynomials rather than the
	finite \(\tau\)-determinant fallback.  For the elementary step
	\(\boldsymbol\theta^{[b-1]}\mapsto\boldsymbol\theta^{[b]}\), the monic
	connection has the form
\(P_{n+1}^{[b-1]}(z) = P_{n+1}^{[b]}(z) + \left(\widetilde L_b^{\mathrm{MB}}\right)_{n+1,n} P_n^{[b]}(z).\)
	Evaluating at \(z=0\) gives
	\eqref{eq:multiple-bessel-L-typeII-poly}.  Notice that the coefficient is a
	difference of two transformed values; replacing it by only
	\(-P_{n+1}^{[b-1]}(0)/P_n^{[b]}(0)\) would omit the
	term \(P_{n+1}^{[b]}(0)\) and gives the wrong lower factor.

	The product form follows as follows.  Applying
	Proposition~\ref{prop:multiple-bessel-transformed-values} and the
	Gamma--Vandermonde formula of Proposition~\ref{prop:ABV-determinant} gives
	\[
	-
	\frac{P_{n+1}^{[a]}(0)}{P_n^{[a+1]}(0)}
	=
	\frac{1}{\theta_a+n+2}
	\prod_{r=1}^{n}
	\frac{\theta_{a+r}+n+1}{\theta_{a+r}+n+2}
	\]
	and
	\[
	-
	\frac{P_{n+1}^{[a+1]}(0)}{P_n^{[a+1]}(0)}
	=
	\frac{1}{\theta_{a+n+1}+n+2}
	\prod_{r=1}^{n}
	\frac{\theta_{a+r}+n+1}{\theta_{a+r}+n+2}.
	\]
	Subtracting these two identities according to
	\eqref{eq:multiple-bessel-L-typeII-poly} gives
	\eqref{eq:multiple-bessel-L-typeII-product}.  The unique row step is the full
	column cycle \(\boldsymbol\theta^{[0]}\mapsto\boldsymbol\theta^{[p]}\), and
	the scalar Christoffel connection gives \eqref{eq:multiple-bessel-U-typeII-poly};
	the same determinant evaluation gives \eqref{eq:multiple-bessel-U-typeII}.
\end{proof}

The preceding formulas also give the signs and bounds of the entries in the
ordered real chamber.  In the present monic normalization the upper bidiagonal
parameters are positive, whereas the subdiagonal entries of the lower factors
are negative.

\begin{proposition}[Explicit signed bounded bidiagonal coefficients]
	\label{prop:multiple-bessel-positive-bounded-factors}
	Assume
	\begin{equation}
		\label{eq:ordered-positive-chamber}
		-1<\alpha_1<\cdots<\alpha_p<\alpha_1+1.
	\end{equation}
	Put
	\[
	u_n\coloneq
	\left(\widetilde U_1^{\mathrm{MB}}\right)_{n,n},
	\qquad
	\ell_{b,n}\coloneq
	-\left(\widetilde L_b^{\mathrm{MB}}\right)_{n+1,n}.
	\]
	Then
	\begin{equation}
		\label{eq:explicit-U-positive-product}
		u_n
		=
		\frac{1}{\theta_n+n+2}
		\prod_{r=0}^{n-1}
		\frac{\theta_r+n+1}{\theta_r+n+2},
	\end{equation}
	and, for \(b\in\{1,\ldots,p\}\), with \(a=b-1\),
	\begin{equation}
		\label{eq:explicit-L-positive-product}
		\ell_{b,n}
		=
		\left(
		\frac{1}{\theta_a+n+2}
		-
		\frac{1}{\theta_{a+n+1}+n+2}
		\right)
		\prod_{r=1}^{n}
		\frac{\theta_{a+r}+n+1}{\theta_{a+r}+n+2}
	\end{equation}
	or, equivalently,
	\[
		\ell_{b,n}
		=
		\frac{\theta_{a+n+1}-\theta_a}
		     {(\theta_a+n+2)(\theta_{a+n+1}+n+2)}
		\prod_{r=1}^{n}
		\frac{\theta_{a+r}+n+1}{\theta_{a+r}+n+2}.
	\]
	Consequently,
\(u_n>0,\) \(\ell_{b,n}>0,\) \(\left(\widetilde L_b^{\mathrm{MB}}\right)_{n+1,n}=-\ell_{b,n}<0,\)
	and
\(u_n=\mathrm{O}(n^{-1}),\) \(\ell_{b,n}=\mathrm{O}(n^{-1}),\) \(n\to\infty.\)
\end{proposition}

\begin{proof}
	Formula \eqref{eq:explicit-L-positive-product} is the signed version of
	\eqref{eq:multiple-bessel-L-typeII-product}.  Substituting the
	Gamma--Vandermonde formula of Proposition~\ref{prop:ABV-determinant} into
	\eqref{eq:multiple-bessel-U-typeII} gives the upper factor.  The equality
	\(\boldsymbol\theta^{[p]}=(\theta_0+1,\theta_1+1,\ldots)\) gives the
	cancellations in the upper factor.  In the ordered chamber all factors in
	\eqref{eq:explicit-U-positive-product} are positive.  Also
	\(\theta_{a+n+1}>\theta_a\), so the displayed expression for
	\(\ell_{b,n}\) is positive.  The asymptotic bounds follow immediately from
	the product formulas.
\end{proof}

\begin{corollary}[Signed bidiagonal factorization in the ordered chamber]
	\label{cor:multiple-bessel-positive-factorization}
	Under \eqref{eq:ordered-positive-chamber}, let \(\Lambda\) be the scalar shift
	used throughout the paper, \(\Lambda X(z)=zX(z)\).  Set
\(E_b\coloneq\operatorname{diag}(\ell_{b,0},\ell_{b,1},\ldots),\) \(U\coloneq\operatorname{diag}(u_0,u_1,\ldots).\)
	Then
	\[
	T_{\mathrm{II}}^{\mathrm{MB}}
	=
	\left(I-\Lambda^{\top}E_1\right)\cdots
	\left(I-\Lambda^{\top}E_p\right)
	\left(U+\Lambda\right),
	\]
	where \(\ell_{b,n}>0\) and \(u_n>0\).
\end{corollary}

\begin{proof}
	This is Theorem~\ref{thm:multiple-bessel-factorization} rewritten with the
	positive parameters of
	Proposition~\ref{prop:multiple-bessel-positive-bounded-factors}.  By
	definition,
	\[
	(\Lambda^{\top}E_b)_{n+1,n}=\ell_{b,n}.
	\]
	Hence the corresponding lower bidiagonal entries are \(-\ell_{b,n}\),
	while \(U+\Lambda\) has diagonal entries \(u_n\) and unit superdiagonal.
\end{proof}

\begin{remark}[Boundedness in the ordered chamber]
	The sequences \(u_n\) and \(\ell_{b,n}\) are uniformly bounded by
	Proposition~\ref{prop:multiple-bessel-positive-bounded-factors}. Since the
	matrix has finite bandwidth, \(T_{\mathrm{II}}^{\mathrm{MB}}\) defines a
	bounded banded operator, in the ordered chamber, on each of the spaces
	\[
	\ell^s(\mathbb N_0),\quad 1\le s\le\infty,
	\qquad\text{and}\qquad c_0(\mathbb N_0).
	\]
\end{remark}

\subsection{The two-row case \texorpdfstring{\(q=2\)}{q=2}}
\label{sec:q-two-arbitrary-p-bidiagonal-example}

The first genuinely mixed row case is made explicit.  For \(p\ge2\), it is the
closest case to the one-row multiple Bessel reduction, but it already contains
the two different constructions used in the general factorization: the closed
column-Christoffel chain for the lower factors and the finite row
tau-determinants for the upper factors.

Let \(q=2\), let \(p\) be arbitrary.  When \(p=1\) this section reduces to the
one-column Wolfs case; the genuinely mixed instances have \(p\ge2\).  Write
\[
\boldsymbol b=\begin{bNiceMatrix}b_1&b_2\end{bNiceMatrix},
\qquad
\overline j\coloneq3-j,
\qquad
\overline H\coloneq3-H.
\]
Thus the numerator length \(r\) can only be \(0\) or \(1\).  The weight matrix is
\[
W(z)=
\begin{bNiceMatrix}
	w_{1,1}(z)&w_{1,2}(z)&\Cdots&w_{1,p}(z)\\
	w_{2,1}(z)&w_{2,2}(z)&\Cdots&w_{2,p}(z)
\end{bNiceMatrix},
\]
where
\[
\begin{aligned}
	w_{1,i}(z)
	&=
	f_0\left(z;\boldsymbol a+\kappa_i\one_r,
	(b_1+\kappa_i+1,b_2+\kappa_i)\right),\\
	w_{2,i}(z)
	&=
	f_0\left(z;\boldsymbol a+\kappa_i\one_r,
	(b_1+\kappa_i,b_2+\kappa_i+1)\right).
\end{aligned}
\]
Equivalently, the reciprocal moments are
\[
\MB_{(u,j),(v,i)}
=
\frac{
	\Gamma((u+v+1+\kappa_i)\one_r+\boldsymbol a)
}{
	\Gamma((u+v+1+\kappa_i)\one_q+\boldsymbol b+\boldsymbol e_j)
},
\qquad
j\in\{1,2\},
\quad
 i\in\{1,\ldots,p\}.
\]

For the step-line one has
\(\boldsymbol n_N=\boldsymbol\sigma_p(N),\) \(\boldsymbol m_N=\boldsymbol\sigma_2(N+1).\)
Hence, for \(s\in\N_0\),
\[
\begin{aligned}
	\boldsymbol m_{2s}&=\begin{bNiceMatrix}s+1&s\end{bNiceMatrix},
	&r_{2s}&=1,&
	\boldsymbol m_{2s+1}&=\begin{bNiceMatrix}s+1&s+1\end{bNiceMatrix},
	&r_{2s+1}&=2.
\end{aligned}
\]
For \(q=2\), it is useful to write the raw two-row coefficient before
multiplication by the normalization constant in \eqref{eq:final-B-nu}.  It is
\begin{multline}
	\label{eq:q-two-B-coefficients}
	\mathsf b_j^{(2)}(\boldsymbol n,\boldsymbol m;d\mid\boldsymbol\kappa)
	=
	\frac{(\boldsymbol a-b_j\one_r)_1}{b_{\overline j}-b_j}
	\sum_{H=1}^{2}
	\sum_{\substack{K=0\\ d\le K-1+\delta_{j,H}}}^{m_H-1}
	\pi_{H,K}^{(2)}
	\frac{b_{\overline j}-b_H-K}
	{(\boldsymbol a-(b_H+K)\one_r)_1}
	\\
	\times
	\frac{
		(b_1-b_H-K+1-\delta_{j,1})_d
		(b_2-b_H-K+1-\delta_{j,2})_d
	}{
		(\boldsymbol a-(b_H+K)\one_r+\one_r)_d
	},
\end{multline}
where the two-row partial-fraction coefficients are
\begin{equation}
	\label{eq:q-two-pi-coefficients}
	\pi_{H,K}^{(2)}
	=
	-
	\frac{
		\prod_{i=1}^{p}(\kappa_i+b_H+K+1)_{n_i}
	}{
		(-1)^K K!(m_H-K-1)!
		(b_{\overline H}-b_H-K)_{m_{\overline H}}
	}.
\end{equation}
For later use, two pieces of finite data are separated.  First, using the
coefficients in \eqref{eq:q-two-pi-coefficients}, the raw constant term of the
explicit two-row \(B\)-polynomial is
\begin{equation}
	\label{eq:q-two-B-constant-data}
	\begin{aligned}
	\mathcal B_j^{(2)}(\boldsymbol n,\boldsymbol m\mid\boldsymbol\kappa)
	&\coloneq
	\frac{(\boldsymbol a-b_j\one_r)_1}{b_{\overline j}-b_j}
	\\[-1mm]
	&\quad\times
	\left[
	\sum_{K=0}^{m_j-1}
	\pi_{j,K}^{(2)}
	\frac{b_{\overline j}-b_j-K}
	{(\boldsymbol a-(b_j+K)\one_r)_1}
	+
	\sum_{K=1}^{m_{\overline j}-1}
	\pi_{\overline j,K}^{(2)}
	\frac{-K}
	{(\boldsymbol a-(b_{\overline j}+K)\one_r)_1}
	\right].
	\end{aligned}
\end{equation}
The second sum is empty if its upper limit is smaller than its lower limit.
This is just \eqref{eq:q-two-B-coefficients} with \(d=0\), written without
hidden hypergeometric notation.

Second, for the row selected by the step-line in
\eqref{eq:general-active-indices}, the top coefficient has only one contributing
partial-fraction block.  If \(j\) denotes this row, set
\begin{equation}
	\label{eq:q-two-active-leading-data}
	\begin{aligned}
		\mathcal G_j^{(2)}(\boldsymbol n,\boldsymbol m\mid\boldsymbol\kappa)
		&\coloneq
		\frac{(\boldsymbol a-b_j\one_r)_1}{b_{\overline j}-b_j}
		\pi_{j,m_j-1}^{(2)}
		\frac{b_{\overline j}-b_j-m_j+1}
		{(\boldsymbol a-(b_j+m_j-1)\one_r)_1}
		\\
		&\quad\times
		\frac{
			(-m_j+1)_{m_j-1}
			(b_{\overline j}-b_j-m_j+2)_{m_j-1}
		}{
			(\boldsymbol a-(b_j+m_j-1)\one_r+\one_r)_{m_j-1}
		}.
	\end{aligned}
\end{equation}
Thus the corresponding unnormalized leading coefficients are the finite products
\begin{equation}
	\label{eq:q-two-active-leading-even-odd}
	\begin{aligned}
		\mathcal G_{2s}
		&\coloneq
		\mathcal G_1^{(2)}
		\left(\boldsymbol\sigma_p(2s),(s+1,s)\mid\boldsymbol\kappa\right),&
		\mathcal G_{2s+1}
		&\coloneq
		\mathcal G_2^{(2)}
		\left(\boldsymbol\sigma_p(2s+1),(s+1,s+1)\mid\boldsymbol\kappa\right).
	\end{aligned}
\end{equation}
Similarly define the four constant-term sequences
\begin{equation}
	\label{eq:q-two-constant-even-odd}
	\begin{aligned}
		\mathcal B_{2s,j}
		&\coloneq
		\mathcal B_j^{(2)}
		\left(\boldsymbol\sigma_p(2s),(s+1,s)\mid\boldsymbol\kappa\right),
		\qquad j\in\{1,2\},\\
		\mathcal B_{2s+1,j}
		&\coloneq
		\mathcal B_j^{(2)}
		\left(\boldsymbol\sigma_p(2s+1),(s+1,s+1)\mid\boldsymbol\kappa\right),
		\qquad j\in\{1,2\}.
	\end{aligned}
\end{equation}
The values entering the row Christoffel determinants are the normalized values
obtained from the raw constants in \eqref{eq:q-two-constant-even-odd} and the
raw leading coefficients in \eqref{eq:q-two-active-leading-even-odd}:
\begin{equation}
	\label{eq:q-two-normalized-B-values-expanded}
	\mathfrak b_{N,j}^{[0]}
	=
	\frac{\mathcal B_{N,j}}{\mathcal G_N},
	\qquad
	N\in\N_0,
	\quad
	j\in\{1,2\}.
\end{equation}
For \(q=2\), the row tau-functions defined in \eqref{eq:Bessel-tau-B} reduce
to one scalar value and one \(2\times2\) determinant.  With
\begin{equation}
	\label{eq:q-two-Delta-row}
	\Delta_n^{(2)}
	\coloneq
	\mathcal B_{n,1}\mathcal B_{n+1,2}
	-
	\mathcal B_{n,2}\mathcal B_{n+1,1},
\end{equation}
one has
\begin{equation}
	\label{eq:q-two-tau-functions-expanded}
	\tau^B_{1,n}
	=
	\frac{\mathcal B_{n,1}}{\mathcal G_n},
	\qquad
	\tau^B_{2,n}
	=
	\frac{\Delta_n^{(2)}}{\mathcal G_n\mathcal G_{n+1}}.
\end{equation}
Thus the two upper bidiagonal factors, obtained from \eqref{eq:Bessel-finite-U-tau}, are finite-product/finite-sum expressions:
\begin{equation}
	\label{eq:q-two-upper-factors-expanded}
	\begin{aligned}
		(U_1)_{n,n}
		&=
		-
		\frac{\mathcal B_{n+1,1}\mathcal G_n}
		{\mathcal B_{n,1}\mathcal G_{n+1}},&
		(U_2)_{n,n}
		&=
		-
		\frac{
			\mathcal B_{n,1}\Delta_{n+1}^{(2)}\mathcal G_{n+1}
		}{
			\mathcal B_{n+1,1}\Delta_n^{(2)}\mathcal G_{n+2}
		}.
	\end{aligned}
\end{equation}
If \((U_b)_{n,n+1}=1\), then
\[
U_2U_1
=
\Lambda^2
+
\operatorname{diag}\left((U_2)_{n,n}+(U_1)_{n+1,n+1}\right)_{n\ge0}\Lambda
+
\operatorname{diag}\left((U_2)_{n,n}(U_1)_{n,n}\right)_{n\ge0}.
\]
The lower factors are obtained from the closed column chain.  The same cyclic
shift as in Section~\ref{sec:final-bidiagonal-bessel-data} is used:
\[
\boldsymbol\kappa^{\langle a\rangle}
=
(\kappa_{a+1},\ldots,\kappa_p,
\kappa_1+1,\ldots,\kappa_a+1),
\qquad
 a\in\{0,\ldots,p\}.
\]
Let
\(\boldsymbol n_N^{\,*}=\boldsymbol\sigma_p(N+1),\) \(\boldsymbol m_N^{\,*}=\boldsymbol\sigma_2(N),\) \(c_N=1+(N\bmod p).\)
For the transformed system with column parameters
\(\boldsymbol\kappa^{\langle a\rangle}\), define
\begin{equation}
	\label{eq:q-two-A-leading-data}
	\mathcal A_N^{[a]}
	\coloneq
	\LC\left(
	(A_{\mathrm L,N}^{[a]})^{(c_N)}
	\right).
\end{equation}
Equivalently,
\(\mathcal A_N^{[a]}
=
\mathsf a_{c_N}
\left(
\boldsymbol n_N^{\,*},\boldsymbol m_N^{\,*};
 (\boldsymbol n_N^{\,*})_{c_N}-1\mid
\boldsymbol\kappa^{\langle a\rangle}
\right),\)
where the coefficient is read in the compatible near-diagonal normalization
of the transformed step-line system.  The lower bidiagonal factors are
\begin{equation}
	\label{eq:q-two-lower-factors}
	(L_a)_{n+1,n}
	=
	\frac{\mathcal A_n^{[a]}}
	{\mathcal A_{n+1}^{[a-1]}},
	\qquad
	a\in\{1,\ldots,p\}.
\end{equation}
Combining \eqref{eq:q-two-upper-factors-expanded} and \eqref{eq:q-two-lower-factors}
gives the explicit first mixed factorization
\(\mathscr T_N = L_1^{[N]}\cdots L_p^{[N]}U_2^{[N]}U_1^{[N]}.\)
No \(A\)-side tau-determinants are needed in this case: the column chain is
closed and is evaluated by transformed polynomial leading coefficients, while
only the two row factors require the finite \(B\)-side tau-determinants.

\subsection{A numerical mixed example: \texorpdfstring{\(p=3\), \(q=2\)}{p=3, q=2}}
\label{sec:p-three-q-two-numerical-example}

The section closes with a concrete mixed example illustrating the first
recurrence coefficients and their Christoffel--Gauss--Borel factorization.
Take
\[
	p=3,
	\qquad
	q=2,
	\qquad
	r=1,
	\qquad
	\boldsymbol a=\begin{bNiceMatrix}\frac12\end{bNiceMatrix},
	\qquad
	\boldsymbol b=\begin{bNiceMatrix}\frac13&\frac23\end{bNiceMatrix},
	\qquad
	\boldsymbol\kappa=\begin{bNiceMatrix}0&\frac14&\frac34\end{bNiceMatrix}.
\]
The bimoments are
\[
	\MB_{(u,j),(v,i)}
	=
	\frac{
		\Gamma(u+v+\kappa_i+\frac32)
	}{
		\Gamma(u+v+\kappa_i+\frac43+\delta_{j,1})
		\Gamma(u+v+\kappa_i+\frac53+\delta_{j,2})
	},
	\qquad
	u,v\in\N_0.
\]
The first monic step-line \(B\)-vectors are
\[
	\mathbf B_0(z)
	\doteq
	\begin{bNiceMatrix}[small]1&0\end{bNiceMatrix},
	\qquad
	\mathbf B_1(z)
	\doteq
	\begin{bNiceMatrix}[small]-0.8&1\end{bNiceMatrix},
	\qquad
	\mathbf B_2(z)
	\doteq
	\begin{bNiceMatrix}[small]-1.376+z&1.237\end{bNiceMatrix}.
\]
With rows and columns indexed from zero, the first principal block of the
\((3,2)\)-banded recurrence matrix is
\[
	\mathscr T^{[5]}
	\doteq
	\begin{bNiceArray}[small,margin=2pt]{rrrrr}
		0.3857 & -1.237 & 1 & 0 & 0 \\
		0.02893 & 0.05236 & 0.1791 & 1 & 0 \\
		-0.004499 & -0.05155 & 0.02238 & -0.7294 & 1 \\
		-3.247\cdot10^{-5} & -0.001782 & 0.006488 & 0.01203 & 0.1221 \\
		0 & 8.156\cdot10^{-5} & -0.001244 & -0.02211 & 0.01037
	\end{bNiceArray}.
\]
The lower factors have unit diagonal. Listing their subdiagonal entries from
top to bottom gives
\[
\begin{aligned}
	\operatorname{subdiag}L_1^{[5]}
	&\doteq
	(0.02609,-0.1934,0.03815,-0.1852),\\
	\operatorname{subdiag}L_2^{[5]}
	&\doteq
	(0.03598,-0.1707,0.03351,-0.1957),\\
	\operatorname{subdiag}L_3^{[5]}
	&\doteq
	(0.01293,-0.09051,0.02572,-0.1417).
\end{aligned}
\]
The upper factors have unit superdiagonal, with diagonal entries
\[
\begin{aligned}
	\operatorname{diag}U_1^{[5]}
	&\doteq
	(0.8,-1.719,0.1886,-0.7065,0.09807),\\
	\operatorname{diag}U_2^{[5]}
	&\doteq
	(0.4821,-0.08442,0.4316,-0.07331,0.4009).
\end{aligned}
\]
Using the unrounded arbitrary-precision values, these truncations satisfy
\[
	\max_{0\le n,m\le4}
	\left|
	\mathscr T^{[5]}_{n,m}
	-
	\left(
	L_1^{[5]}L_2^{[5]}L_3^{[5]}U_2^{[5]}U_1^{[5]}
	\right)_{n,m}
	\right|
	<1.1\times10^{-73}.
\]
Thus the recurrence matrix obtained from the step-line polynomial basis agrees
with the Christoffel--Gauss--Borel factorization; the same computation verifies
the corresponding biorthogonality relations.

\section{Conclusions and outlook}

This article constructs a family of Bessel multiple orthogonal
polynomials of mixed type for a \(q\times p\) matrix of weights on the unit
circle. Unlike the previously known separable examples, this matrix is not a
rank-one product: for generic regular parameters it has rank
\(\min\{p,q\}\), except possibly at finitely many points of the circle. Its
moments are explicit quotients of Gamma functions. The specializations
\(q=1\) and \(p=1\) recover, respectively, the multiple Bessel system and
the Bessel-like system of Wolfs.

For balanced index pairs with a near-diagonal row multi-index, both polynomial
vectors are given explicitly by terminating generalized hypergeometric
formulas. The column multi-index is unrestricted; hence the formulas include
families far from the diagonal, under the stated regularity and admissibility
conditions on the parameters.
Their mixed orthogonality and uniqueness up to normalization are proved, and
the exact conditions under which each component attains its prescribed
degree are identified. The \(B\)-components also admit finite representations
by Kamp\'e de F\'eriet polynomials and a matrix Rodrigues-type formula. When
\(q=1\), the single Kamp\'e de F\'eriet block evaluated at \((-z,1)\) reduces
to one generalized hypergeometric polynomial; the additional parameter pairs
are determined by a reflected type-II multiple Hahn polynomial.

On the step-line, the two polynomial sequences satisfy dual recurrences whose
matrix has \(p\) subdiagonals and \(q\) superdiagonals. Every recurrence
coefficient is expressed by a finite Gamma--Pochhammer formula. Successive
Christoffel transformations give a bidiagonal factorization of this
recurrence matrix. The lower bidiagonal factors are explicit, whereas for
\(q>1\) the upper factors are expressed as determinants of shifted moment
minors. In the one-row case \(q=1\), Gamma--Vandermonde evaluations give all
the factors in closed form. Since the Bessel weights and recurrence
coefficients are generally complex, this is an algebraic bidiagonal
factorization, not a positive one.

The relation with the Jacobi system of mixed type requires a transformation of the
orthogonality problem. The rank-one measures on the interval do not converge
to a finite matrix measure. After passing to the equivalent
Markov--Stieltjes contour representation and rescaling the variable, however,
the matrices of weights converge to the full-rank Bessel matrix. Thus
the increase of rank occurs at the level of the contour representation, not
as a pointwise limit of rank-one interval densities.

Several concrete questions remain open.
\begin{enumerate}[label=\textup{(\arabic*)},leftmargin=*]
	\item Determine explicitly the finite exceptional set on the unit circle
	where the matrix of weights loses maximal rank, and decide whether it is
	empty in natural real parameter regions.
	\item Analyze the exceptional hypersurfaces already identified for degree
	loss in the \(B\)-components: determine their real points, intersections,
	and connected parameter regions on which strong normality holds.
	\item Remove the near-diagonal restriction on the row multi-index and
	determine parameter conditions under which the terminating hypergeometric
	and Rodrigues-type formulas extend to arbitrary balanced index pairs.
	\item Evaluate the upper bidiagonal factors for \(q>1\) without unevaluated
	moment determinants, and determine their zeros, poles, and exceptional
	parameter sets.
	\item Establish the Jacobi-to-Bessel confluence directly for the polynomial
	vectors, recurrence coefficients, and bidiagonal factors, rather than only
	for the matrices of weights and their moments.
	\item Study the large-degree behavior of the polynomial vectors, their
	zeros, and the recurrence coefficients, and develop the corresponding
	spectral theory for the full-rank banded recurrence operator.
\end{enumerate}

\section*{Acknowledgements}
The author thanks Am\'ilcar Branquinho and Ana Foulqui\'e-Moreno for
valuable discussions on mixed-type orthogonality and Bessel polynomials.

\section*{Declarations}

\noindent\textbf{Funding.}
The author was supported by the research project PID2024-155133NB-I00,
\emph{Ortogonalidad, aproximaci\'on e integrabilidad: aplicaciones en procesos
estoc\'asticos cl\'asicos y cu\'anticos}.

\medskip
\noindent\textbf{Competing interests.}
The author declares no competing interests.

\medskip
\noindent\textbf{Data availability.}
No datasets were generated or analyzed during the current study.

\printbibliography

\end{document}